\documentclass[11pt,english,notitlepage]{report}
\PassOptionsToPackage{refpage}{nomencl}
\usepackage[T1]{fontenc}
\usepackage{textcomp}
\usepackage[utf8]{inputenc}
\usepackage{parskip}
\usepackage{color}
\usepackage{babel}
\usepackage{prettyref}
\usepackage{mathrsfs}
\usepackage{mathtools}
\usepackage{enumitem}
\usepackage{amsmath}
\usepackage{amsthm}
\usepackage{amssymb}
\usepackage{setspace}
\usepackage[all]{xy}
\usepackage{nomencl}
\makenomenclature
\usepackage{microtype}
\usepackage[bookmarks=true,bookmarksnumbered=true,bookmarksopen=false,
 breaklinks=false,pdfborder={0 0 0},pdfborderstyle={},backref=false,colorlinks=true]
 {hyperref}
\hypersetup{
 linkcolor=blue,  urlcolor=blue, citecolor=blue, psdextra=true}

\makeatletter

\let\pr@chap=\pr@cha
\newcommand{\noun}[1]{\textsc{#1}}

\theoremstyle{plain}
\newtheorem{thm}{\protect\theoremname}[section]
\theoremstyle{plain}
\newtheorem{cor}[thm]{\protect\corollaryname}
\theoremstyle{plain}
\newtheorem{prop}[thm]{\protect\propositionname}
\theoremstyle{definition}
\newtheorem{defn}[thm]{\protect\definitionname}
\theoremstyle{plain}
\newtheorem{lem}[thm]{\protect\lemmaname}
\theoremstyle{definition}
\newtheorem{example}[thm]{\protect\examplename}
\theoremstyle{remark}
\newtheorem{rem}[thm]{\protect\remarkname}
\newcommand\thmsname{\protect\theoremname}
\newcommand\nm@thmtype{theorem}
\theoremstyle{plain}

\newenvironment{namedthm}[1][Undefined Theorem Name]{
  \ifx{#1}{Undefined Theorem Name}\renewcommand\nm@thmtype{theorem*}
  \else\renewcommand\thmsname{#1}\renewcommand\nm@thmtype{namedtheorem}
  \fi
  \begin{\nm@thmtype}}
  {\end{\nm@thmtype}}
\providecommand*{\code}[1]{\texttt{#1}}

\@ifundefined{date}{}{\date{}}
\usepackage[scale=0.72]{geometry} %
\usepackage[msc-links,nobysame,abbrev]{amsrefs} %

\makeatletter
\@ifpackageloaded{latexml}{}{
\usepackage[theoremfont]{stickstootext} %
\usepackage{sourcesanspro} %
\usepackage{sourcecodepro}
\usepackage[stix2,varbb,uprightscript]{newtxmath} %
\usepackage[cal=stixtwoplain]{mathalpha} %
}
\makeatother

\newrefformat{thm}{Theorem \ref{#1}}
\newrefformat{cor}{Corollary \ref{#1}}
\newrefformat{lem}{Lemma \ref{#1}}
\newrefformat{prop}{Proposition \ref{#1}}
\newrefformat{conjecture}{Conjecture \ref{#1}}
\newrefformat{fact}{Fact \ref{#1}}
\newrefformat{defn}{Definition \ref{#1}}
\newrefformat{def}{Definition \ref{#1}}
\newrefformat{example}{Example \ref{#1}}
\newrefformat{exa}{Example \ref{#1}}
\newrefformat{problem}{Problem \ref{#1}}
\newrefformat{prob}{Problem \ref{#1}}
\newrefformat{xca}{Exercise \ref{#1}}
\newrefformat{sol}{Solution \ref{#1}}
\newrefformat{rem}{Remark \ref{#1}}
\newrefformat{claim}{Claim \ref{#1}}
\newrefformat{problem}{Problem \ref{#1}}
\newrefformat{xca}{Exercise \ref{#1}}
\newrefformat{exer}{Exercise \ref{#1}}
\newrefformat{sol}{Solution \ref{#1}}

\newrefformat{cha}{Chapter \ref{#1}}
\newrefformat{sec}{Section \ref{#1}}
\newrefformat{subsec}{Subsection \ref{#1}}

\newrefformat{namedthm}{\nameref{#1}}

\usepackage{etoolbox}
\makeatletter
\@ifpackageloaded{latexml}{}{
\newcommand{\plainqed}[1][\blacklozenge]{%
  \leavevmode\unskip\penalty9999 \hbox{}\nobreak\hfill
  \quad\hbox{\ensuremath{#1}}}
\newcommand{\definitionqed}[1][\lozenge]{%
  \leavevmode\unskip\penalty9999 \hbox{}\nobreak\hfill
  \quad\hbox{\ensuremath{#1}}}
\newcommand{\remarkqed}[1][\lozenge]{%
  \leavevmode\unskip\penalty9999 \hbox{}\nobreak\hfill
  \quad\hbox{\ensuremath{#1}}}
\AtEndEnvironment{thm}{\plainqed}
\AtEndEnvironment{thm*}{\plainqed}
\AtEndEnvironment{cor}{\plainqed}
\AtEndEnvironment{cor*}{\plainqed}
\AtEndEnvironment{lem}{\plainqed}
\AtEndEnvironment{lem*}{\plainqed}
\AtEndEnvironment{prop}{\plainqed}
\AtEndEnvironment{prop*}{\plainqed}
\AtEndEnvironment{conjecture}{\plainqed}
\AtEndEnvironment{conjecture*}{\plainqed}
\AtEndEnvironment{fact}{\plainqed}
\AtEndEnvironment{fact*}{\plainqed}
\AtEndEnvironment{defn}{\definitionqed}
\AtEndEnvironment{defn*}{\definitionqed}
\AtEndEnvironment{example}{\definitionqed}
\AtEndEnvironment{example*}{\definitionqed}
\AtEndEnvironment{problem}{\definitionqed}
\AtEndEnvironment{problem*}{\definitionqed}
\AtEndEnvironment{xca}{\definitionqed}
\AtEndEnvironment{xca*}{\definitionqed}
\AtEndEnvironment{sol}{\definitionqed}
\AtEndEnvironment{sol*}{\definitionqed}
\AtEndEnvironment{rem}{\remarkqed}
\AtEndEnvironment{rem*}{\remarkqed}
\AtEndEnvironment{claim}{\remarkqed}
\AtEndEnvironment{claim*}{\remarkqed}
}
\makeatother

\usepackage{mleftright}
\mleftright

\renewcommand\[{\begin{equation}}
\renewcommand\]{\end{equation}}

\usepackage{fancyhdr}
\fancyheadinit{\sffamily\footnotesize}%

\makeatletter
\@ifpackageloaded{latexml}{}{
\renewcommand*{\tableofcontents}{\@starttoc{toc}}
}
\makeatother

\makeatother

\providecommand{\corollaryname}{Corollary}
\providecommand{\definitionname}{Definition}
\providecommand{\examplename}{Example}
\providecommand{\lemmaname}{Lemma}
\providecommand{\propositionname}{Proposition}
\providecommand{\remarkname}{Remark}
\providecommand{\theoremname}{Theorem}

\begin{document}
\global\long\def\Identity{\operatorname{id}}%

\global\long\def\image{\operatorname{im}}%

\global\long\def\Isometry{\operatorname{Isom}}%

\global\long\def\eval{\operatorname{eval}}%

\global\long\def\projection{\operatorname{pr}}%

\global\long\def\MCG{\operatorname{MCG}}%

\global\long\def\coker{\operatorname{coker}}%

\global\long\def\subsetopen{\ensuremath{\mathrel{\subseteq\hspace*{-0.733em}\raisebox{0.133em}{\ensuremath{{\scriptstyle \circ}}}}}}%

\global\long\def\supsetopen{\ensuremath{\mathrel{\supseteq\hspace*{-0.866em}\raisebox{0.133em}{\ensuremath{{\scriptstyle \circ}}}}}}%

\makeatletter
\@ifpackageloaded{latexml}{
\global\long\def\subsetopen{\mathrel{\mathring{\subseteq}}}%
\global\long\def\supsetopen{\mathrel{\mathring{\supseteq}}}%
}{}
\makeatother

\global\long\def\ZPrim#1{\mathbb{Z}_{\mathrm{prim}}^{#1}}%

\global\long\def\AffineGroupZ{\mathrm{Aff}\left(1,\mathbb{Z}\right)}%

\global\long\def\gradient{\operatorname{grad}}%

\global\long\def\StandardFibration{\xi}%

\global\long\def\Exp{\operatorname{\mathcal{E}}}%

\global\long\def\tildetimes{\mathbin{\widetilde{\times}}}%

\global\long\def\Disk#1{D^{#1}}%

\global\long\def\Sphere#1{S^{#1}}%

\global\long\def\Torus#1{T^{#1}}%

\global\long\def\Lens#1#2{L\left(#1,#2\right)}%

\global\long\def\RP#1{\mathbb{R}P^{#1}}%

\global\long\def\CP#1{\mathbb{C}P^{#1}}%

\global\long\def\MappingTorus#1#2{\mathrm{MT}_{#2}\left(#1\right)}%

\global\long\def\GL#1#2{\mathrm{GL}\left(#1,\mathbb{#2}\right)}%

\global\long\def\SL#1#2{\mathrm{SL}\left(#1,\mathbb{#2}\right)}%

\global\long\def\Orthogonal#1{\mathrm{O}\left(#1\right)}%

\global\long\def\SOrthogonal#1{\mathrm{SO}\left(#1\right)}%

\global\long\def\Unitary#1{\mathrm{U}\left(#1\right)}%

\global\long\def\SUnitary#1{\mathrm{SU}\left(#1\right)}%

\global\long\def\trace{\operatorname{tr}}%

\global\long\def\Center{\operatorname{Center}}%

\global\long\def\Nil{\mathrm{Nil}}%

\global\long\def\Sol{\mathrm{Sol}}%

\global\long\def\SLtilde{\widetilde{\mathrm{SL}_{2}}}%

\global\long\def\CircleGroup{\mathbb{T}}%

\global\long\def\Conformal{\operatorname{Conf}}%

\global\long\def\Affine{\operatorname{Aff}}%

\global\long\def\Heat#1{\operatorname{\mathcal{H}eat}^{#1}}%

\global\long\def\PBundle{\operatorname{\mathsf{PBun}}}%

\global\long\def\Bundle{\operatorname{\mathsf{Bun}}}%

\global\long\def\ClassFib{\operatorname{\mathsf{Fib}}}%

\global\long\def\ClassMan{\operatorname{\mathsf{Man}}}%

\global\long\def\Classifying{\operatorname{B\!}}%

\global\long\def\Homeo{\operatorname{\mathscr{H}omeo}}%

\global\long\def\Diff{\operatorname{\mathscr{D}iff}}%

\global\long\def\Emb{\operatorname{\mathscr{E}mb}}%

\global\long\def\Subm{\operatorname{\mathscr{S}ubm}}%

\global\long\def\Immersion{\operatorname{\mathscr{I}mm}}%

\global\long\def\Aut{\operatorname{\mathscr{A}ut}}%

\global\long\def\Vau{\operatorname{\mathscr{V}au}}%

\global\long\def\Gauge{\operatorname{\mathscr{G}au}}%

\global\long\def\Shape{\operatorname{\mathscr{S}hap}}%

\global\long\def\Fib{\operatorname{\mathscr{F}ib}}%

\global\long\def\MappingSpace{\operatorname{\mathscr{C}}}%

\global\long\def\Tautological{\operatorname{\mathscr{T}aut}}%

\global\long\def\LieExp#1{\operatorname{\exp}_{#1}^{\mathsf{Lie}}}%

\global\long\def\HTransport{\tau^{\mathrm{h}}}%

\global\long\def\xtwoheadrightarrow#1#2{\xrightarrow[#1]{#2}\mathrel{\mkern-14mu  }\rightarrow}%

\global\long\def\StandardProjection{\pi}%

\global\long\def\AutProjection{\Pi}%

\global\long\def\DiffProjection{Q}%

\global\long\def\smallblacktriangleright{\blacktriangleright}%

\newgeometry{tmargin=.5in,bmargin=1in}
\singlespacing
\title{On Topology of the Infinite-Dimensional Space of Fibrations}
\author{Ziqi Fang}

\maketitle
\thispagestyle{empty}
\begin{abstract}
\noindent This work serves as an opening and basis of an ongoing
program investigating topological and geometric aspects of the moduli
space of smooth fibering structures on a manifold. The present paper
focuses on the algebraic and differential topology of this space,
and particularly addresses the following three quests in a top-down
manner: the classification, for each class the path components, and
for each component the homotopy type (as loosely analogous to the
those three for studying the moduli of smooth structures: exotic manifolds,
mapping class groups, and Smale-type conjectures). The last of the
three is infinite-dimensional in nature, as the corresponding moduli
space is shown to inherit the structure of a smooth Fréchet manifold
from the diffeomorphism group through a (infinite-dimensional) principal
bundle, with which we establish further connections with the Lie theory
of gauge symmetries from one perspective, and with the geometric analysis
of extrinsic flows from another. Concretely, we tackle the problem
of finding the ``homotopy core'': a minimal deformation retract
that encodes the topological structure of such a moduli space of fibrations;
as our first examples, we gave explicit homotopy calculations for
various low-dimensional cases which, combined with earlier known cases
from others' work, complete the solution to this problem for dimensions
up to three.
\end{abstract}
\tableofcontents{}

\onehalfspacing

\restoregeometry

\newpage{}

\chapter*{Introduction\label{chap:Introduction}}

\addcontentsline{toc}{chapter}{Introduction}

\section*{Preface to the General Program}

\paragraph*{The main object of study}

This research program focuses primarily on the topological and geometric
aspects of the \emph{(moduli) space of fiberings}. Intuitively, this
space is an infinite-dimensional geometric entity that parametrizes
the ways to partition a given manifold into unlabeled fibers in a
certain coherent manner—whether they be smooth fiber bundles or other
fibration-like structures. More generally, moduli spaces of various
related differential-geometric structures on manifolds are among the
objects of interest.

\paragraph*{Main questions under consideration}

We found that this program can be effectively guided by a particular
question ``what is the shape of the space of fiberings?'' or, more
precisely, the quest for its homotopy type. Our view is through three
lenses: categorial, discrete, and continuum, as corresponding to the
following three questions.
\begin{itemize}[noitemsep]
\item How to classify these fiberings up to fibering equivalence?
\item How to count the path components in the moduli space of fiberings
for each class?
\item How to model the homotopy type of each component of this moduli space?
\end{itemize}
To draw analogy, this can be loosely compared to the study on the
space of smooth structures, where the analogous three lenses are:
the classification of exotic smooth manifolds, the determination of
mapping class groups, and the homotopy modeling of the identity-isotopic
diffeomorphism groups—the third and last of which is exemplified by
various generalizations of the Smale conjecture, which shares a common
infinite-dimensional theme with our study where the (infinite-dimensional)
topology and geometry naturally emerges.

\paragraph*{Sample theorems}

This research yields precise information on low-dimensional manifolds:
we model the homotopy type of the space of fiberings on a certain
minimal deformation retract (its ``core''), where privileged symmetries
on the ambient manifold manifest themselves. To give a taste, here
are two visualizable theorems sampled from the current paper. One
theorem shows that the space of oriented circle fiberings on the 2-torus
deformation retracts to an affine model consisting of rational linear
fiberings, and thereby has the homotopy type of a discrete space of
coprime pairs. The other theorem shows that the space of oriented
circle fiberings on the real projective 3-space deformation retracts
to an isometric model consisting of Hopf fiberings, and thereby has
the homotopy type of a pair of disjoint 2-spheres. Each of these moduli
spaces is endowed with a natural (infinite-dimensional) smooth manifold
structure, which further allows us to deduce its topological type
as well as to perform differential geometry on it.

\paragraph*{Relation to Existing work}

The intended general program emerges from the author's doctoral dissertation\nocite{me2024thesis},
revised and expanded in the current manuscript as an opening project
of the program. In turn, it initially grew out of a collaborative
study with \noun{DeTurck}, \noun{Gluck}, \noun{Lichtenfelz}, \noun{Merling},
\noun{Wang}, and \noun{Yang} \cite{deturck2025homotopytypespacefibrations}
determining the homotopy type of the space of fiberings on the 3-sphere,
with respect to which the present work is intended to formalize a
general theory out of the special instances. Existing works by other
authors in various different aspects have been consulted. This includes,
in most direct relevance, work of \noun{Hong}, \noun{Kalliongis},
\noun{McCullough}, \noun{Rubinstein}, and \noun{Soma} \cite{MR2976322,MR3024309}
proving the component-contractibility of the space of Seifert fiberings
for a large share of three-dimensional cases (e.g., for those compact
orientable Seifert 3-manifolds that are either Haken or over a hyperbolic
base), where a robust framework was developed to connect the underlying
topologies between the spaces of Seifert fiberings and the groups
of diffeomorphisms—this framework is taken as one of the bases that
the present work is built upon. In particular, this enables access
to extensive studies regarding homotopy types of diffeomorphism groups
such as, in low dimensions, the Smale-type conjectures as proved by
work of \noun{Smale}, \noun{Earle}, \noun{Eells}, \noun{Hatcher},
\noun{Cerf}, \noun{Ivanov}, \noun{Gabai}, \noun{Hong}, \noun{Kalliongis},
\noun{McCullough}, \noun{Rubinstein}, \noun{Soma}, \noun{Bamler},
and \noun{Kleiner} \cite{MR112149,MR276999,MR229250,MR448370,MR661467,MR624946,MR701256,MR1895350,MR2976322,MR3024309,bamler2019ricci,MR4536904,MR4685084},
relying on which we set out to tailor our theory to complete the answer
of our three main questions for all cases in dimensions up to three
(recovering the above contractible cases and settling the remaining
noncontractible cases), firstly for regular fiberings (as accomplished
in the present paper) and then for Seifert fiberings in general (as
undertaken in the ensuing papers). On the other hand, our program
is intended to advance beyond the plain topological category and contextualize
the homotopy types of our moduli spaces. In particular, one of the
overarching objectives in this program is to perform infinite-dimensional
differential geometry on our moduli space (in the same spirit as the
line of research in geometries on diffeomorphisms by other authors,
as traced back to, e.g., \noun{Arnold} \cite{MR202082}, \noun{Ebin}
and \noun{Marsden} \cite{MR0271984}). To this end, the present work
devotes space to setting up the scene in the smooth Fréchet category
(e.g., endowing our moduli space with a homogeneous smooth structure
of the Lie group of total diffeomorphisms), and more generally introducing
tools on infinite-dimensional manifolds and Lie groups from work of
\noun{Palais}, \noun{Hamilton}, \noun{Milnor}, \noun{Michor},
\noun{Neeb}, \noun{Glöckner}, and \noun{Wockel}, and \noun{Schmeding}
\cite{MR189028,MR0583436,MR0656198,MR830252,glockner2016fundamentalssubmersionsimmersionsinfinitedimensional,MR1935553,MR2261066,MR2353707,MR2743767,MR4376531}
(also the textbook account by \noun{Schmeding} \cite{MR4505843},
as well as monographs by \noun{Kriegl \& Michor} \cite{MR1471480},
\noun{Khesin \& Wendt} \cite{MR2456522}, and (forthcoming) \noun{Glöckner \& Neeb}
\cite{GlocknerNeeb202x}). To further deepen the connection between
the infinite-dimensional topology on our moduli space and the finite-dimensional
geometry on individual fibrations themselves, we bring in the dynamical
perspective by promoting each homotopy type to a concrete minimal
deformation retract (the ``core''), and thereby facilitates our
ensuing study on classical realizations in terms of extrinsic geometric
flows; as a first example, we realize the homotopy type of the space
of fiberings on the 2-torus by applying the curve-shortening flow
from work of \noun{Gage}, \noun{Hamilton}, and \noun{Grayson} \cite{MR742856,MR840401,MR906392,MR979601,MR1046497}.
Note that in turn, this example can be reversely fed into our machinery
to recover the aforementioned theorem of \noun{Earle} and \noun{Eells}
\cite{MR276999} on the homotopy type of the diffeomorphism group
of the 2-torus; this motivates us to complement our topological approach
orthogonally with a constructive, geometric-analytic approach that
seeks to deform our fibering spaces with extrinsic geometric flows,
as can be related to recent work of \noun{Bamler} and \noun{Kleiner}
\cite{MR4536904} that sought to deform their diffeomorphism groups
with Ricci flows. Last but not least, throughout there is an implicit
but deep influence on this research by the seminal work of \noun{Gluck}
and \noun{Warner} \cite{MR0700132} on the space of great-circle
fiberings on the 3-sphere, which I perceive as a \emph{soul} for
our general study—in both the metaphorical and the mathematical sense.\footnote{It is the classical \emph{soul theorem} \cite{MR309010} that I meant
to make reference to in this pun, where the intended assertion is
that the space $\Fib$ of circle fiberings deformation retracts to
the subspace $\Fib_{\mathrm{g}}$ of great-circle fiberings as its
\emph{soul}, in the sense that $\Fib_{\mathrm{g}}$ is a special
closed embedded submanifold—satisfying the totally-geodesic and totally-convex
conditions formulated suitably—such that its normal bundle $N\Fib_{\mathrm{g}}$
is diffeomorphic to full space $\Fib$, as supported by our results
that both $\Fib_{\mathrm{g}}$ and $\Fib$ deformation retract to
the same core (i.e., a pair of projective planes consisting of the
Hopf fibrations).}

\section*{Overview of the present study}

\paragraph*{Basic structures on the space of fiberings}

For simplicity of exposition, let us focus on regular fiberings to
start with, meaning those modeled on (compact) smooth fibrations without
singularity. Given such a model fibration $\StandardFibration\colon F\hookrightarrow E\to B$,
we topologize its equivalence class into a desired moduli space. One
approach is to endow this class with a homogeneous structure by viewing
it as the orbit of $\StandardFibration$ under the ambient transformations
``along for the ride''. More precisely, such ambient transformations
on $E$ are supplied by the diffeomorphism group $\Diff\left(E\right)$
equipped with the $C^{\infty}$-topology, while those that stabilize
$\StandardFibration$ form a subgroup $\Aut\left(\StandardFibration\right)$
called the \emph{automorphism group} of $\StandardFibration$; the
resulting coset space, denoted by $\Fib\left(\StandardFibration\right)$,
thus gives a formulation of the \emph{(moduli) space of fiberings}
modeled on $\StandardFibration$:
\[
\Fib\left(\StandardFibration\right)=\Diff\left(E\right)/\Aut\left(\StandardFibration\right).
\]
These spaces admit structures of infinite-dimensional topological
manifolds modeled on separable Fréchet spaces, which allow us to promote
their weak homotopy types into strong ones, and to further deduce
their topological types. Various desirable strengthening of these
structural results is being further investigated in the present work,
particularly including: i) promotion of these homotopy types to \emph{strong deformation retracts},
and ii) promotion of these topological structures to \emph{(infinite-dimensional) smooth structures}.
This leads to differentio-geometric considerations on these infinite-dimensional
transformation groups and moduli spaces, whose interplay with the
geometry on the pertinent compact manifolds (i.e., the ambient $E$,
the fiber $F$, and the base $B$) makes one of the main themes in
this research.

\paragraph*{Interplay with diffeomorphism groups}

A recurring theme in our study of the space of fiberings is to establish
its connections with the \emph{diffeomorphism groups} of the ambient
$E$, the fiber $F$, and the base $B$. To illustrate, let me give
a terse synopsis that samples some of such ideas. We start with a
complete classification of smooth fiberings, which we achieve by approximations
to topological classifying objects such as the homotopy set $\left[B,\mathrm{BDiff}\left(F\right)\right]$
or the Čech cohomology set $\check{H}^{1}\left(B,\Diff\left(F\right)\right)$,
of which we take the orbit space under the natural action of the mapping
class group $\pi_{0}\Diff\left(B\right)$. Each orbit will be the
underlying set of a moduli space $\Fib\left(\StandardFibration\right)$
as defined above, which we then study by means of successively fibrating
the ambient transformation group $\Diff\left(E\right)$, starting
with the following two (infinite-dimensional) fibrations:
\[
\Aut\left(\StandardFibration\right)\hookrightarrow\Diff\left(E\right)\to\Fib\left(\StandardFibration\right),\qquad\Vau\left(\StandardFibration\right)\hookrightarrow\Aut\left(\StandardFibration\right)\to\Diff\left(B\right)_{\StandardFibration}.
\]
Such fibrations allow us to pass topological information back and
forth among these spaces. Here,
\begin{itemize}
\item $\Vau\left(\StandardFibration\right)$, the \emph{vertical automorphism group}
of $\StandardFibration$, is a closed normal subgroup of $\Aut\left(\StandardFibration\right)$
consisting of those that preserve individual fibers. We view it in
two ways: as a Fréchet Lie group modeled on the Lie algebra of vertical
vector fields for $\StandardFibration$, and as an isotropy subgroup
in the Čech cochain group $\check{C}\left(\mathcal{U},\Diff\left(F\right)\right)$
that stabilizes the structural cocycle of $\StandardFibration$.
\item $\Diff\left(B\right)_{\StandardFibration}$, the \emph{basic transformation group}
of $\StandardFibration$, is an open subgroup of $\Diff\left(B\right)$
consisting of those that lift to automorphisms. We view it in two
ways: as a Fréchet Lie group modeled on the Lie algebra of basic vector
fields for $\StandardFibration$, and as an isotropy subgroup in the
mapping class group $\pi_{0}\Diff\left(B\right)$ that stabilizes
the classifying class of $\StandardFibration$.
\end{itemize}
This demonstrates a chain of interrelated objects central to our study,
all of which in turn have tight connections with various diffeomorphism
groups. For manifolds of dimensions up to three, homotopy types of
diffeomorphism groups have been extensively studied and understood
due to work of many mathematicians as mentioned above, which thereby
provides a solid ground for our low-dimensional investigations. In
what follows we shall sample some archetypal cases illustrating a
variety of ideas, and toward the end we shall also see how the space
of fiberings can be studied in its own right hence, conversely, shedding
light on the diffeomorphism groups.

\paragraph*{Spaces of regular fiberings on surfaces}

In dimension 2, there is essentially only one regular case: The only
closed surface that admits oriented circle fiberings is the 2-torus
$\Torus 2$, on which all such fiberings form a single equivalence
class, as represented by the standard coordinate projection $\StandardFibration_{0}\colon\Torus 2\to\Sphere 1$.
The corresponding moduli space of fiberings, $\Fib\left(\StandardFibration_{0}\right)$,
is proven in this work to deformation retract onto a ``core'' consisting
of \emph{rational linear fiberings} (those affinely equivalent to
$\StandardFibration_{0}$), and hence has the homotopy type of a discrete
space of coprime pairs (by telling the ``slope''):
\[
\Fib\left(\StandardFibration_{0}\colon\Sphere 1\hookrightarrow\Torus 2\to\Sphere 1\right)\simeq\left\{ \left(a,b\right)\in\mathbb{Z}^{2}\mid\gcd\left(a,b\right)=1\right\} .
\]
This is contrasted with the complementary space of irrational linear
foliations on torus, which naturally leads us to study spaces of fiberings
modeled on foliations. Yet for a different purpose, we shall revisit
this example later when discussing dynamics under extrinsic geometric
flows. A minor variation allowing more fibering classes to emerge
is by waiving the orientation; for example, the twisted product of
two circles gives an unoriented fibering on the Klein bottle, for
which the corresponding space of fiberings is proven in my work%
to be contractible. This is only the beginning of a rich study about
fibering a surface, which we shall discuss more below.

\paragraph*{Spaces of regular fiberings on 3-manifolds}

In dimension 3, our machinery in principle also suffices to complete
a full picture for all kinds of 3-manifolds, which can be considered
case by case based on Thurston's classification of 3-dimensional geometries.
For example, consider the case of elliptic geometry: The only closed
elliptic $3$-manifolds that admit oriented circle fiberings are the
lens spaces $L\left(e,1\right)$ for $e>0$, on each of which all
such fiberings form a single equivalence class, as represented by
the standard Hopf map $\StandardFibration_{e}\colon L\left(e,1\right)\to\Sphere 2$
(descended from the 3-sphere). The corresponding moduli space of fiberings,
$\Fib\left(\StandardFibration_{e}\right)$, is non-contractible only
if $e=1$ or $e=2$, in which case it is proven in this work to deformation
retract onto a ``core'' consisting of \emph{Hopf fiberings} (those
metrically congruent to $\StandardFibration_{e}$), and hence has
the homotopy type of a pair of disjoint 2-spheres (by telling the
``direction'' and ``chirality''):
\[
\Fib\left(\StandardFibration_{e}\colon\Sphere 1\hookrightarrow L\left(e,1\right)\to\Sphere 2\right)\simeq\begin{cases}
\Sphere 2\sqcup\Sphere 2 & \text{if \ensuremath{e=1,2},}\\
\Sphere 0 & \text{if \ensuremath{e\geq3}.}
\end{cases}
\]
Here, the particular case $e=1$ concerning $\Sphere 3$ was due to
joint work with \noun{DeTurck}, \noun{Gluck}, \noun{Lichtenfelz},
\noun{Merling}, \noun{Wang}, and \noun{Yang} \cite{deturck2025homotopytypespacefibrations}.
On the other hand, there is also the case ``$e=0$'' concerning
$\Sphere 2\times\Sphere 1$ (with $\Sphere 2\times\mathbb{R}$ geometry),
for which the corresponding moduli space is expected to admit no finite-dimensional
deformation retract (as first observed by \cite{wang2024homotopytypespacefiberings}).
In sharp contrast, we emphasize that by work of \noun{Hong}, \noun{Kalliongis},
\noun{McCullough}, \noun{Rubinstein}, and \noun{Soma} \cite{MR2976322,MR3024309},
the space of Seifert fiberings for a large share of 3-manifolds is
contractible or at least has contractible components. As a simplest
visualizable example illustrating this, the space of (regular) circle
fiberings on the solid torus has the homotopy type of $\mathbb{Z}$,
as modeled by twisting the trivial fibering along the meridian disk
by multiple full turns.

\paragraph*{Spaces of surface fiberings and tight contact structures}

On 3-manifolds we can also consider the space of \emph{surface fiberings}.
Here we highlight the case of flat geometry which features a simplest
instance of \emph{duality}: The only closed orientable flat $3$-manifold
that admits oriented circle fiberings is the 3-torus $\Torus 3$,
on which all such fiberings form a single equivalence class — and
so do all the oriented torus fiberings on $\Torus 3$ — as represented
by the complementary coordinate projections $\StandardFibration_{0}\colon\Torus 3\to\Torus 2$
and $\StandardFibration_{0}^{\vee}\colon\Torus 3\to\Sphere 1$, respectively.
The corresponding pair of moduli spaces, $\Fib\left(\StandardFibration_{0}\right)$
and $\Fib\left(\StandardFibration_{0}^{\vee}\right)$, are proven
in this work to concurrently deformation retract onto the rational
linear ones, and hence both have the homotopy type of a discrete space
of coprime triples:
\[
\Fib\left(\StandardFibration_{0}\colon\Sphere 1\hookrightarrow\Torus 3\to\Torus 2\right)\simeq\left\{ \left(a,b,c\right)\in\mathbb{Z}^{3}\mid\gcd\left(a,b,c\right)=1\right\} \simeq\Fib\left(\StandardFibration_{0}^{\vee}\colon\Torus 2\hookrightarrow\Torus 3\to\Sphere 1\right).
\]
Spaces of coexistent circle and torus fiberings are also studied in
$\mathrm{Nil}$-geometry, where the role of $\Torus 3$ is taken by
(a positive integers' worth of) mapping tori of $\Torus 2$ with ``reducible''
monodromy; yet, for an adaptation of the above duality, we are led
to study the space of \emph{(geometric) tight contact structures}.
The moduli spaces of surface fiberings and of tight contact structures
on 3-manifolds in general are being further investigated in this ongoing
program.

\paragraph*{Spaces of Seifert fibrations and singular foliations}

One of the natural generalizations of regular circle fibrations of
3-manifolds are \emph{Seifert fibrations}, for which we have already
mentioned results of \cite{MR2976322,MR3024309} providing large coverage
of the contractibility of such spaces of fiberings. It is one of the
goals in the present program to fully complete such a landscape, determining
homotopy types of those spaces of Seifert fiberings that were not
known to be contractible, and identifying their ``cores'' (minimal
deformation retracts). On the other hand, there is a different kind
of singular fiberings considered in my work:%
those modeled on \emph{singular foliations} with exceptionally lower-dimensional
singular fibers. This is illustrated on the 2-sphere $\Sphere 2$
(which was known to admit no regular fibrations), where marking two
singular points allows certain singular fiberings to emerge. Specifically,
a pair of prototypical examples $\StandardFibration_{\mathrm{long}}$
and $\StandardFibration_{\mathrm{lat}}$ are given by the two spherical-coordinate
projections, with (oriented) fibers being the longitudes and latitudes,
respectively. It is an ongoing project of mine to develop smoothing
techniques to streamline the handling of such singularities, conditioned
on which we have the following expected theorem: the corresponding
pair of moduli spaces, $\Fib\left(\StandardFibration_{\mathrm{long}}\right)$
and $\Fib\left(\StandardFibration_{\mathrm{lat}}\right)$, concurrently
deformation retract onto those congruent to the standard ones, and
hence both have the homotopy type of the 2-sphere:
\[
\Fib\left(\StandardFibration_{\mathrm{long}}\colon\Disk 1\hookrightarrow\Sphere 2\rightsquigarrow\Sphere 1\right)\simeq\Sphere 2\simeq\Fib\left(\StandardFibration_{\mathrm{lat}}\colon\Sphere 1\hookrightarrow\Sphere 2\rightsquigarrow\Disk 1\right).
\]
This admits a natural generalization to the $n$-sphere $\Sphere n$
with a great-sphere pole set $\Sphere{k-1}$ (for $k$ ranging from
$1$ to $n-1$), where the \emph{longitudinal (disk) fibering} consists
of a family of $\Disk k$-fibers bounded by the pole set, and transversely
the \emph{latitudinal (sphere) fibering} consists of a family of
$\Sphere{n-k}$-fibers collapsing into the pole set. Yet another variation
concerns singular circle fiberings of $\Sphere 2$ over a Y-shaped
graph, with three poles and one exceptional $\theta$-shaped fiber
over the endpoints and center of the Y-shape, respectively — this
can be viewed as the base case of certain \emph{pants-decomposition fiberings}
on $\CP n$ by $\Torus n$-fibers. All these singular fiberings arise
naturally in topology and geometry, whose corresponding moduli spaces
are among the objects of study in this ongoing program.%

\paragraph*{Fiberings as families of shapes}

We remark on an alternative view on the space of fiberings: a fibering
can be thought as some $B$-family of unlabelled $F$-submanifolds
embedded in $E$, filling up and fitting together nicely. This realizes
$\Fib\left(\StandardFibration\right)$ as a certain subspace sitting
inside some manifold of mappings into a (infinite-dimensional) \emph{shape space},
and a crux is to find suitable characterizations for this subspace
(which family of fiber shapes form a fibration?). This leads us to
the 1983 work of \noun{Gluck} and \noun{Warner}, who proved that
the space of (oriented) \emph{great-circle fiberings} on the 3-sphere
deformation retracts to the subspace of Hopf fiberings. In essence,
they viewed each such fibering as a certain $\Sphere 2$-family of
$\Sphere 1$-shapes in $\Sphere 3$, and by restricting such shapes
to great circles, they got a good grasp of the corresponding shape
space as being identified with $\Sphere 2\times\Sphere 2$, inside
which those $\Sphere 2$-families associated to great-circle fibrations
were neatly characterized as the graphs of distance-decreasing maps
from either $\Sphere 2$ factor to the other, and thereby a desired
deformation retraction manifested itself. Carrying forward this beautiful
idea, an ongoing project of mine is to seek its higher-dimensional
generalization to various great-sphere fiberings on the one hand and
— with our original theme in mind — to pursue its nonlinear adaptation
to general circle fiberings on the other hand.%

\paragraph*{Deforming the space of fiberings by extrinsic flows}

Let us return to our first example concerning fiberings on the 2-torus
by simple closed curves, for which we have already mentioned a theorem
in the current paper that the corresponding moduli space deformation
retracts onto its core of rational linear fiberings. In view of the
preceding discussion (where a fibering is regarded as some family
of fiber shapes filling up and fitting together nicely), the following
quest from a geometric-analytic perspective is in order: can we find
a flow that achieves such a deformation retraction? More precisely,
we quest for an \emph{extrinsic geometric flow} for simple closed
curves in the torus that induces — by means of simultaneously evolving
the family of fiber shapes for every member in the moduli space of
fiberings — a deformation retraction of this moduli space onto the
expected core. This was accomplished in the following theorem in the
current paper:%
\begin{quote}
\emph{There is a curve-shortening flow on the torus that induces a strong
deformation retraction of its moduli space of circle fiberings onto
the subspace of rational linear fiberings.}
\end{quote}
This idea is being further investigated in my work: given a model
fibration $\StandardFibration\colon F\hookrightarrow E\to B$, we
quest for — in a suitable geometric setup — an extrinsic geometric
flow for $F$-submanifolds in $E$ that deformation retracts the space
of fiberings $\Fib\left(\StandardFibration\right)$ onto a desired
core. We devised a general criterion, with which the codimension-one
case makes a suitable place to start, where e.g.\ the \emph{mean curvature flow}
provides a primitive candidate for further adaptation. This makes
part of the ongoing program to explore the connection between spaces
of fiberings and extrinsic flows, which has a potential link with
recent work of \noun{Bamler} and \noun{Kleiner} on the connection
between groups of diffeomorphisms and Ricci flows.

\section*{Notations and Assumptions}

The followings are some standing notations and assumptions throughout
this paper. We adopt the following notations for some common equivalence
relations:
\begin{itemize}
\item $\simeq$\nomenclature[01eq]{$\simeq$}{(weak) homotopy equivalence}:
(weak) homotopy equivalence
\item $\approx$\nomenclature[02eq]{$\approx$}{homeomorphism or diffeomorphism}:
homeomorphism or diffeomorphism
\item $\cong$\nomenclature[03eq]{$\cong$}{isomorphism in the appropriate category depending on context}:
isomorphism in the appropriate category depending on context
\item $\leftrightarrow$\nomenclature[03eq]{$\leftrightarrow$}{bijection (for sets or pointed sets)}:
bijection (for sets or pointed sets),
\end{itemize}
as well as the following notations for some common order relations:
\begin{itemize}
\item $\subsetopen$\nomenclature[04open]{$\subsetopen$}{open subset containment}:
open subset
\item $\leq$: topological/Lie subgroup, or Lie subalgebra
\item $\trianglelefteq$: topological/Lie normal subgroup, or Lie ideal.
\end{itemize}
We shall use $F\hookrightarrow E\to B$ to denote a fiber bundle,
where $E,F,B$ are always smooth, closed, connected manifolds (which
will be further equipped with some Riemannian metrics if needed).
The mapping spaces will always be endowed with the following structures:
\begin{itemize}
\item For closed manifold $B$, the topological (resp., smooth) mapping
space $\MappingSpace\left(B,N\right)$ (resp., $\MappingSpace^{\infty}\left(B,N\right)$)
will always be endowed with the compact-open $C^{0}$ (resp., $C^{\infty}$)
topology.\nomenclature[CBN]{$\MappingSpace\left(B,N\right)$}{the space of continuous maps from $B$ to $N$ (or ``topological current group'' if $N$ is a group)}\nomenclature[CBNinf]{$\MappingSpace^{\infty}\left(B,N\right)$}{the space of smooth maps from $B$ to $N$ (or ``smooth current group'' if $N$ is a group)}
In the case when the target space $G$ further has a group structure,
then $\MappingSpace\left(B,G\right)$ and $\MappingSpace^{\infty}\left(B,G\right)$
will be endowed with the group structure induced pointwise from $G$.
\item For closed manifold $M$, the diffeomorphism group $\Diff\left(M\right)$\nomenclature[D]{$\Diff\left(M\right)$}{the group of diffeomorphisms of $M$}
is endowed with the canonical structure of a (metrizable and separable)
Fréchet Lie group, with the subspace topology induced from the $C^{\infty}$
topology on $\MappingSpace^{\infty}\left(M,M\right)$ and the smooth
manifold structure modeled on the space $\mathfrak{X}\left(M\right)$
of smooth vector fields.
\end{itemize}
Moreover, we shall use $\Identity_{X}$ to denote the identity map
of any set $X$, and $\ast$\nomenclature[0ast]{$\ast$}{the one-point space}
to denote the one-point set. Lastly, to help make a distinction between
finite and infinite dimensions, our choice of notation for an infinite-dimensional
manifold (or a related object of infinite-dimensional nature) often
has its first letter typeset in script font, as already seen in e.g.,
the mapping space $\MappingSpace^{\infty}\left(B,N\right)$ and the
diffeomorphism group $\Diff\left(M\right)$.

\chapter{Backgrounds\label{chap:Backgrounds}}

\pagestyle{fancy}

\section{Infinite-Dimensional Topology\label{sec:Infinite-dimensional-topology}}

In this section, we set up some well-behaved classes of topological
spaces $X$, with an eye towards being flexible enough to encompass
many important infinite-dimensional manifolds $X$ arising in topology
and geometry, while also rigid enough to ensure nice homotopy properties
for $X$. This will be used to facilitate various homotopy arguments
in this work, where $X$ will be taken as a variety of mapping spaces
(and the subspaces or quotient spaces thereof) related to a fiber
bundle. We recommend interested readers to \cite{MR189028,MR4179591}
for expositions in much greater depth and generality.

We start by recalling an essential class of topological spaces whose
homotopy properties have received extensive study; namely, the class
of \emph{CW-complex}. This notion was first developed by Whitehead
\cite{MR0030759}, who proved the following theorem that underlies
the significance of CW-complexes in homotopy theory:
\begin{thm}[\noun{Whitehead}]
\label{thm:Whitehead}Every weak homotopy equivalence between CW-complexes
is a homotopy equivalence.
\end{thm}

In spite of the powerful homotopy implication by Whitehead's theorem,
most of the compact manifolds $M$ are (homeomorphic to) finite CW-complexes;
specifically, this is the case if $M$ is smoothable (by smooth triangulation
\cite{MR1563139,MR0002545}, or by handle decomposition via Morse
theory \cite{MR1873233}), or if $\dim M\neq4$ (\cite{MR0242166,MR0645390,MR0679069}).
Moving on to infinite dimensions, a naive attempt is to consider CW-complexes
that are countable rather than finite, but the crux of the problem
is not about how many cells are allowed: most of the interesting infinite-dimensional
manifolds in global analysis are not homeomorphic to CW-complexes
at all (countable or not). This prompts us to enlarge the class to
include, rather than full-on homeomorphism types, the spaces that
have homotopy types of CW-complexes (said to be of \emph{CW-type}
for short). Whitehead's theorem extends to spaces of CW-type, which
thus form a class considered to have nice homotopy. However, a space
of nice homotopy type may well have pathological point-set topology,
and so spaces of CW-type may look nothing like a CW-complex. For
this reason (and for easy verification), we choose to restrict attentions
to the subclass of \emph{absolute neighborhood retract (ANR)},\footnote{Recall that a metrizable space $A$ is called an \emph{absolute neighborhood retract (ANR)}
if $A$ is a retract of some neighborhood of $X$ for every metrizable
space $X$ that contains $A$ as a closed subspace.} whose usefulness is justified by the following theorem is due to
Milnor \cite[Thm.\ 2]{MR0100267} and Palais \cite[Thm.\ 14]{MR189028}
(the separable case was proven by Milnor \cite[Thm.\ 1]{MR0100267},
based on results of Hanner \cite{MR0043459} and Whitehead \cite{MR0035997}):
\begin{thm}[\noun{Milnor–Palais}]
\label{thm:Milnor-Palais}Every ANR (resp., separable ANR) has the
homotopy type of a CW-complex (resp., countable CW-complex).
\end{thm}

As explained above, combining Milnor and Palais's theorem (\prettyref{thm:Milnor-Palais})
with Whitehead's theorem (\prettyref{thm:Whitehead}) yields the following
result:
\begin{cor}
\label{cor:corollary-of-Whitehead-Milnor-Palais}Every weak homotopy
equivalence between ANR's is a homotopy equivalence.
\end{cor}

As another important property of CW-complexes shared by ANR's, recall
that every CW-pair satisfies the homotopy extension property, which
allows a promotion from a weakly-equivalent pair to a deformation
retract (see e.g., \cite[Chapter 0]{MR1867354}). This carries over
to any ANR-pair (which means a pair of ANR's such that one is a closed
subset of the other):
\begin{prop}
\label{lem:deformation-retract-promotion-for-ANR}Suppose that $\left(X,A\right)$
is an ANR pair. Then the inclusion map $A\hookrightarrow X$ satisfies
the homotopy extension property; in particular, $A$ is a deformation
retract of $X$ whenever the inclusion is a homotopy equivalence.
\end{prop}

Having demonstrated some desirable properties of ANR's in homotopy
theory, we next introduce a large class of infinite-dimensional manifolds
(encompassing all the ones of our interest) that turn out to be ANR's
automatically. Recall that a \emph{Hilbert space} is an inner product
space whose associated metric is complete, and that its separability
can be characterized by the existence of a countable orthonormal basis.
Thus up to isomorphism, there is only one infinite-dimensional separable
Hilbert space: the space of square-summable sequences, which is denoted
by $\ell^{2}$.

\begin{defn}
\label{def:l2-manifold}A \emph{topological $\ell^{2}$-manifold}
is a Hausdorff second-countable space covered by open subsets that
are homeomorphic to open subsets in $\ell^{2}$.
\end{defn}

In this definition for a $\ell^{2}$-manifold $\mathcal{N}$, the
Hausdorff assumption is standard, while the second countability is
to ensure two desirable properties of $\mathcal{N}$: one is being
separable, which ensures that $\mathcal{N}$ has the same density
as the model space $\ell^{2}$; the other is being paracompact (\cite[Thm.\ 2]{MR189028}),
which ensures that $\mathcal{N}$ is completely-metrizable as deduced
from the local complete-metrizability of the model space $\ell^{2}$
(\cite[Thm.\ 3]{MR189028}) — such a space, being separable and completely-metrizable,
is said to be a \emph{Polish space}. But it was known that every
metrizable manifold is an ANR (see e.g., \cite[Thm.\ 5]{MR189028}),
thus we have obtained the following:
\begin{prop}
\label{prop:l2-manifold-is-ANR}Every $\ell^{2}$-manifold is a Polish
ANR.
\end{prop}

Thus $\ell^{2}$-manifolds enjoy all the nice properties of ANR by
virtue of \prettyref{prop:l2-manifold-is-ANR}. Note that this proposition
only relies on the model space being Polish (i.e., separable and completely-metrizable),
so one may well consider topological manifolds modeled on general
Polish locally-convex spaces; i.e., \emph{separable Fréchet spaces}.
But our exclusive focus on topological $\ell^{2}$-manifolds is justified
by the following theorem of Kadec and Anderson \cite{MR0201951,MR0209804,MR0190888,MR0205212}:
\begin{thm}[\noun{Kadec–Anderson}]
\label{thm:Kadec-Anderson}Every infinite-dimensional separable Fréchet
spaces is homeomorphic to $\ell^{2}$.
\end{thm}

We have seen in \prettyref{prop:l2-manifold-is-ANR} that every $\ell^{2}$-manifold
is a Polish ANR; the converse is true up to homotopy type, or even
better, up to an $\ell^{2}$ factor. This is the content of a version
of the Torunczyk factor theorem \cite[Theorem 2.2.14]{MR4179591}:
\begin{thm}[\noun{Torunczyk}]
\label{thm:Torunczyk-factor-theorem}If $X$ is a Polish ANR, then
$X\times\ell^{2}$ is an $\ell^{2}$-manifold.
\end{thm}

The Torunczyk factor theorem implies that every homotopy type of countable
CW-complex can be represented by an $\ell^{2}$-manifold (indeed,
by \cite[Theorem 1]{MR0100267}, it can be represented by the underlying
space of a countable locally-finite simplicial complex, which is Polish
ANR, hence $X\times\ell^{2}$ is a desired $\ell^{2}$-manifold representative).
In fact, such $\ell^{2}$-manifold models for a homotopy type are
unique up to homeomorphism, as shown in the following theorem of Henderson
and Schori \cite{MR251749}:
\begin{thm}[\noun{Henderson–Schori}]
\label{thm:Henderson-Schori}Two $\ell^{2}$-manifolds are homeomorphic
whenever they are homotopy equivalent.
\end{thm}

As explained above, the Torunczyk factor theorem (\prettyref{thm:Torunczyk-factor-theorem})
and the Henderson–Schori theorem (\prettyref{thm:Henderson-Schori})
implies the existence and uniqueness, respectively, of the ``$\ell^{2}$-model''
of any countable CW-type:
\begin{cor}
\label{cor:corollary-of-Torunczyk-and-Henderson-Schori}Every homotopy
type of countable CW-complex can be represented by an $\ell^{2}$-manifold
unique up to homeomorphism.
\end{cor}

From a more practical perspective, we also have the following immediate
corollary of the Torunczyk factor theorem (\prettyref{thm:Torunczyk-factor-theorem})
and the Henderson–Schori theorem (\prettyref{thm:Henderson-Schori}),
which allows us to write down the homeomorphism type of an $\ell^{2}$-manifold
$\mathcal{N}$ as long as we can find a suitable, preferably much
simpler, space $X$ of the same homotopy type:
\begin{cor}
\label{cor:corollary-of-Henderson-Schori-Torunczyk}If an $\ell^{2}$-manifold
$\mathcal{N}$ has the homotopy type of a Polish ANR $X$, then $\mathcal{N}$
is homeomorphic to $X\times\ell^{2}$.
\end{cor}

In this paper, the homotopy type of $\mathcal{N}$ will be modeled
on a subspace $N_{0}\subseteq\mathcal{N}$. In this case, the previous
consideration about ANR pairs $\left(\mathcal{N},X\right)$ will also
come into play. All of this culminates in the following theorem that
underlies our approach to understand the topology and homotopy of
$\ell^{2}$-manifolds: 
\begin{thm}
\label{thm:main-theorem-for-infinite-dimensional-topology}Let $\mathcal{N}$
be a topological $\ell^{2}$-manifold (or any separable Fréchet manifold).
Suppose that $N_{0}\subseteq\mathcal{N}$ is a closed subset and an
ANR, whose inclusion into $\mathcal{N}$ is a weak homotopy equivalence.
Then $\mathcal{N}$ is homeomorphic to $N_{0}\times\ell^{2}$, and
$\mathcal{N}$ contains $N_{0}$ as a deformation retract.
\end{thm}

\begin{proof}
By the Kadec–Anderson theorem (\prettyref{thm:Kadec-Anderson}), every
separable Fréchet manifold (i.e., Hausdorff second-countable manifold
modeled on infinite-dimensional separable Fréchet spaces) must be
a topological $\ell^{2}$-manifold, thus we may assume that $\mathcal{N}$
is a topological $\ell^{2}$-manifold. Thus $\mathcal{N}$ is a Polish
ANR by \prettyref{prop:l2-manifold-is-ANR}, hence the subset $N_{0}$
is also a Polish ANR (where being Polish is due to being closed, and
ANR is by assumption). By \prettyref{cor:corollary-of-Whitehead-Milnor-Palais}
(which combines Whitehead's theorem in \prettyref{thm:Whitehead}
with Milnor and Palais's theorem in \prettyref{thm:Milnor-Palais}),
we see that the weak homotopy equivalence $\iota\colon N_{0}\hookrightarrow\mathcal{N}$
can be promoted to a homotopy equivalence. Thus the two desired results
follow:
\begin{itemize}
\item $\mathcal{N}$ is homeomorphic to $N_{0}\times\ell^{2}$: Since we
have shown that $N_{0}$ is a Polish ANR, the desired result follows
from \prettyref{cor:corollary-of-Henderson-Schori-Torunczyk} (which
combines the Henderson–Schori theorem in \prettyref{thm:Henderson-Schori}
and the Torunczyk factor theorem in \prettyref{thm:Torunczyk-factor-theorem}).
\item $\mathcal{N}$ deformation retracts to $N_{0}$: Since we have shown
that $\left(\mathcal{N},N_{0}\right)$ is an ANR pair, the desired
result follows from its homotopy extension property (\prettyref{lem:deformation-retract-promotion-for-ANR}).
\end{itemize}
This completes the proof of \prettyref{thm:main-theorem-for-infinite-dimensional-topology}.
\end{proof}
In spirit, \prettyref{thm:main-theorem-for-infinite-dimensional-topology}
allows us to understand the topology and homotopy of an infinite-dimensional
manifold $\mathcal{N}$ by that of a deformation retract $N_{0}$
— the ``simpler and smaller'' the better. Thus let us call a deformation
retract $N_{0}$ \emph{minimal} if it cannot further deformation
retract to any proper subset, and we shall think of such a minimal
deformation retract $N_{0}$ as a ``\emph{(topological) core}''
of $\mathcal{N}$. A nice important case that guarantees a core is
the following:
\begin{cor}
\label{cor:main-theorem-for-infinite-dimensional-topology-compact-case}Let
$\mathcal{N}$ be a topological $\ell^{2}$-manifold (or any separable
Fréchet manifold). Suppose that $N_{0}\subseteq\mathcal{N}$ is a
manifold with compact components (without boundary), whose inclusion
into $\mathcal{N}$ is a weak homotopy equivalence. Then $\mathcal{N}$
is homeomorphic to $N_{0}\times\ell^{2}$, and $\mathcal{N}$ contains
$N_{0}$ as a minimal deformation retract.
\end{cor}

\begin{proof}
By local connectedness, every manifold is a disjoint union of its
components, so we may assume that both $\mathcal{N}$ and $N_{0}$
are connected. Being a compact manifold, $N_{0}$ is of course a closed
subset and an ANR, so \prettyref{thm:main-theorem-for-infinite-dimensional-topology}
applies. Moreover, again by the assumption that $N_{0}$ is a compact
manifold without boundary, we see that the deformation retract $N_{0}$
is clearly minimal (for otherwise if there were a further deformation
retract $N_{1}\subsetneq N_{0}$, then the top cohomology of $N_{0}\simeq N_{1}$
would isomorphic to the zeroth relative homology with respect to a
nonempty subset $N_{0}-N_{1}$, which is zero — a contradiction).
\end{proof}
In the current work, given a fiber bundle $\StandardFibration\colon F\hookrightarrow E\to B$,
we shall apply the above study to understand $\mathcal{N}$ where
$\mathcal{N}$ is taken as various diffeomorphism groups $\Diff\left(E\right)$,
$\Diff\left(F\right)$, $\Diff\left(B\right)$, the vertical automorphism
group $\Vau\left(\StandardFibration\right)$, the automorphism group
$\Aut\left(\StandardFibration\right)$, and finally our main object:
the moduli space $\Fib\left(\StandardFibration\right)$. By virtue
of \prettyref{thm:main-theorem-for-infinite-dimensional-topology}
and \prettyref{cor:main-theorem-for-infinite-dimensional-topology-compact-case},
this is reduced to the quest for a suitable core $N_{0}$ of $\mathcal{N}$
in each case, for which in turn we shall draw on the Riemannian geometry
of the finite-dimensional fibration $\StandardFibration\colon F\hookrightarrow E\to B$
as well as the differential topology of the infinite-dimensional fibrations
among all these $\mathcal{N}$'s.

\section{Infinite-Dimensional Smooth Manifolds\label{sec:Infinite-dimensional-smooth-manifolds}}

In this section, we review some basics about infinite-dimensional
smooth manifold structures, especially focusing on a certain geometric
construction of the ``canonical'' smooth structure on the smooth
mapping space $\MappingSpace^{\infty}\left(M,N\right)$ from a compact
domain $M$, and further on the diffeomorphism group $\Diff\left(M\right)$
of $M$.\footnote{Recall our notational convention: to help make a distinction between
finite and infinite dimensions, our choice of notation for an infinite-dimensional
object often has its first letter in script font; e.g., $\MappingSpace^{\infty}\left(M,N\right)$
and $\Diff\left(M\right)$.} This will be used to facilitate our study of smooth structures on
various symmetry groups in this work, where $M$ will be taken as
the ambient space, the model fiber, or the base space of a fiber bundle.
We recommend interested readers to \cite{MR1471480,MR4505843} for
expositions in much greater depth and generality.

We start by recalling the Kadec–Anderson theorem (\prettyref{thm:Kadec-Anderson})
from the preceding section: every infinite-dimensional separable Fréchet
space is homeomorphic to $\ell^{2}$, and so Fréchet manifolds in
the topological category are just $\ell^{2}$-manifolds. However,
this is no longer the case in the smooth category: given a smooth
Fréchet manifold, while its underlying topological space always admits
a (unique) structure of a smooth Hilbert manifold, it is generally
not the case that these two smooth structures are diffeomorphic. This
is especially so for the infinite-dimensional manifolds under consideration
in this work; e.g., the diffeomorphism group $\Diff\left(M\right)$
— being infinite-dimensional itself but acting on the compact manifold
$M$ effectively, transitively, and smoothly — can never be a Hilbert
or even a Banach Lie group (see \cite{MR0579603}). As such, for our
study of infinite-dimensional differential geometry in mind, we shall
extend the scope of consideration as follows:
\begin{defn}
\label{def:smooth-separable-frechet-manifold}A (possibly infinite-dimensional)
\emph{smooth manifold} is a Hausdorff space with an equivalence class
of atlases, whose charts are modeled on locally convex topological
vector spaces with all transition maps being $C^{\infty}$-smooth.
Moreover, a second-countable smooth manifold is further called a \emph{smooth (separable) Fréchet manifold}
if its model spaces are separable Fréchet spaces.
\end{defn}

Instead of building up the general theory of infinite-dimensional
smooth manifolds in full, we focus on a class of examples that is
most relevant to us in this work; namely, the manifold $\MappingSpace^{\infty}\left(M,\mathcal{N}\right)$
whose underlying set consists of all smooth maps from a \emph{compact}
smooth manifold $M$ to a (possibly infinite-dimensional) smooth manifold
$\mathcal{N}$.\footnote{It is worth emphasizing the crucial assumption that the domain $M$
is always \emph{compact}.} Just like the set $\MappingSpace\left(M,\mathcal{N}\right)$ of continuous
mappings is equipped with the compact-open topology, the set $\MappingSpace^{\infty}\left(M,\mathcal{N}\right)$
of smooth mappings will be topologized by accounting for compact-open
topologies for the derivatives of all degrees, as in the following
definition:
\begin{defn}
\label{def:compact-open-C-inf-topology}The \emph{compact-open $C^{\infty}$-topology}
on $\MappingSpace^{\infty}\left(M,\mathcal{N}\right)$ is the initial
topology with respect to the following maps given by iterated tangent-map
constructions for all degrees:
\[
T^{k}\colon\MappingSpace^{\infty}\left(M,\mathcal{N}\right)\to\MappingSpace\left(T^{k}M,T^{k}\mathcal{N}\right),\quad f\mapsto T^{k}f\qquad\left(\forall k=0,1,\dots\right),
\]
where the target space $\MappingSpace\left(T^{k}M,T^{k}\mathcal{N}\right)$
for each $k$ is equipped with the usual compact-open topology.
\end{defn}

Having determined a preferred topology on $\MappingSpace^{\infty}\left(M,\mathcal{N}\right)$,
we next consider its smooth structures. As to which maps into $\MappingSpace^{\infty}\left(M,\mathcal{N}\right)$
are to be deemed smooth, there is a natural preference provided by
the exponential law. More specifically, consider the general ``currying''
operation that transforms each $X$-parametrized family $f$ of maps
on $M$ to its adjoint $f^{\wedge}$ on the product $X\times M$ as
follows:
\begin{equation}
\wedge\colon\mathrm{Map}\left(X,\mathrm{Map}\left(M,\mathcal{N}\right)\right)\xrightarrow{f\mapsto f^{\wedge}}\mathrm{Map}\left(X\times M,\mathcal{N}\right),\qquad f^{\wedge}\left(x,m\right)\coloneqq f\left(x\right)\left(m\right).\label{eq:adjoint-correspondence}
\end{equation}
This is clearly a bijective correspondence, whose inverse is given
by the ``co-currying'' operation $g\mapsto g^{\vee}$ sending each
map $g\colon X\times M\to\mathcal{N}$ to the map $g^{\vee}\colon X\to\mathrm{Map}\left(M,\mathcal{N}\right)$
given by $g^{\vee}\left(x\right)\left(m\right)\coloneqq g\left(x,m\right)$.
In the topological category, the compact-open topology on $\MappingSpace\left(M,\mathcal{N}\right)$
is exactly such that any map $f$ into $\MappingSpace\left(M,\mathcal{N}\right)$
is continuous if and only if its adjoint $f^{\wedge}$ is continuous.
Following \cite{MR4188745,MR4505843}, we choose an analogous preferred
smooth structure as follows:
\begin{defn}
\label{def:canonical-smooth-structure}A smooth manifold structure
on $\MappingSpace^{\infty}\left(M,\mathcal{N}\right)$ is said to
be \emph{canonical} if its underlying topology is the compact-open
$C^{\infty}$-topology and it satisfies the following \emph{exponential law}:
\begin{equation}
\text{\ensuremath{f\colon\mathcal{L}\to\MappingSpace^{\infty}\left(M,\mathcal{N}\right)} is smooth}\iff\text{\ensuremath{f^{\wedge}\colon\mathcal{L}\times M\to\mathcal{N}} is smooth},\label{eq:exponential-law-general-in-appendix}
\end{equation}
for every (possibly infinite-dimensional) smooth manifold $\mathcal{L}$.
\end{defn}

In other words, the exponential law demands that the adjoint correspondence
in \prettyref{eq:adjoint-correspondence} restricts to the following
bijection (in fact, a homeomorphism):
\[
\MappingSpace^{\infty}\left(\mathcal{L},\MappingSpace^{\infty}\left(M,\mathcal{N}\right)\right)\approx\MappingSpace^{\infty}\left(\mathcal{L}\times M,\mathcal{N}\right),\qquad f\mapsto f^{\wedge},
\]
which will underlie our communications between infinite and finite
dimensions. As a first example, currying the identity map on $\MappingSpace^{\infty}\left(M,\mathcal{N}\right)$
yields the following smooth map called the \emph{evaluation map}:
\begin{equation}
\mathrm{ev}\coloneqq\left(\Identity_{\MappingSpace^{\infty}\left(M,\mathcal{N}\right)}\right)^{\wedge}\colon\MappingSpace^{\infty}\left(M,\mathcal{N}\right)\times M\to\mathcal{N},\qquad\mathrm{ev}\left(h,x\right)\coloneqq\mathrm{ev}_{x}\left(h\right)\coloneqq h\left(x\right).\label{eq:evaluation-map}
\end{equation}
Since this evaluation map is still the adjoint of the setwise identity
map on $\MappingSpace^{\infty}\left(M,\mathcal{N}\right)$ even if
the target manifold $\MappingSpace^{\infty}\left(M,\mathcal{N}\right)$
is equipped with another canonical smooth structure, we see that the
setwise identity map between any two canonical smooth structures is
always smooth, hence a diffeomorphism. Thus the uniqueness of canonical
smooth structure is justified:
\begin{lem}
\label{lem:canonical-smooth-structure-is-unique}The canonical smooth
manifold structure on $\MappingSpace^{\infty}\left(M,\mathcal{N}\right)$
is unique (if exists).
\end{lem}

There are two important classes of smooth maps between canonical smooth
manifold structures, called the \emph{pushforward} and the \emph{pullback},
induced by smooth maps between the source manifolds and between the
target manifolds, respectively:
\begin{lem}
\label{lem:pushforward-and-pullback}For any smooth maps $f\colon\mathcal{N}_{1}\to\mathcal{N}_{2}$
and $h\colon M_{1}\to M_{2}$, the respective pushforward $f_{\ast}$
and pullback $h^{\ast}$:
\begin{alignat}{2}
f_{\ast}\colon\MappingSpace^{\infty}\left(M,\mathcal{N}_{1}\right) & \to\MappingSpace^{\infty}\left(M,\mathcal{N}_{2}\right),\qquad & g & \mapsto f\circ g\\
h^{\ast}\colon\MappingSpace^{\infty}\left(M_{2},\mathcal{N}\right) & \to\MappingSpace^{\infty}\left(M_{1},\mathcal{N}\right),\qquad & g & \mapsto g\circ h
\end{alignat}
are smooth maps with respect to the canonical smooth structures on
smooth mapping spaces.
\end{lem}

\begin{proof}
By the exponential law \prettyref{eq:exponential-law-general-in-appendix},
it suffices to prove the smoothness of their adjoints $\left(f_{\ast}\right)^{\wedge}$
and $\left(h^{\ast}\right)^{\wedge}$; but this is clear since they
can be expressed in terms of the evaluation map $\mathrm{ev}$ given
in \prettyref{eq:evaluation-map} as follows:
\begin{alignat}{2}
\left(f_{\ast}\right)^{\wedge}\colon\MappingSpace^{\infty}\left(M,\mathcal{N}_{1}\right)\times M & \to\mathcal{N}_{2},\qquad & \left(g,x\right) & \mapsto\mathrm{ev}\left(f,\mathrm{ev}\left(g,x\right)\right)\\
\left(h^{\ast}\right)^{\wedge}\colon\MappingSpace^{\infty}\left(M_{2},\mathcal{N}\right)\times M_{1} & \to\mathcal{N},\qquad & \left(g,x\right) & \mapsto\mathrm{ev}\left(g,\mathrm{ev}\left(h,x\right)\right),
\end{alignat}
both of which are smooth as desired.
\end{proof}
We now proceed to a key construction of the canonical smooth structure
on $\MappingSpace^{\infty}\left(M,\mathcal{N}\right)$. This is motivated
by the following trivial case: If $\mathcal{N}=E$ were a locally
convex space, then so would be $\MappingSpace^{\infty}\left(M,E\right)$
(with pointwise operations), which would then have a trivial canonical
smooth structure. In the general case where the target manifold $\mathcal{N}$
is arbitrary, a suitable substitution for an addition operation can
be given by a certain kind of ``local addition'' $\alpha$ defined
in a neighborhood $W$ around the zero section of the tangent bundle:
\[
\alpha\colon T\mathcal{N}\supsetopen W\to\mathcal{N},
\]
which captures the notion of ``adding'' a sufficiently small tangent
vector $v\in T\mathcal{N}$ to its foot-point and resulting a displaced
new point $\alpha\left(v\right)\in\mathcal{N}$. In contract to naively
using the local charts of $\mathcal{N}$, such a local addition $\alpha$
is universally defined for all foot-points $p\in\mathcal{N}$, so
that the above pointwise addition operation can be performed along
any smooth map into $\mathcal{N}$. The precise condition on $\alpha$
making this work is given in the following definition:

\begin{defn}
\label{def:local-addition}Let $\mathcal{N}$ be a (possibly infinite-dimensional)
smooth manifold, and let $\pi_{T\mathcal{N}}\colon T\mathcal{N}\to\mathcal{N}$
denote its tangent bundle. Then a smooth map $\alpha\colon T\mathcal{N}\supsetopen W\to\mathcal{N}$
defined near the zero section is called a \emph{local addition on $\mathcal{N}$}
if $\alpha$ induces a diffeomorphism
\begin{equation}
\left(\left.\pi_{T\mathcal{N}}\right|_{W},\alpha\right)\colon T\mathcal{N}\supsetopen W\xrightarrow{\approx}W'\subsetopen\mathcal{N}\times\mathcal{N}\label{eq:local-addition-diff-requirement}
\end{equation}
from the open neighborhood $W$ around the zero section $0_{\mathcal{N}}$
of $T\mathcal{N}$ onto an open neighborhood $W'$ around the diagonal
$\Delta_{\mathcal{N}}$ of $\mathcal{N}\times\mathcal{N}$, such that
it particularly maps $0_{\mathcal{N}}$ to $\Delta_{\mathcal{N}}$.
\end{defn}

As explained above, given a local condition $\alpha$ on $\mathcal{N}$,
we can perturb a given map $f\in\MappingSpace^{\infty}\left(M,\mathcal{N}\right)$
by ``adding'' any vector field $Y$ along $f$ ; or conversely by
virtue of \prettyref{eq:local-addition-diff-requirement}, we can
measure the ``difference'' between $h$ and any nearby map in $\MappingSpace^{\infty}\left(M,\mathcal{N}\right)$
by an element in the following space:
\begin{equation}
\MappingSpace_{h}^{\infty}\left(M,T\mathcal{N}\right)\coloneqq\left\{ Y\in\MappingSpace^{\infty}\left(M,T\mathcal{N}\right)\mid\pi_{T\mathcal{N}}\circ Y=h\right\} .\label{eq:space-of-vector-fields-along-a-map}
\end{equation}
Of course, this space is canonically identified with the space of
smooth sections of the pullback bundle $h^{\ast}T\mathcal{N}$, hence
serves as a locally convex model space. This yields a smooth structure
on $\MappingSpace^{\infty}\left(M,\mathcal{N}\right)$ that turns
out to be canonical, as in the following theorem:
\begin{thm}
\label{thm:local-addition-implies-canonical-smooth-manifold-structure}Suppose
that the (possibly infinite-dimensional) smooth manifold $\mathcal{N}$
admits a local addition $\alpha$ (as defined in \prettyref{def:local-addition}).
Then for any compact smooth manifold $M$, the smooth mapping space
$\MappingSpace^{\infty}\left(M,\mathcal{N}\right)$ admits a canonical
smooth structure, whose (inverse) chart near each $h\in\MappingSpace^{\infty}\left(M,\mathcal{N}\right)$
can be given by the following homeomorphism:
\begin{equation}
\varphi_{h}\colon\MappingSpace_{h}^{\infty}\left(M,T\mathcal{N}\right)\supsetopen V_{h}\xrightarrow{\approx}U_{h}\subsetopen\MappingSpace^{\infty}\left(M,\mathcal{N}\right),\qquad Y\mapsto\alpha\circ Y.\label{eq:canonical-smooth-structure-chart}
\end{equation}
Moreover, such an induced smooth structure on $\MappingSpace^{\infty}\left(M,\mathcal{N}\right)$
is unique, independent of the choice of local addition $\alpha$.
\end{thm}

\begin{proof}[Proof (Sketch)]
Let $\alpha\colon T\mathcal{N}\supsetopen W\to\mathcal{N}$ be a
local addition on $\mathcal{N}$. Then by requiring the vector fields
$Y\in\MappingSpace_{h}^{\infty}\left(M,T\mathcal{N}\right)$ along
$h\colon M\to\mathcal{N}$ to take values in the neighborhood $W\subsetopen T\mathcal{N}$
of the zero section, we obtain a desired neighborhood $V_{h}\subsetopen\MappingSpace_{h}^{\infty}\left(M,T\mathcal{N}\right)$
around the zero vector field along $h$, such that the map $\varphi_{h}$
given in \prettyref{eq:canonical-smooth-structure-chart} is well-defined.
Conversely, recall from \prettyref{def:local-addition} that the defining
condition on the local addition $\alpha$ guarantees a diffeomorphism
\[
\left(\left.\pi_{T\mathcal{N}}\right|_{W},\alpha\right)\colon T\mathcal{N}\supsetopen W\xrightarrow{\approx}W'\subsetopen\mathcal{N}\times\mathcal{N}.
\]
Then by requiring the map $f\in\MappingSpace^{\infty}\left(M,\mathcal{N}\right)$
such that $\left(h,f\right)$ take values in the neighborhood $W'\subsetopen\mathcal{N}\times\mathcal{N}$
of the diagonal, we obtain a desired neighborhood $U_{h}\subsetopen\MappingSpace^{\infty}\left(M,\mathcal{N}\right)$
around $h$, such that the inverse map $\varphi_{h}^{-1}$ given by
the following formula is well-defined:
\[
\varphi_{h}^{-1}\colon\MappingSpace^{\infty}\left(M,\mathcal{N}\right)\supsetopen U_{h}\to V_{h}\subsetopen\MappingSpace_{h}^{\infty}\left(M,T\mathcal{N}\right),\qquad f\mapsto\left(\left.\pi_{T\mathcal{N}}\right|_{W},\alpha\right)^{-1}\circ\left(h,f\right).
\]
Since the transition map $\varphi_{h_{2}}^{-1}\circ\varphi_{h_{1}}$
of zero-neighborhoods between $\MappingSpace_{h_{1}}^{\infty}\left(M,T\mathcal{N}\right)$
and $\MappingSpace_{h_{2}}^{\infty}\left(M,T\mathcal{N}\right)$ is
clearly smooth, we see that $\left(U_{h},\varphi_{h}\right)$ is indeed
a smooth chart for $\MappingSpace^{\infty}\left(M,\mathcal{N}\right)$
near $h$ modeled on the locally convex space $\MappingSpace_{h}^{\infty}\left(M,T\mathcal{N}\right)$.
To show that this smooth structure on $\MappingSpace^{\infty}\left(M,\mathcal{N}\right)$
is canonical, we need to verify the exponential law \prettyref{eq:exponential-law-general-in-appendix},
for which the proof is omitted and can be found in e.g., \cite[Lemma C.11]{MR4505843}.
Lastly, being canonical, such a smooth structure is unique by \prettyref{lem:canonical-smooth-structure-is-unique}.
This completes the proof of \prettyref{thm:local-addition-implies-canonical-smooth-manifold-structure}.
\end{proof}
The model space $\MappingSpace_{h}^{\infty}\left(M,T\mathcal{N}\right)$
in the preceding theorem can be viewed as the fiber over $h\in\MappingSpace^{\infty}\left(M,\mathcal{N}\right)$
of the following smooth vector bundle:
\begin{equation}
{\pi_{T\mathcal{N}}}_{\ast}\colon\MappingSpace^{\infty}\left(M,T\mathcal{N}\right)\to\MappingSpace^{\infty}\left(M,\mathcal{N}\right),\qquad Y\mapsto\pi_{T\mathcal{N}}\circ Y.\label{eq:vector-bundle-from-M-to-TN}
\end{equation}
In fact, this yields a natural description of the tangent bundle of
such a mapping manifold $\MappingSpace^{\infty}\left(M,\mathcal{N}\right)$,
as stated in the following proposition (whose proof is omitted and
can be found in e.g., \cite[Theorem A.12]{MR4188745} or \cite[Theorem 42.17]{MR1471480}):
\begin{prop}
Suppose that $\MappingSpace^{\infty}\left(M,\mathcal{N}\right)$ is
equipped with the canonical smooth structure as constructed in \prettyref{thm:local-addition-implies-canonical-smooth-manifold-structure}.
Then its tangent bundle admits a natural vector bundle isomorphism
with the vector bundle in \prettyref{eq:vector-bundle-from-M-to-TN}
over $\MappingSpace^{\infty}\left(M,\mathcal{N}\right)$:
\begin{equation}
T\MappingSpace^{\infty}\left(M,\mathcal{N}\right)\cong\MappingSpace^{\infty}\left(M,T\mathcal{N}\right),\label{eq:canonical-isomorphism-of-tangent-bundle}
\end{equation}
such that under the resulting continuous linear isomorphism $T_{h}\MappingSpace^{\infty}\left(M,\mathcal{N}\right)\cong\MappingSpace_{h}^{\infty}\left(M,T\mathcal{N}\right)$
of their fibers over each $h\in\MappingSpace^{\infty}\left(M,\mathcal{N}\right)$,
the tangent map $T_{h}\mathrm{ev}_{x}\colon T_{h}\MappingSpace^{\infty}\left(M,\mathcal{N}\right)\to T_{h\left(x\right)}\mathcal{N}$
of the evaluation map $\mathrm{ev}_{x}$ at $x\in M$ just sends each
$Y\in\MappingSpace_{h}^{\infty}\left(M,T\mathcal{N}\right)$ to its
evaluated tangent vector $Y\left(x\right)\in T_{h\left(x\right)}\mathcal{N}$.
\end{prop}

Local additions arise naturally from many common structures on a smooth
manifold, some of which are illustrated in the following example:
\begin{example}
\label{exa:examples-of-local-additions}Each of the following manifolds
$\mathcal{N}$ with extra structures admits a local addition:
\begin{enumerate}[label=\textup{(\roman*)}]
\item \label{enu:example-of-affine-local-addition}For $\mathcal{N}=V$
a locally convex space, the global chart on $V$ induces a global
trivialization of the tangent bundle $TV\cong V\times V$, so that
we can define a (affine) local addition by
\[
\alpha\colon TV\to V,\qquad T_{v}V\ni w\mapsto v+w\in V.
\]
\item For $\mathcal{N}=G$ a Lie group (possibly infinite-dimensional, modeled
on a locally convex space), say with a centered local chart $\varphi\colon G\supsetopen U\to V\subsetopen T_{e}G$,
the right translation $\rho^{g}$ on $G$ by $g\in G$ induces a global
trivialization of the tangent bundle $TG\cong G\times T_{e}G$, so
that we can define a (right-invariant) local addition by
\[
\alpha\colon TG\supsetopen G\times V\to G,\qquad T_{g}G\ni T_{e}\rho^{g}\left(v\right)\mapsto\rho^{g}\left(\varphi^{-1}\left(v\right)\right)\in G.
\]
\item \label{enu:example-of-riemannian-local-addition}For $\mathcal{N}=N$
a finite-dimensional manifold equipped with any Riemannian metric,
the Riemannian exponential map $\exp$ can serve as a (geometric)
local addition:
\[
\alpha\colon TN\supsetopen U\to N,\qquad T_{p}N\ni v\mapsto\exp\left(v\right)\coloneqq\gamma_{v}\left(1\right)\in N,
\]
where $\gamma_{v}\left(t\right)$ denotes the geodesic in $N$ with
initial velocity $v$. Here, the required diffeomorphism \prettyref{eq:local-addition-diff-requirement}
can be justified by the usual inverse function theorem thanks to the
assumption of finite dimension for $N$.
\end{enumerate}
Therefore, it follows from \prettyref{thm:local-addition-implies-canonical-smooth-manifold-structure}
that for each of the above manifolds $\mathcal{N}$ as the codomain,
its local addition induces a canonical smooth structure on $\MappingSpace^{\infty}\left(M,\mathcal{N}\right)$
for any compact domain $M$.
\end{example}

The last of the preceding three classes of examples — namely, $\MappingSpace^{\infty}\left(M,N\right)$
with a finite-dimensional target manifold $N$ (and a compact source
manifold $M$, as always assumed), endowed with the canonical smooth
structure induced from the Riemannian exponential — will be most relevant
to the current work due to its geometric flavor.\footnote{As we shall see later in \prettyref{sec:Lie-Group-Extension-of-the-Symmetries},
another geometric structure that we shall use to induce a local addition
on $M$ is that of a Riemannian fibration; or in other words, a smooth
fibration of $M$ together with a fiberwise metric, a connection,
and a base metric.} In fact, we shall need the canonical smooth structures to be inherited
by the subsets of submersions and of embeddings, respectively, which
is justified by the following lemma: 
\begin{lem}
\label{lem:subm-and-emb-are-open-subsets}For compact source manifold
$M$ and finite-dimensional target manifold $N$, the smooth mapping
space has the following two open subsets:
\[
\Subm\left(M,N\right)\subsetopen\MappingSpace^{\infty}\left(M,N\right)\qquad\text{and}\qquad\Emb\left(M,N\right)\subsetopen\MappingSpace^{\infty}\left(M,N\right)
\]
consisting of those maps $M\to N$ that are smooth submersions and
smooth embeddings, respectively.
\end{lem}

A particular important instance for us is when the target manifold
$N=M$ coincides with the (compact) source, in which case we are further
interested in the \emph{diffeomorphism group} of $M$, which is denoted
by $\Diff\left(M\right)$ and consists of all $C^{\infty}$-self-diffeomorphisms
of $M$, with the canonical group structure given by mapping compositions.
The preceding lemma (\prettyref{lem:subm-and-emb-are-open-subsets})
implies that it is an open subset
\[
\Diff\left(M\right)\subsetopen\MappingSpace^{\infty}\left(M,M\right),
\]
which thus inherits the compact-open $C^{\infty}$-topology, and furthermore
the canonical smooth structure, whose model space $\MappingSpace_{h}^{\infty}\left(M,TM\right)$
given in \prettyref{eq:space-of-vector-fields-along-a-map} is just
the right translation of the Fréchet space $\mathfrak{X}\left(M\right)$
of smooth vector fields on $M$:
\begin{equation}
\mathfrak{X}\left(M\right)\circ h\coloneqq\left\{ X\circ h\in\MappingSpace^{\infty}\left(M,TM\right)\mid X\in\mathfrak{X}\left(M\right)\right\} \qquad\left(\forall h\in\Diff\left(M\right)\right).\label{eq:right-translation-of-frechet-space}
\end{equation}
The thus obtained canonical smooth structure on the diffeomorphism
group is summarized in the following theorem:

\begin{thm}
\label{thm:Canonical-Smooth-Structure-on-DiffM-appendix}For any compact
smooth manifold $M$, the diffeomorphism group $\Diff\left(M\right)$
admits the canonical structure of a Fréchet Lie group with Lie algebra
$\mathfrak{X}\left(M\right)$ (with Lie bracket being the negative
of the bracket of vector fields). Specifically, its (inverse) chart
around each $h\in\Diff\left(M\right)$ can be given by pushforward
of the Riemannian exponential map $\exp_{M}\colon TM\to M$ with respect
to any choice of Riemannian metric on $M$:
\begin{equation}
\varphi_{h}\colon\mathfrak{X}\left(M\right)\circ h\supsetopen V_{h}\xrightarrow{\approx}U_{h}\subsetopen\Diff\left(M\right),\qquad X\circ h\mapsto\exp_{M}\circ X\circ h.\label{eq:charts-for-diff}
\end{equation}
Here, the Fréchet model space $\mathfrak{X}\left(M\right)\circ h$
is the right-translated space given in \prettyref{eq:right-translation-of-frechet-space},
which also serves as the tangent space $T_{h}\Diff\left(M\right)$
under the canonical identification \prettyref{eq:canonical-isomorphism-of-tangent-bundle},
and further identifies the Lie algebra $T_{\Identity}\Diff\left(M\right)=\mathfrak{X}\left(M\right)$.
\end{thm}

\begin{proof}[Proof (Sketch)]
As mentioned in \prettyref{exa:examples-of-local-additions}\ref{enu:example-of-riemannian-local-addition},
the Riemannian exponential $\exp_{M}\colon TM\supsetopen W\to M$
provides a local addition of $M$. In turn, \prettyref{thm:local-addition-implies-canonical-smooth-manifold-structure}
implies that this local addition induces smooth charts for the canonical
smooth structure of $\MappingSpace^{\infty}\left(M,M\right)$ by post-composition,
which confirms the charts as asserted in \prettyref{eq:charts-for-diff}.
Here, the model space and tangent space at each $h\colon M\to M$
is indeed the Fréchet space $\mathfrak{X}\left(M\right)\circ h$ given
in \prettyref{eq:right-translation-of-frechet-space}, by comparing
it with the definition of $\MappingSpace_{h}^{\infty}\left(M,TM\right)$
in \prettyref{eq:space-of-vector-fields-along-a-map}. This proves
the canonical smooth structure on $\MappingSpace^{\infty}\left(M,M\right)$;
but since by \prettyref{lem:subm-and-emb-are-open-subsets} the submersions
and embeddings are open in $\MappingSpace^{\infty}\left(M,M\right)$,
so are the diffeomorphisms, hence the open subset $\Diff\left(M\right)\subsetopen\MappingSpace^{\infty}\left(M,M\right)$
inherits the canonical smooth structure. Note that the underlying
compact-open $C^{\infty}$-topology on $\Diff\left(M\right)$ is indeed
metrizable and second-countable because $M$ is so (see \cite[Corollary 41.12]{MR1471480}).
Next we need to prove that this smooth manifold $\Diff\left(M\right)$
is a Lie group, which amounts to showing the smoothness of the operations
of taking compositions and inverses:
\[
\mathrm{comp}\colon\Diff\left(M\right)\times\Diff\left(M\right)\xrightarrow{\left(g_{1},g_{2}\right)\mapsto g_{1}\circ g_{2}}\Diff\left(M\right),\qquad\mathrm{inv}\colon\Diff\left(M\right)\xrightarrow{g\mapsto g^{-1}}\Diff\left(M\right).
\]
By the exponential law \prettyref{eq:exponential-law-general-in-appendix},
it suffices to prove the smoothness of their adjoints $\mathrm{comp}^{\wedge}$
and $\mathrm{inv}^{\wedge}$. For the composition operation, the adjoint
$\mathrm{comp}^{\wedge}$ is clearly smooth since it can be written
in terms of the evaluation map $\mathrm{ev}$ given in \prettyref{eq:evaluation-map}
as follows:
\[
\mathrm{comp}^{\wedge}\colon\Diff\left(M\right)\times\Diff\left(M\right)\times M\to M,\qquad\mathrm{comp}^{\wedge}\left(g_{1},g_{2},x\right)\coloneqq\mathrm{ev}\left(g_{1},\mathrm{ev}\left(g_{2},x\right)\right).
\]
On the other hand, for the inverse operation, the adjoint $\mathrm{inv}^{\wedge}$
sending $\left(g,x\right)$ to $g^{-1}\left(x\right)$ can be expressed
in terms of the evaluation map by the following implicit equation:
\begin{equation}
\mathrm{inv}^{\wedge}\colon\Diff\left(M\right)\times M\to M,\qquad\mathrm{ev}\left(g,\mathrm{inv}^{\wedge}\left(g,x\right)\right)=x.\label{eq:implicit-equation-for-inverse-of-diff}
\end{equation}
Note that when each $g\in\Diff\left(M\right)$ is held fixed, the
partial differentiation $T\mathrm{ev}\left(g,-\right)\colon TM\to TM$
coincides with the tangent map $Tg$, which is everywhere invertible
since $g$ is a diffeomorphism. This prompts to apply a suitable implicit
function theorem to the implicit equation in \prettyref{eq:implicit-equation-for-inverse-of-diff},
but mind that the ordinary implicit function theorem in finite dimensions
does not apply here since the first argument has an infinite-dimensional
domain $\Diff\left(M\right)$. There are several ways to work around:
in \cite[Example 3.5]{MR4505843}, a generalized implicit function
theorem from \cite[Theorem 2.3]{MR2269430} is invoked, which is particularly
applicable to the implicit equation in \prettyref{eq:implicit-equation-for-inverse-of-diff}
and hence establish the desired smoothness of $\mathrm{inv}^{\wedge}$;
in addition, the same literature also refers to an alternative proof
from \cite[Theorem 11.11]{MR0583436}, which verifies the desired
smoothness directly. The takeaway is that $\Diff\left(M\right)$ is
a Lie group. Its Lie algebra $\mathfrak{X}^{\mathrm{R}}\left(\Diff\left(M\right)\right)$
consisting of all right-invariant vector fields on $\Diff\left(M\right)$
can be described by the tangent space $T_{\Identity}\Diff\left(M\right)\cong\mathfrak{X}\left(M\right)$
at the identity via the following linear isomorphism:
\begin{equation}
\mathfrak{X}\left(M\right)\cong\mathfrak{X}^{\mathrm{R}}\left(\Diff\left(M\right)\right),\quad X\mapsto X^{\mathrm{R}}\qquad\text{with}\qquad X^{\mathrm{R}}\left(h\right)\coloneqq X\circ h,\quad\forall h\in\Diff\left(M\right).\label{eq:tangent-space-identity-of-Diff}
\end{equation}
Then one may verify, e.g., as in \cite[Example 3.25]{MR4505843},
that the Lie bracket of two such right-invariant vector fields $X_{1}^{\mathrm{R}},X_{2}^{\mathrm{R}}\in\mathfrak{X}^{\mathrm{R}}\left(\Diff\left(M\right)\right)$
is corresponding to the negative bracket $-\left[X_{1},X_{2}\right]$
of their spanning vector fields $X_{1},X_{2}\in\mathfrak{X}\left(M\right)$
on $M$, as desired. This completes the proof of \prettyref{thm:Canonical-Smooth-Structure-on-DiffM-appendix}.
\end{proof}

\section{Infinite-Dimensional Principal Bundles}

In this section, we establish some criteria for the coset projection
$G\to G/H$ of a topological group (resp., a Lie group) to be a principal
bundle (resp., a smooth principal bundle), with an eye towards the
case where $G,H$ are typically infinite-dimensional. This will be
used to facilitate various proofs of bundle structures in this work,
where $G$ will be taken as various transformation groups of a fiber
bundle. We recommend interested readers to \cite{glockner2016fundamentalssubmersionsimmersionsinfinitedimensional,MR3943433}
for expositions in much greater depth and generality.

We begin with letting $G$ be a topological group and $H\leq G$ be
a closed subgroup. Then the canonical projection $q$ onto the left
coset space $G/H$ gives a sequence
\begin{equation}
H\hookrightarrow G\xrightarrow{q}G/H,\label{eq:coset-quotient-sequence}
\end{equation}
with the total space $G$ being equipped with the canonical $H$-action
by right translation:
\begin{equation}
\rho\colon G\times H\to G,\qquad\rho\left(g,h\right)\coloneqq\rho^{h}\left(g\right)\coloneqq g\cdot h.\label{eq:right-translation}
\end{equation}
We aim to characterize when this becomes a principal bundle structure.
For this we may take hints from elementary algebra: recall the usual
splitting lemma that for any short exact sequence $0\to A\to B\to C\to0$
in an abelian category, a splitting structure can be given by any
of the following three equivalent constructs: a right split $C\to B$
(i.e., a section), a left split $B\to A$ (i.e., a retraction), and
a splitting complement $C'\subseteq B$ that splits $B$ as a direct
sum $A\oplus C'$ (i.e., a slice). The analogous constructs for locally
trivial structure on the coset quotient sequence \prettyref{eq:coset-quotient-sequence}
are as follows:
\begin{enumerate}[label=\textup{(\roman*)}]
\item a \emph{local cross section $\sigma$} over some open neighborhood
$W$ around the identity coset in $G/H$:
\begin{equation}
\sigma\colon G/H\supsetopen W\to G,\qquad q\circ\sigma\equiv\Identity_{W},\label{eq:local-section}
\end{equation}
which particularly maps the identity coset to the identity: $\sigma\left(\left[\mathbf{1}\right]\right)=\mathbf{1}$.
\item an \emph{equivariant neighborhood retraction $r$} from some (saturated)
open neighborhood $q^{-1}\left(W\right)$ around $H$ in $G$:
\begin{equation}
r\colon G\supsetopen q^{-1}\left(W\right)\to H,\qquad r\circ\rho^{h}\equiv\rho^{h}\circ r,\ \forall h\in H,\label{eq:neighborhood-retraction}
\end{equation}
which particularly maps the identity to itself: $r\left(\mathbf{1}\right)=\mathbf{1}$.
\item a \emph{local slice $S$} restricted to which the right-translation
map $\rho$ as in \prettyref{eq:right-translation} is a (equivariant)
homeomorphism onto an (saturated) open subset $q^{-1}\left(W\right)$
in $G$:
\begin{equation}
S\subseteq G\qquad\text{with}\qquad\rho\colon S\times H\xrightarrow{\approx}q^{-1}\left(W\right)\subsetopen G\quad\text{homeomorphism},\label{eq:local-slice}
\end{equation}
which particularly contains the identity: $S\ni\mathbf{1}$.
\end{enumerate}
The analogous ``splitting lemma'' for characterizing the locally
trivial structures on a coset quotient projection is justified as
follows:
\begin{lem}
\label{lem:splitting-lemma}Let $G$ be a topological group and $H\leq G$
be a closed subgroup, consider the canonical projection $q\colon G\to G/H$
onto the left coset space $G/H$. Then for any open neighborhood $W\subsetopen G/H$
around the identity coset, the following conditions are equivalent:
\begin{enumerate}[label=\textup{(\roman*)}]
\item There exists a local cross section $\sigma\colon G/H\supsetopen W\to G$\emph{
as in }\prettyref{eq:local-section}.
\item There exists an equivariant neighborhood retraction $r\colon G\supsetopen q^{-1}\left(W\right)\to H$\emph{
as in }\prettyref{eq:neighborhood-retraction}.
\item There exists a local slice $S\subseteq G$ (with the tube homeomorphism
$\rho\colon S\times H\approx q^{-1}\left(W\right)$) as in \prettyref{eq:local-slice}.
\end{enumerate}
Further, if these conditions hold for some identity neighborhood $W\subsetopen G/H$,
then by left translation there induces a topological principal $H$-bundle
structure on the projection $q\colon G\to G/H$ with a trivialization
cover $\left\{ g\cdot W\right\} _{g\in G}$.
\end{lem}

\begin{proof}
Let us first prove the equivalence between $\sigma$, $r$, and $S$.
Indeed, this follows from their constructions, whose relationship
is shown by the three dotted arrows fitting in the following commutative
diagram:
\begin{equation}
\begin{gathered}\xymatrix{H\ar@{^{(}->}[d] &  & H &  & H\ar@{=}[ll]\\
G\ar[d]_{q} & \supsetopen & q^{-1}\left(W\right)\ar[d]_{q}\ar@{..>}[u]^{r}\ar@{..>}[rr]^{{\rho^{-1}}} &  & S\times H\ar[d]^{\mathrm{pr}_{1}}\ar[u]_{\mathrm{pr}_{2}}\\
G/H & \supsetopen & W\ar@{..>}[rr]_{\sigma} &  & S
}
\end{gathered}
\label{eq:diagram-splitting-lemma}
\end{equation}
More precisely, given a local slice $S\subseteq G$ with the tube
homeomorphism $\rho$ from $S\times H$ onto $q^{-1}\left(W\right)\subsetopen G$,
its inverse $\rho^{-1}$ locally splits $q^{-1}\left(W\right)\subsetopen G$
into the $S$-component and $H$-component that yield the section
$\sigma$ and the retraction $r$, respectively:
\begin{equation}
\rho^{-1}\colon G\supsetopen q^{-1}\left(W\right)\xrightarrow{\approx}S\times H,\qquad\rho^{-1}\left(f\right)=\left(\sigma\left(\left[f\right]\right),r\left(f\right)\right).\label{eq:inverse-tube-homeomorphism}
\end{equation}
Here, $\sigma$ is indeed well-defined and $r$ is indeed equivariant
by virtue of the property that the right translation $\rho$ is $H$-equivariant
on $S\times H$ (with respect to the right $H$-action on $S\times H$
by right translating the $H$-component and fixing the $S$-component):
\[
\left(\sigma\left(\left[f\cdot h\right]\right),r\left(f\cdot h\right)\right)=\rho^{-1}\left(f\cdot h\right)=\rho^{-1}\left(f\right)\cdot h=\left(\sigma\left(\left[f\right]\right),r\left(f\right)\right)\cdot h=\left(\sigma\left(\left[f\right]\right),r\left(f\right)\cdot h\right).
\]
Conversely, given a local cross section $\sigma$ or an equivariant
neighborhood retraction $r$, we can see from the above formula \prettyref{eq:inverse-tube-homeomorphism}
that the corresponding local slice $S\subseteq G$ is induced by their
image and kernel, respectively:
\[
S=\image\left(\sigma\right)\qquad\text{and}\qquad S=\ker\left(r\right)\coloneqq r^{-1}\left(e\right),
\]
where $S$ indeed satisfies the slice condition since the tube homeomorphism
$\rho$ is justified by its continuous inverse given in \prettyref{eq:inverse-tube-homeomorphism}.
Since the local splitting reads $f=\sigma\left(\left[f\right]\right)\cdot r\left(f\right)$,
we can also see the direct equivalence between $\sigma$ and $r$
as follows:
\[
\sigma\left(\left[f\right]\right)=f\cdot\left(r\left(f\right)\right)^{-1}\qquad\text{and}\qquad r\left(f\right)=\left(\sigma\left(\left[f\right]\right)\right)^{-1}\cdot f.
\]
In sum, we have shown the desired equivalence between $\sigma$, $r$,
and $S$ as in the diagram \prettyref{eq:diagram-splitting-lemma}.
It also follows from this diagram that the fiber-preserving (i.e.,
$H$-equivariant) homeomorphism $S\times H\approx q^{-1}\left(W\right)$
given by the right translation $\rho$ descends to a homeomorphism
$S\approx W$ given by restricting the projection $q$ to the slice:
\begin{equation}
\left.q\right|_{S}\colon G\supseteq S\xrightarrow{\approx}W\subsetopen G/H,\label{eq:homeomorphism-from-slice}
\end{equation}
so that composing the two fiber-preserving homeomorphisms $\rho^{-1}$
and $\left.q\right|_{S}\times\Identity_{H}$ yields a local trivialization
for the desired bundle structure on $q\colon G\to G/H$ near the identity,
as in the following diagram:
\begin{equation}
\xymatrix@C=4pc{q^{-1}\left(W\right)\ar[r]_{\approx}^{{\rho^{-1}}}\ar[rd]_{q} & S\times H\ar[r]_{\approx}^{\left.q\right|_{S}\times\Identity_{H}} & W\times H\ar[ld]^{\projection_{1}}\\
 & W
}
\label{eq:local-trivialization-from-slice}
\end{equation}
Lastly, the homogeneous $G$-structure on the left coset space $G/H$
allows us to translate this local trivialization near the identity
to any other point. This completes the proof of \prettyref{lem:splitting-lemma}.
\end{proof}
Among the three equivalent conditions in the preceding lemma (\prettyref{lem:splitting-lemma}),
the local section $\sigma$ is arguably the most familiar, while the
equivariant neighborhood retraction $r$ will be more intuitive in
the absence of concrete description for the coset space; but ultimately,
it is the slice description that provides the most robust perspective
when promoting to the (infinite-dimensional) smooth category. To explain
this, we emphasize the absence of a general inverse function theorem
in infinite dimensions, which will be a foremost source of subtleties
and challenges throughout the study of infinite-dimensional differential
geometry: the infinitesimal information at the level of tangent spaces
no longer suffices to imply local information at the level of manifolds
themselves. For example, in finite dimensions, the familiar immersion
theorem and submersion theorem, which are consequences of the (finite-dimensional)
inverse function theorem, tell us that if a map has injective (respectively,
surjective) tangent map then it can be locally expressed as a coordinate
injection $\mathbb{R}^{m}\to\mathbb{R}^{m}\times\mathbb{R}^{n}$ (respectively,
a coordinate projection $\mathbb{R}^{m}\times\mathbb{R}^{n}\to\mathbb{R}^{m}$)
of the Euclidean model spaces — however, these results no longer
hold in general for locally convex model spaces in infinite dimensions.
This discrepancy is significant because most of the important properties
of immersions and submersions rely on the existence of such linear
local representatives, and so we shall let this latter stronger condition
take the role of defining ``immersions'' and ``submersions'' in
infinite dimensions:

\begin{defn}
\label{def:Embeddings-and-Submersions}Let $\mathcal{M},\mathcal{N}$
be (infinite-dimensional) smooth manifolds modeled on locally convex
spaces. Then a smooth map $\mathcal{M}\to\mathcal{N}$ is called an
\emph{immersion} (resp., \emph{submersion}) if it is everywhere
locally represented by a linear map between locally convex spaces
in the form of a canonical injection into a product $A\hookrightarrow A\times C$
(resp., a canonical projection from a product $A\times C\twoheadrightarrow C$).
\end{defn}

As usual, an immersion that is further a topological embedding is
called a \emph{(smooth) embedding}. In particular, for a subset $S\subseteq\mathcal{M}$
the property of having its inclusion being an embedding provides an
alternative characterization of a \emph{split submanifold}. Examples
of split submanifolds include the fibers $\mathcal{M}_{b}\coloneqq q^{-1}\left(b\right)$
of a submersion $q\colon\mathcal{M}\to\mathcal{N}$, whose tangent
space $T_{p}\mathcal{M}_{b}$ at each point can be identified with
the kernel subspace $\ker T_{p}q\subseteq T_{p}\mathcal{M}$. In our
case, we are most interested in a split submanifold $H\leq G$ that
is a closed subgroup of a Lie group, which we shall call a \emph{closed Lie subgroup}
for short, and we aim to characterize when the coset quotient projection
$q\colon G\to G/H$ is a smooth submersion and a smooth principal
bundle. Note that in finite dimensions, this is automatic: every closed
subgroup $H$ in a finite-dimensional Lie group $G$ must be an embedded
Lie subgroup, and the resulting coset space $G/H$ has a unique smooth
structure such that the coset quotient projection $G\to G/H$ is a
smooth submersion, hence a smooth principal bundle. However, these
structures are no longer guaranteed in infinite dimensions in general;
instead, we have the following ``\emph{(smooth) slice lemma}''
that provides us with a useful characterization:

\begin{prop}
\label{prop:manifold-slice-condition}Let $G$ be a (infinite-dimensional)
Lie group modeled on locally convex space, and $H\subseteq G$ be
a closed Lie subgroup. Then the coset space $G/H$ admits a smooth
structure onto which the canonical projection $q\colon G\to G/H$
is a smooth principal $H$-bundle if (and only if) the following condition
holds:
\begin{quote}
There exists a smooth submanifold $S\subseteq G$ containing the identity,
such that the canonical right-translation map $\rho\colon S\times H\to G$
is a diffeomorphism onto an open image $SH\subsetopen G$.
\end{quote}
In this case, the thus obtained smooth manifold structure on $G/H$
is unique, and is a (Hausdorff, second countable) smooth Fréchet manifold
in the sense of \prettyref{def:smooth-separable-frechet-manifold}
if $G$ is assumed to be so.
\end{prop}

\begin{proof}
We focus on constructing local charts for $G/H$ near the identity
coset, for then the homogeneous $G$-structure on the left coset space
$G/H$ would allow us to translate this identity chart to any other
point. Let $S\subseteq G$ be a slice, and let $\rho$ be the corresponding
tube homeomorphism from $S\times H$ onto the tubular neighborhood
$SH\subsetopen G$. Then recall from the proof of \prettyref{lem:splitting-lemma}
(specifically, the homeomorphism \prettyref{eq:homeomorphism-from-slice})
that the fiber-preserving (i.e., $H$-equivariant) homeomorphism $\rho\colon S\times H\approx SH$
descends to a homeomorphism $\left.q\right|_{S}\colon S\approx q\left(S\right)$,
as in the following diagram:
\[
\begin{gathered}\xymatrix{G\ar[d]_{q} & \supsetopen & SH\ar[d]_{q} &  & S\times H\ar[d]^{\mathrm{pr}_{1}}\ar[ll]_{\rho}^{\approx}\\
G/H & \supsetopen & q\left(S\right) &  & S\ar[ll]_{\approx}^{\left.q\right|_{S}}
}
\end{gathered}
.
\]
Now suppose further that $S$ is a smooth slice; i.e., the subset
$S\subseteq G$ is a smooth submanifold and the homeomorphism $\rho\colon S\times H\approx SH$
is a diffeomorphism. Then from the above diagram we see that $q\left(S\right)$
inherits a natural smooth structure from $S$ such that the homeomorphism
$\left.q\right|_{S}\colon S\approx q\left(S\right)$ becomes a diffeomorphism
and the quotient map $q\colon SH\to q\left(S\right)$ becomes a submersion;
in particular, the resulting smooth structure on $G/H$ must be unique
by the universal property of a submersion. Moreover, the following
composition of fiber-preserving homeomorphisms as constructed in \prettyref{eq:local-trivialization-from-slice}
becomes a fiber-preserving diffeomorphism and hence serves as a local
trivialization on $SH\subsetopen G$ over $q\left(S\right)\subsetopen G/H$
for the desired smooth bundle structure:
\[
\left(\left.q\right|_{S}\times\Identity_{H}\right)\circ\rho^{-1}\colon SH\xrightarrow{\approx}q\left(S\right)\times H.
\]
Lastly, if we insist on the assumption that a manifold is always Hausdorff
and second-countable, then the coset space $G/H$ will also inherit
these two properties from $G$. Indeed, in this case $G$ (and hence
$H$) will be metrizable, and so the assumption that the subgroup
$H\leq G$ is closed implies that the canonical action $G\curvearrowleft H$
is proper, hence the orbit space $G/H$ is Hausdorff; on the other
hand, the coset quotient projection $q\colon G\to G/H$ is open, so
the image $G/H$ inherits the second countability from $G$ too, as
desired. Further, if the Lie group $G$ is modeled on a separable
Fréchet space $B$ so that the closed Lie subgroup $H\leq G$ is modeled
on a separable closed Fréchet subspace $A\leq B$, then the complementary
subspace will serve as the model space for $q\left(S\right)\subsetopen G/H$
which is a separable closed Fréchet subspace, as desired. This completes
the proof of \prettyref{prop:manifold-slice-condition}.
\end{proof}

\chapter{Algebraic Topology\label{chap:Algebraic-Topology}}

There are two main goals in this chapter. The first goal is to construct
for smooth $F$-fibrations on $E$ their moduli space, which will
particularly rely on two aspects: the classification problem $\ClassFib\left(E,F\right)$
and the symmetry analysis $\Aut\left(\StandardFibration\right)$ for
each class, as packaged in the following display:\footnote{Recall again our notational convention: to help make a distinction
between finite and infinite dimensions, our choice of notation for
an infinite-dimensional manifold often includes a script letter; e.g.,
$\Diff\left(M\right)$, $\Aut\left(\StandardFibration\right)$, and
$\Fib\left(\StandardFibration\right)$ as seen here.}
\[
\bigsqcup_{\StandardFibration\in\ClassFib\left(E,F\right)}\Fib\left(\StandardFibration\right),\qquad\Fib\left(\StandardFibration\right)\coloneqq\Diff\left(E\right)/\Aut\left(\StandardFibration\right).
\]
The second goal is to study the classification and the symmetry from
the perspective of algebraic topology. As we shall see, both aspects
have natural places in each of the following two frameworks: cohomology
or homotopy; more specifically, the former draws on the cohomology
theory of Čech cocycles with values in the group $\Diff\left(F\right)$
of fiber diffeomorphisms, and the latter draws on the homotopy theory
of classifying maps with values in the space $\Shape\left(F,\ell^{2}\right)$
of (universal) fiber shapes. As a particular common theme in this
chapter, relation between the topological and smooth categories will
be explored.

\section{Smooth Fibrations: Classification, Symmetry, Moduli\label{sec:Smooth-Fibrations}}

The goal of this section is to construct the moduli space $\Fib\left(E,F\right)$
of smooth regular $F$-fibration structures (or ``$F$-fiberings''
for short) on $E$. More specifically, we shall introduce the following
three key ingredients in constructing the moduli for such fiberings:
the classification $\ClassFib\left(E,F\right)$ (\prettyref{def:Classification}),
the total transformation group $\Diff\left(E\right)$ (\prettyref{def:transformation}),
and the symmetry group $\Aut\left(\StandardFibration\right)$ for
each fibering $\StandardFibration$ (\prettyref{def:automorphism});
these three ingredients will then culminate in the following definition
of our desired moduli space (\prettyref{def:space-of-fiberings}):
\[
\Fib\left(E,F\right)\coloneqq\bigsqcup_{\StandardFibration_{\alpha}\in\ClassFib\left(E,F\right)}\Diff\left(E\right)/\Aut\left(\StandardFibration_{\alpha}\right).
\]
This space $\Fib\left(E,F\right)$, with the natural topology and
geometry inherited from the diffeomorphism group $\Diff\left(E\right)$,
will be our main object of study.

\subsection{The \textquotedblleft space of structures\textquotedblright{} and
regular fiberings}

In general, for a suitable class of differentio-geometric structures
on the given manifold $E$, one may construct the corresponding ``moduli''
space (parametrizing all such structures on $E$) as the following
disjoint union of orbit spaces:
\begin{equation}
\bigsqcup_{\left\{ \StandardFibration_{\alpha}\right\} _{\alpha}}\mathscr{D}\left(E\right)/\mathscr{A}\left(\StandardFibration_{\alpha}\right).\label{eq:constructing-moduli-in-general}
\end{equation}
Here, the three ingredients are described as follows: first, the \emph{classification}
$\left\{ \StandardFibration_{\alpha}\right\} _{\alpha}$ is a complete
set of representatives of such structures modulo the natural structure-preserving
diffeomorphisms on $E$; second, the \emph{deformation} $\mathscr{D}\left(E\right)$
is a space consisting of diffeomorphisms on $E$ under which each
representative structure $\StandardFibration_{\alpha}$ is deformed
all through its equivalence class; and third, the \emph{symmetry}
(or \emph{automorphism}) $\mathscr{A}\left(\StandardFibration_{\alpha}\right)$
of each representative structure $\StandardFibration_{\alpha}$ is
group sitting inside $\mathscr{D}\left(\StandardFibration_{\alpha}\right)$
characterized by the property of preserving $\StandardFibration_{\alpha}$.
In our case, the structures on $E$ of our interest are those fibration-like
structures. More specifically, in the present work we shall restrict
attention to a most primitive type of these:
\[
\Fib\left(E,F\right)=\left\{ \text{smooth regular \ensuremath{F}-fibration structures (or ``fiberings'') on \ensuremath{E}}\right\} .
\]
By definition, such a (smooth, regular) $F$-fibering on $E$ can
be represented by a smooth fiber bundle $\StandardFibration$ with
total space $E$ and model fiber $F$, uniquely up to changes of coordinates
on base spaces by diffeomorphisms. In other words, each $F$-fibering
on $E$ has a unique representative $\xi$ of the form
\begin{equation}
\StandardFibration\colon F\hookrightarrow E\xrightarrow{\StandardProjection}\left[B\right],\qquad\left[B\right]\in\ClassMan\left(m\right),\label{eq:fibering-represented-by-bundle}
\end{equation}
where $\ClassMan\left(m\right)$ denotes the set of all diffeomorphism
types of closed smooth $m$-manifolds (with $m$ necessarily equal
to the codimension $\dim E-\dim F$), and by \prettyref{eq:fibering-represented-by-bundle}
we mean that any two of these $F$-fiber bundles $\pi_{i}\colon E\to B_{i}$
($i=1,2$) represent the same fibering if there exists a diffeomorphism
$\beta\colon B_{1}\approx B_{2}$ with $\beta\circ\pi_{1}=\pi_{2}$.
This perspective via the fiber-bundle representation \prettyref{eq:fibering-represented-by-bundle}—or
sometimes written as $F\hookrightarrow E\to B$ by a slight abuse
of notation—will help us study fibering structures in a concrete and
precise way; in particular, we shall next use this to describe explicitly
for regular fiberings the three ingredients needed in \prettyref{eq:constructing-moduli-in-general}:
classification, deformations, and symmetries.

\subsection{The classification}

The first ingredient, the classification, is given by a complete set
of representatives for the $F$-fiberings on $E$ modulo the natural
fibering equivalence. Here, this equivalence relation can be described
concretely in terms of their bundle representations \prettyref{eq:fibering-represented-by-bundle}:
for any two $F$-fiberings $\StandardFibration$ and $\StandardFibration'$
represented by fiber bundles $\StandardProjection_{\StandardFibration}\colon E\to B$
and $\StandardProjection_{\mathcal{\StandardFibration}'}\colon E\to B'$,
respectively, we say that they are \emph{(fibering) equivalent} if
there exist a pair of diffeomorphisms $h\colon E\to E$ and $\underline{h}\colon B\to B'$
satisfying $\StandardProjection_{\StandardFibration'}=\underline{h}\circ\StandardProjection_{\StandardFibration}\circ h^{-1}$,
as in the following commutative diagram:
\begin{equation}
\begin{split}\xymatrix@C=4pc{E\ar[r]_{\approx}^{\exists h}\ar[d]_{\StandardProjection_{\StandardFibration}} & E\ar[d]^{\StandardProjection_{\mathcal{\StandardFibration}'}}\\
B\ar[r]_{\exists\underline{h}}^{\approx} & B'
}
\end{split}
\label{eq:the-classification-diagram}
\end{equation}
Note that such a relation $\StandardFibration\sim\StandardFibration'$
for fiberings is well-defined (independent of the choice of bundle
representations $\StandardProjection_{\StandardFibration}$ and $\StandardProjection_{\mathcal{\StandardFibration}'}$).
Thus we obtain a desired formulation for the classification of fiberings:
\begin{defn}
\label{def:Classification}The \emph{classification} of $F$-fiberings
on $E$, denoted by $\ClassFib\left(E,F\right)$, is a complete list
of representatives for smooth $F$-fiber bundle structures on $E$
modulo the fibering equivalence $\StandardFibration\sim\StandardFibration'$
according to \prettyref{eq:the-classification-diagram}.
\end{defn}

The determination of $\ClassFib\left(E,F\right)$ will be addressed
later in this chapter, where theories of classifying spaces, nonabelian
Čech cohomology, and mapping class group actions, among others, will
be drawn on. For now, let us fix an $F$-fibering $\StandardFibration$
on $E$ (as represented by an $F$-fiber bundle $\StandardProjection_{\StandardFibration}\colon E\to B$),
and consider next how $\StandardFibration$ is deformed all through
its equivalence class.

\subsection{The deformations (diffeomorphisms)}

The second ingredient, deformations of a fibering $\StandardFibration$,
are universally supplied by the diffeomorphisms on $E$ that deforms
$\StandardFibration$ to all its equivalent fiberings. Here, these
deformations for fiberings can be described concretely in terms of
their bundle representations \prettyref{eq:fibering-represented-by-bundle}:
under a diffeomorphism $f\colon E\to E$, the \emph{pushforward}
of $\StandardFibration$ is an $F$-fibering $f_{\ast}\StandardFibration$
on $E$ represented by the fiber bundle $\StandardProjection_{f_{\ast}\StandardFibration}\coloneqq\StandardProjection_{\StandardFibration}\circ f^{-1}$,
as in the following commutative diagram
\begin{equation}
\begin{split}\xymatrix@C=4pc{E\ar[r]_{\approx}^{f}\ar[d]_{\StandardProjection_{\StandardFibration}} & E\ar@{-->}[d]^{\StandardProjection_{f_{\ast}\StandardFibration}}\\
B\ar@{=}[r]_{\Identity_{B}} & B
}
\end{split}
\label{eq:the-deformations-diagram}
\end{equation}
Note that such deformations $\StandardFibration\mapsto f_{\ast}\StandardFibration$
for fiberings are well-defined (independent of the choice $\StandardProjection_{\StandardFibration}$)
and exhaustive. Thus we obtain a desired formulation for the deformations
of fiberings:
\begin{defn}
\label{def:transformation}The \emph{deformation} of $\StandardFibration$
refers to the topological group $\Diff\left(E\right)$ consisting
of all diffeomorphisms $f\colon E\to E$, together with its transitive
action on the equivalence class of $\StandardFibration$ by the pushforward
$\StandardFibration\mapsto f_{\ast}\StandardFibration$ according
to \prettyref{eq:the-deformations-diagram}.
\end{defn}

Various aspects concerning the topology and geometry of the diffeomorphism
group $\Diff\left(E\right)$ (with its canonical structure constructed
in \prettyref{thm:Canonical-Smooth-Structure-on-DiffM-appendix})
have been extensively studied by others, for which we recommend interested
readers to \cite{MR1445290,MR2568571,MR2456522,hatcher201250}. We
shall see in the present chapter and the next that the space of fiberings
inherits these rich structures as a homogeneous space of the diffeomorphism
group.

\subsection{The symmetries (automorphisms)}

The third ingredient, symmetries of a fibering $\StandardFibration$,
are supplied by those diffeomorphisms on $E$ that push forward $\StandardFibration$
to itself. Here, these symmetries for fiberings can be described concretely
in terms of their bundle representations \prettyref{eq:fibering-represented-by-bundle}:
a diffeomorphism $h\colon E\to E$ is said to \emph{preserve} $\StandardFibration$
if there exists a diffeomorphism $\underline{h}\colon B\to B$ satisfying
$\StandardProjection_{\StandardFibration}\circ h=\underline{h}\circ\StandardProjection_{\StandardFibration}$,
as in the following commutative diagram:
\begin{equation}
\begin{split}\xymatrix@C=4pc{E\ar[r]_{\approx}^{h}\ar[d]_{\StandardProjection_{\StandardFibration}} & E\ar[d]^{\StandardProjection_{\mathcal{\StandardFibration}}}\\
B\ar[r]_{\exists\underline{h}}^{\approx} & B
}
\end{split}
\label{eq:the-symmetries-diagram}
\end{equation}
Note that such symmetries $h_{\ast}\StandardFibration=\StandardFibration$
for fiberings are well-defined (independent of the choice of $\StandardProjection_{\StandardFibration}$)
and exhaustive. Thus we obtain a desired formulation for the symmetries
of fiberings:
\begin{defn}
\label{def:automorphism}The \emph{symmetry} or \emph{automorphism group}
of $\StandardFibration$, denoted by $\Aut\left(\StandardFibration\right)$,
is the topological subgroup of $\Diff\left(E\right)$ consisting of
those diffeomorphisms $h\colon E\to E$ that preserve $\StandardFibration$
according to \prettyref{eq:the-symmetries-diagram}.
\end{defn}

The structures (topological, homotopy, differential, etc.)\ of the
automorphism group $\Aut\left(\StandardFibration\right)$ will be
addressed in the present chapter and the next, where theories of gauge
analysis, (a recurrence of) nonabelian cohomology, infinite-dimensional
group extension, and Riemannian geometry of fibrations, among others,
will be drawn on.

\subsection{The moduli space}

Having defined for regular fiberings the classification $\ClassFib\left(E,F\right)$,
the deformation $\Diff\left(E\right)$, and the symmetry $\Aut\left(\StandardFibration\right)$,
we are now ready to construct the space of fiberings. By the very
constructions of these three ingredients, we see that for each representative
$\StandardFibration_{\alpha}\in\ClassFib\left(E,F\right)$, all the
$F$-fiberings on $E$ within the equivalence class of $\StandardFibration_{\alpha}$
can be bijectively parametrized by the coset space $\Diff\left(E\right)/\Aut\left(\StandardFibration_{\alpha}\right)$.
Thus our considerations culminate in the following definition:
\begin{defn}
\label{def:space-of-fiberings}The \emph{space of (smooth, regular) $F$-fiberings on $E$},
denoted by $\Fib\left(E,F\right)$, is the following disjoint union
of left coset spaces:
\[
\Fib\left(E,F\right)\coloneqq\bigsqcup_{\StandardFibration\in\ClassFib\left(E,F\right)}\Fib\left(\StandardFibration\right)\qquad\text{with}\qquad\Fib\left(\StandardFibration\right)\coloneqq\Diff\left(E\right)/\Aut\left(\StandardFibration\right).
\]
\nomenclature[Fib-EF]{$\Fib\left(E,F\right)$}{the moduli space of smooth fiberings of $E$ by $F$}\nomenclature[Fib]{$\Fib\left(\StandardFibration\right)$}{the moduli space of smooth fiberings modeled on $\StandardFibration$ (or the oriented version thereof)}Specifically,
each subspace $\Fib\left(\StandardFibration\right)\subseteq\Fib\left(E,F\right)$
is called the \emph{space of fiberings modeled on $\StandardFibration$},
which consists of left cosets $f\Aut\left(\StandardFibration\right)$ (for $f\in\Diff\left(E\right)$) parametrizing the equivalent fiberings
$f_{\ast}\StandardFibration$.
\end{defn}

This is the main object of our study; especially, its homotopy type
will be one of the driving quests in our investigation. In particular,
in what follows we shall lay general groundwork in this regard, equipped
with which in \prettyref{chap:First-Examples} we shall determine
case by case the homotopy type of $\Fib\left(E,F\right)$ for $\dim E\leq3$.%
Before we embark on this, we remark that although $\Fib\left(\StandardFibration\right)$
was constructed as the space of cosets for diffeomorphisms modulo
automorphisms, in practice it is often convenient to realize these
cosets concretely; in other words, besides viewing it as a homogeneous
space of $\Diff\left(E\right)$, we seek to embed $\Fib\left(\StandardFibration\right)$
into certain intuitive (infinite-dimensional) manifolds, as explained
in what follows.

\subsection{Two views of a fibration: submersion vs.\  family of shapes}

Here we offer two realizations of $\Fib\left(\StandardFibration\right)$
in \prettyref{eq:fiberings-as-submersions} and in \prettyref{eq:fiberings-as-families-of-shapes}
below as corresponding to two perspectives for a fibration: as a bundle
projection and as a fiber family, respectively. On the one hand (in
the spirit of our construction above), the first perspective realizes
a fibering as an unindexed \emph{submersion}. More precisely, we
view the bundle projection $\StandardProjection_{\StandardFibration}\colon E\to B$
as a member in the set of all smooth submersions of $E$ onto $B$.
This set is denoted by $\Subm\left(E,B\right)$, which is an open
subset $\Subm\left(E,B\right)\subsetopen\MappingSpace^{\infty}\left(E,B\right)$
(\prettyref{lem:subm-and-emb-are-open-subsets}), and thereby inherits
a natural topology and smooth structure from the smooth mapping space
(see \prettyref{sec:Infinite-dimensional-smooth-manifolds}). Further,
this space admits two (commuting, smooth) natural Lie-group actions:
\[
\Diff\left(E\right)\curvearrowright\Subm\left(E,B\right)\curvearrowleft\Diff\left(B\right),
\]
where the pullback action by $\Diff\left(B\right)$ is given by post-composition
with inverses of base diffeomorphisms (which accounts for the ``base
forgetting'' that turns fiber bundles into fiberings), while the
pushforward action by $\Diff\left(E\right)$ is given by pre-composition
with inverses of total diffeomorphisms (which accounts for the deformations
of a fibering $\StandardFibration$ all through its equivalence class).
Thus in sum, the key takeaway is that $\Fib\left(\StandardFibration\right)$
can be realized as the $\Diff\left(E\right)$-orbit of $\StandardFibration$
living inside a space of unindexed submersions as follows:
\begin{equation}
\Fib\left(\StandardFibration\right)\ \hookrightarrow^{\Diff\left(E\right)}\ \Subm\left(E,B\right)/\Diff\left(B\right).\label{eq:fiberings-as-submersions}
\end{equation}
On the other hand, the second perspective realizes a fibering as an
unindexed \emph{family of shapes}. More precisely, we view each fiber
$E_{x}\coloneqq\StandardProjection_{\StandardFibration}^{-1}\left(x\right)$
as a member in the set of all smooth embedded $F$-submanifolds in
$E$ or, even better, in the orbit space $\Emb\left(F,E\right)/\Diff\left(F\right)$
of all smooth embeddings of $F$ into $E$ modulo the reparametrization
action by diffeomorphisms on $F$. This set is denoted by $\Shape\left(F,E\right)$
(and whose elements are sometimes called ``$F$-shapes'' in $E$),
which is an open quotient of the open subset $\Emb\left(F,E\right)\subsetopen\MappingSpace^{\infty}\left(F,E\right)$
(\prettyref{lem:subm-and-emb-are-open-subsets}), and thereby inherits
a natural topology and smooth structure from the smooth mapping space
(see \prettyref{sec:Infinite-dimensional-smooth-manifolds}). Thus
with all fiber-shapes of $\StandardFibration$ bundled together, the
resulting family $\left(E_{x}\right)_{x\in B}$ belongs to the space
$\MappingSpace^{\infty}\left(B,\Shape\left(F,E\right)\right)$ consisting
of all smooth $B$-families of $F$-shapes in $E$. Further, this
space admits two (commuting, smooth) natural Lie-group actions:
\[
\Diff\left(E\right)\curvearrowright\MappingSpace^{\infty}\left(B,\Shape\left(F,E\right)\right)\curvearrowleft\Diff\left(B\right),
\]
where the pullback action by $\Diff\left(B\right)$ is given by pre-composition
with base diffeomorphisms (which accounts for the ``base forgetting''
that turns fiber bundles into fiberings), while the pushforward action
by $\Diff\left(E\right)$ is given by deforming the fiber-shapes ``along
for the ride'' under ambient diffeomorphisms (which accounts for
the deformations of a fibering $\StandardFibration$ all through its
equivalence class). Thus in sum, the takeaway is that $\Fib\left(\StandardFibration\right)$
can be realized as the $\Diff\left(E\right)$-orbit of $\StandardFibration$
living inside a space of unindexed families of shapes as follows:\footnote{In fact, such a family of embedded shapes resulting from a fiber bundle
must be an embedded family itself (interpreted appropriately), so
one may say that $\Fib\left(\StandardFibration\right)$ is contained
in $\Shape\left(B,\Shape\left(F,E\right)\right)$, the ``shape space
squared''.}%
\begin{equation}
\Fib\left(\StandardFibration\right)\ \hookrightarrow^{\Diff\left(E\right)}\ \MappingSpace^{\infty}\left(B,\Shape\left(F,E\right)\right)/\Diff\left(B\right).\label{eq:fiberings-as-families-of-shapes}
\end{equation}
It is worth noting that while every member of $\Subm\left(E,B\right)$
in \prettyref{eq:fiberings-as-submersions} yields a fiber bundle
hence a fibering (as a consequence of Ehresmann's fibration theorem),
this is no longer the case%
for $\MappingSpace^{\infty}\left(B,\Shape\left(F,E\right)\right)$
in \prettyref{eq:fiberings-as-families-of-shapes} — an arbitrary
family of fiber shapes in $E$ need not be bundled together to form
a fiber bundle structure, and the crux of the matter is in seeking
suitable characterizations of the success of such bundlings. In any
event, we shall find each of these two realizations useful in different
places throughout this work, and henceforth we shall freely think
of a fibering as an unindexed submersion
\prettyref{eq:fiberings-as-submersions} or as an unindexed family
of shapes \prettyref{eq:fiberings-as-families-of-shapes}.

\subsection{The spacetime symmetries (base transformations)}

Observe that between the symmetries of a fibering (\prettyref{def:automorphism})
and the classification of fiberings (\prettyref{def:Classification}),
there is a subtle interplay regarding coordinate changes of the base.
This can be encoded in the following definition:
\begin{defn}
\label{def:The-base-transformation-group}For each smooth fiber bundle
$\StandardFibration\colon F\hookrightarrow E\xrightarrow{\StandardProjection_{\StandardFibration}}B$,%
the group of \emph{base transformations} of $\StandardFibration$,
denoted by $\Diff\left(B\right)_{\StandardFibration}$, is a subgroup
of $\Diff\left(B\right)$ consisting of those base diffeomorphisms
that are covered by automorphisms of $\StandardFibration$:
\[
\Diff\left(B\right)_{\StandardFibration}\coloneqq\left\{ \beta\in\Diff\left(B\right)\mid\StandardProjection_{\StandardFibration}\circ h=\beta\circ\StandardProjection_{\StandardFibration},\ \exists h\in\Aut\left(\StandardFibration\right)\right\} .
\]
Further endow $\Diff\left(B\right)_{\StandardFibration}$ with the
subspace topology induced from the $C^{\infty}$-topology of $\Diff\left(B\right)$.
\end{defn}

In other words, the base transformation group $\Diff\left(B\right)_{\StandardFibration}$
of $\StandardFibration$ is the image of the continuous homomorphism
\begin{equation}
\AutProjection_{\StandardFibration}\colon\Aut\left(\StandardFibration\right)\to\Diff\left(B\right),\qquad h\mapsto\underline{h},\label{eq:main-projection-from-Aut}
\end{equation}
as in the following commutative diagram:
\[
\begin{split}\xymatrix@C=4pc{E\ar[r]_{\approx}^{h}\ar[d]_{\pi_{\StandardFibration}} & E\ar[d]^{\pi_{\StandardFibration}}\\
B\ar[r]_{\underline{h}}^{\approx} & B
}
\end{split}
\]
We shall return to this map $\AutProjection_{\StandardFibration}$
in later sections to study its other aspects (e.g., its kernel, its
fibration structure, etc.). But for now, let us focus on its image
$\Diff\left(B\right)_{\StandardFibration}$. We start with the trivial
case:
\begin{lem}
\label{prop:Base-transformation-group-for-trivial-bundles}If $\StandardFibration_{0}$
is a trivial bundle over $B$, then its base transformation group
$\Diff\left(B\right)_{\StandardFibration_{0}}$ attains the full diffeomorphism
group $\Diff\left(B\right)$.
\end{lem}

\begin{proof}
Let $\varphi\colon E\to B\times F$ be a global trivialization of
the trivial bundle $\StandardFibration_{0}$. Then every base diffeomorphism
$\beta\in\Diff\left(B\right)$ is clearly covered by an automorphism
$\widetilde{\beta}\in\Aut\left(\StandardFibration_{0}\right)$ given
by the formula $\widetilde{\beta}=\varphi^{-1}\circ\left(\beta\times\Identity_{F}\right)\circ\varphi$.
Thus we indeed have $\Diff\left(B\right)_{\StandardFibration_{0}}=\Diff\left(B\right)$
as desired.
\end{proof}
Let us next turn to the general case for an arbitrary fibration $\StandardFibration$.
The following lemma provides an alternative description for $\Diff\left(B\right)_{\StandardFibration}$
as the stabilizer of $\StandardFibration$ under the previously studied
pullback action, hence justifies our choice of the notation $\Diff\left(B\right)_{\StandardFibration}$.
\begin{prop}
\label{prop:The-base-transformation-group-as-stabilizer}The base
transformation group $\Diff\left(B\right)_{\StandardFibration}$ consists
exactly of those base diffeomorphisms that pull back $\StandardFibration$
to isomorphic fiber bundles:
\[
\Diff\left(B\right)_{\StandardFibration}=\left\{ \beta\in\Diff\left(B\right)\mid\StandardFibration\cong\beta^{\ast}\StandardFibration\ \left(\text{i.e., \ensuremath{\pi_{\StandardFibration}=\pi_{\beta^{\ast}\StandardFibration}\circ\eta} for some diffeo \ensuremath{\eta\colon E\to\beta^{\ast}E}}\right)\right\} .
\]
\nomenclature[Baut]{$\Diff\left(B\right)_{\StandardFibration}$}{the group of base transformations of a fibering $\StandardFibration$ (or the oriented version thereof)}In
other words, $\Diff\left(B\right)_{\StandardFibration}$ is the stabilizer
of $\StandardFibration$ under the natural pullback action of $\Diff\left(B\right)$
on the set of isomorphism classes of smooth $F$-fiber bundles over
$B$.
\end{prop}

\begin{proof}
Let $\beta\in\Diff\left(B\right)$ be an arbitrary diffeomorphism
of the base space. Recall that the pullback bundle $\beta^{\ast}\StandardFibration$
is an $F$-fiber bundle with bundle projection $\beta^{\ast}\pi\colon\beta^{\ast}E\to B$,
as in the following commutative diagram:
\begin{equation}
\begin{split}\xymatrix@C=4pc{\beta^{\ast}E\ar[r]^{\pi^{\ast}\beta}\ar[d]_{\beta^{\ast}\pi} & E\ar[d]^{\pi}\\
B\ar[r]_{\beta} & B
}
\end{split}
\label{eq:pullback-diagram}
\end{equation}
On the one hand, if $\beta\in\Diff\left(B\right)_{\StandardFibration}$,
say being covered by a fibering automorphism $h\in\Aut\left(\StandardFibration\right)$,
then the desired bundle isomorphism $\eta\colon\left(E,\StandardFibration\right)\to\left(\beta^{\ast}E,\beta^{\ast}\xi\right)$
is given by $\eta=\left(\pi^{\ast}\beta\right)^{-1}\circ h$; indeed,
the required condition $\pi=\beta^{\ast}\pi\circ\eta$ can be verified
directly, or by the following diagrammatic argument (which is basically
concatenating the inverse of the pullback diagram \prettyref{eq:pullback-diagram}):
\[
\begin{split}\xymatrix@C=4pc{E\ar[r]^{h}\ar[d]_{\pi} & E\ar[d]^{\pi}\\
B\ar[r]_{\beta} & B
}
\end{split}
\implies\begin{split}\xymatrix@C=4pc{E\ar[r]^{h}\ar[d]_{\pi} & E\ar[d]^{\pi}\ar[r]^{\left(\pi^{\ast}\beta\right)^{-1}} & \beta^{\ast}E\ar[d]^{\beta^{\ast}\pi}\\
B\ar[r]_{\beta} & B\ar[r]_{\beta^{-1}} & B
}
\end{split}
\]
On the other hand, if $\StandardFibration\cong\beta^{\ast}\StandardFibration$,
say being given by a bundle isomorphism $\eta\colon\left(E,\StandardFibration\right)\to\left(\beta^{\ast}E,\beta^{\ast}\xi\right)$,
then the desired fibering automorphism $h\in\Aut\left(\StandardFibration\right)$
covering $\beta$ is given by $h=\pi^{\ast}\beta\circ\eta$; indeed,
the required condition $\pi\circ h=\beta\circ\pi$ can be verified
directly, or by the following diagrammatic argument (which is basically
concatenating the pullback diagram \prettyref{eq:pullback-diagram}):
\[
\begin{split}\xymatrix@C=4pc{E\ar[d]_{\pi}\ar[r]^{\eta} & \beta^{\ast}E\ar[d]^{\beta^{\ast}\pi}\\
B\ar@{=}[r]_{\Identity_{B}} & B
}
\end{split}
\implies\begin{split}\xymatrix@C=4pc{E\ar[d]_{\pi}\ar[r]^{\eta} & \beta^{\ast}E\ar[r]^{\pi^{\ast}\beta}\ar[d]^{\beta^{\ast}\pi} & E\ar[d]^{\pi}\\
B\ar@{=}[r]_{\Identity_{B}} & B\ar[r]_{\beta} & B
}
\end{split}
\]
This completes the proof of the equivalence between $\beta\in\Diff\left(B\right)_{\StandardFibration}$
and $\StandardFibration\cong\beta^{\ast}\StandardFibration$, as desired.
\end{proof}
Thus there emerges a close relation between the way how bundle classes
get coarsened into fibering classes on the one hand, and how base
transformations are characterized among base diffeomorphisms on the
other. Namely, they are exactly the orbits and the stabilizers, respectively,
for the natural action on the set
\[
\Bundle_{F}\left(B\right)\coloneqq\left\{ \text{isomorphism classes of \ensuremath{F}-fiber bundles over \ensuremath{B}}\right\} 
\]
\nomenclature[BunFB]{$\Bundle_{F}\left(B\right)$}{the isomorphism classes of $F$-fiber bundles over $B$}by
the base diffeomorphism group or, better still, by the base mapping
class group 
\[
\mathrm{Mod}\left(B\right)\coloneqq\pi_{0}\Diff\left(B\right).
\]
\nomenclature[ModB]{$\mathrm{Mod}\left(B\right)$}{the mapping class group of $B$ (i.e., $\pi_{0}\Diff\left(B\right)$)}As
such, the orbit-stabilizer relation for this action says that the
bundle classes in the pullback orbit of $\StandardFibration$ are
exactly corresponding to the cosets of $\mathrm{Mod}\left(B\right)_{\StandardFibration}\coloneqq\pi_{0}\Diff\left(B\right)_{\StandardFibration}$
in $\mathrm{Mod}\left(B\right)$, as in the following corollary:
\begin{cor}
\label{cor:Orbit-stabilizer-relation-for-pullback}The base transformation
group $\Diff\left(B\right)_{\StandardFibration}\leq\Diff\left(B\right)$
is an open subgroup of the base diffeomorphism group, and hence equal
to the union of components corresponding to the subgroup $\mathrm{Mod}\left(B\right)_{\StandardFibration}\leq\mathrm{Mod}\left(B\right)$
of the base mapping-class group. Further, via the orbit-stabilizer
theorem, the natural pullback action of $\mathrm{Mod}\left(B\right)$
on $\Bundle_{F}\left(B\right)$ induces the following bijective correspondence:
\[
\mathrm{Mod}\left(B\right)/\mathrm{Mod}\left(B\right)_{\StandardFibration}\leftrightarrow\left\{ \beta^{\ast}\StandardFibration\in\Bundle_{F}\left(B\right)\mid\beta\in\Diff\left(B\right)\right\} .
\]
In particular, $\Diff\left(B\right)_{\StandardFibration}$ attains
the full diffeomorphism group $\Diff\left(B\right)$ if and only if
the bundle isomorphism class of $\StandardFibration$ remains unchanged
under all pullbacks.
\end{cor}

\subsection{The gauge symmetries (vertical automorphisms)}

Complemented by the base transformations are the ``internal'' or
``gauge'' symmetries, as given in the following definition:
\begin{defn}
For each smooth fiber bundle $\StandardFibration\colon F\hookrightarrow E\xrightarrow{\StandardProjection_{\StandardFibration}}B$,
the group of \emph{vertical automorphisms} of $\StandardFibration$,
denoted by $\Vau\left(\StandardFibration\right)$, is a subgroup of
$\Aut\left(\StandardFibration\right)$ consisting of those automorphisms
that map each fiber of $\StandardFibration$ to itself:
\[
\Vau\left(\StandardFibration\right)\coloneqq\left\{ k\in\Diff\left(E\right)\mid\StandardProjection_{\StandardFibration}\circ k=\StandardProjection_{\StandardFibration}\right\} .
\]
\nomenclature[Vaut]{$\Vau\left(\StandardFibration\right)$}{the group of vertical automorphisms of a fibering $\StandardFibration$ (or the oriented version thereof)}Further
endow $\Vau\left(\StandardFibration\right)$ with the subspace topology
induced from the $C^{\infty}$-topology of $\Diff\left(E\right)$.
\end{defn}

The vertical automorphism group $\Vau\left(\StandardFibration\right)$
is clearly a closed, normal subgroup of $\Aut\left(\StandardFibration\right)$,
but it is a priori not clear what general results can be said about
its topological or geometric structures. To get hold of $\Vau\left(\StandardFibration\right)$,
let us first consider the case of \emph{trivial} bundle $\StandardFibration_{0}$,
for which under any global trivialization $\varphi\colon E\to B\times F$,
each vertical automorphism $k\in\Vau\left(\StandardFibration_{0}\right)$
admits a global coordinate representation
\begin{equation}
\kappa\colon B\to\Diff\left(F\right),\qquad\varphi\circ k\circ\varphi^{-1}\left(x,y\right)=\left(x,\kappa\left(x\right)\left(y\right)\right).\label{eq:global-coordinate-representation-of-gauge}
\end{equation}
Thus for trivial bundle $\StandardFibration_{0}$, the assignment
$k\mapsto\kappa$ yields a description of $\Vau\left(\StandardFibration_{0}\right)$
in terms of the space of mappings from $B$ to $\Diff\left(F\right)$.
In general, for any possibly infinite-dimensional Lie group $\mathcal{K}$,
the topological (resp., smooth) mapping space of the form $C\left(B,\mathcal{K}\right)$
(resp. $C^{\infty}\left(B,\mathcal{K}\right)$) is called a topological
(resp., smooth) \emph{current group}\emph{ }with values in $\mathcal{K}$,
where the group structure refers to the one induced from $\mathcal{K}$
pointwise. Suitable smooth approximations for such current groups
— especially for those with coefficient groups such as $\mathcal{K}=\Diff\left(F\right)$
— are already available; e.g., see \cite[Appendix A.3]{MR1935553}
or \cite{MR4157915}, which applies to show that the following natural
inclusion is a homotopy equivalence with dense image:
\begin{equation}
C^{\infty}\left(B,\Diff\left(F\right)\right)\xhookrightarrow{\simeq}C\left(B,\Diff\left(F\right)\right).\label{eq:smooth-approx-of-current-groups}
\end{equation}
In summary, from the preceding discussion we deduce that the vertical
automorphism group for a trivial bundle can be described in terms
of a suitable current group, as in the following lemma:
\begin{lem}
\label{lem:The-Vau-group-for-trivial-bundles}Suppose that $\StandardFibration_{0}$
is a trivial $F$-fiber bundle over $B$. Then its vertical automorphism
group is isomorphic to the smooth current group, which is in turn
homotopy equivalent to the topological current group:
\[
\Vau\left(\StandardFibration_{0}\right)\cong C^{\infty}\left(B,\Diff\left(F\right)\right)\simeq C\left(B,\Diff\left(F\right)\right).
\]
More precisely, the first map is an isomorphism given by taking the
global coordinate representation as in \prettyref{eq:global-coordinate-representation-of-gauge},
and the second map is an injective homotopy equivalence with dense
image induced by the forgetful map.
\end{lem}

For a general fibration $\StandardFibration$ that is merely locally
trivial, we shall see that the preceding lemma can serve as a partial
description of $\Vau\left(\StandardFibration\right)$ over each trivializing
chart. There is much more to say concerning this group, for which
we shall defer to the subsequent sections for a detailed study.

\subsection{Principal bundles and tangentially-smooth fibrations}

Smooth approximation, such as the homotopy equivalence \prettyref{eq:smooth-approx-of-current-groups}
used in the proof of \prettyref{lem:The-Vau-group-for-trivial-bundles}
above, will become a recurring theme throughout our study. A natural
way to explain this in the current contexts is via the theory of \emph{principal}
bundles. To make the connection, recall that for any smooth model
fiber $F$ and smooth base space $B$, there is a natural correspondence:%
\begin{equation}
\text{\ensuremath{C^{\infty}}-smooth \ensuremath{F}-fiber bundles over \ensuremath{B}}\ \leftrightarrow\ \text{smooth principal \ensuremath{\Diff\left(F\right)}-bundles over \ensuremath{B}}.\label{eq:correspondence-bun-and-pbun}
\end{equation}
Here, in the forward construction, corresponding to every smooth $F$-fiber
bundle $\StandardFibration$ over $B$ there is a \emph{nonlinear frame bundle}
$\mathrm{Fr}_{\Diff\left(F\right)}\left(\StandardFibration\right)$;
more specifically, this is a smooth principal $\Diff\left(F\right)$-bundle
over $B$ whose fiber over each $x$ consists of all diffeomorphisms
from the model fiber $F$ to the concrete fiber $\pi_{\StandardFibration}^{-1}\left(x\right)$
(i.e., all parametrizations of that concrete fiber):
\begin{equation}
\mathrm{Fr}_{\Diff\left(F\right)}\left(\StandardFibration\right)\coloneqq\bigsqcup_{x\in B}\Diff\left(F,\pi_{\StandardFibration}^{-1}\left(x\right)\right).\label{eq:nonlinear-frame-bundle}
\end{equation}
Conversely, in the backward construction, corresponding to every smooth
principal $\Diff\left(F\right)$-bundle $P$ over $B$ there is an
\emph{associated fiber bundle} $P\times_{\Diff\left(F\right)}F$;
more specifically, this is a smooth $F$-fiber bundle given by the
usual quotient construction on the product $P\times F$ under the
natural $\Diff\left(F\right)$-actions on both factors (i.e., the
principal right-action on $P$ and the evaluation left-action on $F$):
\begin{equation}
P\times_{\Diff\left(F\right)}F\coloneqq{P\times F}/{\sim}\qquad\text{with}\qquad\left(p\cdot k,y\right)\sim\left(p,k\left(y\right)\right)\enskip\left(\text{for }k\in\Diff\left(F\right)\right).\label{eq:associated-fiber-bundle}
\end{equation}
In this view, there emerges a natural regularity class for fiber bundles
that particularly fits in the topological setting. Indeed, in the
above correspondence \prettyref{eq:correspondence-bun-and-pbun} if
we considered continuous principal $\Diff\left(F\right)$-bundles
in the topological category, then we would be naturally led to the
notion of ``\emph{$C^{0,\infty}$}'' or ``\emph{tangentially smooth}''
$F$-fiber bundles; namely, topological fiber bundles with structure
group $\Diff\left(F\right)$. The above correspondence \prettyref{eq:correspondence-bun-and-pbun}
can then be extended to the more flexible topological category as
follows:
\[
\text{\ensuremath{C^{0,\infty}}-smooth \ensuremath{F}-fiber bundles over \ensuremath{B}}\ \leftrightarrow\ \text{topological principal \ensuremath{\Diff\left(F\right)}-bundles over \ensuremath{B}}.
\]
In this way, the ``forgetful map'' \prettyref{eq:smooth-approx-of-current-groups}
for $\Diff\left(F\right)$-valued current groups can be viewed as
the natural inclusion of the vertical automorphism group of a trivial
$F$-fiber bundle $\StandardFibration_{0}$ (viewed as a $C^{\infty}$-smooth
bundle by default) into that of $\StandardFibration_{0}$ viewed as
a $C^{0,\infty}$-smooth bundle. This exemplifies a recurring theme
throughout this work, where suitable smooth approximations will be
established to allow for reducing the study of $C^{\infty}$ fibrations
to that of $C^{0,\infty}$ fibrations in the more flexible topological
realm.

\subsection{The big picture}

Before we begin the detailed study, let us outline the big picture
connecting all these key ingredients just introduced in this section.
The first step is classification, for which recall that we can draw
on the classical classification theory of bundles and of manifolds
in the following way:
\[
\ClassFib\left(E,F\right)=\bigsqcup_{B\in\ClassMan\left(m\right)}\Bundle_{F}\left(B\right)/\mathrm{Mod}\left(B\right).
\]
Then for each fibering class in $\ClassFib\left(E,F\right)$, say
with a representative fibration $\StandardFibration\colon F\hookrightarrow E\xrightarrow{\StandardProjection_{\StandardFibration}}B$,
recall that there are three transformation groups at play; i.e., diffeomorphisms
(the deformation), automorphisms (the symmetry), and vertical automorphisms
(the gauge), which form a chain of (Lie) subgroups as follows:
\[
\Vau\left(\StandardFibration\right)\trianglelefteq\Aut\left(\StandardFibration\right)\leq\Diff\left(E\right).
\]
Between them, the quotient (resp., coset space) construction yields
the following sequence of groups (resp., pointed spaces):
\[
\Vau\left(\StandardFibration\right)\hookrightarrow\Aut\left(\StandardFibration\right)\xtwoheadrightarrow{}{\AutProjection_{\xi}}\Diff\left(B\right)_{\StandardFibration},\qquad\Aut\left(\StandardFibration\right)\hookrightarrow\Diff\left(E\right)\xtwoheadrightarrow{}{\DiffProjection_{\xi}}\Fib\left(\StandardFibration\right).
\]
These spaces at play will have rich structures as inherited from the
diffeomorphism groups of various pertinent compact manifolds (i.e.,
the ambient $E$, the fiber $F$, and the base $B$); for example,
we have seen that $\Diff\left(B\right)_{\StandardFibration}$ is a
union of path components corresponding to a certain subgroup of the
mapping class group $\pi_{0}\Diff\left(B\right)$, while $\Vau\left(\StandardFibration\right)$
is a certain twisted version of the current group $\MappingSpace^{\infty}\left(B,\Diff\left(F\right)\right)$.
Individually, these are central objects studied in a diverse range
of fields; e.g., the following gives a taste:
\[
\begin{split}\xymatrix@R=1pc{\Vau\left(\StandardFibration\right)\ar@{^{(}->}[d]\ar@{..}[r] & \text{infinite-dimensional gauge theory}\\
\Aut\left(\StandardFibration\right)\ar@{->>}[d]\ar@{..}[r] & \text{infinite-dimensional Lie group extensions}\\
\Diff\left(B\right)_{\StandardFibration}\ar@{..}[r] & \text{mapping class group action on certain cohomology/homotopy}
}
\end{split}
\]
and
\[
\begin{split}\xymatrix@R=1pc{\Aut\left(\StandardFibration\right)\ar@{^{(}->}[d]\ar@{..}[r] & \text{symmetries of a fibering, embedded as a Lie subgroup}\\
\Diff\left(E\right)\ar@{->>}[d]\ar@{..}[r] & \text{generalized ``Smale conjectures''; intrinsic flows (e.g., Ricci)}\\
\Fib\left(\StandardFibration\right)\ar@{..}[r] & \text{geometric analysis of extrinsic flows}.
}
\end{split}
\]
Lastly, an essential feature of both these sequences, as we shall
see, is that they are in fact themselves (infinite-dimensional) \emph{principal fibrations}
in various categories: topological, smooth, and even Riemannian (in
suitable sense), so that the topological and geometric structures
are deeply interconnected among these spaces. With this in mind, we
are ready to embark on our study.

\section{Cohomology via Čech Theory\label{sec:Cohomology-via-Cech-theory}}

\subsection{Prelude: nonabelian Čech cohomology\label{subsec:Nonabelian-Cech-cohomology}}

In this section, we review the basics of Čech cohomology theory with
values in a group $\mathcal{K}$, with an eye towards the case where
$\mathcal{K}$ is typically infinite-dimensional and nonabelian. This
will be used to facilitate various local studies in this work, where
$\mathcal{K}$ will be typically taken as the fiber diffeomorphism
group of a fiber bundle.

Let $\mathcal{U}\coloneqq\left(U_{i}\right)_{i\in I}$ be an open
cover of $B$,\footnote{We use the traditional notation $\mathcal{U}$ for open cover, even
though it clashes with our notational convention that calligraphic/script
letters generally indicate certain infinite dimensionality in nature.
(This will be the only exception in this work.)} whose index set $I$ is assumed
to be finite (by virtue of the compactness assumption of $B$) and
will be often suppressed from the notation henceforth. Then just as
in the familiar abelian case, the $q$-cochain group for each $q\in\mathbb{N}_{0}$
is defined as the following product of current groups:
\begin{equation}
\check{C}^{q}\left(\mathcal{U},\mathcal{K}\right)\coloneqq\prod_{i_{0},\dots,i_{q}\in I}\MappingSpace\left(U_{i_{0}}\cap\dots\cap U_{i_{q}},\mathcal{K}\right),\label{eq:cochain-groups}
\end{equation}
where the cases $q=0,1$ are spelled out as follows because they will
be of our primary interest in the nonabelian contexts:
\[
\check{C}^{0}\left(\mathcal{U},\mathcal{K}\right)\coloneqq\prod_{i\in I}\MappingSpace\left(U_{i},\mathcal{K}\right)\qquad\text{and}\qquad\check{C}^{1}\left(\mathcal{U},\mathcal{K}\right)\coloneqq\prod_{i,j\in I}\MappingSpace\left(U_{i}\cap U_{j},\mathcal{K}\right).
\]
Here, recall that the group structure on $\check{C}^{q}\left(\mathcal{U},\mathcal{K}\right)$
is induced from the coefficient group $\mathcal{K}$ pointwise at
each point $x\in B$ of the base. To ease the notations we shall suppress
the base variable $x$, which is understood to range over a well-defined
common domain that is clear from the context; for example, the usual
0-cocycle condition ``$\kappa_{i}\left(x\right)=\kappa_{j}\left(x\right)$
for all $x\in U_{i}\cap U_{j}$'' will be written as ``$\kappa_{i}=\kappa_{j}$''
for short. Since this 0-cocycle condition just says that local sections
agree on overlapped domains, we see that the $0$-cocycles form a
group that is naturally identified with the familiar \emph{current group}:
\begin{equation}
\check{Z}^{0}\left(\mathcal{U},\mathcal{K}\right)\coloneqq\left\{ \left(\kappa_{i}\right)_{i}\in\prod_{i\in I}\MappingSpace\left(U_{i},\mathcal{K}\right)\mid\kappa_{i}=\kappa_{j}\right\} \cong\MappingSpace\left(B,\mathcal{K}\right).\label{eq:zero-cocycle}
\end{equation}
However, for higher degrees, the abelian case no longer carries over
to the nonabelian case nicely. For instance, the naive coboundary
operator sending each 0-cochain $\left(\kappa_{i}\right)_{i}$ to
the 1-cochain $\left(\kappa_{i}^{-1}\kappa_{j}\right)_{i,j}$ is generally
no longer a group homomorphism. Similarly, the $1$-cocycles generally
no longer form a group but only a pointed set (which we call the \emph{cocycle set}):
\begin{equation}
\check{Z}^{1}\left(\mathcal{U},\mathcal{K}\right)\coloneqq\left\{ \left(g_{ij}\right)_{i,j}\in\prod_{i,j\in I}\MappingSpace\left(U_{i}\cap U_{j},\mathcal{K}\right)\mid g_{ij}g_{jk}=g_{ik}\right\} .\label{eq:one-cocycle-set-prelim}
\end{equation}
Though, at least for degree one, there is still a way to define the
notion of cohomologous in the general nonabelian case. Namely, instead
of quotienting out the coboundary-image of $\check{C}^{0}\left(\mathcal{U},\mathcal{K}\right)$
from $\check{Z}^{1}\left(\mathcal{U},\mathcal{K}\right)$ as in the
abelian case, we declare the desired cohomologous relation to be given
by the following natural action of $\check{C}^{0}\left(\mathcal{U},\mathcal{K}\right)$
on $\check{Z}^{1}\left(\mathcal{U},\mathcal{K}\right)$ (which we
call the \emph{coboundary action}):
\begin{equation}
\check{Z}^{1}\left(\mathcal{U},\mathcal{K}\right)\times\check{C}^{0}\left(\mathcal{U},\mathcal{K}\right)\to\check{Z}^{1}\left(\mathcal{U},\mathcal{K}\right),\qquad\left(\left(g_{ij}\right)_{i,j\in I},\left(\kappa_{i}\right)_{i\in I}\right)\mapsto\left(\kappa_{i}^{-1}g_{ij}\kappa_{j}\right)_{i,j\in I}.\label{eq:coboundary-action-prelim}
\end{equation}
Then for this coboundary action we can take the orbit set, for which
we can further take the direct limit over all open covers $\mathcal{U}$
of $B$ (with respect to the natural refinement homomorphisms); this
yields the \emph{(first) Čech cohomology set} (of $B$ with values
in $\mathcal{K}$):
\begin{equation}
\check{H}^{1}\left(B,\mathcal{K}\right)\coloneqq\varinjlim_{\mathcal{U}}{\check{H}^{1}\left(\mathcal{U},\mathcal{K}\right)}\qquad\text{with}\qquad\check{H}^{1}\left(\mathcal{U},\mathcal{K}\right)\coloneqq{\check{Z}^{1}\left(\mathcal{U},\mathcal{K}\right)}/{\check{C}^{0}\left(\mathcal{U},\mathcal{K}\right)}.\label{eq:cohomology-set-prelim}
\end{equation}
Complementarily,  for the above coboundary action we can also take
the stabilizer subgroup of each 1-cocycle $\tau$; this yields the
\emph{(first) Čech cocycle-stabilizer group of $\tau$} (over $\mathcal{U}$
with values in $\mathcal{K}$):
\begin{equation}
\check{C}^{0}\left(\mathcal{U},\mathcal{K}\right)_{\tau}\coloneqq\left\{ \left(\kappa_{i}\right)_{i}\in\prod_{i\in I}\MappingSpace\left(U_{i},\mathcal{K}\right)\mid\ensuremath{\kappa_{i}^{-1}\tau_{ij}\kappa_{j}=\tau_{ij}}\right\} ,\qquad\forall\tau\in\check{Z}^{1}\left(\mathcal{U},\mathcal{K}\right).\label{eq:cocycle-stablizer-prelim}
\end{equation}
If $\mathcal{K}$ is abelian, or more generally if $\tau$ centralizes
$\mathcal{K}$, then the stabilizing condition $\kappa_{i}^{-1}\tau_{ij}\kappa_{j}=\tau_{ij}$
is equivalent to the $0$-cocycle condition $\kappa_{i}=\kappa_{j}$,
so that the cocycle-stabilizer recovers the 0-cocycle group (i.e.,
the current group) $\MappingSpace\left(B,\mathcal{K}\right)$ in \prettyref{eq:zero-cocycle}.
However, this is far from being the case in general, especially for
our interested case where $\mathcal{K}$ is a diffeomorphism group
which is never centralized by any nontrivial 1-cocycle.%
\begin{rem}
\label{rem:smooth-cech-cohomology-theory}The preceding preliminary
exposition of nonabelian Čech theory, which was set up in the topological
category, can be easily adapted to the smooth category. More precisely,
given any Lie group $\mathcal{K}$ (possibly infinite-dimensional;
e.g., $\mathcal{K}=\Diff\left(F\right)$), the smooth counterpart
of the above exposition can be obtained by simply changing the building
blocks \prettyref{eq:cochain-groups} from topological cochain groups
to the following smooth cochain groups:
\[
\check{C}_{\mathrm{smooth}}^{q}\left(\mathcal{U},\mathcal{K}\right)\coloneqq\prod_{i_{0},\dots,i_{q}\in I}\MappingSpace^{\infty}\left(U_{i_{0}}\cap\dots\cap U_{i_{q}},\mathcal{K}\right).
\]
Here, the factors $\MappingSpace^{\infty}\left(U,\mathcal{K}\right)$
on the right-hand side are the smooth current groups as already seen
previously; i.e., the smooth mapping spaces equipped with the $C^{\infty}$-topology
(\prettyref{def:compact-open-C-inf-topology}) and the pointwise group
structure induced from $\mathcal{K}$. Then everything in the above
exposition carries over, mutatis mutandis, to the smooth category.
In particular, one can proceed verbatim to define the smooth counterparts
of the cocycle set \prettyref{eq:one-cocycle-set-prelim}, the coboundary
action \prettyref{eq:coboundary-action-prelim}, and hence the cohomology
set \prettyref{eq:cohomology-set-prelim} and the cocycle stabilizer
\prettyref{eq:cocycle-stablizer-prelim}.
\end{rem}

In what follows, we shall see that when specializing to the case $\mathcal{K}=\Diff\left(F\right)$,
the cohomology set $\check{H}^{1}\left(B,\mathcal{K}\right)$ and
the cocycle-stabilizer groups $\check{C}^{0}\left(\mathcal{U},\mathcal{K}\right)_{\tau}$'s
will account for the classification and the gauge analysis of $F$-fiber
bundles, respectively.

\subsection{Classification via cohomology: Čech cohomology}

Let us first consider the classification problem in the framework
of Čech cohomology theory. Our arguments below will work equally well
for either tangentially-smooth ($C^{0,\infty}$) or smooth ($C^{\infty}$)
fiber bundles, as long as we consider either topological or smooth
Čech cohomology accordingly (see \prettyref{rem:smooth-cech-cohomology-theory});
roughly speaking, the difference in the definitions between these
two categories resides in whether the topological or smooth cochain
groups are being used as building blocks:
\[
\check{C}^{q}\left(\mathcal{U},\mathcal{K}\right)\coloneqq\prod_{i_{0},\dots,i_{q}\in I}\MappingSpace\left(U_{i_{0},\dots,i_{q}},\mathcal{K}\right)\quad\text{or}\quad\check{C}_{\mathrm{smooth}}^{q}\left(\mathcal{U},\mathcal{K}\right)\coloneqq\prod_{i_{0},\dots,i_{q}\in I}\MappingSpace^{\infty}\left(U_{i_{0},\dots,i_{q}},\mathcal{K}\right).
\]
Thus for the purpose of exposition, let us restrict attention to
tangentially-smooth ($C^{0,\infty}$) fiber bundles, with the structure
group $\mathcal{K}=\Diff\left(F\right)$ set to be the full diffeomorphism
group. Let $\StandardFibration$ be such an $F$-fiber bundle over
$B$, and choose an open cover $\mathcal{U}\coloneqq\left(U_{i}\right)_{i}$
with local trivializations $\varphi_{i}\colon\pi^{-1}\left(U_{i}\right)\to U_{i}\times F$.
Then the corresponding transition maps can be described by the following
$\mathcal{K}$-valued data:
\begin{equation}
\tau_{ij}\colon U_{i}\cap U_{j}\to\mathcal{K}=\Diff\left(F\right),\qquad\varphi_{i}\circ\varphi_{j}^{-1}\bigl(x,y\bigr)=\left(x,\tau_{ij}\left(x\right)\left(y\right)\right).\label{eq:transition-cocycle}
\end{equation}
Note that such $\tau\coloneqq\left(\tau_{ij}\right)_{i,j}$ is characterized
by the condition $\tau_{ij}\tau_{jk}=\tau_{ik}$; in other words,
$\tau$ is a Čech $1$-cochain lying in the subspace of 1-cocycles:
\begin{equation}
\check{Z}^{1}\left(\mathcal{U},\Diff\left(F\right)\right)\coloneqq\left\{ \left(\tau_{ij}\right)_{i,j}\in\prod_{i,j}\MappingSpace\left(U_{i}\cap U_{j},\Diff\left(F\right)\right)\mid\tau_{ij}\tau_{jk}=\tau_{ik}\right\} .\label{eq:Cech-1-cocycle}
\end{equation}
Further, any two $F$-fiber bundles $\StandardFibration$ and $\StandardFibration'$
over a common trivialization cover $\mathcal{U}$ are equivalent to
each other if and only if their transition cocycles $\tau$ and $\tau'$
satisfy the relation $\kappa_{i}^{-1}\tau_{ij}\kappa_{j}=\tau'_{ij}$
for all $\kappa_{i}\colon U_{i}\to\Diff\left(F\right)$; in other
words, if and only if $\tau$ and $\tau'$ satisfy the cohomologous
relation given by the following coboundary action as seen in \prettyref{eq:coboundary-action-prelim}:
\begin{equation}
\check{Z}^{1}\left(\mathcal{U},\Diff\left(F\right)\right)\curvearrowleft\check{C}^{0}\left(\mathcal{U},\Diff\left(F\right)\right)\coloneqq\prod_{i}\MappingSpace\left(U_{i},\Diff\left(F\right)\right),\qquad\left(\tau,\kappa\right)\mapsto\left(\kappa_{i}^{-1}\tau_{ij}\kappa_{j}\right)_{i,j}.\label{eq:cochain-action-on-cocycle}
\end{equation}
Although a priori this depends on the choice of local trivialization,
changing charts amounts to conjugating $\tau$ by a cochain (which
gets identified under the cochain action), and refining $\mathcal{U}$
amounts to restricting $\tau$ (hence gets identified in the direct
limit). The orbit set of this action \prettyref{eq:cochain-action-on-cocycle}
has already made an appearance in the prelude subsection above, whose
definition is reprised as follows:
\begin{defn}
Fix a coefficient group $\mathcal{K}=\Diff\left(F\right)$ or any
topological subgroup thereof. Then the \emph{(first) Čech cohomology set}
of $B$, denoted by $\check{H}^{1}\left(B,\mathcal{K}\right)$, is
the pointed set
\[
\check{H}^{1}\left(B,\mathcal{K}\right)\coloneqq\varinjlim_{\mathcal{U}}\left(\check{Z}^{1}\left(\mathcal{U},\mathcal{K}\right)/\sim\right),\qquad\tau\sim\tau'\Leftrightarrow\exists\kappa\in\prod_{i}\MappingSpace\left(U_{i},\mathcal{K}\right),\,\kappa_{i}^{-1}\tau_{ij}\kappa_{j}=\tau_{ij}'.
\]
\nomenclature[Hcheck1BK]{$\check{H}^{1}\left(B,\mathcal{K}\right)$}{the (first) Čech cohomology set of $B$ (with coefficient group $\mathcal{K}$)}In
other words, $\check{H}^{1}\left(B,\mathcal{K}\right)$ is the direct
limit of the orbit sets for the natural cochain action \prettyref{eq:cochain-action-on-cocycle},
where the limit is taken over all open covers $\mathcal{U}$ of $B$
with respect to natural refinements.
\end{defn}

With this definition, the above discussion can be interpreted as giving
a map $\left[\StandardFibration\right]\mapsto\left[\tau\right]$ from
the classification set for $F$-fibrations to the $\Diff\left(F\right)$-valued
cohomology set. Conversely, every cohomology class can be represented
by a transition cocycle of some smooth fiber bundle (unique up to
smooth equivalence), as constructed in the following lemma:%

\begin{lem}
\label{lem:construction-of-bundle-from-cocycle}Every cohomology class
in $\check{H}^{1}\left(B,\Diff\left(F\right)\right)$ admits smooth
cocycle representatives $\tau$, each of which can be realized as
the transition cocycle of a smooth $F$-fiber bundle $E_{\tau}$ over
$B$ given by the following smooth manifold (equipped with the natural
quotient-manifold structure and bundle projection):
\begin{equation}
E_{\tau}\coloneqq\left(\bigsqcup_{i}U_{i}\times F\right)/\sim,\qquad\bigl(x,y\bigr)\sim\bigl(x',y'\bigr)\Leftrightarrow U_{i}\ni x=x'\in U_{j},\,y=\tau_{ij}\bigl(x\bigr)\bigl(y'\bigr).\label{eq:construction-of-bundle-from-cocycle}
\end{equation}
Moreover, any other choice of smooth representative $\tau$ results
in a smoothly equivalent fibration $E_{\tau}$.
\end{lem}

\begin{proof}
Given a cohomology class $\left[\tau^{0}\right]\in\check{H}^{1}\left(B,\Diff\left(F\right)\right)$
represented by any cocycle $\tau^{0}$, not necessarily smooth. Then
the same construction in \prettyref{eq:construction-of-bundle-from-cocycle}
yields a topological manifold $E_{\tau^{0}}$ with a canonical tangentially-smooth
fibration over $B$. I claim that this is topologically equivalent
to some smooth fibration $\widetilde{E_{\tau^{0}}}$, unique up to
smooth equivalence. Once we justify this claim, the transition cocycle
$\tau$ of $\widetilde{E_{\tau^{0}}}$ would yield a desired smooth
representative of $\left[\tau^{0}\right]$, with the uniqueness property
that the smooth fibration $E_{\tau'}$ constructed from any other
choice of smooth representative $\tau'$ would lie in the same smooth
equivalence class of $\widetilde{E_{\tau^{0}}}\cong E_{\tau}$, as
desired. Now to justify the above claim about the existence and uniqueness
of $\widetilde{E_{\tau^{0}}}$, we just need to establish a suitable
smoothing lemma for classification of fiber bundles. For easiness
of distinction in the following result, let us denote by $\Bundle_{F}^{\infty}\left(B\right)$
(resp., $\Bundle_{F}^{0,\infty}\left(B\right)$) the set of isomorphism
classes of smooth (resp., tangentially smooth) $F$-fiber bundles
over $B$. Then I claim the following:
\begin{quote}
\emph{The natural (pointed) map $\Bundle_{F}^{\infty}\left(B\right)\to\Bundle_{F}^{0,\infty}\left(B\right)$
induced by the forgetful functor is a bijective correspondence.}
\end{quote}
To prove this, we draw on the corresponding smoothing result from
the classification theory of principal bundles. Thus let $\PBundle^{\mathsf{Diff}}\left(B;\mathcal{K}\right)$
(respectively, $\PBundle^{\mathsf{Top}}\left(B;\mathcal{K}\right)$)
denote the set of equivalence classes of smooth (resp., topological)
principal $\mathcal{K}$-bundles over $B$, with the structure group
$\mathcal{K}=\Diff\left(F\right)$ set to be the full diffeomorphism
group. Then the goal of our proof is summarized in the following diagram:
\[
\begin{split}\xymatrix@C=4pc{\Bundle_{F}^{\infty}\left(B\right)\ar@{-->}[r]^{\cong} & \Bundle_{F}^{0,\infty}\left(B\right)\\
\PBundle^{\mathsf{Diff}}\left(B;\Diff\left(F\right)\right)\ar[r]_{\cong}\ar[u]^{\cong} & \PBundle^{\mathsf{Top}}\left(B;\Diff\left(F\right)\right)\ar[u]_{\cong}
}
\end{split}
.
\]
Here, each of the two vertical arrows is given by the construction
of associated fiber bundles as in \prettyref{eq:associated-fiber-bundle},
which is straightforwardly seen to induce a bijection on isomorphism
classes with inverse given by the construction of nonlinear frame
bundles as in \prettyref{eq:nonlinear-frame-bundle}. Alternatively,
one may also show this by representing both principal $\Diff\left(F\right)$-bundle
classes and $F$-fiber bundle classes by the same cocycle data up
to the coboundary action; more precisely, both sets on the left (respectively,
right) are identified with the $\Diff\left(F\right)$-valued first
Čech cohomology set in the smooth (respectively, topological) category.
The takeaway is, it suffices to show that the forgetful map on the
bottom row is indeed bijective on isomorphism classes for principal
bundles. More specifically,
\begin{itemize}
\item the surjectivity means that every topological principal $\Diff\left(F\right)$-bundle
over $B$ is equivalent to a smooth principal $\Diff\left(F\right)$-bundle;
and
\item the injectivity means that every pair of topologically-equivalent
smooth principal $\Diff\left(F\right)$-bundles over $B$ must be
smoothly equivalent too.
\end{itemize}
But both claims were already known to be true for $\mathcal{K}$-bundles
with $\mathcal{K}$ being any (possibly infinite-dimensional) Lie
group $\mathcal{K}$ modeled on a locally convex space — see e.g.,
\cite{MR2574141} — which can be applied to our case since $\Diff\left(F\right)$
admits a canonical Fréchet Lie group structure, whose underlying topology
indeed coincides with the $C^{\infty}$-topology. This completes the
proof that the forgetful map induces a bijection between smooth and
topological principal $\Diff\left(F\right)$-bundle classes, and hence
a bijection between smooth and tangentially-smooth $F$-fibration
classes. This completes the proof of the desired smooth approximation
for classification of fiber bundles, and hence the proof of \prettyref{lem:construction-of-bundle-from-cocycle}
as desired.
\end{proof}
By the preceding lemma, we in particular see that every $\Diff\left(F\right)$-valued
cohomology class can be assigned a well-defined diffeomorphism type,
hence the following definition is justified:
\begin{defn}
\label{def:diffeomorphism-type-of-cohomology}Associated with each
cohomology class in $\check{H}^{1}\left(B,\Diff\left(F\right)\right)$
is its \emph{total diffeomorphism type}, which is defined as the
unique diffeomorphism type of the smooth manifold $E_{\tau}$ as constructed
in \prettyref{lem:construction-of-bundle-from-cocycle} from any smooth
representative $\tau$. Given any smooth manifold $E$, denote the
subset 
\[
\check{H}^{1}\left(B,\Diff\left(F\right)\right)_{E}\coloneqq\left\{ \alpha\in\check{H}^{1}\left(B,\Diff\left(F\right)\right)\mid\text{\ensuremath{\alpha} has the total diff. type of \ensuremath{E}}\right\} .
\]
This subset of $\check{H}^{1}\left(B,\Diff\left(F\right)\right)$
will be called the ``\emph{cohomology subset of total type $E$}''
for short.
\end{defn}

We are now in a position to state and prove the following cohomology-theoretic
classification of smooth fiberings on a given smooth total space $E$.
(Below and throughout, recall that $\ClassMan\left(m\right)$ denotes
the set of all diffeomorphism types of closed smooth $m$-manifolds.)
\begin{prop}
\label{prop:classification-theorem-via-cohomology}Let $m\coloneqq\dim E-\dim F$.
The smooth $F$-fiberings on $E$ are classified — for each $m$-dimensional
diffeomorphism type for the base space $B$ — by the cohomology subset
of $\check{H}^{1}\left(B,\Diff\left(F\right)\right)$ of total type
$E$ (\prettyref{def:diffeomorphism-type-of-cohomology}) modulo pullback
by diffeomorphisms:
\begin{equation}
\ClassFib\left(E,F\right)\leftrightarrow\bigsqcup_{B\in\ClassMan\left(m\right)}{\check{H}^{1}\left(B,\Diff\left(F\right)\right)_{E}/\mathrm{Mod}\left(B\right)}.\label{eq:classification-theorem-via-cohomology}
\end{equation}
More precisely, for each class on the right-hand side, any choice
of smooth cocycle representative $\tau$ is assembled into a smooth
bundle $E_{\tau}$ as in \prettyref{eq:construction-of-bundle-from-cocycle},
which induces a desired smooth fibration on $E$ via a diffeomorphism
$E\approx E_{\tau}$ as guaranteed by assumption.
\end{prop}

\begin{proof}
Let $\Bundle_{F}\left(B\right)$ denote the set of equivalence classes
of smooth $F$-fiber bundles over $B$. Then by the preceding discussion,
the construction of taking transition cocycles for fiber bundles induces
the following isomorphism of pointed sets, with its inverse given
in \prettyref{lem:construction-of-bundle-from-cocycle}:
\[
\Bundle_{F}\left(B\right)\cong\check{H}^{1}\left(B,\Diff\left(F\right)\right),\qquad\left[E_{\tau}\right]\leftrightarrow\left[\tau\right].
\]
Besides $\Bundle_{F}\left(B\right)$, the cohomology set $\check{H}^{1}\left(B,\Diff\left(F\right)\right)$
is also equipped with a natural pullback action by the base diffeomorphism
group $\Diff\left(B\right)$ or, better still, by the base mapping
class group $\mathrm{Mod}\left(B\right)$. Since by construction the
above correspondence is $\mathrm{Mod}\left(B\right)$-equivariant,
we thereby obtain a descended correspondence
\[
\Bundle_{F}\left(B\right)/\mathrm{Mod}\left(B\right)\cong\check{H}^{1}\left(B,\Diff\left(F\right)\right)/\mathrm{Mod}\left(B\right).
\]
Letting $B$ range over all $m$-dimensional diffeomorphism types,
we then have the correspondence
\[
\bigsqcup_{B\in\ClassMan\left(m\right)}\Bundle_{F}\left(B\right)/\mathrm{Mod}\left(B\right)\cong\bigsqcup_{B\in\ClassMan\left(m\right)}\check{H}^{1}\left(B,\Diff\left(F\right)\right)/\mathrm{Mod}\left(B\right).
\]
Observe that the left-hand side of this last correspondence can be
viewed as the set of equivalence classes of $m$-codimensional smooth
$F$-fiberings on all possible diffeomorphism types of total spaces,
with our desired fibering classification $\ClassFib\left(E,F\right)$
embedded as the subset corresponding to the particular diffeomorphism
type of $E$. On the right-hand side, this amounts to restricting
$\check{H}^{1}\left(B,\Diff\left(F\right)\right)$ to a certain cohomology
subset accordingly for $E$ — but this subset has already been identified
in \prettyref{def:diffeomorphism-type-of-cohomology}; namely, it
is exactly the set $\check{H}^{1}\left(B,\Diff\left(F\right)\right)_{E}$
of those cohomology classes with total diffeomorphism type of $E$.
Therefore, we obtain the correspondence \prettyref{eq:classification-theorem-via-cohomology}
as desired.
\end{proof}

Once we complete the classification of all smooth $F$-fiberings on
$E$, we shall then focus on each fibering class as represented by
a model fibration $\StandardFibration\colon F\hookrightarrow E\to B$.
We have seen in \prettyref{sec:Smooth-Fibrations} that the fibering
classification interacts with a certain part of symmetries of $\StandardFibration$;
namely, the base transformation group $\Diff\left(B\right)_{\StandardFibration}$
(\prettyref{def:The-base-transformation-group}) or, better still,
its component group $\mathrm{Mod}\left(B\right)_{\StandardFibration}\coloneqq\pi_{0}\Diff\left(B\right)_{\StandardFibration}$.
This can thus be described explicitly in terms of cohomology of transition
cocycles by virtue of the preceding proposition:
\begin{cor}
For any model fibration $\StandardFibration\colon F\hookrightarrow E\to B$,
the base transformation subgroup $\mathrm{Mod}\left(B\right)_{\StandardFibration}\leq\mathrm{Mod}\left(B\right)$
(resp., $\Diff\left(B\right)_{\StandardFibration}\leq\Diff\left(B\right)$)
of the base mapping class group (resp., of the base diffeomorphism
group) coincides with the stabilizer of the cohomology class of any
transition cocycle $\tau_{\StandardFibration}$ of $\StandardFibration$:
\[
\mathrm{Mod}\left(B\right)_{\StandardFibration}=\left\{ \left[\beta\right]\in\mathrm{Mod}\left(B\right)\mid\left[\beta^{\ast}\tau_{\StandardFibration}\right]=\left[\tau_{\StandardFibration}\right]\in\check{H}^{1}\left(B,\Diff\left(F\right)\right)\right\} ,
\]
under the natural pullback action on cohomology.
\end{cor}

\begin{proof}
This immediately follows from the preceding cohomology-theoretic description
for classification (\prettyref{prop:classification-theorem-via-cohomology})
combined with the stabilizer characterization for base transformations
(\prettyref{prop:The-base-transformation-group-as-stabilizer}, and/or
\prettyref{cor:Orbit-stabilizer-relation-for-pullback}).
\end{proof}

\subsection{Gauge analysis via cohomology: Čech cocycle-stabilizer}

Let us next consider the gauge analysis in the framework of Čech cohomology
theory. Our arguments below will work equally well for either tangentially-smooth
($C^{0,\infty}$) or smooth ($C^{\infty}$) fiber bundles, as long
as we consider either topological or smooth Čech cohomology accordingly;
roughly speaking, the difference in the definitions between these
two categories resides in whether the topological or smooth Čech cochain
groups are used as building blocks:
\[
\check{C}^{q}\left(\mathcal{U},\mathcal{K}\right)\coloneqq\prod_{i_{0},\dots,i_{q}\in I}\MappingSpace\left(U_{i_{0}}\cap\dots\cap U_{i_{q}},\mathcal{K}\right)\quad\text{or}\quad\check{C}^{q}\left(\mathcal{U},\mathcal{K}\right)^{\infty}\coloneqq\prod_{i_{0},\dots,i_{q}\in I}\MappingSpace^{\infty}\left(U_{i_{0}}\cap\dots\cap U_{i_{q}},\mathcal{K}\right).
\]
Thus for the purpose of exposition, let us restrict attention to tangentially-smooth
($C^{0,\infty}$) fiber bundles, with the structure group $\mathcal{K}=\Diff\left(F\right)$
set to be the full diffeomorphism group. Let $\StandardFibration$
be such an $F$-fiber bundle over $B$, and choose an open cover $\mathcal{U}\coloneqq\left(U_{i}\right)_{i}$
with local trivializations $\varphi_{i}\colon\pi^{-1}\left(U_{i}\right)\to U_{i}\times F$.
Recall that this determines a Čech (1-)cocycle $\tau\in\check{Z}^{1}\left(\mathcal{U},\Diff\left(F\right)\right)$
given by the transition maps $\tau_{ij}\colon U_{i}\cap U_{j}\to\Diff\left(F\right)$
(as in \prettyref{eq:transition-cocycle}). Then for each vertical
automorphism $k\colon E\to E$, the corresponding local representations
can be described by the following $\mathcal{K}$-valued data:
\begin{equation}
\kappa_{i}\colon U_{i}\to\mathcal{K}=\Diff\left(F\right),\qquad\varphi_{i}\circ k\circ\varphi_{i}^{-1}\bigl(x,y\bigr)=\left(x,\kappa_{i}\left(x\right)\left(y\right)\right).\label{eq:local-rep-of-gauge-transformations-new}
\end{equation}
Note that such $\kappa\coloneqq\left(\kappa_{i}\right)_{i}$ is characterized
by the condition $\tau_{ij}\kappa_{j}{\tau_{ij}}^{-1}=\kappa_{i}$,
or equivalently $\kappa_{i}^{-1}\tau_{ij}\kappa_{j}=\tau_{ij}$; in
other words, $\kappa$ is a Čech 0-cochain lying in the stabilizer
subgroup at $\tau$ for the following coboundary action as seen in
\prettyref{eq:coboundary-action-prelim}:
\begin{equation}
\check{Z}^{1}\left(\mathcal{U},\Diff\left(F\right)\right)\curvearrowleft\check{C}^{0}\left(\mathcal{U},\Diff\left(F\right)\right)\coloneqq\prod_{i}\MappingSpace\left(U_{i},\Diff\left(F\right)\right),\qquad\left(\tau,\kappa\right)\mapsto\left(\kappa_{i}^{-1}\tau_{ij}\kappa_{j}\right)_{i,j}.\label{eq:cochain-action-new}
\end{equation}
Although a priori this depends on the choice of local trivialization,
changing charts amounts to conjugating $\kappa$ by a cocycle,%
and refining $\mathcal{U}$ amounts to restricting $\kappa$ which
can always be recovered by gluing. The stabilizer subgroup of this
action \prettyref{eq:cochain-action-new} has already made an appearance
in the prelude subsection above, whose definition is reprised as follows:
\begin{defn}
Fix a coefficient group $\mathcal{K}=\Diff\left(F\right)$ or any
topological subgroup thereof. Then for any cocycle $\tau\in\check{Z}^{1}\left(\mathcal{U},\mathcal{K}\right)$,
the \emph{(first) Čech cocycle-stabilizer} of $\tau$, denoted by
$\check{C}^{0}\left(\mathcal{U},\mathcal{K}\right)_{\tau}$, is the
topological group
\[
\check{C}^{0}\left(\mathcal{U},\mathcal{K}\right)_{\tau}\coloneqq\Bigl\{\kappa\in\prod_{i}\MappingSpace\left(U_{i},\mathcal{K}\right)\mid\kappa_{i}^{-1}\tau_{ij}\kappa_{j}=\tau_{ij}\Bigr\}.
\]
\nomenclature[Ccheck0UKt]{$\check{C}^{0}\left(\mathcal{U},\mathcal{K}\right)_{\tau}$}{the (first) Čech cocycle-stabilizer of $\tau$ (with coefficient group $\mathcal{K}$)}In
other words, $\check{C}^{0}\left(\mathcal{U},\mathcal{K}\right)_{\tau}$
is the stabilizer subgroup of $\tau$ for the natural cochain action
\prettyref{eq:cochain-action-new}.
\end{defn}

With this definition, the above discussion can be interpreted as giving
a map $k\mapsto\kappa$ from the vertical automorphism group of the
$F$-fibration $\StandardFibration$ to the $\Diff\left(F\right)$-valued
cocycle-stabilizer of a transition cocycle $\tau$. Conversely, it
is clear that every $\tau$-stabilizing cochain $\kappa$ can be glued
back into a vertical automorphism thanks to the condition $\tau_{ij}\kappa_{j}{\tau_{ij}}^{-1}=\kappa_{i}$.
Therefore we have the following result:%
\begin{prop}
\label{prop:homotopy-type-of-Vau-group-by-cocycle-stabilizer}Let
$\StandardFibration$ be a smooth $F$-fibration, with any choice
of smooth transition cocycle $\tau$ over $\mathcal{U}$. Then the
vertical automorphism group of $\StandardFibration$ has the homotopy
type of the Čech cocycle-stabilizer of $\tau$:
\[
\Vau\left(\StandardFibration\right)\simeq\check{C}^{0}\left(\mathcal{U},\Diff\left(F\right)\right)_{\tau}.
\]
More specifically, the desired map from $\Vau\left(\StandardFibration\right)$
to $\check{C}^{0}\left(\mathcal{U},\Diff\left(F\right)\right)_{\tau}$
is given by taking local representations as in \prettyref{eq:local-rep-of-gauge-transformations-new},
which is a continuous injective homomorphism with dense image and
also a homotopy equivalence.
\end{prop}

\begin{proof}
For easiness of distinction throughout this proof, let us denote by
$\Vau^{\infty}\left(\StandardFibration\right)$ (resp., $\Vau^{0,\infty}\left(\StandardFibration\right)$)
the vertical automorphism group of $\StandardFibration$ as a smooth
(resp., tangentially smooth) $F$-fibration over $B$. Then consider
the following commutative diagram:
\[
\begin{split}\xymatrix@C=4pc{\Vau^{\infty}\left(\StandardFibration\right)\ar[d]_{\cong}\ar[r] & \Vau^{0,\infty}\left(\StandardFibration\right)\ar[d]^{\cong}\\
\check{C}^{0}\left(\mathcal{U},\Diff\left(F\right)\right)_{\tau}^{\infty}\ar[r] & \check{C}^{0}\left(\mathcal{U},\Diff\left(F\right)\right)_{\tau}
}
\end{split}
.
\]
From this we see that it suffices to prove a suitable homotopy equivalence
$\Vau^{\infty}\left(\StandardFibration\right)\simeq\Vau^{0,\infty}\left(\StandardFibration\right)$;
more precisely, it suffices to prove the following smoothing lemma
for gauge groups of fiber bundles:
\begin{quote}
\emph{The natural homomorphism $\Vau^{\infty}\left(\StandardFibration\right)\to\Vau^{0,\infty}\left(\StandardFibration\right)$
induced by the forgetful functor is a continuous injection with dense
image and a homotopy equivalence, and moreover these two spaces are
homeomorphic.}
\end{quote}
But this can be shown in a similar way as the smoothing lemma for
classification in the proof of \prettyref{lem:construction-of-bundle-from-cocycle},
by drawing on the corresponding smoothing result from the gauge theory
of principal bundles. More precisely, the desired homotopy equivalence
$\Vau^{\infty}\left(\StandardFibration\right)\simeq\Vau^{0,\infty}\left(\StandardFibration\right)$
is proved as in the following commutative diagram:%
\[
\begin{split}\xymatrix@C=4pc{\Vau^{\infty}\left(\StandardFibration\right)\ar@{-->}[r]^{\simeq}\ar@{<->}[d]_{\cong} & \Vau^{0,\infty}\left(\StandardFibration\right)\ar@{<->}[d]^{\cong}\\
\Gauge^{\mathsf{Diff}}\left(\mathrm{Fr}_{\Diff\left(F\right)}\left(\StandardFibration\right);\Diff\left(F\right)\right)\ar[r]_{\simeq} & \Gauge^{\mathsf{Top}}\left(\mathrm{Fr}\left(\StandardFibration\right);\Diff\left(F\right)\right)
}
\end{split}
.
\]
Here, the crux is the homotopy equivalence in the bottom row. This
amounts to the corresponding smooth approximations result for the
principal-bundle gauge group, which is already available in e.g.,
\cite{MR4632309}. This completes
the proof of the desired smooth approximation for gauge analysis of
fiber bundles, and hence the proof of \prettyref{prop:homotopy-type-of-Vau-group-by-cocycle-stabilizer}
as desired.
\end{proof}

\subsection{Gauge reduction: equivariant homotopy of the diffeomorphism group}

Our main application of the preceding proposition (\prettyref{prop:homotopy-type-of-Vau-group-by-cocycle-stabilizer})
will take advantages of the following elementary property of the cocycle-stabilizer:%
\begin{lem}
Suppose that $\tau\in\check{Z}^{1}\left(\mathcal{U},\Diff\left(F\right)\right)$
takes values in a subgroup $T\leq\Diff\left(F\right)$, and consider
the $T$-action on $\Diff\left(F\right)$ by conjugation. Then for
any two subgroups $K_{1},K_{2}\leq\Diff\left(F\right)$ containing
$T$, we have
\[
K_{1}\simeq^{T}K_{2}\implies\check{C}^{0}\left(\mathcal{U},K_{1}\right)_{\tau}\simeq\check{C}^{0}\left(\mathcal{U},K_{2}\right)_{\tau}.
\]
In words, for each $T$-equivariant homotopy equivalence between the
coefficient groups $K_{1}$ and $K_{2}$, there induces a homotopy
equivalence between their corresponding cocycle stabilizers $\check{C}^{0}\left(\mathcal{U},K_{1}\right)_{\tau}$
and $\check{C}^{0}\left(\mathcal{U},K_{2}\right)_{\tau}$.
\end{lem}

Thus we have the following:\nomenclature[06eq]{$\simeq^{T}$}{$T$-equivariant homotopy equivalence}
\begin{prop}
\label{prop:Homotopy-reduction-of-Vau}Let $\StandardFibration$ be
represented by a transition cocycle $\tau\in\check{Z}^{1}\left(\mathcal{U},\Diff\left(F\right)\right)$
taking values in a subgroup $T\leq\Diff\left(F\right)$, and consider
the $T$-action on $\Diff\left(F\right)$ by conjugation. Then for
any subgroup $K\leq\Diff\left(F\right)$ containing $T$, we have
\[
\Diff\left(F\right)\simeq^{T}K,\enspace T\subseteq\Center\left(K\right)\implies\Vau\left(\StandardFibration\right)\simeq\MappingSpace^{\infty}\left(B,K\right)\simeq\MappingSpace\left(B,K\right).
\]
More precisely, if $\Diff\left(F\right)$ admits a $T$-equivariant
deformation retraction onto $K$ and if $T$ is contained in the center
of $K$, then $\Vau\left(\StandardFibration\right)$ admits a deformation
retraction onto $\MappingSpace^{\infty}\left(B,K\right)$, which in
turn is injected into $\MappingSpace\left(B,K\right)$ by a homotopy
equivalence with dense image.
\end{prop}

\begin{proof}
This is clear from the following diagram:
\[
\begin{split}\xymatrix@C=4pc{\Vau^{\infty}\left(\StandardFibration\right)\ar[r]^{\simeq} & \Vau^{0,\infty}\left(\StandardFibration\right)\\
\check{C}^{0}\left(\mathcal{U},\Diff\left(F\right)\right)_{\tau}^{\infty}\ar@{<->}[u]^{\cong}\ar[r]^{\simeq} & \check{C}^{0}\left(\mathcal{U},\Diff\left(F\right)\right)_{\tau}\ar@{<->}[u]^{\cong}\\
\check{C}^{0}\left(\mathcal{U},K\right)_{\tau}^{\infty}\ar@{^{(}->}[u]^{\simeq}\ar[r]^{\simeq} & \check{C}^{0}\left(\mathcal{U},K\right)_{\tau}\ar@{^{(}->}[u]^{\simeq}\\
\MappingSpace^{\infty}\left(B,K\right)\ar@{<->}[u]^{\cong}\ar[r]^{\simeq} & \MappingSpace\left(B,K\right)\ar@{<->}[u]^{\cong}
}
\end{split}
.
\]
\end{proof}

We remark two extreme cases of the preceding proposition where the
condition $T\subseteq\Center\left(K\right)$ holds trivially. The
first case is when $\StandardFibration$ is a trivial bundle, so that
$T$ is the trivial group, and thus in this case we can take $K=\Diff\left(F\right)$
the full structure group and \prettyref{prop:Homotopy-reduction-of-Vau}
recovers \prettyref{lem:The-Vau-group-for-trivial-bundles} for trivial
bundles. The other case is when the full structure group has an abelian
subgroup $K$ as an equivariant deformation retract, so that the center
of $K$ attains the full group, and thus in this case $\Vau\left(\StandardFibration\right)$
can again be homotopically reduced to the current group $\MappingSpace\left(B,K\right)$
by \prettyref{prop:Homotopy-reduction-of-Vau} — this will be put
into action when studying oriented circle fibrations in \prettyref{chap:First-Examples}.

\section{Homotopy via Shape Analysis\label{sec:Homotopy-via-shape-analysis}}

\subsection{Prelude: classifying space of diffeomorphism group}

In this subsection, we shall describe a geometric model for the classifying
space of the diffeomorphism group $\Diff\left(F\right)$. As we shall
see, this can be viewed as an adaptation from the classification theory
of vector bundles, where the classifying space of the linear, finite-dimensional
structure group $\mathrm{GL}\left(k\right)$ or $\Orthogonal k$ can
be modeled by the infinite Grassmannian manifold $\mathrm{Gr}\left(k,\infty\right)$.
In the current case of the nonlinear, infinite-dimensional structure
group $\Diff\left(F\right)$, a candidate by analogy for the classifying
space would be the ``nonlinear Grassmannian manifold'' that consists
of all smooth $F$-submanifolds embedded in the infinite-dimensional
separable Hilbert space $\ell^{2}$ — such nonlinear Grassmannians
have already made an appearance in previous sections (e.g., see \prettyref{eq:fiberings-as-families-of-shapes}),
whose definition (specialized in the current case) is reprised as
follows:
\begin{defn}
The space of $F$-shapes embedded in $\ell^{2}$ (or \emph{shape space}
for short), denoted by $\Shape\left(F,\ell^{2}\right)$, is the following
orbit space with the quotient topology:
\[
\Shape\left(F,\ell^{2}\right)\coloneqq\Emb\left(F,\ell^{2}\right)/\Diff\left(F\right),
\]
where the space $\Emb\left(F,\ell^{2}\right)$ (consisting of all
smooth embeddings of $F$ into $\ell^{2}$) is acted upon by $\Diff\left(F\right)$
from the right via pre-composition (i.e., reparametrization).
\end{defn}

Here, our universal ambient manifold is deliberately chosen to be
the Hilbert space $\ell^{2}$ (rather than the naive direct limit
$\mathbb{R}^{\infty}\coloneqq\varinjlim_{n}\mathbb{R}^{n}$); as we
shall see, this will be a more natural setup for studying the infinite-dimensional
differential geometric structures, and yet it will also be homotopy
equivalent to the direct-limit model so that the usual stability arguments
can be exploited. The latter point is illustrated in the following
proof of the contractibility of $\Emb\left(F,\ell^{2}\right)$, which
can be viewed as Whitney's embedding theorem in disguise:\footnote{Alternatively, a concrete proof can be found in e.g., \cite[§44.22]{MR1471480},
where an desired contraction is constructed explicitly.}
\begin{lem}
\label{lem:shape-space-in-l2-is-contractible}The space $\Emb\left(F,\ell^{2}\right)$
is (nonempty and) contractible.
\end{lem}

\begin{proof}
Consider the smooth closed expanding system $\left(\mathbb{R}^{n}\right)_{n}$
(given by the natural embeddings of $\mathbb{R}^{n}$ into $\mathbb{R}^{n+1}$).
This direct system has $\ell^{2}$ as a \emph{homotopy direct limit},
meaning that the embeddings of $\mathbb{R}^{n}$ into $\ell^{2}$
induce a homotopy equivalence $j\colon\varinjlim_{n}\mathbb{R}^{n}\simeq\ell^{2}$.
By applying the functor $\Emb\left(F,-\right)$, we obtain an induced
closed expanding system $\left(\Emb\left(F,\mathbb{R}^{n}\right)\right)_{n}$.
Now the crux is a desirable property that this functor $\Emb\left(F,-\right)$
commutes with homotopy direct limits, which is already shown in \cite{MR0319206,MR0322900}.
Thus in our case, this implies that the system $\left(\Emb\left(F,\mathbb{R}^{n}\right)\right)_{n}$
must have the space $\Emb\left(F,\ell^{2}\right)$ as a homotopy direct
limit; more precisely, the above homotopy equivalence $j\colon\varinjlim_{n}\mathbb{R}^{n}\simeq\ell^{2}$
induces a desired homotopy equivalence
\[
j_{\ast}\colon\varinjlim_{n}\Emb\left(F,\mathbb{R}^{n}\right)\simeq\Emb\left(F,\ell^{2}\right).
\]
Therefore, for contractibility of $\Emb\left(F,\ell^{2}\right)$
it suffices to show that for each $m\geq0$, the space $\Emb\left(F,\mathbb{R}^{n}\right)$
is nonempty and $m$-connected for sufficiently large $n$. But this
is already well known as a consequence of Whitney's embedding theorem
(\cite{MR1503303,MR0010274}); alternatively, we may invoke a theorem
of Hansen (\cite{MR0322900}) or Dax (\cite{MR0321110}) which implies
that the natural inclusion of the embedding space $\Emb\left(F,\mathbb{R}^{n}\right)$
into the mapping space $\MappingSpace^{\infty}\left(F,\mathbb{R}^{n}\right)$
is $k$-connected for sufficiently large $n$ (where the latter space
is clearly contractible), hence we have the desired contractibility:
\[
\pi_{m}\varinjlim_{n}\Emb\left(F,\mathbb{R}^{n}\right)\cong\pi_{m}\varinjlim_{n}\MappingSpace^{\infty}\left(F,\mathbb{R}^{n}\right)=0,\qquad\forall m\geq0.
\]
This completes the proof of \prettyref{lem:shape-space-in-l2-is-contractible}.
\end{proof}
Having shown in the preceding lemma that the embedding space $\Emb\left(F,\ell^{2}\right)$
is contractible, we next confirm its principal $\Diff\left(F\right)$-bundle
structure over the shape space $\Shape\left(F,\ell^{2}\right)$, as
in the following lemma:\footnote{Existing proofs for the case of finite-dimensional target manifold
can be found in e.g., \cite[§44.1]{MR1471480} (and originally in
\cite{MR0613004}). Our proof here offers a slight variation that
highlights the dependence on two desirable structures on the target
Hilbert space: an inner product and a compatible local addition.}
\begin{lem}
\label{lem:bundle-structure-over-shape-space}The canonical quotient
projection associated with the natural right action of $\Diff\left(F\right)$
on $\Emb\left(F,\ell^{2}\right)$ gives a topological principal bundle
\[
\Diff\left(F\right)\hookrightarrow\Emb\left(F,\ell^{2}\right)\to\Shape\left(F,\ell^{2}\right).
\]
Further, $\Shape\left(F,\ell^{2}\right)$ admits a unique smooth structure
for which the above bundle becomes a smooth principal bundle; more
specifically, it is a (separable) smooth Fréchet manifold locally
modeled on—at any embedded $F$-shape $S\subseteq\ell^{2}$—the Fréchet
subspace $\mathfrak{X}^{\bot}\left(S\right)\subseteq\mathfrak{X}\left(S\right)$
of normal vector fields, and for which the above bundle admits a smooth
local splitting modeled on the Fréchet-space splitting $\mathfrak{X}\left(S\right)=\mathfrak{X}^{\bot}\left(S\right)\oplus\mathfrak{X}^{\top}\left(S\right)$
of vector fields along $S$ into normal and tangential parts .
\end{lem}

\begin{proof}
Fix an arbitrary $F$-embedding $\iota_{0}\in\Emb\left(F,\ell^{2}\right)$,
so that $\Diff\left(F\right)$ is embedded as its orbit $\mathscr{D}\left(F\right)\subseteq\Emb\left(F,\ell^{2}\right)$.
Let $S_{0}\coloneqq\iota_{0}\left(F\right)\in\Shape\left(F,\ell^{2}\right)$
be the corresponding $F$-shape. To confirm the bundle structure we
need to show that there exists a local section $\sigma\colon U\to q^{-1}\left(U\right)$
with $q\circ\sigma=\Identity_{U}$ and $\sigma\left(S_{0}\right)=\iota_{0}$.
Intuitively, this amounts to show that the shapes $S$'s near $S_{0}$
can be continuously parametrized by $F$; but this is clear since
every shape $S$ sufficiently close to $S_{0}$ must be a normal graph
in $\ell^{2}$ over $S_{0}$ (or more invariantly, a section of the
normal bundle in $T\ell^{2}$ over $S_{0}$), and thereby inherits
a parametrization from the one $\iota_{0}$ for $S_{0}$. More precisely,
letting $\projection^{\perp}$ denote the normal projection onto $S_{0}$
from some tubular neighborhood in $\ell^{2}$, over some sufficiently
small domain $U$ we can construct the desired local section as follows:
\[
\sigma\colon\Shape\left(F,\ell^{2}\right)\supsetopen U\xrightarrow{S\mapsto\iota_{S}}\Emb\left(F,\ell^{2}\right)\qquad\text{with}\qquad\iota_{S}\left(y\right)-\iota_{0}\left(y\right)\in T_{\iota_{0}\left(y\right)}^{\perp}S_{0}.
\]
This confirms the topological principal $\mathscr{D}\left(F\right)$-bundle
structure on $\Emb\left(F,\ell^{2}\right)$ over $\Shape\left(F,\ell^{2}\right)$,
as desired. In order to promote such a bundle structure to the smooth
category, we switch to a more robust description of this local splitting:
taking the image of the above local section $\sigma$ induces a local
slice
\begin{equation}
\mathcal{S}\coloneqq\image\sigma=\left\{ \iota\in\Emb\left(F,\ell^{2}\right)\mid\iota=\iota_{S},\,\exists S\in U\subsetopen\Shape\left(F,\ell^{2}\right)\right\} ,\label{eq:slice-for-shape-space}
\end{equation}
onto which the restriction of the right $\mathscr{D}\left(F\right)$-action
yields a tube homeomorphism
\begin{equation}
\rho\colon\mathcal{S}\times\mathscr{D}\left(F\right)\xrightarrow[\approx]{\left(\iota,\kappa\right)\mapsto\iota\circ\kappa}q^{-1}\left(U\right)\subsetopen\Emb\left(F,\ell^{2}\right).\label{eq:tube-homeomorphism-for-shape-space}
\end{equation}
We now move on to the smooth category. Let us start with recalling
the smooth structure on $\Emb\left(F,\ell^{2}\right)$: since the
target manifold $\ell^{2}$ being a locally convex space admits a
(affine) local addition (\prettyref{exa:examples-of-local-additions}\ref{enu:example-of-affine-local-addition}),
it follows from \prettyref{thm:local-addition-implies-canonical-smooth-manifold-structure}
that there induces a canonical smooth structure on $\MappingSpace^{\infty}\left(F,\ell^{2}\right)$,
and hence on its open subset $\Emb\left(F,\ell^{2}\right)$; more
explicitly, this smooth structure on $\Emb\left(F,\ell^{2}\right)$
is modeled on the Fréchet space $\mathfrak{X}\left(S_{0}\right)$
of smooth vector fields along $S_{0}\coloneqq\iota_{0}\left(F\right)\subseteq\ell^{2}$
via a global chart induced by the canonical identification $\ell^{2}\cong T\ell^{2}$
centered at $S_{0}$. Under this smooth structure for $\Emb\left(F,\ell^{2}\right)$
modeled on the Fréchet space $\mathfrak{X}\left(S_{0}\right)$, the
subset $\mathscr{D}\left(F\right)$ becomes a smooth submanifold modeled
on the Fréchet subspace $\mathfrak{X}^{\top}\left(S_{0}\right)$ of
tangential vector field along $S_{0}$. Similarly, the slice $\mathcal{S}$
given in \prettyref{eq:slice-for-shape-space} becomes a smooth submanifold
modeled on the Fréchet subspace $\mathfrak{X}^{\bot}\left(S_{0}\right)$
of normal vector fields to $S_{0}$, with respect to which the tube
homeomorphism $\rho$ given in \prettyref{eq:tube-homeomorphism-for-shape-space}
becomes a diffeomorphism. Indeed, the smooth map $\rho$ admits a
smooth inverse as simply given by the orthogonal splitting
\[
\mathfrak{X}\left(S_{0}\right)=\mathfrak{X}^{\bot}\left(S_{0}\right)\oplus\mathfrak{X}^{\top}\left(S_{0}\right).
\]
Thus just as in the proof of \prettyref{prop:manifold-slice-condition},
such a ``smooth slice'' $\mathcal{S}$ induces both the smooth manifold
structure on $\Shape\left(F,\ell^{2}\right)$ and the smooth bundle
structure on $\Emb\left(F,\ell^{2}\right)$ over $\Shape\left(F,\ell^{2}\right)$,
as desired.
\end{proof}

In particular, a topological implication of the preceding lemma is
that the shape space $\Shape\left(F,\ell^{2}\right)$ can serve as
the unique topological $\ell^{2}$-model of its countable CW-homotopy
type (see \prettyref{cor:corollary-of-Torunczyk-and-Henderson-Schori}).
This makes it a natural choice of the classifying space of $\Diff\left(F\right)$;
more precisely, as an immediate consequence of the preceding two lemmas
(\prettyref{lem:shape-space-in-l2-is-contractible} and \prettyref{lem:bundle-structure-over-shape-space}),
we have the following result:
\begin{prop}
\label{prop:classifying-space-of-diffeomorphism}The space $\Shape\left(F,\ell^{2}\right)$
is the unique (up to homeomorphism) $\ell^{2}$-model for the classifying
space of the diffeomorphism group $\Diff\left(F\right)$, over which
the universal bundle is given by the natural projection
\begin{equation}
\Diff\left(F\right)\hookrightarrow\Emb\left(F,\ell^{2}\right)\to\Shape\left(F,\ell^{2}\right).\label{eq:universal-bundle}
\end{equation}
Moreover, this can be promoted to a classifying object in the category
of smooth Fréchet manifolds.
\end{prop}

\subsection{Classification via homotopy: family of fiber-shapes}

With the Hilbert model $\Shape\left(F,\ell^{2}\right)$ of the classifying
space for $\Diff\left(F\right)$ as explained above, we are in a position
to tackle the classification problem with a global, more geometric
perspective, in the framework of homotopy theory. As before, our arguments
below will work equally well for either tangentially-smooth ($C^{0,\infty}$)
or smooth ($C^{\infty}$) fiber bundles, as long as we consider either
topological or smooth homotopy accordingly:
\[
\left[B,\mathcal{N}\right]\coloneqq\pi_{0}\MappingSpace\left(B,\mathcal{N}\right)\quad\text{or}\quad\left[B,\mathcal{N}\right]^{\infty}\coloneqq\pi_{0}\MappingSpace^{\infty}\left(B,\mathcal{N}\right).
\]
Our goal is to construct a suitable universal object for fiber bundles.
To this end, let us follow the analogous procedure in the linear case
of classifying vector bundles: the classifying space of the structure
group $\mathrm{GL}\left(k\right)$ is the infinite Grassmannian manifold
$\mathrm{Gr}\left(k,\infty\right)$, over which we have the universal
principal $\GL kR$-bundle given by the infinite Stiefel manifold
$\mathrm{St}\left(k,\infty\right)$; in turn, associated to the universal
bundle $\mathrm{St}\left(k,\infty\right)$ there is a universal rank-$k$
vector bundle $\mathrm{St}\left(k,\infty\right)\times_{\GL kR}\mathbb{R}^{k}$,
which can be simply described by the tautological bundle over $\mathrm{Gr}\left(k,\infty\right)$
consisting of all those pairs $\left(S,p\right)$ in $\mathrm{Gr}\left(k,\infty\right)\times\mathbb{R}^{k}$
that satisfy $p\in S$. The same arguments apply, mutatis mutandis,
to the nonlinear case of classifying fiber bundles; more precisely,
we can construct a kind of universal object called the \emph{nonlinear tautological bundle}
over $\Shape\left(F,\ell^{2}\right)$ as given by the following manifold
\begin{equation}
\Tautological\left(F,\ell^{2}\right)\coloneqq\left\{ \left(S,p\right)\in\Shape\left(F,\ell^{2}\right)\times\ell^{2}\mid p\in S\right\} \cong\Emb\left(F,\ell^{2}\right)\times_{\Diff\left(F\right)}F.\label{eq:tautological-bundle}
\end{equation}
The resulting fiber bundle $\Tautological\left(F,\ell^{2}\right)\to\Shape\left(F,\ell^{2}\right)$
given by the canonical projection is called the \emph{universal $F$-fiber bundle}
because, by construction, every $F$-fiber bundle $\StandardFibration\colon E\to B$
is equivalent to the pullback bundle $f^{\ast}\Tautological\left(F,\ell^{2}\right)$
for some classifying map $f\colon B\to\Shape\left(F,\ell^{2}\right)$
unique up to homotopy. Up to equivalence, this pullback bundle under
$f$ (resp., the classifying map associated with $\StandardFibration$)
can be given a quite intuitive description, as in the statement (resp.,
proof) of the following lemma:
\begin{lem}
\label{lem:construction-of-bundle-from-classifying-map}Every homotopy
class in $\left[B,\Shape\left(F,\ell^{2}\right)\right]$ admits smooth
mapping representatives $f$, each of which can serve as the classifying
map of a smooth $F$-fiber bundle $E_{f}$ over $B$ given by the
following smooth manifold (equipped with the natural submanifold structure
and bundle projection):
\[
E_{f}\coloneqq\left\{ \left(x,p\right)\in B\times\ell^{2}\mid p\in f\left(x\right)\right\} .
\]
Moreover, any other choice of smooth representative $f$ results in
a smoothly equivalent fibration $E_{f}$.
\end{lem}

\begin{proof}
Given a homotopy class $\left[f\right]\in\left[B,\Shape\left(F,\ell^{2}\right)\right]$
represented by any map $f$. By definition, the $f$-pullback of the
tautological bundle is given by
\[
f^{\ast}\left(\Tautological\left(F,\ell^{2}\right)\right)\coloneqq\left\{ \left(\left(x,S\right),p\right)\in\Gamma_{f}\times\ell^{2}\mid p\in S\right\} .
\]
Here, $\Gamma_{f}\subseteq B\times\Shape\left(F,\ell^{2}\right)$
denotes the graph of the classifying map $f\colon B\to\Shape\left(F,\ell^{2}\right)$.
Thus its canonical projection $\Gamma_{f}\to B$ induces the desired
bundle isomorphism (over $B$) from the above pullback bundle to the
proposed bundle $E_{f}$:
\[
f^{\ast}\left(\Tautological\left(F,\ell^{2}\right)\right)\xrightarrow{\cong}E_{f},\qquad\left(\left(x,S\right),p\right)\mapsto\left(x,p\right).
\]
As for the claim concerning smooth approximation, we can invoke the
smoothing lemma for classification of fiber bundles as shown in the
proof of \prettyref{lem:construction-of-bundle-from-cocycle}, but
in the current case we can also carry out the smoothing procedure
directly on the classifying maps from $B$ into $\Shape\left(F,\ell^{2}\right)$;
indeed, such a smoothing was well-known for maps between finite-dimensional
manifolds, and its generalization to the case where the target is
an infinite-dimensional (locally convex) manifold is already available
— e.g., see \cite{MR2591666}. The takeaway is that the natural forgetful
map induces a bijection
\[
\pi_{0}\MappingSpace^{\infty}\left(B,\Shape\left(F,\ell^{2}\right)\right)\leftrightarrow\pi_{0}\MappingSpace^{0}\left(B,\Shape\left(F,\ell^{2}\right)\right).
\]
This complete the proof of \prettyref{lem:construction-of-bundle-from-classifying-map}.
\end{proof}
By the preceding lemma, we in particular see that every $\Diff\left(F\right)$-classifying
homotopy class can be assigned a well-defined diffeomorphism type,
hence the following definition is justified:
\begin{defn}[total diffeomorphism type of homotopy]
\label{def:diffeomorphism-type-of-homotopy}Associated with each
homotopy class in $\left[B,\Shape\left(F,\ell^{2}\right)\right]$
is its \emph{total diffeomorphism type}, which is defined as the
unique diffeomorphism type of the smooth manifold $E_{f}$ as constructed
in \prettyref{lem:construction-of-bundle-from-classifying-map} from
any smooth representative $f$. Given any smooth manifold $E$, denote
the subset 
\[
\left[B,\Shape\left(F,\ell^{2}\right)\right]_{E}\coloneqq\left\{ \eta\in\left[B,\Shape\left(F,\ell^{2}\right)\right]\mid\text{\ensuremath{\eta} has the total diff. type of \ensuremath{E}}\right\} .
\]
This subset of $\left[B,\Shape\left(F,\ell^{2}\right)\right]$ will
be called the ``\emph{homotopy subset of total type $E$}'' for
short.
\end{defn}

We are now in a position to state and prove the following homotopy-theoretic
classification of smooth fiberings on a given smooth total space $E$.
(Below and throughout, recall that $\ClassMan\left(m\right)$ denotes
the set of all diffeomorphism types of closed smooth $m$-manifolds.)
\begin{prop}
\label{prop:classification-theorem-via-homotopy}Let $m\coloneqq\dim E-\dim F$.
The smooth $F$-fiberings on $E$ are classified — for each $m$-dimensional
diffeomorphism type for the base space $B$ — by the homotopy subset
of $\left[B,\Shape\left(F,\ell^{2}\right)\right]$ of total type $E$
(\prettyref{def:diffeomorphism-type-of-homotopy}) modulo pullback
by diffeomorphisms:
\begin{equation}
\ClassFib\left(E,F\right)\leftrightarrow\bigsqcup_{B\in\ClassMan\left(m\right)}{\left[B,\Shape\left(F,\ell^{2}\right)\right]_{E}/\mathrm{Mod}\left(B\right)}.\label{eq:classification-theorem-via-homotopy}
\end{equation}
More precisely, for each class on the right-hand side, any choice
of smooth mapping representative $f$ is assembled into a smooth bundle
$E_{f}$ as in \prettyref{lem:construction-of-bundle-from-classifying-map},
which induces a desired smooth fibration on $E$ via a diffeomorphism
$E\approx E_{f}$ as guaranteed by assumption.
\end{prop}

\begin{proof}
Let $\Bundle_{F}\left(B\right)$ denote the set of equivalence classes
of smooth $F$-fiber bundles over $B$. Then the $\Diff\left(F\right)$–classifying
property of $\Shape\left(F,\ell^{2}\right)$ (\prettyref{prop:classifying-space-of-diffeomorphism})
as well as \prettyref{lem:construction-of-bundle-from-classifying-map}
implies the following isomorphism of pointed sets
\[
\Bundle_{F}\left(B\right)\cong\left[B,\Shape\left(F,\ell^{2}\right)\right],\qquad\left[E_{f}\right]\leftrightarrow\left[f\right].
\]
Besides $\Bundle_{F}\left(B\right)$, the homotopy set $\left[B,\Shape\left(F,\ell^{2}\right)\right]$
is also equipped with a natural pullback action by the base diffeomorphism
group $\Diff\left(B\right)$ or, better still, by the base mapping
class group $\mathrm{Mod}\left(B\right)$. Since by construction the
above correspondence is $\mathrm{Mod}\left(B\right)$-equivariant,
we thereby obtain a descended correspondence
\[
\Bundle_{F}\left(B\right)/\mathrm{Mod}\left(B\right)\cong\left[B,\Shape\left(F,\ell^{2}\right)\right]/\mathrm{Mod}\left(B\right).
\]
Letting $B$ range over all $m$-dimensional diffeomorphism types,
we then have the correspondence
\[
\bigsqcup_{B\in\ClassMan\left(m\right)}\Bundle_{F}\left(B\right)/\mathrm{Mod}\left(B\right)\cong\bigsqcup_{B\in\ClassMan\left(m\right)}\left[B,\Shape\left(F,\ell^{2}\right)\right]/\mathrm{Mod}\left(B\right).
\]
Observe that the left-hand side of this last correspondence can be
viewed as the set of equivalence classes of $m$-codimensional smooth
$F$-fiberings on all possible diffeomorphism types of total spaces,
with our desired fibering classification $\ClassFib\left(E,F\right)$
embedded as the subset corresponding to the particular diffeomorphism
type of $E$. On the right-hand side, this amounts to restricting
$\left[B,\Shape\left(F,\ell^{2}\right)\right]$ to a certain homotopy
subset accordingly for $E$ — but this subset has already been identified
in \prettyref{def:diffeomorphism-type-of-homotopy}; namely, it is
exactly the set $\left[B,\Shape\left(F,\ell^{2}\right)\right]_{E}$
of those homotopy classes with total diffeomorphism type of $E$.
Therefore, we obtain the correspondence \prettyref{eq:classification-theorem-via-homotopy}
as desired.
\end{proof}
Once we complete the classification of all smooth $F$-fiberings on
$E$, we shall then focus on each fibering class as represented by
a model fibration $\StandardFibration\colon F\hookrightarrow E\to B$.
We have seen in \prettyref{sec:Smooth-Fibrations} that the fibering
classification interacts with a certain part of symmetries of $\StandardFibration$;
namely, the base transformation group $\Diff\left(B\right)_{\StandardFibration}$
(\prettyref{def:The-base-transformation-group}), or better still
$\mathrm{Mod}\left(B\right)_{\StandardFibration}\coloneqq\pi_{0}\Diff\left(B\right)_{\StandardFibration}$,
concerning the coordinate changes of the base that preserve $\StandardFibration$
up to bundle isomorphisms. This can thus be described explicitly in
terms of homotopy of classifying maps by virtue of the preceding proposition:
\begin{cor}
For any model fibration $\StandardFibration\colon F\hookrightarrow E\to B$,
the base transformation subgroup $\mathrm{Mod}\left(B\right)_{\StandardFibration}\leq\mathrm{Mod}\left(B\right)$
(resp., $\Diff\left(B\right)_{\StandardFibration}\leq\Diff\left(B\right)$)
of the base mapping class group (resp., of the base diffeomorphism
group) coincides with the stabilizer of the homotopy class of any
classifying map $f_{\StandardFibration}$ of $\StandardFibration$:
\[
\mathrm{Mod}\left(B\right)_{\StandardFibration}=\left\{ \left[\beta\right]\in\mathrm{Mod}\left(B\right)\mid\left[\beta^{\ast}f_{\StandardFibration}\right]=\left[f_{\StandardFibration}\right]\in\left[B,\Shape\left(F,\ell^{2}\right)\right]\right\} ,
\]
under the natural pullback action on homotopy.
\end{cor}

\begin{proof}
This immediately follows from the preceding homotopy-theoretic description
for classification (\prettyref{prop:classification-theorem-via-homotopy})
combined with the stabilizer characterization for base transformations
(\prettyref{prop:The-base-transformation-group-as-stabilizer}, and/or
\prettyref{cor:Orbit-stabilizer-relation-for-pullback}).
\end{proof}

\subsection{}

\subsection{Gauge analysis via homotopy: fiberwise embeddings and looping/delooping}

\chapter{Differential Topology\label{chap:Differential-Topology}}

The goal of this chapter is to ``lift'' the differential geometry
— from the finite-dimensional fiber bundle $\StandardFibration\colon F\hookrightarrow E\xtwoheadrightarrow{}{\StandardProjection}B$
— to the infinite-dimensional symmetry groups and moduli space, as
packaged in the following two sequences:
\[
\Vau\left(\StandardFibration\right)\hookrightarrow\Aut\left(\StandardFibration\right)\xtwoheadrightarrow{}{\AutProjection}\Diff\left(B\right)_{\StandardFibration},\qquad\Aut\left(\StandardFibration\right)\hookrightarrow\Diff\left(E\right)\xtwoheadrightarrow{}{\DiffProjection}\Fib\left(\StandardFibration\right).
\]
As we shall show, with the canonical Lie group structure on the diffeomorphism
group $\Diff\left(E\right)$, the symmetry group $\Aut\left(\StandardFibration\right)$
(hence $\Vau\left(\StandardFibration\right)$) inherits a Lie subgroup
structure, and the moduli space $\Fib\left(\StandardFibration\right)$
inherits a homogeneous smooth structure, for which both sequences
above become infinite-dimensional smooth principal fibrations. There
are two recurring themes in our construction: infinitesimal%
and geometric, so that the resulting smooth fibrations above will
have the nature of Lie integration and Riemannian fibration. Before
we embark on this infinite-dimensional differential geometric study,
let us first set up the geometric scene for our fiber bundle $\StandardFibration$
in the finite-dimensional world. This is summarized in the following
assumption and will be explained in the preliminary subsection immediately
afterward.
\begin{namedthm}[Assumption for \prettyref{chap:Differential-Topology}]
In this chapter, we assume (without loss of generality) that the
fibered total space $\left(E,\StandardFibration\right)$ is equipped
with a projectable (bundle-like) Riemannian metric $g$, so that it
induces the following three data:
\begin{enumerate}
\item a fiberwise metric $\bigsqcup_{x\in B}g_{E_{x}}$ on the vertical
distribution $T\StandardFibration=\bigsqcup_{x\in B}TE_{x}$, as in
\prettyref{eq:fiberwise-metric}.
\item a connection or horizontal distribution $H\StandardFibration$ complement
to $T\StandardFibration$, as in \prettyref{eq:connection-or-horizontal-distribution}.
\item a base metric $g_{B}$ on the base tangent bundle $TB$ as descended
from $H\StandardFibration$ (by virtue of the projectable assumption),
as in \prettyref{eq:base-metric}.
\end{enumerate}
This will be referred to as ``setting up $\StandardFibration$ as
a Riemannian fibration'' for short.
\end{namedthm}
One may think of these three ingredients — a fiberwise metric $\bigsqcup_{x\in B}g_{E_{x}}$,
a connection $H\StandardFibration$, and a base metric $g_{B}$ —
as what are really being relied on; for then the 2-tensor field $\left(\left.\pi^{\ast}g_{B}\right|_{H\StandardFibration}\right)\oplus\bigsqcup_{x\in B}g_{E_{x}}$
on $H\StandardFibration\oplus T\StandardFibration$ will recover a
suitable metric on the total space. As we progress through this chapter,
it is worth recognizing which one(s) of these three is in effect.

\subsection*{Setting the scene: Riemannian fibration}

Let $\left(E,\StandardFibration\right)$ be our fibered manifold,
and let the tangent bundle of $\StandardFibration$ be embedded as
the vertical distribution $T\StandardFibration=\bigsqcup_{x\in B}TE_{x}$
in $TE$. Suppose that $E$ is equipped with a Riemannian metric $g$.
Then we have the following two induced geometric structures with respect
to $\StandardFibration$. The first structure is a \emph{fiberwise metric}
$\bigsqcup_{x\in B}g_{E_{x}}$ on the vertical distribution, which
is induced from $g$ by taking the restriction:
\begin{equation}
g_{E_{x}}\colon TE_{x}\oplus TE_{x}\to\mathbb{R},\quad g_{E_{x}}\coloneqq\left.g\right|_{E_{x}}\qquad\left(\forall x\in B\right).\label{eq:fiberwise-metric}
\end{equation}
The other structure is a \emph{connection} or \emph{horizontal distribution}
$H\StandardFibration$ complement to the vertical distribution, which
is induced from $g$ by taking the orthogonal complement:
\begin{equation}
H\StandardFibration\coloneqq\ker\left(TE\xrightarrow{v\mapsto v^{\top}}T\StandardFibration\right)\coloneqq\left(T\StandardFibration\right)^{\perp}.\label{eq:connection-or-horizontal-distribution}
\end{equation}
Such a horizontal distribution $H\StandardFibration\subseteq TE$
determines at each point a distinguished horizontal subspace $H_{p}\StandardFibration\subseteq T_{p}E$
to be canonically isomorphic to the base tangent space $T_{\pi\left(p\right)}B$
via (the restriction of) $d\pi_{p}$; in other words, there induces
the following smooth vector bundle isomorphism which we shall call
the \emph{horizontal lift}:
\begin{equation}
\pi^{\ast}TB\cong H\StandardFibration,\qquad\left.d\pi_{p}\right|_{H_{p}\StandardFibration}^{-1}\colon T_{x}B\cong H_{p}\StandardFibration\quad\left(\forall p\in E_{x}\right).\label{eq:horizontal-lift-vector-space-iso}
\end{equation}
We say that our metric $g$ is \emph{projectable} (or \emph{bundle-like})
if the horizontal subspaces along each fiber are all isometric to
each other via the canonical linear isomorphisms induced from \prettyref{eq:horizontal-lift-vector-space-iso}.
By definition, a projectable metric $g$ exactly ensures that each
base tangent space $T_{x}B$ inherits a unique inner product from
a horizontal subspace $H_{p}\StandardFibration$ via horizontal lift
\prettyref{eq:horizontal-lift-vector-space-iso}, independent of the
choice of $p\in E_{x}$. Hence the terminology: a projectable metric
$g$ admits a well-defined projection; i.e., a \emph{base metric}
$g_{B}$ uniquely determined by its $\pi$-relation to the horizontal
restriction of $g$:
\begin{equation}
g_{B}\colon TB\oplus TB\to\mathbb{R},\qquad\left.\pi^{\ast}g_{B}\right|_{H\StandardFibration}\coloneqq\left.g\right|_{H\StandardFibration}.\label{eq:base-metric}
\end{equation}
Just as the induced fiberwise metric $\left\{ g_{E_{x}}\right\} _{x}$
can be uniquely characterized by the property of making every concrete
fiber $E_{x}\coloneqq\pi^{-1}\left(x\right)\subseteq E$ of $\StandardFibration$
a Riemannian submanifold, the induced base metric can be uniquely
characterized by the property of making the bundle projection $\pi\colon E\to B$
of $\StandardFibration$ a Riemannian submersion. We shall refer to
this as ``setting up $\StandardFibration$ as a Riemannian fibration''
and take it as our primary setup for this chapter.

\section{Tangent Spaces and Lie Algebras\label{sec:Tangent-Spaces-and-Lie-Algebras}}

To embark on our (infinite-dimensional) differential geometric study,
let us first understand the corresponding structures at the \emph{infinitesimal}
level, by which we mean the expected tangent spaces and tangent maps.
It is worth making two comments in this regard: on the one hand, just
like in finite dimensions, information at the infinitesimal level
often provides a simplified and clarified way to access the infinite-dimensional
smooth manifolds, which will be carried out in this section; but on
the other hand, in contrast with the finite-dimensional case, information
at the infinitesimal level is generally not sufficient to capture
the local picture on the infinite-dimensional smooth manifolds, which
will be treated later in subsequent sections.

\subsection{Lie subalgebra of projectable vector fields\label{subsec:Lie-subalgebra-of-projectable}}

In this subsection we approach, from the infinitesimal level, the
desired structure of infinite-dimensional \emph{Lie subgroups} on
the following:
\begin{equation}
\Vau\left(\StandardFibration\right)\trianglelefteq\Aut\left(\StandardFibration\right)\leq\Diff\left(E\right).\label{eq:to-be-Lie-subgroups}
\end{equation}
Recall that for the diffeomorphism group $\Diff\left(E\right)$ with
its canonical Lie group structure, its Lie algebra $T_{\Identity}\Diff\left(E\right)$
is identified (via the exponential law) with the space $\mathfrak{X}\left(E\right)$
of smooth vector fields. In this way, the to-be determined infinitesimal
version of \prettyref{eq:to-be-Lie-subgroups} will take the form
of a Lie subalgebra sequence
\begin{equation}
\mathfrak{X}^{\mathrm{V}}\left(\StandardFibration\right)\trianglelefteq\mathfrak{X}^{\mathrm{A}}\left(\StandardFibration\right)\leq\mathfrak{X}\left(E\right).\label{eq:Lie-algebra-of-to-be-Lie-subgroups}
\end{equation}
Let us start with the more restrictive one: for the vertical automorphism
group $\Vau\left(\StandardFibration\right)$, the expected Lie algebra
turns out to be just the space of \emph{vertical vector fields} (i.e.,
sections of the vertical distribution):
\[
\mathfrak{X}^{\mathrm{V}}\left(\StandardFibration\right)\coloneqq\Gamma\left(T\StandardFibration\right)=\left\{ X\in\mathfrak{X}\left(E\right)\mid d\pi\circ X=0\right\} .
\]
In other words, a vertical field is characterized by the property
of being $\pi$-related to the zero field on $B$. This has an obvious
generalization that will turn out to be the expected Lie algebra of
$\Aut\left(\StandardFibration\right)$, as given in the following
definition:

\begin{defn}
\label{def:projectable-or-vertical-vector-fields}The space of \emph{projectable vector fields}
(or \emph{aligned vector fields}) of $\StandardFibration$, denoted
by $\mathfrak{X}^{\mathrm{A}}\left(\StandardFibration\right)$, is
a closed Fréchet Lie subalgebra of $\mathfrak{X}\left(E\right)$ given
as follows:
\[
\mathfrak{X}^{\mathrm{A}}\left(\StandardFibration\right)\coloneqq\left\{ X\in\mathfrak{X}\left(E\right)\mid d\pi\circ X=Y\circ\pi,\ \exists Y\in\mathfrak{X}\left(B\right)\right\} .
\]
In words, a smooth vector field $X$ on $E$ is projectable if it
is $\pi$-related to some smooth vector field on the base (such a
base field is unique if exists, and will be denoted by $\pi_{\ast}X$).
\end{defn}

By the preceding definition, the defining condition for $X$ to be
projectable exactly ensures its projection $\pi_{\ast}X$ to be well-defined
(hence the choice of terminology), with the vertical ones characterized
by the condition $\pi_{\ast}X=0$. This projection procedure is clearly
a Lie algebra homomorphism by the naturality of the Lie bracket, hence
we obtain a Lie algebra extension:
\begin{equation}
0\to\mathfrak{X}^{\mathrm{V}}\left(\StandardFibration\right)\to\mathfrak{X}^{\mathrm{A}}\left(\StandardFibration\right)\xrightarrow{\pi_{\ast}}\mathfrak{X}\left(B\right)\to0.\label{eq:infinitesimal-symmetry-sequence}
\end{equation}
In particular, we see that the Lie subalgebra $\mathfrak{X}^{\mathrm{V}}\left(\StandardFibration\right)$
is an ideal of $\mathfrak{X}^{\mathrm{A}}\left(\StandardFibration\right)$.
In fact, $\mathfrak{X}^{\mathrm{A}}\left(\StandardFibration\right)$
is the ``idealizer'' of $\mathfrak{X}^{\mathrm{V}}\left(\StandardFibration\right)$,
in the sense of the characterization
\[
X\in\mathfrak{X}^{\mathrm{A}}\left(\StandardFibration\right)\iff\left[X,Y\right]\in\mathfrak{X}^{\mathrm{V}}\left(\StandardFibration\right),\ \forall Y\in\mathfrak{X}^{\mathrm{V}}\left(\StandardFibration\right).
\]
In other words, a vector field is projectable if and only if its associated
Lie derivative operator sends vertical fields to vertical fields.
In fact, we have the following consideration in terms of flows. Recall
that as a regular Lie group, the total diffeomorphism group $\Diff\left(E\right)$
admits a \emph{Lie group exponential}. In general, this refers to
the map sending every element in the Lie algebra to the time-1 value
of the corresponding one-parameter subgroup of the Lie group. In our
case concerning the total diffeomorphism group $\Diff\left(E\right)$,
the Lie algebra $T_{\Identity}\Diff\left(E\right)$ is the space $\mathfrak{X}\left(E\right)$
of smooth vector fields on the total space, and each $V\in\mathfrak{X}\left(E\right)$
(or the associated right-invariant vector field thereof) generates
a one-parameter subgroup of $\Diff\left(E\right)$ that can be identified
— via the exponential law \prettyref{eq:exponential-law-general-in-appendix}
— with the usual flow $\mathrm{Fl}_{t}\left(V\right)$ on $E$ generated
by $V$. In sum, the Lie group exponential of $\Diff\left(E\right)$
is just given by the time-1 flow on $E$:
\begin{equation}
\LieExp{\Diff\left(E\right)}\colon T_{\Identity}\Diff\left(E\right)=\mathfrak{X}\left(E\right)\ni V\mapsto\mathrm{Fl}_{1}\left(V\right)\in\Diff\left(E\right),\qquad V_{p}=\left.\frac{\partial}{\partial t}\right|_{t=0}\mathrm{Fl}_{t}\left(V\right)\left(p\right)\in T_{p}E.\label{eq:Lie-exp}
\end{equation}
Under this correspondence, the Lie subalgebras $\mathfrak{X}^{\mathrm{V}}\left(\StandardFibration\right)$
and $\mathfrak{X}^{\mathrm{A}}\left(\StandardFibration\right)$ defined
above turn out to be associated with the to-be Lie subgroups $\Vau\left(\StandardFibration\right)$
and $\Aut\left(\StandardFibration\right)$, respectively, as shown
in the following lemma:
\begin{lem}
\label{lem:infinitesimal-Lie-subgroup}Under the Lie group exponential
$\LieExp{\Diff\left(E\right)}\colon\mathfrak{X}\left(E\right)\to\Diff\left(E\right)$
for the diffeomorphism group of $E$, we have the correspondences
\begin{align*}
\LieExp{\Diff\left(E\right)}\left(tX\right)\in\Aut\left(\StandardFibration\right),\,\forall t & \iff X\in\mathfrak{X}^{\mathrm{A}}\left(\StandardFibration\right);\\
\LieExp{\Diff\left(E\right)}\left(tX\right)\in\Vau\left(\StandardFibration\right),\,\forall t & \iff X\in\mathfrak{X}^{\mathrm{V}}\left(\StandardFibration\right).
\end{align*}
In words, the one-parameter subgroups of $\Diff\left(E\right)$ contained
in the subset $\Aut\left(\StandardFibration\right)$ (resp., $\Vau\left(\StandardFibration\right)$)
are those with infinitesimal generators in $\mathfrak{X}\left(E\right)$
contained in the Lie subalgebra $\mathfrak{X}^{\mathrm{A}}\left(\StandardFibration\right)$
(resp. $\mathfrak{X}^{\mathrm{V}}\left(\StandardFibration\right)$)
consisting of projectable vector fields (resp., vertical vector fields),.
\end{lem}

\begin{proof}
Recall from \prettyref{eq:Lie-exp} that via the exponential law,
the Lie group exponential of $\Diff\left(E\right)$ is just given
by the time-1 flow on $E$. Thus it suffices to show that a smooth
vector field generates automorphism flow (resp., vertical automorphism
flow) for all time if and only if it is a projectable field (resp.,
vertical field). Let $X\in\mathfrak{X}\left(E\right)$, and let $\mathrm{Fl}_{t}\left(X\right)\in\Diff\left(E\right)$
denote its flow at any time $t$. Recall that by definition, it is
an automorphism $\mathrm{Fl}_{t}\left(X\right)\in\Aut\left(\StandardFibration\right)$
if and only if it can be projected to some base transformation:
\begin{equation}
\pi\circ\mathrm{Fl}_{t}\left(X\right)=\beta_{t}\circ\pi,\qquad\exists\beta_{t}\in\Diff\left(B\right).\label{eq:projection-of-automorphism-in-the-proof}
\end{equation}
Taking infinitesimal generators yields the corresponding condition
at the infinitesimal level:
\[
\left.\frac{\partial}{\partial t}\right|_{t=0}\pi\left(\mathrm{Fl}_{t}\left(X\right)\left(p\right)\right)=\left.\frac{\partial}{\partial t}\right|_{t=0}\beta_{t}\left(\pi\left(p\right)\right).
\]
But by the chain rule, this recovers the defining condition for $X$
to be projectable $X\in\mathfrak{X}^{\mathrm{A}}\left(\StandardFibration\right)$
(with its projection $\pi_{\ast}X=Y$ being the infinitesimal generator
of the flow $\beta_{t}$ in the base), as desired:
\begin{equation}
d\pi\left(X\left(p\right)\right)=Y\left(\pi\left(p\right)\right)\qquad\text{with}\qquad Y\in\mathfrak{X}\left(B\right),\,Y\left(x\right)\coloneqq\left.\frac{\partial}{\partial t}\right|_{t=0}\beta_{t}\left(x\right).\label{eq:projecting-vector-fields-infinitesimal}
\end{equation}
Conversely, suppose that $X\in\mathfrak{X}^{\mathrm{A}}\left(\StandardFibration\right)$.
Then the condition \prettyref{eq:projecting-vector-fields-infinitesimal}
is satisfied for $Y\coloneqq\pi_{\ast}X$; but by the naturality of
flows (see e.g., \cite[Proposition 9.13]{MR2954043}), we have
\begin{equation}
\pi\circ\mathrm{Fl}_{t}\left(X\right)=\mathrm{Fl}_{t}\left(\pi_{\ast}X\right)\circ\pi,\label{eq:naturality-of-flows}
\end{equation}
so that the condition \prettyref{eq:projection-of-automorphism-in-the-proof}
is satisfied for $\beta_{t}\coloneqq\mathrm{Fl}_{t}\left(\pi_{\ast}X\right)$,
hence $\mathrm{Fl}_{t}\left(X\right)\in\Aut\left(\StandardFibration\right)$
as desired. Therefore, we have established the desired relation between
automorphism flows and projectable field.
Moreover, for $\mathrm{Fl}_{t}\left(X\right)$ to be a vertical automorphism
$\mathrm{Fl}_{t}\left(X\right)\in\Vau\left(\StandardFibration\right)$,
the defining condition amounts to requiring $\beta_{t}$ to be the
identity map, whose infinitesimal generator $Y$ is then the zero
field, which just recovers the defining condition $\pi_{\ast}X=0$
for $X$ to be vertical $X\in\mathfrak{X}^{\mathrm{V}}\left(\StandardFibration\right)$,
as desired.
\end{proof}
In finite dimensions, the infinitesimal information (such as the results
proved in this subsection) suffices to imply local information at
the level of manifolds themselves. For example, the familiar (finite-dimensional)
immersion theorem and submersion theorem, which are consequences of
the inverse function theorem, tell us that if a map has injective
(respectively, surjective) tangent map then it can be locally expressed
as a coordinate injection $\mathbb{R}^{m}\to\mathbb{R}^{m}\times\mathbb{R}^{n}$
(respectively, a coordinate projection $\mathbb{R}^{m}\times\mathbb{R}^{n}\to\mathbb{R}^{m}$)
of the Euclidean model spaces — however, these results no longer
hold in general for locally convex model spaces in infinite dimensions.
In particular, while one might be tempted by \prettyref{lem:infinitesimal-Lie-subgroup}
to use the Lie group exponential $\mathrm{Fl}_{1}\colon\mathfrak{X}\left(E\right)\to\Diff\left(E\right)$
to construct a desired submanifold chart for the automorphism group,
this will turn out to be false for the reason that it is not even
a chart. Indeed, the diffeomorphism group $\Diff\left(E\right)$ is
known to \emph{fail to be locally exponential} for any compact manifold
$E$ (see \cite[§51.3]{MR1471480} for the case $E=\Sphere 1$),
meaning that the Lie group exponential fails to be a local diffeomorphism
around $0$, even though its differential at $0$ is invertible. Again,
the missing ingredient here is the inverse function theorem in infinite
dimensions, which is a source of challenges in general in the study
of infinite-dimensional differential geometry: the infinitesimal information
at the level of tangent spaces no longer suffices to imply local information
at the level of manifolds themselves.\footnote{An important general strategy to overcome such challenges — which
though will not be pursued here — is to seek a suitable generalized
inverse function theorem that can be applied to the specific case
under consideration. See e.g., the Nash–Moser inverse function theorem
\cite{MR0656198}.} While we shall return to the study of such more subtle local structures
in subsequent sections, for now let us continue our infinitesimal
consideration: having understood the Fréchet subspaces $\mathfrak{X}^{\mathrm{V}}\left(\StandardFibration\right)\trianglelefteq\mathfrak{X}^{\mathrm{A}}\left(\StandardFibration\right)\leq\mathfrak{X}\left(E\right)$,
we next study their complements and splittings.

\subsection{Basic vector fields and horizontal lift}

In this subsection we approach, from the infinitesimal level, the
desired structure of infinite-dimensional \emph{Lie group extension}
— meaning that the projection is not only a surjective Lie group homomorphism
but also an infinite-dimensional \emph{smooth principal fibration}
— on the following:
\begin{equation}
\Vau\left(\StandardFibration\right)\hookrightarrow\Aut\left(\StandardFibration\right)\xrightarrow{\AutProjection}\Diff\left(B\right)_{\StandardFibration}.\label{eq:Lie-group-extension}
\end{equation}
The first condition, which requires $\AutProjection$ to be a surjective
Lie group homomorphism, is more familiar. At the infinitesimal level,
we have already met a natural candidate of the corresponding surjective
Lie algebra homomorphism; namely, the defining projection $\pi_{\ast}$
for projectable vector fields, as in the following Lie algebra extension:
\begin{equation}
0\to\mathfrak{X}^{\mathrm{V}}\left(\StandardFibration\right)\to\mathfrak{X}^{\mathrm{A}}\left(\StandardFibration\right)\xrightarrow{\pi_{\ast}}\mathfrak{X}\left(B\right)\to0.\label{eq:Lie-algebra-extension}
\end{equation}
The following lemma justifies such a desired correspondence under
the Lie group exponentials \prettyref{eq:Lie-exp} for diffeomorphism
groups:
\begin{lem}
\label{lem:infinitesimal-projection-of-projectable-field}Under the
Lie group exponential $\LieExp{\Diff\left(\cdot\right)}\colon\mathfrak{X}\left(\cdot\right)\to\Diff\left(\cdot\right)$
for the diffeomorphism groups of $E$ and of $B$, we have the relation
\[
\AutProjection\left(\LieExp{\Diff\left(E\right)}\left(tX\right)\right)=\LieExp{\Diff\left(B\right)}\left(t\pi_{\ast}\left(X\right)\right),\qquad\forall X\in\mathfrak{X}^{\mathrm{A}}\left(\StandardFibration\right).
\]
In words, the projection of one-parameter subgroups under $\AutProjection\colon\Aut\left(\StandardFibration\right)\to\Diff\left(B\right)_{\StandardFibration}$
is infinitesimally generated by the Lie algebra homomorphism $\pi_{\ast}\colon\mathfrak{X}^{\mathrm{A}}\left(\StandardFibration\right)\to\mathfrak{X}\left(B\right)$
given by the canonical projection of projectable fields.
\end{lem}

\begin{proof}
Recall from \prettyref{eq:Lie-exp} that via the exponential law,
the Lie group exponential of $\Diff\left(E\right)$ is just given
by the time-1 flow on $E$. Thus it suffices to show that for any
projectable vector field, the flow of the projection agrees with the
projection of its flow for all time:
\[
\AutProjection\left(\mathrm{Fl}_{t}\left(X\right)\right)=\mathrm{Fl}_{t}\left(\pi_{\ast}\left(X\right)\right),\qquad\forall X\in\mathfrak{X}^{\mathrm{A}}\left(\StandardFibration\right).
\]
But this is just equivalent to \prettyref{eq:naturality-of-flows}
in the proof of \prettyref{lem:infinitesimal-Lie-subgroup}, which
holds as a consequence of the naturality of flows. This completes
the proof.
\end{proof}
In summary, via the Lie group exponentials of diffeomorphism groups,
we have established relations between the (to-be) Lie group homomorphisms
in \prettyref{eq:Lie-group-extension} with their infinitesimal version
in \prettyref{eq:Lie-algebra-extension}, as illustrated in the following
commutative diagram:
\[
\xymatrix@C=4pc{\Vau\left(\StandardFibration\right)\ar@{^{(}->}[r] & \Aut\left(\StandardFibration\right)\ar@{->>}[r]^{\AutProjection} & \Diff\left(B\right)_{\StandardFibration}\\
\mathfrak{X}^{\mathrm{V}}\left(\StandardFibration\right)\ar@{^{(}->}[r]\ar[u]^{\LieExp{\Diff\left(E\right)}} & \mathfrak{X}^{\mathrm{A}}\left(\StandardFibration\right)\ar@{->>}[r]^{\pi_{\ast}}\ar[u]^{\LieExp{\Diff\left(E\right)}} & \mathfrak{X}\left(B\right)\ar[u]^{\LieExp{\Diff\left(B\right)}}
}
.
\]
Next, we turn to the second condition for \prettyref{eq:Lie-group-extension}
to be a genuine Lie group extension; namely, $\AutProjection$ should
also be a smooth principal fibration. This%
is a feature in infinite dimensions. Indeed, while being automatic
in finite dimensions, the local splitting (for a smooth bundle structure)
of an infinite-dimensional Lie group by a closed Lie subgroup does
not always exist, nor does the splitting of an infinite-dimensional
Fréchet space by a closed Fréchet subspace. In our case concerning
\prettyref{eq:Lie-group-extension}, the desired local splitting can
be conveniently described by a smooth local section
\[
\sigma\colon\Diff\left(B\right)_{\StandardFibration}\supsetopen U\to\Aut\left(\StandardFibration\right),\qquad\AutProjection\circ\sigma=\Identity_{U}.
\]
At the infinitesimal level \prettyref{eq:Lie-algebra-extension},
this amounts to a continuous linear (hence smooth)
section of Fréchet spaces:
\begin{equation}
\sigma'\colon\mathfrak{X}\left(B\right)\to\mathfrak{X}^{\mathrm{A}}\left(\StandardFibration\right),\qquad{\pi}_{\ast}\circ\sigma'=\Identity_{\mathfrak{X}\left(B\right)}.\label{eq:section-condition}
\end{equation}
In words, we need a continuous, linear procedure to ``lift'' each
vector field on the base space to a projectable vector field. Since
vertical fields are characterized by having zero projection, what
we are seeking a certain procedure of ``horizontal lifting'' along
each fiber — there is just a canonical one associated with a choice
of connection, as constructed in the following lemma:
\begin{lem}
\label{lem:horizontal-lifting}Given a connection of $\StandardFibration$,
the canonical Fréchet space projection of $\mathfrak{X}^{\mathrm{A}}\left(\StandardFibration\right)$
onto $\mathfrak{X}\left(B\right)$ admits a natural (continuous linear)
section $\sigma'\colon\mathfrak{X}\left(B\right)\to\mathfrak{X}^{\mathrm{A}}\left(\StandardFibration\right)$.
\end{lem}

\begin{proof}
Fix a connection $H\StandardFibration\subseteq TE$ of $\StandardFibration$.
Recall from \prettyref{eq:horizontal-lift-vector-space-iso} that
it induces a bundle isomorphism $\pi^{\ast}TB\cong H\StandardFibration$
(also called a horizontal lift before) whose inverse is induced by
$d\pi$, hence also a linear (even $\MappingSpace^{\infty}\left(E\right)$-linear)
isomorphism on the spaces of sections $\Gamma\left(\pi^{\ast}TB\right)\cong\Gamma\left(H\StandardFibration\right)$.
Under this identification, the horizontal bundle $H\StandardFibration$
satisfies the following commutative diagram for the pullback bundle
$\pi^{\ast}TB$, with the two dotted arrows indicating the two equivalent
descriptions of its sections:
\[
\begin{aligned}\xymatrix@C=4pc{\pi^{\ast}TB\cong H\StandardFibration\ar[r]^{d\pi}\ar[d] & TB\ar[d]\\
E\ar@/^{1pc}/@{..>}[u]^{X}\ar[r]_{\pi}\ar@{..>}[ru]^{f_{X}} & B
}
\end{aligned}
,\qquad\Gamma\left(\pi^{\ast}TB\right)\xrightarrow[\cong]{X\mapsto f_{X}}\MappingSpace_{\pi}^{\infty}\left(E,TB\right).
\]
In general for a pullback bundle along a fibration $\pi$, we can
``pull back a section'' along $\pi$ as follows. Given any section
$Y\in\Gamma\left(TB\right)$ of the original bundle, we construct
the map $f_{\widetilde{Y}}\coloneqq Y\circ\pi$ in $\MappingSpace_{\pi}^{\infty}\left(E,TB\right)$,
which thus determines a section $\widetilde{Y}\in\Gamma\left(\pi^{\ast}TB\right)$
in the pullback bundle. In sum, this yields a map
\begin{equation}
\sigma'\colon\Gamma\left(TB\right)\xrightarrow{Y\mapsto\widetilde{Y}}\Gamma\left(\pi^{\ast}TB\right)\cong\Gamma\left(H\StandardFibration\right).\label{eq:steps-in-a-infinitesimal-section}
\end{equation}
I claim that this map $\sigma'$ is the desired section. Indeed, $\sigma'$
is continuous and linear because each step in \prettyref{eq:steps-in-a-infinitesimal-section}
is so. Moreover, by construction each image $\sigma'\left(Y\right)$
is a projectable vector field that is projected to $Y$, hence $\sigma'$
satisfies the desired condition of being a section:
\[
\sigma'\colon\mathfrak{X}\left(B\right)\to\mathfrak{X}^{\mathrm{A}}\left(\StandardFibration\right),\qquad{\pi}_{\ast}\circ\sigma'=\Identity_{\mathfrak{X}\left(B\right)}.
\]
This completes the proof of \prettyref{lem:horizontal-lifting}.
\end{proof}
By unravelling the construction in the proof, we see that the linear
section $\sigma'$ in the preceding lemma can be described more explicitly
as lifting each $Y\in\mathfrak{X}\left(B\right)$ to a projectable
field $\sigma'\left(Y\right)\in\mathfrak{X}^{\mathrm{A}}\left(\StandardFibration\right)$
by declaring
\begin{equation}
\sigma'\left(Y\right)\coloneqq\widetilde{Y},\qquad\widetilde{Y}_{p}\coloneqq\left.d\pi_{p}\right|_{H_{p}\StandardFibration}^{-1}Y_{\pi\left(p\right)}\in H_{p}\StandardFibration.\label{eq:horizontal-lifting-explicit-formula}
\end{equation}
The vector field $\widetilde{Y}$ on $E$ is called the \emph{horizontal lift}
of $Y$, and the procedure $\sigma'$ is called the \emph{horizontal lifting of vector fields}
(along the fibers). The
image of $\sigma'$, which consists of those vector fields on $E$
that are horizontally lifted from some fields on the base, warrants
a definition as follows:
\begin{defn}
\label{def:basic-vector-field}The space of \emph{basic vector fields}
of $\StandardFibration$ (with respect to a given connection), denoted
by $\mathfrak{X}^{\mathrm{B}}\left(\StandardFibration\right)$, is
an Fréchet subspace given by the image of horizontal lifting:
\[
\mathfrak{X}^{\mathrm{B}}\left(\StandardFibration\right)\coloneqq\image\sigma'=\left\{ X\in\mathfrak{X}^{\mathrm{A}}\left(\StandardFibration\right)\mid X^{\top}=0\right\} .
\]
In words, a smooth vector field on $E$ is basic if it is projectable
and horizontal.
\end{defn}

By construction, the space $\mathfrak{X}^{\mathrm{B}}\left(\StandardFibration\right)$
of basic vector fields defined above provides a complement of the
Fréchet subspace $\mathfrak{X}^{\mathrm{V}}\left(\StandardFibration\right)$
in $\mathfrak{X}^{\mathrm{A}}\left(\StandardFibration\right)$. In
fact, it is a natural choice of complement in view of our geometric
setting, as the following lemma shows:
\begin{lem}
Let $\StandardFibration$ be set up as a Riemannian fibration with
respect to any Riemannian metric on $E$. Then the $L^{2}$-orthogonal
complement of $\mathfrak{X}^{\mathrm{V}}\left(\StandardFibration\right)\subseteq\mathfrak{X}^{\mathrm{A}}\left(\StandardFibration\right)$
is the space $\mathfrak{X}^{\mathrm{B}}\left(\StandardFibration\right)$
of basic vector fields (\prettyref{def:basic-vector-field}):
\begin{equation}
\mathfrak{X}^{\mathrm{A}}\left(\StandardFibration\right)=\mathfrak{X}^{\mathrm{B}}\left(\StandardFibration\right)\oplus_{\StandardFibration}^{L^{2}}\mathfrak{X}^{\mathrm{V}}\left(\StandardFibration\right).\label{eq:L2-splitting-of-automorphism-by-vertical-and-basic}
\end{equation}
In words, every projectable vector field is expressed uniquely as
the sum of a basic field and a vertical field, which are $L^{2}$-orthogonal
to each other.
\end{lem}

\begin{proof}
As already explained above, the splitting is a consequence of \prettyref{lem:horizontal-lifting}
and the definition $\mathfrak{X}^{\mathrm{B}}\left(\StandardFibration\right)\coloneqq\image\sigma'$.
Furthermore, this splitting is $L^{2}$-orthogonal because $\mathfrak{X}^{\mathrm{B}}\left(\StandardFibration\right)\subseteq\Gamma\left(H\StandardFibration\right)$
being horizontal is even pointwise orthogonal to $\mathfrak{X}^{\mathrm{V}}\left(\StandardFibration\right)$.
\end{proof}

As we shall see later, the $L^{2}$-orthogonal splitting \prettyref{eq:L2-splitting-of-automorphism-by-vertical-and-basic}
in the preceding lemma, as induced from a connection of the finite-dimensional
fibration $F\hookrightarrow E\to B$, will be a key to constructing
a differentio-geometric splitting of the infinite-dimensional fibration
$\Vau\left(\StandardFibration\right)\hookrightarrow\Aut\left(\StandardFibration\right)\to\Diff\left(B\right)_{\StandardFibration}$
in \prettyref{eq:Lie-group-extension}. This will provide one of our
primary instances of ``lifting the differential geometry from finite
dimensions to infinite''.

\subsection{Fair vector fields and horizontal average}

In this subsection we approach, from the infinitesimal level, the
desired structure of \emph{smooth principal fibration} on the following:
\begin{equation}
\Aut\left(\StandardFibration\right)\hookrightarrow\Diff\left(E\right)\xrightarrow{\DiffProjection}\Fib\left(\StandardFibration\right).\label{eq:main-d-bundle}
\end{equation}
Since we have not constructed the suitable differential structure
on $\Fib\left(\StandardFibration\right)$ yet, it is more convenient
to describe the desired bundle structure of \prettyref{eq:main-d-bundle}
by a smooth equivariant neighborhood retraction (cf. \prettyref{lem:splitting-lemma}):
\[
r\colon\Diff\left(E\right)\supsetopen W\to^{\Aut\left(\StandardFibration\right)}\Aut\left(\StandardFibration\right),\qquad\left.r\right|_{\Aut\left(\StandardFibration\right)}=\Identity_{\Aut\left(\StandardFibration\right)}.
\]
At the infinitesimal level \prettyref{eq:Lie-algebra-extension},
this amounts to a continuous linear (hence smooth) retraction of Fréchet
spaces (not to be confused with that of Lie algebras):
\begin{equation}
r'\colon\mathfrak{X}\left(E\right)\to\mathfrak{X}^{\mathrm{A}}\left(\StandardFibration\right),\qquad\left.r'\right|_{\mathfrak{X}^{\mathrm{A}}\left(\StandardFibration\right)}=\Identity_{\mathfrak{X}^{\mathrm{A}}\left(\StandardFibration\right)}.\label{eq:retraction-defining-condition}
\end{equation}
In words, we need a continuous, linear procedure to ``average''
each vector field on the total space to a projectable vector field.
Since vertical fields are characterized by having zero projection,
it is natural to seek a certain procedure of ``horizontal averaging''
over each fiber — there is just a canonical one associated with a
choice of connection and fiberwise measure, as constructed in the
following lemma:
\begin{lem}
\label{lem:horizontal-averaging}Given a connection and a fiberwise
measure of $\StandardFibration$, the canonical Fréchet space injection
$\mathfrak{X}^{\mathrm{A}}\left(\StandardFibration\right)\hookrightarrow\mathfrak{X}\left(E\right)$
admits a natural (continuous linear) retraction $r'\colon\mathfrak{X}\left(E\right)\to\mathfrak{X}^{\mathrm{A}}\left(\StandardFibration\right)$.
\end{lem}

\begin{proof}
Fix a connection $H\StandardFibration$ and a fiberwise measure $\left\{ d\mu_{E_{x}}\right\} _{x}$
of $\StandardFibration$. By declaring that $r'$ restricts to the
identity on the vertical subspace $\mathfrak{X}^{\mathrm{V}}\left(\StandardFibration\right)$,
we can split off $\mathfrak{X}^{\mathrm{V}}\left(\StandardFibration\right)$
from both the domains and codomains of $r'$, so that it suffices
to construct the restricted retraction from horizontal fields to basic
fields, which by an abuse of notation is again denoted by $r'$:
\[
r'\colon\Gamma\left(H\StandardFibration\right)\to\mathfrak{X}^{\mathrm{B}}\left(\StandardFibration\right).
\]
As in the proof of \prettyref{lem:horizontal-lifting}, recall again
from \prettyref{eq:horizontal-lift-vector-space-iso} that a connection
$H\StandardFibration$ induces a bundle isomorphism $\pi^{\ast}TB\cong H\StandardFibration$
(also called a horizontal lift before) whose inverse is induced by
$d\pi$, hence also a linear (even $\MappingSpace^{\infty}\left(E\right)$-linear)
isomorphism on the spaces of sections $\Gamma\left(\pi^{\ast}TB\right)\cong\Gamma\left(H\StandardFibration\right)$.
Under this identification, the horizontal bundle $H\StandardFibration$
satisfies the following commutative diagram for the pullback bundle
$\pi^{\ast}TB$, with the two dotted arrows indicating the two equivalent
descriptions of its sections:
\[
\begin{aligned}\xymatrix@C=4pc{\pi^{\ast}TB\cong H\StandardFibration\ar[r]^{d\pi}\ar[d] & TB\ar[d]\\
E\ar@/^{1pc}/@{..>}[u]^{X}\ar[r]_{\pi}\ar@{..>}[ru]^{f_{X}} & B
}
\end{aligned}
,\qquad\Gamma\left(\pi^{\ast}TB\right)\xrightarrow[\cong]{X\mapsto f_{X}}\MappingSpace_{\pi}^{\infty}\left(E,TB\right).
\]
In general for a pullback bundle along a fibration $\pi$ with a given
fiberwise measure $\left\{ d\mu_{E_{x}}\right\} _{x}$ of $\pi$,
we can ``average out a section'' along $\pi$ as follows. Given
any section $X\in\Gamma\left(\pi^{\ast}TB\right)$ of the pullback
bundle, the associated map $f_{X}\in\MappingSpace_{\pi}^{\infty}\left(E,TB\right)$
is generally not descendable to a map $B\to TB$ because $f_{X}$
is generally not constant along each fiber $E_{x}$; but with a fiberwise
measure $\left\{ d\mu_{E_{x}}\right\} _{x}$, we can take the average
of $f_{X}$ over each fiber $E_{x}$ with respect to $d\mu_{E_{x}}$,
so that we can define a descendable map $\overline{f_{X}}\in\MappingSpace_{\pi}^{\infty}\left(E,TB\right)$
by declaring that it takes the constant averaged value along each
$E_{x}$. The resulting averaged map $\overline{f_{X}}$ can then
be descended to a section $\overline{X}\in\Gamma\left(TB\right)$
of the original bundle, which in turn can be pulled back to the pullback
bundle. In sum, this yields a map $r'$
\begin{equation}
r'\colon\Gamma\left(H\StandardFibration\right)\cong\Gamma\left(\pi^{\ast}TB\right)\xrightarrow{X\mapsto\overline{X}}\underbrace{\Gamma\left(TB\right)\xrightarrow{Y\mapsto\widetilde{Y}}\Gamma\left(\pi^{\ast}TB\right)\cong\Gamma\left(H\StandardFibration\right)}_{\sigma'}\label{eq:step-in-horizontal-averaging}
\end{equation}
I claim that the thus constructed map $r'$ is the desired retraction
(more precisely, the horizontal part thereof). Indeed, $r'$ is continuous
and linear because each step in \prettyref{eq:step-in-horizontal-averaging}
is so. Moreover, by construction each image $r'\left(X\right)$ is
a basic vector field that is horizontally lifted from $\overline{X}$,
hence $r'$ satisfies the desired condition of being a retraction:
\[
r'\colon\Gamma\left(H\StandardFibration\right)\to\mathfrak{X}^{\mathrm{B}}\left(\StandardFibration\right),\qquad\left.r'\right|_{\mathfrak{X}^{\mathrm{B}}\left(\StandardFibration\right)}=\Identity_{\mathfrak{X}^{\mathrm{B}}\left(\StandardFibration\right)}.
\]
This completes the proof of \prettyref{lem:horizontal-averaging}.
\end{proof}
By unravelling the construction in the proof, we see that the linear
retraction $r'\colon\mathfrak{X}\left(E\right)\to\mathfrak{X}^{\mathrm{A}}\left(\StandardFibration\right)$
in the preceding lemma can be described more explicitly as averaging
each $X\in\mathfrak{X}\left(E\right)$ to a projectable field $r'\left(X\right)\in\mathfrak{X}^{\mathrm{A}}\left(\StandardFibration\right)$
by declaring
\begin{equation}
r'\left(X\right)\coloneqq X^{\top}+\sigma'\left(\overline{X}\right),\qquad\overline{X}_{x}\coloneqq\frac{1}{\mathrm{Vol}\left(E_{x}\right)}\int_{p\in E_{x}}d\pi_{p}\left(X_{p}\right)\,d\mu_{E_{x}}\in T_{x}B.\label{eq:horizontal-averaging-explicit-formula}
\end{equation}
The vector field $\overline{X}$ on $B$ is called the \emph{horizontal center}
of $X$, and the procedure $r'$ is called the \emph{horizontal averaging of vector fields}
(over the fibers). The kernel of $r'$, which consists of those vector
fields on $E$ that are horizontally averaged to zero, warrants a
definition as follows:
\begin{defn}
\label{def:fair-vector-field}The space of \emph{fair vector fields}
of $\StandardFibration$ (with respect to a given connection and fiberwise
measure), denoted by $\mathfrak{X}^{\mathrm{F}}\left(\StandardFibration\right)$,
is an Fréchet subspace given by the kernel of horizontal averaging:
\[
\mathfrak{X}^{\mathrm{F}}\left(\StandardFibration\right)\coloneqq\ker r'=\left\{ X\in\mathfrak{X}\left(E\right)\mid X^{\top}=0,\,\overline{X}=0\right\} .
\]
In words, a smooth vector field on $E$ is fair
if it is horizontal and has zero horizontal center.
\end{defn}

By construction, the space $\mathfrak{X}^{\mathrm{F}}\left(\StandardFibration\right)$
of fair vector fields defined above provides a complement of the Fréchet
subspace $\mathfrak{X}^{\mathrm{A}}\left(\StandardFibration\right)$
in $\mathfrak{X}\left(E\right)$. In fact, it is a natural choice
of complement in view of our geometric setting, as the following lemma
shows:
\begin{lem}
\label{lem:L2-splitting-by-fair-vector-fields}Let $\StandardFibration$
be set up as a Riemannian fibration with respect to any Riemannian
metric on $E$. Then the $L^{2}$-orthogonal complement of $\mathfrak{X}^{\mathrm{A}}\left(\StandardFibration\right)\subseteq\mathfrak{X}\left(E\right)$
is the space $\mathfrak{X}^{\mathrm{F}}\left(\StandardFibration\right)$
of fair vector fields (\prettyref{def:fair-vector-field}):
\begin{equation}
\mathfrak{X}\left(E\right)=\mathfrak{X}^{\mathrm{F}}\left(\StandardFibration\right)\oplus_{\StandardFibration}^{L^{2}}\mathfrak{X}^{\mathrm{A}}\left(\StandardFibration\right),\label{eq:L2-splitting-by-fair-vector-fields}
\end{equation}
In words, every vector field is expressed uniquely as the sum of a
fair field and a projectable field, which are $L^{2}$-orthogonal
to each other.
\end{lem}

\begin{proof}
As already explained above, the splitting is a consequence of \prettyref{lem:horizontal-averaging}
and the definition $\mathfrak{X}^{\mathrm{F}}\left(\StandardFibration\right)\coloneqq\ker r'$.
Furthermore, to show that this splitting is $L^{2}$-orthogonal, first
note that $\mathfrak{X}^{\mathrm{F}}\left(\StandardFibration\right)\subseteq\Gamma\left(H\StandardFibration\right)$
being horizontal is pointwise orthogonal to $\mathfrak{X}^{\mathrm{V}}\left(\StandardFibration\right)$,
so it suffices to restrict to horizontal fields and prove the $L^{2}$-orthogonality
of the restricted splitting
\[
\Gamma\left(H\StandardFibration\right)=\mathfrak{X}^{\mathrm{F}}\left(\StandardFibration\right)\oplus_{\StandardFibration}^{L^{2}}\mathfrak{X}^{\mathrm{B}}\left(\StandardFibration\right).
\]
Thus let $X\in\mathfrak{X}^{\mathrm{F}}\left(\StandardFibration\right)$
and $Y\in\mathfrak{X}^{\mathrm{B}}\left(\StandardFibration\right)$.
Then the pointwise inner product at each $p\in E$ is computed as
\begin{align*}
\left\langle X_{p},Y_{p}\right\rangle _{g} & =\left\langle d\pi\left(X_{p}\right),d\pi\left(Y_{p}\right)\right\rangle _{g_{B}} &  & \text{(\ensuremath{\because\,}\ensuremath{\pi} is Riemannian)}\\
 & =\left\langle d\pi\left(X_{p}\right),\left(\pi_{\ast}Y\right)_{x}\right\rangle _{g_{B}} &  & \text{(\ensuremath{\because\,}\ensuremath{Y} is basic)}
\end{align*}
It then follows that the fiberwise inner product along each $E_{x}$
is computed to be zero:
\begin{align*}
\int_{p\in E_{x}}\left\langle X_{p},Y_{p}\right\rangle _{g}\,d\mu_{E_{x}} & =\int_{p\in E_{x}}\left\langle d\pi\left(X_{p}\right),\left(\pi_{\ast}Y\right)_{x}\right\rangle _{g_{B}}\,d\mu_{E_{x}}\\
 & =\left\langle \left(\int_{p\in E_{x}}d\pi\left(X_{p}\right)\,d\mu_{E_{x}}\right),\left(\pi_{\ast}Y\right)_{x}\right\rangle _{g_{B}} &  & \text{(\ensuremath{\because\,}\ensuremath{\left(\pi_{\ast}Y\right)_{x}} independent of \ensuremath{p})}\\
 & =\left\langle 0,\left(\pi_{\ast}Y\right)_{x}\right\rangle _{g_{B}} &  & \text{(\ensuremath{\because\,}\ensuremath{X} is fair)}\\
 & =0
\end{align*}
Therefore, the $L^{2}$-inner product is zero too
\[
\left\langle X,Y\right\rangle _{\StandardFibration}\coloneqq\int_{x\in B}\int_{p\in E_{x}}\left\langle X_{p},Y_{p}\right\rangle _{g}\,d\mu_{E_{x}}\,d\mu_{B}=\int_{x\in B}0\,d\mu_{B}=0.
\]
This completes the proof.
\end{proof}
As we shall see later, the $L^{2}$-orthogonal splitting \prettyref{eq:L2-splitting-by-fair-vector-fields}
in the preceding lemma, as induced from a connection and a fiberwise
measure of the finite-dimensional fibration $F\hookrightarrow E\to B$,
will be a key to constructing a differentio-geometric splitting of
the infinite-dimensional fibration $\Aut\left(\StandardFibration\right)\hookrightarrow\Diff\left(E\right)\to\Fib\left(\StandardFibration\right)$
in \prettyref{eq:main-d-bundle}. Thus this will provide another one
of our primary instances of ``lifting the differential geometry from
finite dimensions to infinite''.

\section{Lie Group Extension of the Symmetries\label{sec:Lie-Group-Extension-of-the-Symmetries}}

\subsection{Prelude: parallel horizontal transports}

Let us start by recalling our symmetry sequence associated with $\StandardFibration$,
which is the following sequence of topological groups (with arrows
being continuous homomorphisms):
\begin{equation}
\Vau\left(\StandardFibration\right)\hookrightarrow\Aut\left(\StandardFibration\right)\xrightarrow{\AutProjection}\Diff\left(B\right)_{\StandardFibration}.\label{eq:symmetry-sequence}
\end{equation}
To study the principal fibration structures on this sequence, we shall
first give a hands-on construction of a (topological) local section.
This will base on the following procedure for fiber bundles as analogous
to the parallel transport for vector bundles.
\begin{defn}
\label{def:Horizontal-transport}Given any two endpoints $x_{0},x\in B$
sufficiently close to each other (so as to be connected by a unique
minimizing geodesic). Then the \emph{parallel horizontal transport}
(along geodesic) from $x_{0}$ to $x$ is the following map between
the concrete fibers $E_{x_{0}}$ and $E_{x}$ over the two given points:
\begin{gather*}
\HTransport_{x_{0},x}\colon E_{x_{0}}\to E_{x},\qquad q\mapsto\widetilde{\gamma}_{q}\left(1\right),\qquad\text{where}\\
\widetilde{\gamma}_{q}\colon\left[0,1\right]\to E,\qquad\widetilde{\gamma}_{q}\left(0\right)=q,\quad\pi\left(\widetilde{\gamma}_{q}\left(t\right)\right)=\gamma\left(t\right)\coloneqq\exp_{x_{0}}\left(t\exp_{x_{0}}^{-1}\left(x\right)\right).
\end{gather*}
In words, let $\gamma\colon\left[0,1\right]\to B$ be the unique minimizing
geodesic from $x_{0}$ to $x$, and for each $q\in E_{x_{0}}$ let
$\widetilde{\gamma}_{q}\colon\left[0,1\right]\to E$ be the unique
horizontal lift of $\gamma$ starting at $q$, then its terminal point
$\widetilde{\gamma}_{q}\left(1\right)\in E_{x}$ is assigned as the
image of $q\in E_{x_{0}}$ under the horizontal transport.
\end{defn}

The preceding definition provides a geometrically natural way of moving
points from one fiber to another, from which we can construct a local
section for the desired bundle structure of automorphisms fibered
by vertical automorphisms, as the following proposition shows.
\begin{prop}
\label{prop:The-topological-fibration-of-Aut-by-Vau}The canonical
quotient projection associated with the (normal) subgroup $\Vau\left(\StandardFibration\right)\leq\Aut\left(\StandardFibration\right)$
gives the following topological principal
bundle:
\[
\Vau\left(\StandardFibration\right)\hookrightarrow\Aut\left(\StandardFibration\right)\xrightarrow{\AutProjection}\Diff\left(B\right)_{\StandardFibration}.
\]
Specifically, the bundle structure is given by a certain local section
$\sigma$ into $\Aut\left(\StandardFibration\right)$, whose construction
is induced by the procedure of parallel horizontal transport between
fibers along base geodesics.
\end{prop}

\begin{proof}
First note that the map $\AutProjection\colon\Aut\left(\StandardFibration\right)\to\Diff\left(B\right)_{\StandardFibration}$
can be viewed as the coset quotient projection for the coset space
$\Aut\left(\StandardFibration\right)/\Vau\left(\StandardFibration\right)$,
thus for the bundle structure it suffices to construct a local section
$\sigma\colon U\to\Aut\left(\StandardFibration\right)$ over some
open neighborhood around the identity map $\Identity_{B}$ (see \prettyref{lem:splitting-lemma}).
To this end, let $U'\subseteq\Diff\left(B\right)_{\StandardFibration}$
be a sufficiently small open neighborhood around $\Identity_{B}$
so that, in particular, for each $\beta\in U'$ and $x\in B$ there
exists a unique minimizing geodesic in $B$ from $x$ to $\beta\left(x\right)$.
Then along such base geodesics we can apply the horizontal transport
$\HTransport$ as in \prettyref{def:Horizontal-transport}, which
yields the following transformation on the total space:
\begin{equation}
\sigma\left(\beta\right)\colon E\to E,\qquad E_{x}\ni q\mapsto\HTransport_{x,\beta\left(x\right)}\left(q\right)\in E_{\beta\left(x\right)}.\label{eq:horizontal-transport-formula}
\end{equation}
It is straightforward to verify that the map $\sigma\left(\beta\right)$
on $E$ is smooth, and that the assignment $\beta\mapsto\sigma\left(\beta\right)$
gives a continuous map $\sigma\colon U'\to\MappingSpace^{\infty}\left(E,E\right)$.
Further, since $\Diff\left(E\right)$ is an open subset of $\MappingSpace^{\infty}\left(E,E\right)$,
its pre-image $U\coloneqq\sigma^{-1}\left(\Diff\left(E\right)\right)$
is an open subset of $U'$, so that $\sigma$ is restricted to a map
$U\to\Diff\left(E\right)$ whose image consists entirely of diffeomorphisms.
Lastly, by construction, each such diffeomorphism $\sigma\left(\beta\right)$
maps fibers to fibers, whose induced base transformation recovers
exactly $\beta$ itself. Therefore, in sum, we obtain the desired
local section
\begin{equation}
\sigma\colon\Diff\left(B\right)_{\StandardFibration}\supsetopen U\to\Aut\left(\StandardFibration\right),\qquad\AutProjection\circ\sigma=\Identity_{U}.\label{eq:local-section-in-small-bundle}
\end{equation}
This completes the proof.
\end{proof}

The preceding proposition (\prettyref{prop:The-topological-fibration-of-Aut-by-Vau})
reveals the principal $\Vau\left(\StandardFibration\right)$-fibration
structure of $\Aut\left(\StandardFibration\right)$ in the topological
category. With in mind the goal of understanding its infinite-dimensional
differential geometry, we shall next promotes this principal fibration
to manifold categories in the order of increasingly rich structures:
firstly of topological Hilbert manifolds, then of smooth Fréchet manifolds,
and lastly alluding to structures of $L^{2}$ weak Riemannian manifolds.

\subsection{Topological manifold structures}

To embark on our (infinite-dimensional) differential geometric study,
a crucial starting point is to endow our symmetry groups with suitable
smooth submanifold structures of the diffeomorphism group. More specifically,
we need to construct a suitable smooth chart for $\Diff\left(E\right)$
modeled on vector fields (near the identity) such that it restricts
to a desired smooth chart for $\Aut\left(\StandardFibration\right)$
modeled on projectable fields:
\[
\Exp_{\StandardFibration}\colon V\cap\mathfrak{X}^{\mathrm{A}}\left(\StandardFibration\right)\xrightarrow{\approx}U\cap\Aut\left(\StandardFibration\right).
\]
Although the canonical smooth structure on $\Diff\left(E\right)$
has been well understood and its charts can be constructed via any
local addition on $E$, it is a priori not clear whether there exist
submanifold charts for $\Aut\left(\StandardFibration\right)$. Previously
in \prettyref{sec:Tangent-Spaces-and-Lie-Algebras}, we considered
(but dismissed) a naive candidate: as a regular Lie group, $\Diff\left(E\right)$
admits a Lie group exponential defined on its Lie algebra (identified
with) $T_{\Identity}\Diff\left(E\right)=\mathfrak{X}\left(E\right)$,
which can be viewed as flowing along the integral curves on $E$ for
a unit time; although this map does enjoy the desired property of
sending the subspace $\mathfrak{X}^{\mathrm{A}}\left(\StandardFibration\right)$
into the subgroup $\Aut\left(\StandardFibration\right)$ (\prettyref{lem:infinitesimal-Lie-subgroup}),
it is however generally not a local diffeomorphism. On the other hand,
as a (weak) $L^{2}$-Riemannian manifold, $\Diff\left(E\right)$ admits
a Riemannian exponential (restricted at the identity) on $T_{\Identity}\Diff\left(E\right)=\mathfrak{X}\left(E\right)$,
which can be viewed as flowing along the Riemannian geodesics on $E$
for a unit time:
\begin{equation}
\Exp\colon\mathfrak{X}\left(E\right)\supsetopen V\xrightarrow{\approx}U\subsetopen\Diff\left(E\right),\qquad X\mapsto\exp_{E}\circ X.\label{eq:exponential-chart-ord}
\end{equation}
This map $\Exp$ is indeed a chart for the canonical smooth structure
of $\Diff\left(E\right)$ because it is induced from a local addition;
namely, the Riemannian exponential $\exp_{E}\colon TE\to E$ on $E$.
However, it is generally not sending the subspace $\mathfrak{X}^{\mathrm{A}}\left(\StandardFibration\right)$
into the subgroup $\Aut\left(\StandardFibration\right)$: an initial
point in a fiber traveling along a geodesic with vertical initial
velocity may very well end up outside the fiber, and two initial points
in the same fiber traveling along geodesics with horizontal initial
velocities of the same projected image in the base may very well end
up in two different fibers. To remedy this, we adapt the exponential
map for $\StandardFibration$ by traveling — instead of along the
geodesic of the total space $E$ — firstly along the geodesic of the
fiber $E_{x}$ and then along the horizontal lift of the geodesic
of the base $B$, as in the following definition:
\begin{defn}
\label{def:adapted-Riemannian-exponential-map}The \emph{adapted Riemannian exponential map}
for $\StandardFibration$, denoted by $\exp_{\StandardFibration}$,
is a map defined on some sufficiently small neighborhood $V$ in $TE$
(around the zero section), given as follows for any $w\in T_{p}E$
(at each $p\in E_{x}$ over each $x\in B$):
\[
\exp_{\StandardFibration}\colon TE\supsetopen V\to E,\qquad w\mapsto\HTransport_{x,\exp_{B}\left(\overline{w}\right)}\left(\exp_{E_{x}}\left(w^{\top}\right)\right).
\]
In words, each image point $\exp_{\StandardFibration}\left(w\right)$
is obtained from the initial point $p$ by two steps: first vertically
traveling for a unit time along the $E_{x}$-geodesic with an initial
velocity of $w^{\top}\in T_{p}E_{x}$ in the fiber, and then horizontally
transporting for a unit time along the $B$-geodesic with an initial
velocity of $\overline{w}\in T_{x}B$ in the base.
\end{defn}

Just as the ordinary exponential $\exp_{E}$ induces a chart \prettyref{eq:exponential-chart-ord}
on $\Diff\left(E\right)$, the adapted Riemannian exponential $\exp_{\StandardFibration}$
for $\StandardFibration$ as defined above induces the following candidate
of a desired adapted chart for $\Aut\left(\StandardFibration\right)$:
\[
\Exp_{\StandardFibration}\colon\mathfrak{X}\left(E\right)\supsetopen V\to\Diff\left(E\right),\qquad X\mapsto\exp_{\StandardFibration}\circ X.
\]
Observe how the restriction of $\Exp_{\StandardFibration}$ to $\mathfrak{X}^{\mathrm{A}}\left(\StandardFibration\right)$
fits perfectly with everything so far in this section, as illustrated
in the following diagram:
\begin{equation}
\begin{aligned}\xymatrix@C=4pc{U\cap\Vau\left(\StandardFibration\right)\ar@{^{(}->}[r] & U\cap\Aut\left(\StandardFibration\right)\ar@{->>}[r]_{\AutProjection} & \overline{U}\cap\Diff\left(B\right)_{\StandardFibration}\ar@/_{1pc}/[l]_{\sigma}\\
V\cap\mathfrak{X}^{\mathrm{V}}\left(\StandardFibration\right)\ar@{^{(}->}[r]\ar[u]_{\approx}^{\bigsqcup_{x\in B}\exp_{E_{x}}\circ\left(-\right)} & V\cap\mathfrak{X}^{\mathrm{A}}\left(\StandardFibration\right)\ar@{-->}[u]_{\approx}^{\Exp_{\StandardFibration}\coloneqq\exp_{\xi}\circ\left(-\right)}\ar@{->>}[r]^{\pi_{\ast}} & \overline{V}\cap\mathfrak{X}\left(B\right)\ar[u]_{\approx}^{\exp_{B}\circ\left(-\right)}\ar@/^{1pc}/[l]^{\sigma'}
}
\end{aligned}
.\label{eq:diagram-construction-of-adapted}
\end{equation}
Here, the local chart for $\Vau\left(\StandardFibration\right)$ on
$\mathfrak{X}^{\mathrm{V}}\left(\StandardFibration\right)$ (resp.,
for $\Diff\left(B\right)_{\StandardFibration}$ on $\mathfrak{X}\left(B\right)$)
is given by post-composition with the fiberwise exponential $\bigsqcup_{x}\exp_{E_{x}}$
(resp., the base exponential $\exp_{B}$), which is just given by
the ordinary Riemannian exponential of the concrete fibers $E_{x}$
(resp., of the base $B$). Then these two charts can be in turn composed
together by virtue of the splittings of both rows, which have already
been justified previously: a linear section $\sigma'$ of the projection
$\mathfrak{X}^{\mathrm{A}}\left(\StandardFibration\right)\to\mathfrak{X}\left(B\right)$
in the Fréchet space category has been provided in \prettyref{lem:horizontal-lifting}
by horizontal lift, while a continuous local section $\sigma$ of
the projection $\Aut\left(\StandardFibration\right)\to\Diff\left(B\right)_{\StandardFibration}$
in the topological category has been provided in \prettyref{prop:The-topological-fibration-of-Aut-by-Vau}
by horizontal transport. Therefore, we have reached the following
proposition:
\begin{prop}
\label{prop:topological-manifold-structures-on-symmetry-groups}The
group $\Aut\left(\StandardFibration\right)$ of automorphisms admits
the structure of a topological manifold modeled on the space $\mathfrak{X}^{\mathrm{A}}\left(\StandardFibration\right)$
of projectable fields, whose local chart near the identity can be
given by post-composition with the adapted Riemannian exponential
(\prettyref{def:adapted-Riemannian-exponential-map}); i.e., the homeomorphism
\begin{equation}
\left.\Exp_{\StandardFibration}\right|_{\mathfrak{X}^{\mathrm{A}}}\colon V\cap\mathfrak{X}^{\mathrm{A}}\left(\StandardFibration\right)\xrightarrow{\approx}U\cap\Aut\left(\StandardFibration\right),\qquad X\mapsto\exp_{\StandardFibration}\circ X.\label{eq:adapted-exponential-on-Aut-chart}
\end{equation}
Further, this chart is constructed exactly such that the diagram \prettyref{eq:diagram-construction-of-adapted}
commutes, and hence the topological principal $\Vau\left(\StandardFibration\right)$-fibration
$\Aut\left(\StandardFibration\right)$ in \prettyref{prop:The-topological-fibration-of-Aut-by-Vau}
is promoted to the category of topological manifolds.
\end{prop}

\begin{proof}
This essentially has been shown from the diagram \prettyref{eq:diagram-construction-of-adapted}
(and the subsequent discussion thereof), which indeed commutes by
directly verifying the following equation for (sufficiently small)
projectable vector field $X$:
\[
\Exp_{\StandardFibration}\left(X\right)=\sigma\left(\exp_{B}\circ\pi_{\ast}X\right)\circ\left(\bigsqcup_{x}\exp_{E_{x}}\circ X^{\top}\right),\qquad X\in\mathfrak{X}^{\mathrm{A}}\left(\StandardFibration\right).
\]
From this formula we see that the map $\Exp_{\StandardFibration}$
between neighborhoods of $0\in\mathfrak{X}^{\mathrm{A}}\left(\StandardFibration\right)$
and $\Identity\in\Aut\left(\StandardFibration\right)$ can be broken
down into a composition of the following three maps:
\begin{align*}
V\cap\mathfrak{X}^{\mathrm{A}}\left(\StandardFibration\right) & \to\left(\overline{V}\cap\mathfrak{X}\left(B\right)\right)\times\left(V\cap\mathfrak{X}^{\mathrm{V}}\left(\StandardFibration\right)\right) & X & \mapsto\left(\pi_{\ast}X,X^{\top}\right)\\
 & \to\left(\overline{U}\cap\Diff\left(B\right)_{\StandardFibration}\right)\times\left(U\cap\Vau\left(\StandardFibration\right)\right) & \left(Y,Z\right) & \mapsto\left(\exp_{B}\circ Y,\bigsqcup_{x}\exp_{E_{x}}\circ Z\right)\\
 & \to U\cap\Aut\left(\StandardFibration\right) & \left(\beta,k\right) & \mapsto\sigma\left(\beta\right)\circ k.
\end{align*}
Thus for justifying $\Exp_{\StandardFibration}$ a chart, it suffices
to show that each of these three maps is a homeomorphism. But this
is clear since, on the one hand, the first and the third maps are
homeomorphisms since they just the splittings of the corresponding
sequences (i.e., the two rows of the diagram in \prettyref{eq:diagram-construction-of-adapted})
given by the sections $\sigma'$ in \prettyref{lem:horizontal-lifting}
and $\sigma$ in \prettyref{prop:The-topological-fibration-of-Aut-by-Vau},
respectively. On the other hand, the second map is also a homeomorphism
since it is induced from $\exp_{B}$ and $\exp_{E_{x}}$ for $x\in B$
which (being just the ordinary Riemannian exponentials) are local
additions. This completes the proof.
\end{proof}
In particular, the preceding proposition implies that the symmetry
groups of $\StandardFibration$ have just the desired type of topological
structures which was set up in \prettyref{sec:Infinite-dimensional-topology},
as recorded in the following corollary:
\begin{cor}
The automorphism group $\Aut\left(\StandardFibration\right)$ is a
topological $\ell^{2}$-manifold (\prettyref{def:l2-manifold}), and
hence has the following property: if it admits a closed subset $N$
that is also a manifold, such that the subset inclusion is a weak
homotopy equivalence, then it is homeomorphic to $N\times\ell^{2}$
and contains $N$ as a (strong) deformation retract:
\begin{equation}
\text{\ensuremath{N\xhookrightarrow{\simeq}\Aut\left(\StandardFibration\right)}\ensuremath{\enskip}weak equivalence}\implies\begin{cases}
\Aut\left(\StandardFibration\right)\xleftrightarrow{\approx}N\times\ell^{2} & \text{homeomorphism}\\
\Aut\left(\StandardFibration\right)\xtwoheadrightarrow{}{\simeq}N & \text{deformation retraction}
\end{cases}.\label{eq:l2-property-for-aut}
\end{equation}
Moreover, this deformation retract $N$ is a minimal deformation retract
(a ``core'') of $\Aut\left(\StandardFibration\right)$ provided
that $N$ has compact components. Similarly, this result holds verbatim
for the vertical automorphism group $\Vau\left(\StandardFibration\right)$.
\end{cor}

\begin{proof}
The preceding proposition (\prettyref{prop:topological-manifold-structures-on-symmetry-groups})
implies that $\Aut\left(\StandardFibration\right)$ is an infinite-dimensional
topological manifold modeled on $\mathfrak{X}^{\mathrm{A}}\left(\StandardFibration\right)$.
Moreover, since $\Aut\left(\StandardFibration\right)$ is a topological
subspace of $\Diff\left(E\right)$, it inherits the property of being
Hausdorff and second-countable; on the other hand, since $\mathfrak{X}^{\mathrm{A}}\left(\StandardFibration\right)$
is a closed linear subspace of $\mathfrak{X}\left(E\right)$, it inherits
the property of being a separable Fréchet space. Therefore, $\Aut\left(\StandardFibration\right)$
is a topological $\ell^{2}$-manifold (in the sense of \prettyref{def:l2-manifold})
by the Kadec–Anderson theorem (\prettyref{thm:Kadec-Anderson}). Thus
the desired properties \prettyref{eq:l2-property-for-aut} for $\Aut\left(\StandardFibration\right)$
follows from the general properties of $\ell^{2}$-manifold (\prettyref{thm:main-theorem-for-infinite-dimensional-topology}
and \prettyref{cor:main-theorem-for-infinite-dimensional-topology-compact-case}).
This completes the proof for the case of $\Aut\left(\StandardFibration\right)$,
which also can be applied verbatim to the case of $\Vau\left(\StandardFibration\right)$
too, as desired.
\end{proof}

\subsection{Smooth manifold structures}

Having used the adapted Riemannian exponential $\exp_{\StandardFibration}$
(\prettyref{def:adapted-Riemannian-exponential-map}) to construct
a topological manifold structure $\Exp_{\StandardFibration}$ on our
symmetry groups $\Aut\left(\StandardFibration\right)$ and $\Vau\left(\StandardFibration\right)$
(\prettyref{prop:topological-manifold-structures-on-symmetry-groups}),
we are now in a position to move on to the (infinite-dimensional)
smooth category. Here, the crux is not whether they are smooth manifolds,
but rather whether they are smooth submanifolds of the diffeomorphism
group $\Diff\left(E\right)$ with its canonical smooth structure.
This is crucial because $\Diff\left(E\right)$, with its canonical
smooth structure provided by the ordinary Riemannian exponential,
is the starting point of our story (where our moduli space $\Fib\left(\StandardFibration\right)$
by definition inherits the structures of $\Diff\left(E\right)$ as
a homogeneous space). The following lemma assures that this privileged
smooth structure is not altered when the exponential is adapted:
\begin{lem}
\label{lem:smooth-compatibility}The (unrestricted) map $\Exp_{\StandardFibration}$
induced by the adapted Riemannian exponential $\exp_{\StandardFibration}$
for $\StandardFibration$ (\prettyref{def:adapted-Riemannian-exponential-map});
i.e., the map
\begin{equation}
\Exp_{\StandardFibration}\colon\mathfrak{X}\left(E\right)\supsetopen V\to\Diff\left(E\right),\qquad X\mapsto\exp_{\StandardFibration}\circ X,\label{eq:adapted-chart-induced-by-adapted-exp}
\end{equation}
is smoothly compatible with the chart $\Exp$ induced from the ordinary
Riemannian exponential $\exp_{E}$ as in \prettyref{eq:exponential-chart-ord},
and hence serve as a chart for the canonical smooth structure of $\Diff\left(E\right)$
near the identity.
\end{lem}

\begin{proof}
In general, recall from \prettyref{thm:local-addition-implies-canonical-smooth-manifold-structure}
that every local addition $\alpha\colon TE\supsetopen W\to E$ on
$E$ induces a chart for $\Diff\left(E\right)$ modeled on $\mathfrak{X}\left(E\right)$
by post-composition with $\alpha$, and it was shown in \prettyref{lem:canonical-smooth-structure-is-unique}
that the resulting canonical smooth structure on $\Diff\left(E\right)$
is independent of the choice of $\alpha$ (i.e., the charts induced
by any two local additions are always smoothly compatible). In the
current case, our candidate $\Exp_{\StandardFibration}$ of a chart
for $\Diff\left(E\right)$ modeled on $\mathfrak{X}\left(E\right)$,
as given in \prettyref{eq:adapted-chart-induced-by-adapted-exp},
is induced by the adapted Riemannian exponential $\exp_{\StandardFibration}$
in the same way (i.e., via pushforward / post-composition). Therefore,
for the current lemma it suffices to prove the claim that
\[
\exp_{\StandardFibration}\colon TE\supsetopen V\xrightarrow{w\mapsto\HTransport_{x,\exp_{B}\left(\overline{w}\right)}\left(\exp_{E_{x}}\left(w^{\top}\right)\right)}E\quad\text{is a local addition on \ensuremath{E}}.
\]
This was already shown in \cite{MR1748878} in the general setting
of Riemannian foliations; in our case, it is straightforward to justify
this claim as follows. We start with setting up $\StandardFibration$
as a Riemannian fibration, so that we can construct a smooth local
trivialization $\varphi_{x}^{-1}$ of $\StandardFibration$ centered
around each fiber $E_{x}$ by horizontal transporting $E_{x}$ to
its nearby fibers:
\begin{equation}
\varphi_{x}^{-1}\colon U_{B}\times E_{x}\xrightarrow{\approx}\pi^{-1}\left(U_{B}\right),\qquad\left(x',q\right)\mapsto\HTransport_{x,x'}\left(q\right).\label{eq:horizontal-transport}
\end{equation}
Here, $\varphi_{x}^{-1}$ is indeed a diffeomorphism by virtue of
the assumption that $\StandardFibration$ is a Riemannian fibration,
so that the horizontal transport $\HTransport_{x,x'}\left(q\right)$
can be obtained by just traversing the horizontal geodesic in $E$
starting from $q$; namely, $\HTransport_{x,x'}\left(q\right)=\exp_{E}\left(\widetilde{u}\right)$
where $\widetilde{u}\in T_{q}^{\perp}E$ is the horizontal lift of
the base vector $\exp_{B,x}^{-1}\left(x'\right)\in T_{x}B$. Thus
for each point $p\in E_{x}$, the restriction of $\exp_{\StandardFibration}$
to the tangent space $T_{p}E=T_{x}B\times T_{p}E_{x}$ can be expressed
as the following composition of diffeomorphisms:
\[
\exp_{\StandardFibration,p}\colon T_{x}B\times T_{p}E_{x}\supsetopen V_{B}\times V_{E_{x}}\xrightarrow[\approx]{\exp_{B}\times\exp_{E_{x}}}U_{B}\times U_{E_{x}}\xrightarrow[\approx]{\varphi_{x}^{-1}}\pi^{-1}\left(U_{B}\right)\subsetopen E.
\]
This justifies that the restricted exponential $\exp_{\StandardFibration,p}$
is a diffeomorphism from some $V_{p}\subsetopen T_{p}E$ onto its
open image in $E$, hence its differential $d_{0_{p}}\exp_{\StandardFibration,p}$
at zero is an invertible linear transformation on $T_{0_{p}}T_{p}E\cong T_{p}E$.
To justify that $\exp_{\StandardFibration}$ is a local addition,
we need to verify the defining condition \prettyref{eq:local-addition-diff-requirement}
in \prettyref{def:local-addition}; i.e., that the induced map 
\[
\left(\left.\pi_{TE}\right|_{V},\exp_{\StandardFibration}\right)\colon TE\supsetopen V\to E\times E
\]
is a diffeomorphism onto an open neighborhood around the diagonal.
But this is clear since for each $p\in E$ its differential at the
zero vector $0_{p}\in T_{p}E$ is given by a linear transformation
on $T_{0_{p}}TE\cong T_{p}E\times T_{p}E$ as follows:
\[
d_{0_{p}}\left(\left.\pi_{TE}\right|_{V},\exp_{\StandardFibration}\right)=\left(\begin{array}{c|c}
I & 0\\
\hline I & d_{0_{p}}\exp_{\StandardFibration,p}
\end{array}\right),
\]
which is invertible because $d_{0_{p}}\exp_{\StandardFibration,p}$
is invertible. Since the domain $TE$ is finite-dimensional by assumption,
we can invoke the usual (finite-dimensional) inverse function theorem
to deduce that the map $\left(\left.\pi_{TE}\right|_{V},\exp_{\StandardFibration}\right)$
is indeed a local diffeomorphism, as desired. This justifies that
the adapted Riemannian exponential $\exp_{\StandardFibration,p}$
is a local addition, hence completes the proof of \prettyref{lem:smooth-compatibility}.
\end{proof}
In conclusion, all of this culminates in the following structure
theorem for our infinite-dimensional symmetry groups: 
\begin{thm}
\label{thm:lie-group-structure-on-symmetry-group}Both $\Aut\left(\StandardFibration\right)$
and $\Vau\left(\StandardFibration\right)$ are Fréchet Lie groups,
which fit in the following series of Fréchet Lie subgroups of diffeomorphisms
on the total space:%
\begin{equation}
\Vau\left(\StandardFibration\right)\leq\Aut\left(\StandardFibration\right)\leq\Diff\left(E\right),\label{eq:Lie-subgroups}
\end{equation}
whose Lie algebras are given by the following series of Fréchet Lie
subalgebras of vector fields on the total space:
\begin{equation}
\mathfrak{X}^{\mathrm{V}}\left(\StandardFibration\right)\leq\mathfrak{X}^{\mathrm{A}}\left(\StandardFibration\right)\leq\mathfrak{X}\left(E\right).\label{eq:Lie-subalgebras}
\end{equation}
Here, the Fréchet smooth manifold structures in the former series
are also given by the Fréchet spaces in the latter series as the model
spaces near the identity, where the submanifold chart is induced by
the adapted Riemannian exponential $\exp_{\StandardFibration}$ by
post-composition (as in \prettyref{lem:smooth-compatibility}).
\end{thm}

\subsection{The smooth fibration structure}

Finally, we are in a position to complete the proof of the following
main theorem in this section, which underlies the infinite-dimensional
differential geometry of our symmetry groups: 
\begin{thm}
\label{thm:main-theorem-about-smooth-symmetry-fibration}Let the diffeomorphism
groups of $E$ and of $B$ be equipped with the canonical smooth structure,
and let the automorphism group of $\left(E,\xi\right)$ be equipped
with the induced smooth structure as a submanifold. Then the topological
principal bundle (as constructed in \prettyref{prop:The-topological-fibration-of-Aut-by-Vau})%
\[
\Xi\colon\Vau\left(\StandardFibration\right)\hookrightarrow\Aut\left(\StandardFibration\right)\xrightarrow{\AutProjection}\Diff\left(B\right)_{\StandardFibration}\qquad\text{is a smooth fibration}.
\]
More precisely, the smooth principal bundle structure on $\Xi$ is
induced from a smooth slice $\mathrm{\mathcal{S}}^{\mathrm{B}}\left(\StandardFibration\right)\subseteq\Aut\left(\StandardFibration\right)$,
which is modeled on the Fréchet space $\mathfrak{X}^{\mathrm{B}}\left(\StandardFibration\right)\subseteq\mathfrak{X}^{\mathrm{A}}\left(\StandardFibration\right)$
of basic vector fields as given in \prettyref{def:basic-vector-field}.
\end{thm}

\begin{proof}
Since $\Xi$ is given by the coset space construction associated with
the Lie subgroup $\Vau\left(\StandardFibration\right)$ of the Lie
group $\Aut\left(\StandardFibration\right)$ (\prettyref{thm:lie-group-structure-on-symmetry-group}),
the theorem will follow from the smooth slice lemma (\prettyref{prop:manifold-slice-condition})
once we verify the assumptions required therein. To this end, consider
the topological slice given by the image of the local section $\sigma$
in \prettyref{eq:local-section-in-small-bundle}:
\[
\mathrm{\mathcal{S}}^{\mathrm{B}}\left(\StandardFibration\right)\coloneqq\left\{ h\in\Aut\left(\StandardFibration\right)\mid h=\sigma\left(\beta\right),\,\exists\beta\in\Diff\left(B\right)_{\StandardFibration}\right\} .
\]
I claim that this subset $\mathrm{\mathcal{S}}^{\mathrm{B}}\subseteq\Aut\left(\StandardFibration\right)$
is a smooth submanifold modeled on the Fréchet subspace $\mathfrak{X}^{\mathrm{B}}\left(\StandardFibration\right)\subseteq\mathfrak{X}^{\mathrm{A}}\left(\StandardFibration\right)$
of basic vector fields, whose submanifold chart (near the identity)
is given by the adapted Riemannian exponential. To see this, recall
by definition that the adapted Riemannian exponential corresponds
a vector field $X$ to a diffeomorphism $\Exp_{\StandardFibration}\left(X\right)$
that moves each point by first traveling along the internal geodesic
and then horizontal transporting along the base geodesic. Now suppose
that $X\in\mathfrak{X}^{\mathrm{B}}\left(\StandardFibration\right)$
is basic. Then its vertical part $X^{\top}\in\mathfrak{X}^{\mathrm{V}}\left(\StandardFibration\right)$
is everywhere zero, hence the associated internal geodesics are all
stationary. In other words, the effect of $\Exp_{\StandardFibration}\left(X\right)$
is simply reduced to horizontal transporting all points along the
base geodesics determined by $\pi_{\ast}X\in\mathfrak{X}\left(B\right)$
— this is exactly our construction of $\sigma$, so that $\Exp_{\StandardFibration}\left(X\right)$
indeed lies in the slice $\mathrm{\mathcal{S}}^{\mathrm{B}}\coloneqq\image\left(\sigma\right)$
as desired:
\[
\Exp_{\StandardFibration}\left(X\right)=\sigma\left(\beta\right),\qquad\beta\coloneqq\Exp_{B}\left(\pi_{\ast}X\right)\in\Diff\left(B\right)_{\StandardFibration}.
\]
Conversely, suppose that $\sigma\left(\beta\right)\in\Aut\left(\StandardFibration\right)$
is an automorphism obtained as the $\sigma$-image of some base transformation
$\beta\in\Diff\left(B\right)_{\StandardFibration}$. Then as long
as $\sigma\left(\beta\right)$ is sufficiently close to the identity,
the Riemannian exponential on the base will associate to $\beta$
its corresponding vector field $Y\in\mathfrak{X}\left(B\right)$ on
the base, which in turn can be horizontally lifted to a vector field
$X\in\mathfrak{X}\left(E\right)$ on the total space. By construction,
this vector field has the properties $X\in\mathfrak{X}^{\mathrm{B}}\left(\StandardFibration\right)$
and $\Exp_{\StandardFibration}\left(X\right)=\sigma\left(\beta\right)$,
as desired. This completes the proof of the smooth structure on the
slice $\mathrm{\mathcal{S}}^{\mathrm{B}}\left(\StandardFibration\right)$.
Thus to complete the proof by the slice lemma (\prettyref{prop:manifold-slice-condition}),
it suffices to verify that the following right-translation map $\rho$
is a diffeomorphism onto an open subset of $\Aut\left(\StandardFibration\right)$:
\[
\rho\colon\mathrm{\mathcal{S}}^{\mathrm{B}}\left(\StandardFibration\right)\times\Vau\left(\StandardFibration\right)\xrightarrow{\approx}W\subsetopen\Aut\left(\StandardFibration\right),\qquad\left(s,h\right)\mapsto s\circ h.
\]
To this end, first recall that by construction, $\mathrm{\mathcal{S}}^{\mathrm{B}}\left(\StandardFibration\right)$
is the image of the local section $\sigma\colon U\to\Aut\left(\StandardFibration\right)$
given in the proof of \prettyref{prop:The-topological-fibration-of-Aut-by-Vau}.
Thus in particular the image of $\rho$ is $W=\pi^{-1}\left(U\right)$,
which is clearly open in $\Aut\left(\StandardFibration\right)$. Since
the right-translation map $\rho$ is obtained from the Lie group operation
on $\Aut\left(\StandardFibration\right)$ by restricting to some smooth
embedded submanifolds, we see that $\rho$ is smooth. Thus it suffices
to show that it admits a smooth inverse $\rho^{-1}$ on $\pi^{-1}\left(U\right)$,
which can be described in terms of the local section $\sigma$ by
\[
\rho^{-1}\colon\Aut\left(\StandardFibration\right)\supsetopen W\xrightarrow{\approx}\mathrm{\mathcal{S}}^{\mathrm{B}}\left(\StandardFibration\right)\times\Vau\left(\StandardFibration\right),\qquad f\mapsto\left(\sigma\left(\AutProjection\left(f\right)\right),\sigma\left(\AutProjection\left(f\right)\right)^{-1}\circ f\right).
\]
Thus it suffices to show that the local section $\sigma$ is smooth.
But this is clear since it is locally modeled (with respect to the
smooth charts given by the adapted Riemannian exponential) by a local
map from $\mathfrak{X}\left(B\right)$ to $\mathfrak{X}^{\mathrm{A}}\left(\StandardFibration\right)$
coinciding with the horizontal lifting $Y\mapsto\sigma'\left(Y\right)=\widetilde{Y}$
as defined in \prettyref{lem:horizontal-lifting}, which is continuous
linear and hence smooth. This confirms that local section $\sigma$
is smooth, and hence so is $\rho^{-1}$ as desired. This completes
the proof of \prettyref{thm:main-theorem-about-smooth-symmetry-fibration}.
\end{proof}
The developments in this section can be re-interpreted from the perspective
of infinite dimensional Riemannian geometry as follows. Recall from
\prettyref{def:adapted-Riemannian-exponential-map} that we associated
with our Riemannian fibration $\xi\colon F\hookrightarrow E\to B$
an adapted Riemannian exponential $\exp_{\left(E,\xi\right)}\colon TE\supsetopen V\to E$
for the fibered total space $\left(E,\xi\right)$, as given by first
vertically traveling for a unit time along the $E_{x}$-geodesic with
an initial velocity of $w^{\top}\in T_{p}E_{x}$ in the fiber, and
then horizontally traveling for a unit time along the $E$-geodesic
with an initial velocity of the horizontal lift of $\overline{w}\in T_{x}B$
from the base. Of course, this construction of an \emph{adapted Riemannian exponential associated with a Riemannian fibration}
carries over verbatim to the infinite-dimensional case (where the Riemannian metrics
involved are allowed to be in the weak sense only). Then just as the
Riemannian exponential for the infinite-dimensional manifold $\Diff\left(E\right)$
with respect to the $L^{2}$ (weak) Riemannian metric is exactly given
by pushing forward the Riemannian exponential for the finite-dimensional
manifold $E$ (see \cite{MR0271984}), the adapted Riemannian exponential
for the infinite-dimensional fibered manifold $\left(\Aut\left(\StandardFibration\right),\Xi\right)$
also admits a close relation with the finite-dimensional fibered manifold
$\left(E,\StandardFibration\right)$, as in the following corollary:
\begin{cor}
\label{cor:Riemannian-fibration-of-Aut-seq}Given a Riemannian fibration
$\xi\colon F\hookrightarrow E\to B$. Let the diffeomorphism groups
of $E$ and of $B$ be equipped with the $L^{2}$ (weak) Riemannian
metric, and let the automorphism group of $\left(E,\xi\right)$ be
equipped with the induced Riemannian metric as a submanifold. Then
the smooth principal bundle (as constructed in \prettyref{thm:main-theorem-about-smooth-symmetry-fibration})%
\[
\Xi\colon\Vau\left(\StandardFibration\right)\hookrightarrow\Aut\left(\StandardFibration\right)\xrightarrow{\AutProjection}\Diff\left(B\right)_{\StandardFibration}\qquad\text{is a Riemannian fibration}.
\]
Moreover, the adapted Riemannian exponential for the infinite-dimensional
fibered manifold $\left(\Aut\left(\StandardFibration\right),\Xi\right)$
is exactly given by pushing forward the adapted Riemannian exponential
for the finite-dimensional fibered manifold $\left(E,\StandardFibration\right)$
(as constructed in \prettyref{prop:topological-manifold-structures-on-symmetry-groups}).
\end{cor}

\begin{proof}
This is clear from the construction, as summarized in the diagram
\prettyref{eq:diagram-construction-of-adapted} and reiterated as
follows:
\begin{equation}
\begin{aligned}\xymatrix@C=4pc{U\cap\Vau\left(\StandardFibration\right)\ar@{^{(}->}[r] & U\cap\Aut\left(\StandardFibration\right)\ar@{->>}[r]_{\AutProjection} & \overline{U}\cap\Diff\left(B\right)_{\StandardFibration}\ar@/_{1pc}/[l]_{\sigma}\\
V\cap\mathfrak{X}^{\mathrm{V}}\left(\StandardFibration\right)\ar@{^{(}->}[r]\ar[u]_{\approx}^{\bigsqcup_{x\in B}\exp_{E_{x}}\circ\left(-\right)} & V\cap\mathfrak{X}^{\mathrm{A}}\left(\StandardFibration\right)\ar@{-->}[u]_{\approx}^{\exp_{\left(E,\xi\right)}\circ\left(-\right)}\ar@{->>}[r]^{\pi_{\ast}} & \overline{V}\cap\mathfrak{X}\left(B\right)\ar[u]_{\approx}^{\exp_{B}\circ\left(-\right)}\ar@/^{1pc}/[l]^{\sigma'}
}
\end{aligned}
,\label{eq:diagram-construction-of-adapted-repeated}
\end{equation}
Here, the local chart for $\Vau\left(\StandardFibration\right)$ on
$\mathfrak{X}^{\mathrm{V}}\left(\StandardFibration\right)$ (resp.,
for $\Diff\left(B\right)_{\StandardFibration}$ on $\mathfrak{X}\left(B\right)$)
is given by post-composition with the fiberwise exponential $\bigsqcup_{x}\exp_{E_{x}}$
(resp., the base exponential $\exp_{B}$), which is just given by
the ordinary Riemannian exponential of the concrete fibers $E_{x}$
(resp., of the base $B$), so that with respect to the induced $L^{2}$
(weak) Riemannian metrics we have the relations
\[
\exp_{\Vau\left(\StandardFibration\right)}=\bigsqcup_{x\in B}\exp_{E_{x}}\circ\left(-\right)\qquad\text{and}\qquad\exp_{\Diff\left(B\right)_{\StandardFibration}}=\exp_{B}\circ\left(-\right).
\]
Then these two charts are in turn composed together by the splittings
of the two rows: a linear section $\sigma'$ of the projection $\mathfrak{X}^{\mathrm{A}}\left(\StandardFibration\right)\to\mathfrak{X}\left(B\right)$
in the Fréchet space category was provided in \prettyref{lem:horizontal-lifting}
by horizontal lift, while a continuous local section $\sigma$ of
the projection $\Aut\left(\StandardFibration\right)\to\Diff\left(B\right)_{\StandardFibration}$
in the topological category was provided in \prettyref{prop:The-topological-fibration-of-Aut-by-Vau}
by horizontal transport, so that we indeed have the desired relation
\[
\exp_{\left(\Aut\left(\StandardFibration\right),\Xi\right)}=\exp_{\left(E,\xi\right)}\circ\left(-\right).
\]
This completes the proof of \prettyref{cor:Riemannian-fibration-of-Aut-seq}.
\end{proof}

\section{Smooth Fibration over the Moduli\label{sec:Smooth-Fibration-over-the-Moduli}}

In this section, we construct the following fibration of the total
diffeomorphism group by the automorphism group over the moduli space,
in the order of increasingly rich structures: firstly a topological
principal bundle, then an infinite-dimensional smooth fibration, and
lastly an infinite-dimensional Riemannian fibration:%
\[
\begin{gathered}\xymatrix{\Aut\left(\StandardFibration\right)\ar@{^{(}->}[r] & \Diff\left(E\right)\ar@{->>}[d]^{\DiffProjection}\\
 & \Fib\left(\StandardFibration\right)
}
\end{gathered}
\]
where the base space $\Fib\left(\StandardFibration\right)$, our main
object of study, will be simultaneously endowed with a (Fréchet) smooth
structure and a (weak) Riemannian structure.

\subsection{Fibrating diffeomorphisms by automorphisms}

Recall from \prettyref{prop:The-topological-fibration-of-Aut-by-Vau}
that we obtained a principal $\Vau\left(\StandardFibration\right)$-bundle
structure on $\AutProjection\colon\Aut\left(\StandardFibration\right)\to\Diff\left(B\right)_{\StandardFibration}$
by constructing a local cross section. One may try proceeding analogously
for our main sequence
\begin{equation}
\Aut\left(\StandardFibration\right)\hookrightarrow\Diff\left(E\right)\xrightarrow{\DiffProjection}\Fib\left(\StandardFibration\right).\label{eq:deformation sequence}
\end{equation}
However, this task is less straightforward than the previous case,
partly because in the current case the base $\Fib\left(\StandardFibration\right)$
is a coset space that lacks a concrete description. To circumvent
this inconvenience, we recall the splitting lemma from \prettyref{lem:splitting-lemma},
whose specialization into the current case is reprised as follows:
\begin{lem}
\label{lem:splitting-lemma-specialized-to-Diff-Aut}The canonical
projection $Q\colon\Diff\left(E\right)\to\Fib\left(\StandardFibration\right)$
onto the left coset space $\Fib\left(\StandardFibration\right)\coloneqq\Diff\left(E\right)/\Aut\left(\StandardFibration\right)$
is a topological principal bundle if and only if for some open neighborhood
$W$ around $\StandardFibration$ in $\Fib\left(\StandardFibration\right)$,
any (hence all) of the following three equivalent conditions is satisfied:
\begin{enumerate}[label=\textup{(\roman*)}]
\item There exists a continuous map $\sigma\colon\Fib\left(\StandardFibration\right)\supsetopen W\to\Diff\left(E\right)$,
sending $\StandardFibration$ to $\Identity_{E}$, such that $\DiffProjection\circ\sigma\equiv\Identity_{W}$.
\item There exists a continuous map $r\colon\Diff\left(E\right)\supsetopen\DiffProjection^{-1}\left(W\right)\to\Aut\left(\StandardFibration\right)$,
sending $\Identity_{E}$ to $\Identity_{E}$, such that $r\left(-\cdot h\right)\equiv r\left(-\right)\cdot h$
for all $\forall h\in\Aut\left(\StandardFibration\right)$.
\item There exists a subset $S\subseteq\Diff\left(E\right)$, containing
$\Identity_{E}$, such that the right translation $\rho$ maps $S\times\Aut\left(\StandardFibration\right)$
homeomorphically onto $\DiffProjection^{-1}\left(W\right)\subsetopen\Diff\left(E\right)$.
\end{enumerate}
In this case, $\sigma$, $r$, and $S$ are called a ``local section'',
an ``equivariant neighborhood retraction'', and a ``local slice''
for this fibration, respectively.
\end{lem}

There is an intuitive interpretation for the second criterion above,
the existence of an equivariant neighborhood retraction
\[
r\colon\Diff\left(E\right)\supsetopen U\to^{\Aut\left(\StandardFibration\right)}\Aut\left(\StandardFibration\right),\qquad\left.r\right|_{\Aut\left(\StandardFibration\right)}=\Identity_{\Aut\left(\StandardFibration\right)}.
\]
Namely, by viewing each $f\in\Diff\left(E\right)$ sufficiently close
to $\Aut\left(\StandardFibration\right)$ as a slightly perturbed
(``wiggly'') fibration $f_{\ast}\StandardFibration$, we can think
of a desired equivariant neighborhood retraction onto $\Aut\left(\StandardFibration\right)$
as a continuous way to ``straighten'' any such wiggly fibration
$f_{\ast}\StandardFibration$ to one that looks the same as $\StandardFibration$,
while without referring to the parametrizations or labels of the fibers.
This in turn breaks down to the following two questions for each wiggly
fiber $S\in\Shape\left(F,E\right)$ of $f_{\ast}\StandardFibration$:
\begin{itemize}
\item Which fiber $E_{x_{S}}\coloneqq\pi^{-1}\left(x_{S}\right)$ of $\StandardFibration$
should $S$ be straightened to?
\item How to match the points between $S$ and the chosen fiber $E_{x_{S}}=\pi^{-1}\left(x_{S}\right)$
of $\StandardFibration$?
\end{itemize}
Our answers to these two questions will be given in \prettyref{def:center-of-mass-answer-to-1st-question}
and \prettyref{def:normal-graph-answer-to-second-question}, respectively.
Let us start with the more straightforward second question: given
a slightly perturbed fibration $f_{\ast}\StandardFibration$, how
to match points between each fiber $S$ of $f_{\ast}\StandardFibration$
and a chosen nearby reference fiber of $\StandardFibration$. Roughly
speaking, we need to come up with a natural procedure to parametrize
the perturbed fiber $S$ with respect to the nearby reference fiber
$E_{x}\coloneqq\pi^{-1}\left(x\right)$ over $x\in B$, just based
on the shape $S\in\Shape\left(F,E\right)$. Here we adopt one such
procedure using the \emph{Riemannian normal exponential} from basic
Riemannian geometry, described as follows. Let $\pi_{NE_{x}}\colon NE_{x}\to E_{x}$
denote the (Riemannian) normal bundle of $E_{x}\subseteq E$, on which
the (Riemannian) normal exponential restricts to the map $\left.\exp_{E}\right|_{NE_{x}}$
taking sufficiently small vectors in $NE_{x}$ to the points near
$E_{x}\subseteq E$ along the normal geodesics. Then as long as the
perturbation of $S$ is small enough, we can choose a uniform tubular
neighborhood $W'\subsetopen E$ around $E_{x}$ and containing $S$,
such that the following tubular-neighborhood projection 
\begin{equation}
E\supsetopen W'\to E_{x}\subseteq E,\qquad p\mapsto\pi_{NE_{x}}\left(\left.\exp_{E}\right|_{NE_{x}}^{-1}\left(p\right)\right)\label{eq:tubular-neighborhood-projection}
\end{equation}
is well-defined and restricts to a diffeomorphism between $S$ and
$E_{x}$ (which is clearly independent of the choice of $W'$). This
justifies the following definition.
\begin{defn}
\label{def:normal-graph-answer-to-second-question}Suppose that $f\in\Diff\left(E\right)$
is sufficiently close to $\Aut\left(\StandardFibration\right)$, and
let $S\subseteq E$ be a fiber of the deformed fibration $f_{\ast}\StandardFibration$.
Given any reference fiber $E_{x}\coloneqq\pi^{-1}\left(x\right)$
of $\StandardFibration$ (over some $x\in B$) sufficiently close
to $S$. Then we define $u_{S,x}$ to be the diffeomorphism between
$S$ and $E_{x}$ uniquely determined by restricting the tubular-neighborhood
projection in \prettyref{eq:tubular-neighborhood-projection} as above:
\[
u_{S,x}\colon S\xrightarrow{\approx}E_{x},\qquad u_{S,x}\left(p\right)\coloneqq\pi_{NE_{x}}\left(\left.\exp_{E}\right|_{NE_{x}}^{-1}\left(p\right)\right).
\]
This procedure is said to be parametrizing a perturbed fiber $S$
as a \emph{normal graph} with respect to $x$.
\end{defn}

\subsection{Interlude: Riemannian center of mass}

Having answered the second question above, we consider next the first
question: given a slightly perturbed fibration $f_{\ast}\StandardFibration$,
how to select for each fiber $S$ of $f_{\ast}\StandardFibration$
an associated fiber of $\StandardFibration$. Roughly speaking, we
need to come up with a natural procedure to label the perturbed fiber
$S$ by a distinguished point $x\in B$, just based on the shape $S\in\Shape\left(F,E\right)$.
Here we adopt one such procedure using the Riemannian center of mass
due to Karcher \cite{MR442975} (see \prettyref{rem:Riemannian-center-of-mass}
below), described as follows. Let $d\hat{\nu}_{S}$ denote the (Riemannian)
normalized volume form of $S\subseteq E$ with the induced metric,
and let $\left.\exp_{B}\right|_{T_{x}B}$ be the restricted (Riemannian)
exponential map taking sufficiently small vectors in $T_{x}B$ to
the points near $x\in B$ along the radial geodesics. Then as long
as the perturbation of $S$ is small enough, we can choose a geodesic
ball $W\subsetopen B$, whose closure $\overline{W}$ is geodesically-convex
and contains the projected image $\StandardProjection\left(S\right)$,
such that the following restricted vector field
\begin{equation}
X_{S}\colon B\supseteq\overline{W}\to TB,\qquad X_{S}\left(x\right)\coloneqq\int_{a\in S}\left(\left.\exp_{B}\right|_{T_{x}B}\right)^{-1}\left(\StandardProjection\left(a\right)\right)\,d\hat{\nu}_{S},\label{eq:restricted-vector-field}
\end{equation}
is well-defined and has a unique zero $x_{S}$ in $W$ (which is clearly
independent of the choice of $W$). This justifies the following definition.
\begin{defn}
\label{def:center-of-mass-answer-to-1st-question}Suppose that $f\in\Diff\left(E\right)$
is sufficiently close to $\Aut\left(\StandardFibration\right)$, and
let $S\subseteq E$ be a fiber of the deformed fibration $f_{\ast}\StandardFibration$.
Then we define $x_{S}$ to be the point in $B$ uniquely determined
by the zero of the restricted vector field \prettyref{eq:restricted-vector-field}
as above:
\[
x_{S}\in B,\qquad\int_{a\in S}\left(\left.\exp_{B}\right|_{T_{x_{S}}B}\right)^{-1}\left(\StandardProjection\left(a\right)\right)\,d\hat{\nu}_{S}=0.
\]
This procedure is said to be labeling a perturbed fiber $S$ by the
\emph{center of mass}.
\end{defn}

The center of mass constructed in the preceding definition is essentially
equivalent to the one used in \cite{MR2976322}, the latter of which
is what more commonly is called ``Riemannian center of mass'' as
it captures the notion of minimizing the total distances. The relation
is explained in the following remark:
\begin{rem}
\label{rem:Riemannian-center-of-mass}Again let $S\in\Shape\left(F,E\right)$
be a deformed fiber by only a sufficiently small perturbation, so
that in particular there is a geodesic ball $W\subsetopen B$ whose
closure $\overline{W}$ is geodesically-convex and contains the projected
image $\StandardProjection\left(S\right)$. Then instead of the restricted
vector field $X_{S}$ in \prettyref{eq:restricted-vector-field} as
above, it was considered by \cite{MR2976322} the following $L^{2}$
distance $P_{S}$ to the projected image $\StandardProjection\left(S\right)$
in the base:
\[
P_{S}\left(x\right)\coloneqq\frac{1}{2}\int_{a\in S}d_{B}\left(x,\StandardProjection\left(a\right)\right)^{2}\,d\hat{\nu}_{S}.
\]
By computing the negative gradient field of $P_{S}$, we recover our
vector field $X_{S}$ in \prettyref{eq:restricted-vector-field}:
\[
-\gradient P_{S}=\int_{a\in S}\left(\left.\exp_{B}\right|_{T_{x}B}\right)^{-1}\left(\StandardProjection\left(a\right)\right)\,d\hat{\nu}_{S}=X_{S}.
\]
Therefore, the zeros of $X_{S}$ correspond to the critical points
of $P_{S}$. A minimum of a map of such form $P_{S}$ is known as
the \emph{($L^{2}$) Riemannian center of mass} of the measurable
subset in $B$. Usages of the Riemannian center of mass can be found
for instance in \cite{MR356104}, and subsequently a more systematic
study was done by Karcher in \cite{MR442975}; see also \cite{MR2736346}
for a general study of the $L^{p}$ Riemannian center of mass for
$1\leq p\leq\infty$. These references in particular address the question
of existence and uniqueness of the center, including the special case
used in our proof of \prettyref{prop:topological-slice-of-our-bundles}
where everything is assumed to be contained in a convex geodesic ball
of a compact manifold.
\end{rem}

\subsection{Fibrating diffeomorphisms by automorphisms (continued)}

The preceding two definitions (\prettyref{def:center-of-mass-answer-to-1st-question}
and \prettyref{def:normal-graph-answer-to-second-question}) provide
a geometrically natural two-step procedure to ``straighten'' any
slightly perturbed fiber $S$: we first label it by the center of
mass $x_{S}$ (\prettyref{def:center-of-mass-answer-to-1st-question}),
with respect to which we then parametrize it as a normal graph (\prettyref{def:normal-graph-answer-to-second-question}).
From this we can construct an equivariant neighborhood retraction
for the desired bundle structure of diffeomorphisms fibered by automorphisms,
as the following proposition shows.
\begin{prop}
\label{prop:topological-slice-of-our-bundles}The canonical quotient
projection associated with the subgroup $\Aut\left(\StandardFibration\right)\leq\Diff\left(E\right)$
gives the following topological principal bundle:
\[
\Aut\left(\StandardFibration\right)\hookrightarrow\Diff\left(E\right)\xrightarrow{\DiffProjection}\Fib\left(\StandardFibration\right).
\]
Specifically, the bundle structure is given by a certain equivariant
neighborhood retraction $r$ onto $\Aut\left(\StandardFibration\right)$,
whose construction is induced by the procedure of parametrizing perturbed
fibers as normal graphs with respect to centers of mass.
\end{prop}

\begin{proof}
By the splitting lemma for topological coset quotient projections
(\prettyref{lem:splitting-lemma-specialized-to-Diff-Aut}), it suffices
to construct an equivariant neighborhood retraction $r$ onto $\Aut\left(\StandardFibration\right)$.
Thus for any diffeomorphism $f\in\Diff\left(E\right)$ sufficiently
close to $\Aut\left(\StandardFibration\right)$, we need to construct
a desired automorphism $r\left(f\right)\in\Aut\left(\StandardFibration\right)$.
Or alternatively — taking the ``Lagrangian'' perspective rather
than ``Eulerian'' — we need to construct a desired displacement
diffeomorphism $\delta_{f}\in\Diff\left(E\right)$ that performs the
retraction $r$ by post-composition:
\[
\delta_{f}\in\Diff\left(E\right),\qquad r\left(f\right)=\delta_{f}\circ f.
\]
To this end, we view the source space $E$ as equipped with the perturbed
fibration $f_{\ast}\StandardFibration$, so that we can construct
the desired displacement $\delta_{f}\colon E\to E$ by describing
its effect on each fiber $S\subseteq E$ of $f_{\ast}\StandardFibration$.
We have already had a candidate; namely, the above two-step procedure
of straightening the perturbed fiber $S$. To recap, i) first, we
label $S$ by the center of mass (\prettyref{def:center-of-mass-answer-to-1st-question}),
which yields a point $x_{S}\in B$ and hence picks out a fiber $S'\coloneqq\StandardProjection^{-1}\left(x_{S}\right)$
of the standard fibration $\StandardFibration$; ii) then, we parametrize
$S$ as a normal graph with respect to $x_{S}$ (\prettyref{def:normal-graph-answer-to-second-question}),
which yields a diffeomorphism $u_{S,x_{S}}$ between $S$ and $S'$.
This last diffeomorphism $u_{S,x_{S}}$, for each fiber $S$ of $f_{\ast}\StandardFibration$,
will serve as the desired displacement $\delta_{f}$ on $E$ when
restricting to $S$:
\[
\delta_{f}\colon E\to E,\qquad\left.\delta_{f}\right|_{S}\coloneqq u_{S,x_{S}}\colon E\supseteq S\xrightarrow{\approx}S'\subseteq E.
\]
To see that the thus defined map $\delta_{f}\colon E\to E$ is smooth,
we consider its local representation with respect to smooth local
coordinates induced by the suitable smooth fibrations. More precisely,
as before we equip the source space $E$ with the local coordinates
$\left(x',y'\right)\in U\times F$ induced by the local trivialization
$\varphi_{\StandardFibration}$ of $\StandardFibration$, and moreover
we equip the target space $E$ with the local coordinates $\left(x,y\right)\in U\times F$
induced by the local trivialization $\varphi_{f_{\ast}\StandardFibration}=\varphi_{\StandardFibration}\circ f^{-1}$
of $f_{\ast}\StandardFibration$. Then with represent to these coordinates,
the map $\delta_{f}\colon E\to E$ admits a local representation
\begin{equation}
\hat{\delta}_{f}=\varphi_{\StandardFibration}\circ\delta_{f}\circ\left(f\circ\varphi_{\StandardFibration}^{-1}\right)\colon U\times F\to U\times F,\qquad\left(x,y\right)\mapsto\left(x',y'\right)\coloneqq\left(x_{S},\hat{u}_{S,x_{S}}\left(y\right)\right)\label{eq:local-rep-of-displacement}
\end{equation}
Here, $x_{S}$ is the center of mass of $S$, which depends smoothly
on $x$, and is independent of $y$; while $\hat{u}_{S,x_{S}}$ is
the coordinate representation of the diffeomorphism $u_{S,x_{S}}$
between $S$ and $S'$, which is hence smooth on $y$, and also smoothly
depends on $x_{S}$ hence on $x$. This shows that our construction
of the displacement $\delta_{f}$ has the desired smoothness:
\[
\delta_{f}\in\MappingSpace^{\infty}\left(E,E\right).
\]
Moreover, the assignment $f\mapsto\delta_{f}$ is continuous (as is
clear from the expression of $\delta_{f}$ in \prettyref{eq:local-rep-of-displacement}),
hence it yields a continuous map from an open subset $U'$ in $\Diff\left(E\right)$
to the smooth mapping space $\MappingSpace^{\infty}\left(E,E\right)$.
But since $\Diff\left(E\right)$ is an open subset of $\MappingSpace^{\infty}\left(E,E\right)$,
we can ensure $\delta_{f}$ to be a diffeomorphism by shrinking the
open subset $U'$ if necessary. In conclusion, we obtain a desired
construction of ``displacement'' diffeomorphisms
\[
\Diff\left(E\right)\supsetopen U'\to\Diff\left(E\right),\qquad f\mapsto\delta_{f},
\]
which in turn induces a desired neighborhood retraction
\[
r\colon\Diff\left(E\right)\supsetopen U\to\Aut\left(\StandardFibration\right),\qquad f\mapsto\delta_{f}\circ f.
\]
Lastly, we need to show that $r$ is equivariant with respect to $\Aut\left(\StandardFibration\right)$;
but this is clear by our construction of $\delta_{f}$. Indeed, since
$\delta_{f}$ is constructed in a way that does not rely on the parametrization
or labelling of the fibers of $f_{\ast}\StandardFibration$ (but only
depends on their shapes), it has the invariant property $\delta_{f}=\delta_{f\circ h}$
for any automorphism $h\in\Aut\left(\StandardFibration\right)$, and
hence $r\left(f\right)\coloneqq\delta_{f}\circ f$ has the equivariant
property $r\left(f\circ h\right)=r\left(f\right)\circ h$, as desired.
This completes the proof.
\end{proof}

\subsection{Smooth fibration structure}

\begin{lem}
\label{lem:intrinsic-horizontal-center}Assume without loss of generality
that $\StandardFibration$ is a harmonic Riemannian submersion. Let
$S\in\Shape\left(F,E\right)$ be sufficiently close to $\left|\StandardFibration\right|$,
and for each nearby fiber $E_{x}$ let $\Psi_{x}\in\Diff\left(S,E_{x}\right)$
denote the horizontal retraction. Then the following vector field
$X_{S}$ defined on a sufficiently small convex geodesic ball in $B$
admits a unique zero $\overline{x}_{S}$ in the interior:
\begin{equation}
X_{S}\left(\overline{x}_{S}\right)=0,\qquad X_{S}\left(x\right)\coloneqq\int_{q\in E_{x}}\left(\left.\exp_{B}\right|_{T_{x}B}\right)^{-1}\left(\pi\left(\Psi_{x}^{-1}\left(q\right)\right)\right)\,d\nu_{E_{x}}\in T_{x}B.\label{eq:intrinsic-vector-field-for-center}
\end{equation}
This point $\overline{x}_{S}\in B$ (or the fiber over it) is called
the \emph{``intrinsic'' horizontal center of $S$}.
\end{lem}

\begin{proof}
Let $W$ be such
a sufficiently small convex geodesic ball containing $\pi\left(S\right)$
in $B$. Then for each $m\in W$, the horizontal retraction $\Psi_{m}\in\Diff\left(S,E_{m}\right)$
induces a volume form on $S$ from the nearby fiber $E_{m}$, with
respect to which we can take the horizontal center of $S$; i.e.,
the unique zero $\overline{x}_{S,m}$ of a certain vector field $X_{S,m}$
given as follows:
\begin{equation}
X_{S,m}\left(\overline{x}_{S,m}\right)=0,\qquad X_{S,m}\left(x\right)\coloneqq\int_{q\in E_{m}}\left(\left.\exp_{B}\right|_{T_{x}B}\right)^{-1}\left(\pi\left(\Psi_{m}^{-1}\left(q\right)\right)\right)\,d\nu_{E_{m}}\in T_{x}B.\label{eq:semi-extrinsic-vector-field-for-center}
\end{equation}
This horizontal center $\overline{x}_{S,m}$ (taken with respect to
the probability measure induced from $E_{m}$) is almost but not exactly
the same as the desired intrinsic horizontal center $\overline{x}_{S}$,
but it yields a map that is up for inductive applications:
\[
\beta_{S}\colon B\supsetopen W\to W\subsetopen B,\qquad m\mapsto\overline{x}_{S,m}.
\]
To examine the regularity and convergence of this map $\beta_{S}$,
let us choose a suitable Riemannian metric on $E$, so that $\StandardFibration$
becomes a harmonic Riemannian submersion. To see that this can be
always arranged, fix a volume form $d\mu_{0}$ on $F$ that is positive
and normalized (i.e., of unit total volume). Then it was shown in
\cite{MR0271984} that the group $\Diff^{+}\left(F\right)$ of orientation-preserving
diffeomorphisms of $F$ deformation retracts to the subgroup $\Diff\left(F,d\mu_{0}\right)$
of volume-preserving ones. Thus we may assume without loss of generality
that $\StandardFibration$ admits a local trivialization with structure
cocycle taking values in $\Diff\left(F,d\mu_{0}\right)$. As a result,
there induces a well-defined fiberwise probability measure $\left\{ d\mu_{x}\right\} _{x\in B}$
on all fibers, which may be assumed to be induced from the original
fiberwise metric $\left\{ g_{x}\right\} _{x\in B}$ upon rescaling.
Then this local trivialization of $\StandardFibration$ induces a
(complete) connection $H\StandardFibration$ for which the horizontal
transport between fibers are all volume-preserving diffeomorphisms.
Now having constructed the desired fiberwise metric $\left\{ g_{x}\right\} _{x\in B}$
and connection $H\StandardFibration$, we may choose any base metric
$g_{B}$ to induce the desired total metric $g$ on $E$ such that
$\StandardFibration$ becomes harmonic, as desired. In this way, the
$L^{2}$ gradient field computed with respect to different measures
induced by $\mu_{m}$ on $E_{m}$ can be written intrinsically in
terms of the measure induced by $\mu_{0}$ on $F$, independent of
the choice of locally trivializing chart $q_{m}\left(y\right)$ parametrizing
each $E_{m}$:
\[
X_{S,m}\left(x\right)=\int_{y\in F}\left(\left.\exp_{B}\right|_{T_{x}B}\right)^{-1}\left(\pi\left(\Psi_{m}^{-1}\left(q_{m}\left(y\right)\right)\right)\right)\,d\mu_{0}\left(y\right).
\]
From this we see that for different reference points $m$ and $m'$,
the difference in the $L^{2}$ gradient fields $X_{S,m}$ and $X_{S,m'}$
is due to the holonomies arising from horizontal transporting along
the geodesic triangles around $m$, $m'$, and points in $\pi\left(S\right)$.
Thus in particular, $X_{S,m}$ is the same for all $m$ if $\StandardFibration$
is flat; more generally, a desired Lipschitz-type control can be deduced
from \cite{MR442975} that the map $\beta_{S}$ is Lipschitz continuous
with Lipschitz constant $\epsilon_{S}$ depending on the curvature
bounds of $\StandardFibration$ over $W$ and factoring by the radius
of $W$, so that $\epsilon_{S}<1$ can be arranged as long as $S$
is sufficiently close to $\left|\StandardFibration\right|$. Thus
$\beta_{S}$ admits a unique fixed point $x_{S}$ given by taking
the limit of successive approximations starting from any $x_{0}\in W$
(say, we can take $x_{0}$ to be the extrinsic center of mass as in
\prettyref{prop:topological-slice-of-our-bundles}):
\[
\overline{x}_{S}\coloneqq\lim_{n\to\infty}x_{n},\qquad x_{n}\coloneqq\beta_{S}\left(x_{n-1}\right).
\]
By construction, this fixed point $\overline{x}_{S}=\beta_{S}\left(\overline{x}_{S}\right)$
is the zero of the vector field $X_{S,m}$ in \prettyref{eq:semi-extrinsic-vector-field-for-center}
with $m=\overline{x}_{S}$, which is thus the zero of the vector field
$X_{S}$ in \prettyref{eq:intrinsic-vector-field-for-center} as desired.
This completes the proof.
\end{proof}
All of the preceding discussion culminates in the following theorem
that underlies the infinite-dimensional differential geometry of our
main sequence:
\begin{thm}
\label{thm:main-theorem-about-smooth-D-fibration}The moduli space
$\Fib\left(\StandardFibration\right)$ admits a unique smooth manifold
structure such that topological principal bundle (as constructed in
\prettyref{prop:topological-slice-of-our-bundles})%
\[
\Aut\left(\StandardFibration\right)\hookrightarrow\Diff\left(E\right)\xrightarrow{\DiffProjection}\Fib\left(\StandardFibration\right)\qquad\text{is a smooth fibration}.
\]
More precisely, both the smooth manifold structure on $\Fib\left(\StandardFibration\right)$
and the smooth bundle structure on $\DiffProjection$%
are induced from the smooth slice $\mathrm{\mathcal{S}}^{\mathrm{F}}\left(\StandardFibration\right)\subseteq\Diff\left(E\right)$,
which is modeled on the Fréchet space $\mathfrak{X}^{\mathrm{F}}\left(\StandardFibration\right)\subseteq\mathfrak{X}\left(E\right)$
of fair vector fields as given in \prettyref{def:fair-vector-field}.
\end{thm}

\begin{proof}
Our
goal is to apply the smooth slice lemma (\prettyref{prop:manifold-slice-condition})
to the coset projection associated with $\Aut\left(\StandardFibration\right)\leq\Diff\left(E\right)$,
which is already shown to be a Lie subgroup in \prettyref{thm:lie-group-structure-on-symmetry-group}.
Thus we need to construct a smooth slice $\mathrm{\mathcal{S}}^{\mathrm{F}}\left(\StandardFibration\right)$;
i.e., a smooth submanifold $\mathrm{\mathcal{S}}^{\mathrm{F}}\left(\StandardFibration\right)$
in $\Diff\left(E\right)$ (containing the identity) such that the
following right-translation map $\rho$ is a diffeomorphism onto its
open image:
\begin{equation}
\rho\colon\mathrm{\mathcal{S}}^{\mathrm{F}}\left(\StandardFibration\right)\times\Aut\left(\StandardFibration\right)\xrightarrow{\approx}W\subsetopen\Diff\left(E\right),\qquad\left(s,h\right)\mapsto s\circ h.\label{eq:right-translation-map}
\end{equation}
To this end, recall from the splitting lemma (\prettyref{lem:splitting-lemma-specialized-to-Diff-Aut})
that in the topological category, such a slice can be constructed
as the kernel of an equivariant retraction $r$ onto $\Aut\left(\StandardFibration\right)$
from some saturated neighborhood $W\coloneqq\DiffProjection^{-1}\left(U\right)$
in $\Diff\left(E\right)$. In turn, such an equivariant retraction
$f\mapsto r\left(f\right)$ can be constructed by letting each perturbed
fiber shape $S\in\Shape\left(F,E\right)$ in $\left|f_{\ast}\StandardFibration\right|$
be horizontally retracted to a certain fiber in $\left|\StandardFibration\right|$
over some horizontal center of mass of $S$. For better adaptation
for our goal of a smooth promotion, we adopt the intrinsic horizontal
center $\overline{x}_{S}$ in \prettyref{lem:intrinsic-horizontal-center}.
Recall that this is the unique zero of a certain intrinsic $L^{2}$
gradient field (defined on a sufficiently small convex geodesic ball
containing $\pi\left(S\right)$ in $B$), as follows:
\begin{equation}
X_{S}\left(\overline{x}_{S}\right)=0,\qquad X_{S}\left(x\right)\coloneqq\int_{q\in E_{x}}\left(\left.\exp_{B}\right|_{T_{x}B}\right)^{-1}\left(\pi\left(\Psi_{x}^{-1}\left(q\right)\right)\right)\,d\mu_{E_{x}}\in T_{x}B.\label{eq:defining-vector-field-for-horizontal-center-intrinsic}
\end{equation}
Then just as in the proof of \prettyref{prop:topological-slice-of-our-bundles},
performing the horizontal retraction $\Psi\left(S,\overline{x}_{S}\right)$
of each perturbed fiber $S\in\left|f_{\ast}\StandardFibration\right|$
to its intrinsic horizontal center $\pi^{-1}\left(\overline{x}_{S}\right)\in\left|\StandardFibration\right|$
yields a smooth displacement map $\delta\left(f\right)$ from $E$
to itself, which we may assume without loss of generality to be a
diffeomorphism by shrinking $W$ if necessary (so that $f\in W$ is
sufficiently close to $\Aut\left(\StandardFibration\right)$):
\begin{equation}
\delta\colon\Diff\left(E\right)\supsetopen W\to\Diff\left(E\right),\qquad\left.\delta\left(f\right)\right|_{S}\coloneqq\Psi\left(S,\overline{x}_{S}\right)\colon S\xrightarrow{\approx}\pi^{-1}\left(\overline{x}_{S}\right)\enskip\left(\forall S\in\left|f_{\ast}\StandardFibration\right|\right).\label{eq:displacement-map-in-the-ENR}
\end{equation}
Since this construction of $\delta\left(f\right)$ only depends on
the shape $S\in\left|f_{\ast}\StandardFibration\right|$ of each perturbed
fiber, not on its parametrization or labelling, we see that $\delta$
is $\Aut\left(\StandardFibration\right)$-invariant. As a result,
post-composition with $\delta\left(f\right)$ yields a desired neighborhood
retraction that is $\Aut\left(\StandardFibration\right)$-equivariant:
\begin{equation}
r\colon\Diff\left(E\right)\supsetopen W\to^{\Aut\left(\StandardFibration\right)}\Aut\left(\StandardFibration\right),\qquad r\left(f\right)\coloneqq\delta\left(f\right)\circ f,\label{eq:equivariant-neighborhood-retraction}
\end{equation}
and inversion of $\delta\left(f\right)$ descends to a desired local
section:
\[
\sigma\colon\Fib\left(\StandardFibration\right)\supsetopen U\to\Diff\left(E\right),\qquad\sigma\left(f_{\ast}\StandardFibration\right)\coloneqq\delta\left(f\right)^{-1}=f\circ\left(r\left(f\right)\right)^{-1}.
\]
Neither of these characteristics in the topological category is fully
sufficient for a smooth proposition, so we focus on the most robust
one — the slice $\mathrm{\mathcal{S}}^{\mathrm{F}}\left(\StandardFibration\right)$,
which by the splitting lemma \prettyref{lem:splitting-lemma-specialized-to-Diff-Aut}
can be given by the kernel (resp., the image) of the retraction $r$
(resp., the section $\sigma$):
\[
\mathrm{\mathcal{S}}^{\mathrm{F}}\left(\StandardFibration\right)\coloneqq\ker r=\image\sigma\subseteq\Diff\left(E\right),
\]
on which the right-translation map $\rho$ in \prettyref{eq:right-translation-map}
is seen to be a homeomorphism as its continuous inverse can be given
by
\begin{equation}
\rho^{-1}\colon\Diff\left(E\right)\supsetopen W\xrightarrow{\approx}\mathrm{\mathcal{S}}^{\mathrm{F}}\left(\StandardFibration\right)\times\Aut\left(\StandardFibration\right),\qquad f\mapsto\left(f\circ\left(r\left(f\right)\right)^{-1},r\left(f\right)\right).\label{eq:inverse-tube-homeo}
\end{equation}
We are now in a position of promote this slice to the smooth category.
To this end, we first need to show that $\mathrm{\mathcal{S}}^{\mathrm{F}}\left(\StandardFibration\right)$
is a smooth submanifold of $\Diff\left(E\right)$; but this is clear
by construction and \prettyref{lem:intrinsic-horizontal-center}.
Indeed, recall from \prettyref{lem:horizontal-averaging} that at
the infinitesimal level we have the horizontal averaging $r'\colon\mathfrak{X}\left(E\right)\to\mathfrak{X}^{\mathrm{A}}\left(\StandardFibration\right)$,
which is a linear retraction explicitly given by \prettyref{eq:horizontal-averaging-explicit-formula};
i.e.,
\[
r'\left(X\right)\coloneqq X^{\top}+\sigma'\left(\overline{X}\right),\qquad\overline{X}_{x}\coloneqq\int_{q\in E_{x}}d\pi_{q}\left(X_{q}\right)\,d\mu_{E_{x}}\in T_{x}B.
\]
For a (sufficiently close-to-identity) diffeomorphism $g\in\Diff\left(E\right)$
to lie in the slice $\mathrm{\mathcal{S}}^{\mathrm{F}}\left(\StandardFibration\right)$,
the defining condition requires that $r\left(g\right)=\Identity_{E}$,
or equivalently $\delta\left(g\right)^{-1}=g$; in words, this says
that for every fiber $g\left(E_{x}\right)$ in the perturbed fibration
$g_{\ast}\StandardFibration$, the intrinsic horizontal center $\overline{x}_{g\left(E_{x}\right)}$
recovers $x$ itself, and to which the horizontal retraction $\overline{\Psi}_{g\left(E_{x}\right)}$
has the same countering effect as the inverse of $g$. By unraveling
these definitions (specifically, by plugging in the defining vector
field \prettyref{eq:defining-vector-field-for-horizontal-center-intrinsic}
for the intrinsic horizontal center), we see that this amounts to
requiring that the corresponding vector field $\Exp^{-1}\left(g\right)$
is eliminated under horizontal averaging $r'$, and hence is a fair
vector field (\prettyref{def:fair-vector-field}):
\[
g\in\mathrm{\mathcal{S}}^{\mathrm{F}}\left(\StandardFibration\right)\iff r\left(g\right)=\Identity_{E}\iff r'\left(\Exp^{-1}\left(g\right)\right)=0\iff\Exp^{-1}\left(g\right)\in\mathfrak{X}^{\mathrm{F}}\left(\StandardFibration\right).
\]
In particular, the adapted Riemannian exponential can serve as a
smooth submanifold chart for $\mathrm{\mathcal{S}}^{\mathrm{F}}\left(\StandardFibration\right)\subseteq\Diff\left(E\right)$,
as modeled on the Fréchet subspace $\mathfrak{X}^{\mathrm{F}}\left(\StandardFibration\right)\subseteq\mathfrak{X}\left(E\right)$.
This confirms that the slice $\mathrm{\mathcal{S}}^{\mathrm{F}}\left(\StandardFibration\right)$
is a smooth submanifold of $\Diff\left(E\right)$, and so for completing
the slice lemma it suffices to show that the right-translation map
$\rho$ in \prettyref{eq:right-translation-map} is a diffeomorphism
onto its open image. But since $\rho$ is obtained from the group
operation of the Lie group $\Diff\left(E\right)$ by restricting to
smooth embedded submanifolds $\Aut\left(\StandardFibration\right)$
and $\mathrm{\mathcal{S}}^{\mathrm{F}}\left(\StandardFibration\right)$,
we see that $\rho$ is smooth. Thus it suffices to prove the smoothness
of the inverse $\rho^{-1}$ in \prettyref{eq:inverse-tube-homeo};
or more specifically, the smoothness of the splitting $f\mapsto\left(f\circ\left(r\left(f\right)\right)^{-1},r\left(f\right)\right)$
for sufficiently small perturbations $f\in\Diff\left(E\right)$. In
this way, we have reduced the proof to the final task of proving the
smoothness of the equivariant neighborhood retraction $r$ (constructed
in \prettyref{eq:equivariant-neighborhood-retraction} via \prettyref{eq:displacement-map-in-the-ENR}).
Continuing with the spirit as above, we compare $r$ with the exponentiation
of the corresponding linear construction:
\[
C\colon\Diff\left(E\right)\supsetopen U_{\Identity}\to\Aut\left(\StandardFibration\right),\qquad C\coloneqq\Exp_{\StandardFibration}\circ r'\circ\Exp_{\StandardFibration}^{-1}.
\]
This map $C$ is clearly smooth (being modeled on the continuous linear
map $r'$ of Fréchet spaces), and was shown above to have the same
kernel with $r$; namely, the slice $\mathrm{\mathcal{S}}^{\mathrm{F}}\left(\StandardFibration\right)$.
However, in general this ``linear'' retraction $C$ does not suffice
to achieve the full effect of $r$, for otherwise $f\circ C\left(f\right)^{-1}$
would always lie in the kernel of $r$ for arbitrary (sufficiently
small) diffeomorphism $f$, which clearly need not be the case. This
prompts a construction for approximating $r$ successively: start
with $r_{0}\left(f\right)\coloneqq\Identity_{E}$ (hence the splitting
$\rho_{0}^{-1}\left(f\right)\coloneqq\left(f,\Identity_{E}\right)$),
then inductively improve $r_{n-1}\left(f\right)$ by composing with
the diffusion of $f\circ r_{n-1}\left(f\right)^{-1}$ under $C$:
\[
r_{n}\left(f\right)\coloneqq\varepsilon_{n}\left(f\right)\circ r_{n-1}\left(f\right),\qquad\varepsilon_{n}\left(f\right)\coloneqq C\left(f\circ r_{n-1}\left(f\right)^{-1}\right)\in\Aut\left(\StandardFibration\right).
\]
In this way, we obtain a sequence of operators $\left(r_{n}\right)_{n}$,
each sends a sufficiently small diffeomorphism $f\in\Diff\left(E\right)$
to an automorphism $r_{n}\left(f\right)\in\Aut\left(\StandardFibration\right)$,
in a way such that the assignment $f\mapsto r_{n}\left(f\right)$
is smooth. I claim that this sequence converges to a smooth operator
$r_{\infty}$ that coincides with $r$. To see this, let us look closer
into the discrepancy between $r_{n}\left(f\right)$ and $r\left(f\right)$,
or equivalently the failure for $\varepsilon_{n+1}\left(f\right)$
to lie in the slice $\mathrm{\mathcal{S}}^{\mathrm{F}}\left(\StandardFibration\right)$,
as measured by the following vector field:
\[
Y_{n}\left(f\right)\coloneqq\Exp_{\StandardFibration}^{-1}\left(\varepsilon_{n+1}\left(f\right)\right)=r'\left(\Exp_{\StandardFibration}^{-1}\left(f\circ r_{n}\left(f\right)^{-1}\right)\right)\in\mathfrak{X}\left(E\right).
\]
For each $x_{0}\in B$, the fiber $E_{x_{0}}$ of $\StandardFibration$
is perturbed to a fiber $S\coloneqq f\left(E_{x}\right)$ of $f_{\ast}\StandardFibration$,
which in turn is horizontally retracted to its center of mass $E_{\overline{x}_{S}}$
of $\StandardFibration$ — this procedure yields a diffeomorphic transport
from $E_{x_{0}}$ to $E_{\overline{x}_{S}}$ that exactly describes
the effect of $r\left(f\right)$ on $E_{x_{0}}$, while $C\left(f\right)$
is described by another diffeomorphic transport of $E_{x_{0}}$ that
is generally different. Both intend to capture the ``average'' of
$S$, but the crux of their discrepancy lies in the fact that the
procedure is not $\Aut\left(\StandardFibration\right)$-invariant
throughout: the coordinates (both horizontal and vertical) of $S$
observed by $E_{x_{0}}$ and $E_{x_{S}}$ need not relate linearly.
More specifically, consider the geodesic triangle $\gamma_{q}$ in
$B$ with vertices $\overline{x}_{S}$, $x_{0}$, and $\pi\left(q\right)$
for each $q\in S$, all of which are contained in a sufficiently small
convex geodesic ball $W$. Then the nonzero horizontal component of
$Y_{n}\left(f\right)$ is contributed by the nonlinearity of $\gamma_{q}$,
which is controlled by the curvature of $B$ in $W$; while the nonzero
vertical component of $Y_{n}\left(f\right)$ is contributed by the
holonomies along $\gamma_{q}$, which is controlled by the curvature
of $\StandardFibration$ in $\pi^{-1}\left(W\right)$. More precisely,
the desired Lipschitz-type controls are available in \cite{MR442975},
with the Lipschitz constant depends on the aforementioned curvature
bounds for $B$ in $W$, for $\StandardFibration$ in $\pi^{-1}\left(W\right)$,
as well as the radius of $W$. In particular, for $f\in\Diff\left(E\right)$
sufficiently close to $\Aut\left(\StandardFibration\right)$, we can
arrange $W$ to be sufficiently small such that $Y_{n}$ smoothly
converges to zero and $r_{n}$ converges to a smooth operator
\[
r_{\infty}\colon\Diff\left(E\right)\supsetopen U_{\Identity}\to\Aut\left(\StandardFibration\right),\qquad r_{\infty}\left(f\right)\coloneqq\lim_{n\to\infty}r_{n}\left(f\right)\in\Aut\left(\StandardFibration\right).
\]
By construction, we have $\varepsilon_{\infty}\left(f\right)=\Identity_{E}$;
in other words, the successive diffusion makes $f\circ r_{\infty}\left(f\right)^{-1}$
to lie in the kernel of $C$ and hence of $r$ (i.e., in the slice
$\mathrm{\mathcal{S}}^{\mathrm{F}}\left(\StandardFibration\right)$).
This implies that $r_{\infty}$ coincides with the desired equivariant
retraction $r$:
\[
r\left(f\right)=r\left(f\circ r_{\infty}\left(f\right)^{-1}\circ r_{\infty}\left(f\right)\right)=r\left(f\circ r_{\infty}\left(f\right)^{-1}\right)\circ r_{\infty}\left(f\right)=r_{\infty}\left(f\right).
\]
This in particular shows that $r$ is smooth, hence so is the splitting
$\rho^{-1}$, as desired. This completes the proof of \prettyref{thm:main-theorem-about-smooth-D-fibration}.
\end{proof}

In particular, the preceding theorem implies that the moduli space
$\Fib\left(\StandardFibration\right)$ has just the desired type of
topological manifold structures which was set up in \prettyref{sec:Infinite-dimensional-topology},
as recorded in the following corollary:
\begin{cor}
\label{prop:Promoting-weak-homotopy-type}The moduli space $\Fib\left(\StandardFibration\right)$
is a topological $\ell^{2}$-manifold (\prettyref{def:l2-manifold}),
and hence has the following property: if it admits a closed subset
$N$ that is also a manifold, such that the subset inclusion is a
weak homotopy equivalence, then it is homeomorphic to $N\times\ell^{2}$
and contains $N$ as a (strong) deformation retract:
\begin{equation}
\text{\ensuremath{N\xhookrightarrow{\simeq}\Fib\left(\StandardFibration\right)}\ensuremath{\enskip}weak equivalence}\implies\begin{cases}
\Fib\left(\StandardFibration\right)\xleftrightarrow{\approx}N\times\ell^{2} & \text{homeomorphism}\\
\Fib\left(\StandardFibration\right)\xtwoheadrightarrow{}{\simeq}N & \text{deformation retraction}
\end{cases}.\label{eq:l2-property-for-Fib}
\end{equation}
Moreover, this deformation retract $N$ is a minimal deformation retract
(a ``core'') of $\Fib\left(\StandardFibration\right)$ provided
that $N$ has compact components.
\end{cor}

\begin{proof}
The preceding theorem (\prettyref{thm:main-theorem-about-smooth-D-fibration})
particularly implies that $\Fib\left(\StandardFibration\right)$ is
an infinite-dimensional Hausdorff topological manifold modeled on
$\mathfrak{X}^{\mathrm{F}}\left(\StandardFibration\right)$. Moreover,
since $\Fib\left(\StandardFibration\right)$ is a open quotient of
$\Diff\left(E\right)$, it inherits the property of being second-countable;
on the other hand, since $\mathfrak{X}^{\mathrm{F}}\left(\StandardFibration\right)$
is a closed linear subspace of $\mathfrak{X}\left(E\right)$, it inherits
the property of being a separable Fréchet space. Therefore, $\Fib\left(\StandardFibration\right)$
is a topological $\ell^{2}$-manifold (in the sense of \prettyref{def:l2-manifold})
by the Kadec–Anderson theorem (\prettyref{thm:Kadec-Anderson}). Thus
the desired properties \prettyref{eq:l2-property-for-Fib} for $\Fib\left(\StandardFibration\right)$
follows from the general properties of $\ell^{2}$-manifold (\prettyref{thm:main-theorem-for-infinite-dimensional-topology}
and \prettyref{cor:main-theorem-for-infinite-dimensional-topology-compact-case}).
This completes the proof.
\end{proof}

\chapter{First Examples\label{chap:First-Examples}}

\section{Oriented Circle Fibrations\label{sec:oriented-circle-fiberings}}

In this section, we specialize the preceding general study in \prettyref{chap:Algebraic-Topology}
to the case with the simplest fiber structure; i.e., \emph{oriented
circle fiberings}. For such fiberings, we shall describe the classification
and the vertical automorphism groups in terms of cohomology of the
base. After laying the general groundwork in this section, we shall
determine case by case the moduli space of oriented circle fiberings
on surfaces and on 3-manifolds in \prettyref{sec:Case-study-dimension2}
and \prettyref{sec:Case-study-dimension3}.

\subsection{The (fiberwise) oriented fiberings }

In general, we say that a smooth $F$-fibering is \emph{(fiberwise)
oriented} if the model fiber $F$ is oriented and the structure group
is reducible to the subgroup $\Diff_{+}\left(F\right)$ of orientation-preserving
diffeomorphisms.\footnote{It should be noted that there is an ambiguity with respect to the
term ``oriented'' for a fiber bundle or for a fibering. While some
authors use it to mean that the bundle is fiberwise oriented as we
have, there are some other authors using it to mean that the bundle
has an oriented total space.} We generally use a subscript ``$+$'' to denote the oriented counterpart
of objects; in particular,
\begin{itemize}
\item $\Fib_{+}\left(E,F\right)$ denotes the moduli space of oriented smooth
fiberings of $E$ by $F$.\nomenclature[Fib+]{$\Fib^{+}\left(E,F\right)$}{the moduli space of oriented smooth fiberings of $E$ by $F$}
\item $\ClassFib_{+}\left(E,F\right)$ denotes the set of equivalence classes
of oriented smooth fiberings of $E$ by $F$.\nomenclature[Fib+]{$\mathsf{Fib}^{+}\left(E,F\right)$}{the classification (i.e., set of equivalence classes) of oriented smooth fiberings of $E$ by $F$}
\end{itemize}
Given an oriented fibering $\StandardFibration\colon F\hookrightarrow E\to B$,
an \emph{automorphism} of $\StandardFibration$ will be given a more
restrictive condition; namely, it is not only an automorphism of $\StandardFibration$
regarded as an unoriented fibering, but also required to preserve
the orientation of $\StandardFibration$. As a result, its various
symmetry groups $\Aut\left(\StandardFibration\right)$, $\Vau\left(\StandardFibration\right)$,
and $\Diff\left(B\right)_{\StandardFibration}$ will acquire the modified
definitions as follows:
\begin{itemize}
\item $\Aut\left(\StandardFibration\right)$, the automorphism group of
$\StandardFibration$ (as an oriented fibering), consists of all the
automorphisms of $\StandardFibration$ (in the above more restrictive
sense); intuitively, mapping each oriented fiber to an oriented fiber.
\item $\Vau\left(\StandardFibration\right)$, the vertical automorphism
group of $\StandardFibration$ (as an oriented fibering), consists
of all those automorphisms of $\StandardFibration$ (in the above
more restrictive sense) that are vertical; intuitively, mapping each
oriented fiber to itself.
\item $\Diff\left(B\right)_{\StandardFibration}$, the basic transformation
group of $\StandardFibration$ (as an oriented fibering), consists
of all those diffeomorphisms of $B$ that are covered by automorphisms
of $\StandardFibration$ (in the above more restrictive sense).
\end{itemize}
Similarly, the associated moduli space $\Fib\left(\StandardFibration\right)$
will refer to the one consisting only of those oriented fiberings
that are orientedly equivalent to $\StandardFibration$. It is straightforward
to verify that the preceding general study about fiberings apply,
mutatis mutandis, to oriented fiberings. In particular, we have the
definitions
\[
\Fib_{+}\left(E,F\right)\coloneqq\bigsqcup_{\StandardFibration\in\ClassFib_{+}\left(E,F\right)}\Fib\left(\StandardFibration\right),\qquad\Fib\left(\StandardFibration\right)\coloneqq\Diff\left(E\right)/\Aut\left(\StandardFibration\right),
\]
as well as the two infinite-dimensional structure fibrations
\[
\Vau\left(\StandardFibration\right)\hookrightarrow\Aut\left(\StandardFibration\right)\to\Diff\left(B\right)_{\StandardFibration},\qquad\Aut\left(\StandardFibration\right)\hookrightarrow\Diff\left(E\right)\to\Fib\left(\StandardFibration\right).
\]
In what follows, we shall specialize to the case of smooth oriented
circle fiberings $\Sphere 1\hookrightarrow E\to B$, for which we
shall see that the structure group $\Diff_{+}\bigl(\Sphere 1\bigr)$
is a fairly special one that make it a simplest nontrivial case in
the study of the space of fiberings.

\subsection{The classification}

We start with classifying smooth oriented circle fiberings. Among
the various approaches developed in \prettyref{chap:Algebraic-Topology},
we choose a homotopy-theoretic one. This idea is to classify bundles
by the homotopy set into the classifying space of the structure group;
this approach is particularly effective for our structure group $\Diff_{+}\bigl(\Sphere 1\bigr)$
because its classifying space admits a fairly simple topological structure
(in particular, it is an Eilenberg–MacLane space). The upshot is the
following result, which classifies smooth oriented circle fiberings
in terms of the second integral cohomology of the base (as well as
how it is being acted upon by the base mapping class group):
\begin{prop}
\label{prop:Classifying-oriented-circle-fiberings}The smooth oriented
circle fiberings are classified by the second integral cohomology
modulo pullback by the mapping class group:
\begin{equation}
\bigsqcup_{E\in\ClassMan\left(m+1\right)}\ClassFib_{+}\left(E,\Sphere 1\right)\leftrightarrow\bigsqcup_{B\in\ClassMan\left(m\right)}H^{2}\left(B;\mathbb{Z}\right)/\mathop{\mathrm{Mod}}\left(B\right).\label{eq:Classifying-oriented-circle-fiberings}
\end{equation}
Here, associated to each bundle class of such a fibration $\StandardFibration$
over $B$ is its \emph{Euler class} $e_{\StandardFibration}$ in
$H^{2}\left(B;\mathbb{Z}\right)$, whose orbit under the $\mathop{\mathrm{Mod}}\left(B\right)$-action
accounts for those bundle classes that merge into the fibering class
of $\StandardFibration$.
\end{prop}

\begin{proof}
Recall that by definition, a fibering class over some base type
$B$ is obtained by merging those bundle classes in $\Bundle\left(B;\Diff_{+}\left(\Sphere 1\right)\right)$
lying in a single $\mathop{\mathrm{Mod}}\left(B\right)$-orbit, where
$\Bundle\left(B;\Diff_{+}\left(\Sphere 1\right)\right)$ denotes the
set of equivalence classes of smooth oriented circle bundles over
$B$. Thus for \prettyref{eq:Classifying-oriented-circle-fiberings}
it suffices to prove that this latter (pointed) set $\Bundle\left(B;\Diff_{+}\left(\Sphere 1\right)\right)$
admits a $\mathop{\mathrm{Mod}}\left(B\right)$-equivariant bijective
correspondence with the second integral cohomology of the base space:
\begin{equation}
\Bundle\left(B;\Diff_{+}\left(\Sphere 1\right)\right)\quad\leftrightarrow^{\mathop{\mathrm{Mod}}\left(B\right)}\quad H^{2}\left(B;\mathbb{Z}\right).\label{eq:Classifying-oriented-circle-bundles}
\end{equation}
To this end, recall from \prettyref{prop:classification-theorem-via-homotopy}
(or more precisely, its adaptation for the oriented case) that the
smooth bundle classification $\Bundle\left(B;\Diff_{+}\left(\Sphere 1\right)\right)$
is in bijective correspondence with the (topological) homotopy set
$\left[B,\Shape_{+}\left(\Sphere 1,\ell^{2}\right)\right]$. But since
$\Diff_{+}\left(\Sphere 1\right)\simeq\CircleGroup$ is an Eilenberg–MacLane
space $K\left(\mathbb{Z},1\right)$, its classifying space $\Shape_{+}\left(\Sphere 1,\ell^{2}\right)$
is a $K\left(\mathbb{Z},2\right)$, into which the desired homotopy
set can be characterized by the Euler class as follows:
\[
\left[B,\Shape_{+}\left(\Sphere 1,\ell^{2}\right)\right]\leftrightarrow\left[B,K\left(\mathbb{Z},2\right)\right]\leftrightarrow\left[B,K\left(\mathbb{Z},2\right)\right]_{\ast}\leftrightarrow H^{2}\left(B;\mathbb{Z}\right).
\]
Here, the correspondence between $\left[B,K\left(\mathbb{Z},2\right)\right]$
(the homotopy set) and $\left[B,K\left(\mathbb{Z},2\right)\right]_{\ast}$
(the pointed homotopy set) is due to the simple-connectedness of the
target space $K\left(\mathbb{Z},2\right)$; and in turn, the correspondence
between $\left[B,K\left(\mathbb{Z},2\right)\right]_{\ast}$ and $H^{2}\left(B;\mathbb{Z}\right)$
is due to a standard result in topology regarding the homotopy-cohomology
connections for Eilenberg–MacLane spaces (cf.\ \cite[Theorem 4.57]{MR1867354}).
Since all the bijective correspondences constructed above are natural
with respect to base diffeomorphisms, the resulting (pointed) bijective
correspondence between $\Bundle\left(B;\Diff_{+}\left(\Sphere 1\right)\right)$
and $H^{2}\left(B;\mathbb{Z}\right)$ is $\mathop{\mathrm{Mod}}\left(B\right)$-equivariant,
which completes the proof of \prettyref{eq:Classifying-oriented-circle-bundles}
hence \prettyref{eq:Classifying-oriented-circle-fiberings}, as desired.%
\end{proof}

\subsection{The base transformation group}

The preceding result allows us to classify smooth oriented circle
fiberings on any given total space $E$ of interest (indeed, the strategy
is to first apply the preceding result to all admissible base diffeomorphism
types, and then select those equivalence classes that have $E$ as
the total space). As such, from now on let us fix a representative fibration $\StandardFibration\colon\Sphere 1\hookrightarrow E\to B$
and study its symmetries. We first consider its basic transformation
group $\Diff\left(B\right)_{\StandardFibration}$. The idea is to
utilize the implicit description in \prettyref{prop:The-base-transformation-group-as-stabilizer},
which allows us to relate $\Diff\left(B\right)_{\StandardFibration}$
with the pullback action, which in turn was already described concretely
in the preceding proof. The upshot is the following result, which
characterizes $\Diff\left(B\right)_{\StandardFibration}$ as the stabilizer
of the Euler class:
\begin{cor}
\label{cor:Base-transformations-of-an-oriented-circle-bundle}Let
$\StandardFibration$ be a smooth oriented circle bundle over $B$,
with Euler class $e_{\StandardFibration}\in H^{2}\left(B;\mathbb{Z}\right)$.
Then its basic transformation group $\Diff\left(B\right)_{\StandardFibration}$
is an open subgroup of $\Diff\left(B\right)$ characterized as follows:
\[
\mathop{\mathrm{Mod}}\left(B\right)_{\StandardFibration}=\left\{ \left[\beta\right]\in\mathop{\mathrm{Mod}}\left(B\right)\mid\text{\ensuremath{\beta^{\ast}e_{\StandardFibration}=e_{\StandardFibration}} in \ensuremath{H^{2}\left(B;\mathbb{Z}\right)}}\right\} .
\]
In other words, $\Diff\left(B\right)_{\StandardFibration}$ is the
disjoint union of those components of $\Diff\left(B\right)$ that
stabilize the Euler class of $\StandardFibration$ under pullback.
\end{cor}

\begin{proof}
Recall from \prettyref{prop:The-base-transformation-group-as-stabilizer}
that $\Diff\left(B\right)_{\StandardFibration}$ is the stabilizer
of the bundle class $\left[\StandardFibration\right]$ with respect
to the pullback action of $\mathop{\mathrm{Mod}}\left(B\right)$ on
$\Bundle\left(B;\Diff_{+}\left(\Sphere 1\right)\right)$. But this
action is equivalent to the pullback action of $\mathop{\mathrm{Mod}}\left(B\right)$
on $H^{2}\left(B;\mathbb{Z}\right)$ via the equivariant correspondence
\prettyref{eq:Classifying-oriented-circle-bundles} between $\Bundle\left(B;\Diff_{+}\left(\Sphere 1\right)\right)$
and $H^{2}\left(B;\mathbb{Z}\right)$ (as shown in the proof of \prettyref{prop:Classifying-oriented-circle-fiberings}),
thus the desired result follows.
\end{proof}

\subsection{The vertical automorphism group}

Having understood the basic transformations from the preceding result,
we next turn to the vertical automorphism group $\Vau\left(\StandardFibration\right)$.
We first make a simple observation that makes the gauge analysis for
oriented circle fiberings relatively accessible:
\begin{rem}
For smooth oriented circle bundles, every equivalence class can be
represented by a \emph{principal} $\CircleGroup$-bundle $\StandardFibration$
where the circle group $\CircleGroup$ acts on the circle fibers of
$\StandardFibration$ by self-translations (indeed, this is due to
the fact that the corresponding embedding $\CircleGroup\hookrightarrow\Diff_{+}\left(\Sphere 1\right)$
is a homotopy equivalence, which thus induces a desired homotopy equivalence
$\Classifying\CircleGroup\simeq\Classifying\Diff_{+}\left(\Sphere 1\right)$
between their classifying spaces). Note that this does \emph{not}
say that $\Vau\left(\StandardFibration\right)$ acts on each fiber
only by translations in $\CircleGroup\subseteq\Diff_{+}\left(\Sphere 1\right)$,
nor can the homotopy equivalence $\CircleGroup\hookrightarrow\Diff_{+}\left(\Sphere 1\right)$
suffice to justify such a reduction through homotopy — though this
turns out to be true as we shall see. 
\end{rem}

By the preceding remark, we may and shall assume that $\StandardFibration$
is a principal $\CircleGroup$-bundle (with $\CircleGroup$ identified
as the subgroup of $\Diff_{+}\left(\Sphere 1\right)$ consisting of
self-translations on $\Sphere 1$). To understand the homotopy of
$\Vau\left(\StandardFibration\right)$, the idea is to utilize the
globalizing procedure in \prettyref{prop:Homotopy-reduction-of-Vau},
which allows us to deform $\Vau\left(\StandardFibration\right)$ by
means of smoothly, equivariantly deforming the structure group; this
approach is particularly effective for our structure group $\Diff_{+}\left(\Sphere 1\right)$
because its topological and geometric structure admits a fairly simple
interplay with its algebraic structure (in particular, it has an abelian
equivariant-deformation retract by a flow). The upshot is the following
result, which reduces the homotopy type of $\Vau\left(\StandardFibration\right)$
to a current group:

\begin{prop}
\label{prop:The-Vau-group-of-oriented-circle-bundle-I}Let $\StandardFibration$
be a smooth oriented circle bundle over $B$, assumed to be principal
$\CircleGroup$-bundle as above. Then the vertical automorphism group
of $\StandardFibration$ deformation retracts onto the smooth $\CircleGroup$-valued
current group:
\[
\Vau\left(\StandardFibration\right)\simeq\MappingSpace^{\infty}\left(B,\CircleGroup\right).
\]
Here, $\MappingSpace^{\infty}\left(B,\CircleGroup\right)$ is embedded
as the subgroup consisting of parametrized translations on the fibers
(where each smooth map from $B$ to $\CircleGroup$ prescribes the
amounts of translations).
\end{prop}

The proof of this proposition will occupy the next subsection.

\subsection{A proof that the Vau group deforms onto the current group}

In this subsection we prove \prettyref{prop:The-Vau-group-of-oriented-circle-bundle-I}.
In view of \prettyref{prop:Homotopy-reduction-of-Vau}, we aim to
deformation retract the structure group $\Diff_{+}\left(\Sphere 1\right)$
onto $\CircleGroup$, smoothly and equivariantly (with respect to
conjugation by $\CircleGroup$):
\begin{equation}
H\colon\Diff_{+}\left(\Sphere 1\right)\times\left[0,1\right]\to^{\CircleGroup}\Diff_{+}\left(\Sphere 1\right),\qquad\left.H\right|_{t=1}\colon\Diff_{+}\left(\Sphere 1\right)\simeq^{\CircleGroup}\CircleGroup.\label{eq:S1-Vau-proof-step-deform-K}
\end{equation}
To ease the notation, let $K\coloneqq\Diff_{+}\left(\Sphere 1\right)$
denote the structure group throughout the proof.

$\smallblacktriangleright$ \noun{Step 1: Setting the scene.} To
be more explicit, let us work on the universal cover $\mathbb{R}$
of $\Sphere 1$ via the standard covering map $\theta\mapsto\exp\left(2\pi i\theta\right)$.
Then each circle map $h\colon\Sphere 1\to\Sphere 1$ can be lifted
to a real function $\widetilde{h}\colon\mathbb{R}\to\mathbb{R}$,
unique up to post-composing with $\mathbb{Z}$-translations. We particularly
concern those lifts $\widetilde{h}$ that are associated with $h\in K$,
which form a topological subgroup $\widetilde{K}$ of $\MappingSpace^{\infty}\left(\mathbb{R}\right)$
(under pointwise additions) and can be characterized by two properties:
${\widetilde{h}'\left(x\right)>0}$ and ${\widetilde{h}\left(x+1\right)=\widetilde{h}\left(x\right)+1}$
for all $x$. Thus to sum up, taking the lifts yields the following
central extension of $K$ into $\widetilde{K}$:
\begin{equation}
\mathbb{Z}\hookrightarrow\widetilde{K}\overset{q}{\twoheadrightarrow}K,\qquad\widetilde{K}=\left\{ f\in\MappingSpace^{\infty}\left(\mathbb{R}\right)\mid f'\left(x\right)>0,\ f\left(x+1\right)=f\left(x\right)+1,\ \forall x\right\} .\label{eq:S1-Vau-proof-step-descend-map}
\end{equation}
In this way, the desired ($\CircleGroup$-equivariant) deformation
retraction $K\simeq\CircleGroup$ can be constructed by descending
a ($\mathbb{R}$-equivariant) deformation retraction $\widetilde{K}\simeq\mathbb{R}$
that is further equivariant under post-composing with $\mathbb{Z}$-translations.
Intuitively, such a desired deformation $\widetilde{K}\simeq\mathbb{R}$
can be accomplished by diffusing the ``pointwise displacements''
towards a constant displacement; this can be made precise by assigning
to each $f\colon\mathbb{R}\to\mathbb{R}$ its \emph{displacement function}
$\delta\left(f\right)\colon\mathbb{R}\to\mathbb{R}$ given by $\delta\left(f\right)\left(x\right)\coloneqq f\left(x\right)-x$.
We particularly concern those displacement functions $\delta\left(f\right)$
that are associated with $f\in\widetilde{K}$, which can be characterized
by two properties: ${\delta\left(f\right)'\left(x\right)>-1}$ and
${\delta\left(f\right)\left(x+1\right)=\delta\left(f\right)\left(x\right)}$
for all $x$. Thus to sum up, taking the displacement functions yields
an open embedding of $\widetilde{K}$ into $\MappingSpace^{\infty}\left(\Sphere 1\right)$
(viewed as the space of periodic real functions):
\begin{equation}
\delta\colon\widetilde{K}\hookrightarrow\MappingSpace^{\infty}\left(\Sphere 1\right),\qquad\delta\left(\widetilde{K}\right)=\left\{ u\in\MappingSpace^{\infty}\left(\Sphere 1\right)\mid u'\left(x\right)>-1\right\} \subsetopen\MappingSpace^{\infty}\left(\Sphere 1\right).\label{eq:S1-Vau-proof-step-open-embedding}
\end{equation}
Having switched the scene from $K$ to $\MappingSpace^{\infty}\left(\Sphere 1\right)$,
we may and shall deform the diffeomorphisms in $K$ by means of running
flows on the corresponding functions in $\MappingSpace^{\infty}\left(\Sphere 1\right)$,
as will be detailed in the next step.

$\smallblacktriangleright$ \noun{Step 2: Deforming circle diffeomorphisms.}
As we shall see here and throughout, the \emph{heat flow} will generally
come in useful whenever we need a natural (smooth, geometric, equivariant)
way to diffuse the amounts of disturbances, which is briefly recapped
as follows. For any smooth, connected, compact Riemannian manifold
$M$, the solution operator of the heat equation $\left(\partial_{t}-\Delta\right)u=0$
gives a flow in $\MappingSpace^{\infty}\left(M\right)$ that induces,
upon suitable rescaling of the time interval, a deformation retraction
$\Heat M$ of $\MappingSpace^{\infty}\left(M\right)$ onto its subspace
of constant functions $\mathbb{R}$:
\begin{equation}
\Heat M\colon\MappingSpace^{\infty}\left(M\right)\times\left[0,1\right]\to\MappingSpace^{\infty}\left(M\right),\qquad\left.\Heat M\right|_{t=1}\colon\MappingSpace^{\infty}\left(M\right)\simeq\mathbb{R},\label{eq:S1-Vau-proof-step-heat-flow}
\end{equation}
where as $t$ approaches to the terminal time, each given initial
function $u$ will even out towards a constant function $\overline{u}\in\mathbb{R}$
equal to the average value of $u$ over $M$. With this tool, we
are now in a position to construct the desired deformation of $K$
(as stated in \prettyref{eq:S1-Vau-proof-step-deform-K}) by following
the plan
\begin{equation}
\MappingSpace^{\infty}\left(\Sphere 1\right)\simeq\mathbb{R}{\quad\xRightarrow{\delta}\quad}\widetilde{K}\simeq\mathbb{R}{\quad\xRightarrow{q}\quad}K\simeq\CircleGroup,\label{eq:S1-Vau-proof-step-plan}
\end{equation}
which is implemented as follows. Firstly, we run the heat flow \prettyref{eq:S1-Vau-proof-step-heat-flow}
to obtain a deformation retraction of $\MappingSpace^{\infty}\left(\Sphere 1\right)$
onto its subgroup $\mathbb{R}$ of constant functions. Secondly, we
pull back this deformation of $\MappingSpace^{\infty}\left(\Sphere 1\right)$
via the the open embedding $\delta\colon\widetilde{K}\hookrightarrow\MappingSpace^{\infty}\left(\Sphere 1\right)$
in \prettyref{eq:S1-Vau-proof-step-open-embedding}, so as to obtain
a deformation retraction of $\widetilde{K}$ onto its subgroup $\mathbb{R}$
of translations — this is justified since the characteristic property
$u'\left(x\right)>-1$ is preserved under the heat flow at all times
(thanks to the maximum principle of the heat flow). Thirdly, we descend
this deformation of $\widetilde{K}$ via the quotient map $q\colon\widetilde{K}\twoheadrightarrow K$
in \prettyref{eq:S1-Vau-proof-step-descend-map}, so as to obtain
the desired deformation retraction $K\simeq\CircleGroup$ — this is
justified since it is equivariant under post-composing with $\mathbb{Z}$-translations
(thanks to the same property of the heat flow and of the displacement
function $\delta$). Thus to sum up, we have constructed a deformation
retraction $K\simeq\CircleGroup$ as follows:
\begin{equation}
H\colon K\times\left[0,1\right]\to K,\qquad H_{t}\circ q\coloneqq q\circ\left.\left(\delta^{-1}\circ\Heat{\Sphere 1}_{t}\circ\delta\right)\right|_{\widetilde{K}},\label{eq:S1-Vau-proof-step-construction}
\end{equation}
whose desired properties of smoothness and equivariance will be verified
in the next step.

$\smallblacktriangleright$ \noun{Step 3: Verifying smoothness and equivariance.}
To verify that the deformation retraction $H$ constructed in \prettyref{eq:S1-Vau-proof-step-construction}
has the desired properties of smoothness and equivariance, we shall
implement the following strategy in view of \prettyref{eq:S1-Vau-proof-step-plan}:
i) show that the heat flow on $\MappingSpace^{\infty}\left(\Sphere 1\right)$
has the corresponding properties; ii) show that these properties are
inherited by the induced flow on $\widetilde{K}$ under the embedding
$\delta\colon\widetilde{K}\hookrightarrow\MappingSpace^{\infty}\left(\Sphere 1\right)$;
and iii) show that these properties are inherited by the induced flow
on $K$ under the projection $q\colon\widetilde{K}\to K$. More specifically:
\begin{itemize}
\item $H$ is smooth. Indeed, we show that: i) the heat flow on $\MappingSpace^{\infty}\left(\Sphere 1\right)$
is smooth, as can be verified by observing that its Fourier coefficients
evolve smoothly with respect to the initial data; ii) the preceding
smoothness is inherited by the induced flow on $\widetilde{K}$ under
the embedding $\delta\colon\widetilde{K}\hookrightarrow\MappingSpace^{\infty}\left(\Sphere 1\right)$,
for the reason that $\delta$ is a smooth open embedding; and iii)
the preceding smoothness is inherited by the induced flow on $K$
under the projection $q\colon\widetilde{K}\to K$, for the reason
that $q$ is a smooth submersion. Therefore, the desired smoothness
of $H$ is confirmed.
\item $H$ is $\CircleGroup$-equivariant. Indeed, we show that: i) the
heat flow on $\MappingSpace^{\infty}\left(\Sphere 1\right)$ is equivariant
(with respect to pre-translations on $\MappingSpace^{\infty}\left(\Sphere 1\right)$),
as can be shown by observing that the heat equation is geometric (i.e.,
isometry-invariant); ii) under the embedding $\delta\colon\widetilde{K}\hookrightarrow\MappingSpace^{\infty}\left(\Sphere 1\right)$,
the induced deformation $\widetilde{K}\to\widetilde{K}$ inherits
an $\mathbb{R}$-equivariance (with respect to the $\mathbb{R}$-action
on $\widetilde{K}$ by conjugation), for the reason that $\delta$
is equivariant (as can be verified straightforwardly); and iii) under
the projection $q\colon\widetilde{K}\to K$, the induced deformation
$K\to K$ inherits a $\CircleGroup$-equivariance (with respect to
the $\CircleGroup$-action on $K$ by conjugation), for the reason
that $q$ is equivariant (as it is even a group homomorphism). Therefore,
the desired equivariance of $H$ is confirmed.
\end{itemize}
This completes the proof of the smooth, equivariant deformation retraction
$K\simeq\CircleGroup$ in \prettyref{eq:S1-Vau-proof-step-deform-K}.

$\smallblacktriangleright$ \noun{Step 4: Deforming circle-bundle automorphisms.}
As mentioned, the preceding smooth, equivariant deformation retraction
$K\simeq\CircleGroup$ in \prettyref{eq:S1-Vau-proof-step-deform-K}
will induce a desired deformation retraction $\Vau\left(\StandardFibration\right)\simeq\MappingSpace^{\infty}\left(B,\CircleGroup\right)$
by \prettyref{prop:Homotopy-reduction-of-Vau}. We reiterate the procedure
here for the purpose of preciseness. Let $\left(U_{i}\right)_{i}$
be some trivializing atlas of $\StandardFibration$, over which let
$\left(\tau_{ij}\right)_{i,j}$ be the structural cocycle of $\StandardFibration$
as given by transition maps $\tau_{ij}\colon U_{i}\cap U_{j}\to\CircleGroup$.
Recall that every vertical automorphism $h\in\Vau\left(\StandardFibration\right)$
can be viewed as a cochain $\left(h_{i}\right)_{i}$ given by smooth
maps $h_{i}\colon U_{i}\to K$, such that every pair of which satisfies
the gluing condition $h_{i}=\tau_{ij}h_{j}{\tau_{ij}}^{-1}$ on overlapped
domain. In this way, pushing forward by the smooth, equivariant deformation
retraction $H_{t}\colon K\to K$ in \prettyref{eq:S1-Vau-proof-step-deform-K}
yields a deformation retraction of $\Vau\left(\StandardFibration\right)$:
\[
H_{\ast}\colon\Vau\left(\StandardFibration\right)\times\left[0,1\right]\to\Vau\left(\StandardFibration\right),\qquad H_{\ast}\left(\left(h_{i}\right)_{i},t\right)\coloneqq\left(H_{t}\circ h_{i}\right)_{i}.
\]
Here, at the terminal time $t=1$, the resulting cochains take values
in the abelian group $\CircleGroup$, so that the gluing condition
simply reads $h_{i}=h_{j}$ on overlapped domains, rendering cochains
into maps globally-defined on $B$:
\[
\left.H_{\ast}\right|_{t=1}\colon\Vau\left(\StandardFibration\right)\xrightarrow{\simeq}\left\{ \left(h_{i}\right)_{i}\in\prod_{i\in I}\MappingSpace^{\infty}\left(U_{i},\CircleGroup\right)\mid\ensuremath{h_{i}=h_{j}},\ \forall x\in U_{i}\cap U_{j}\right\} \cong\MappingSpace^{\infty}\left(B,\CircleGroup\right).
\]
This completes the proof of \prettyref{prop:The-Vau-group-of-oriented-circle-bundle-I},
as desired.

\subsection{The vertical automorphism group (continued)}

By the preceding proposition (\prettyref{prop:The-Vau-group-of-oriented-circle-bundle-I}),
we are led to study the homotopy of the $\CircleGroup$-valued current
group $\MappingSpace^{\infty}\left(B,\CircleGroup\right)$. The weak
homotopy type of this group was already well-understood thanks to
the uncomplicated topological structure of $\CircleGroup$; in order
to keep track of the subspace that realizes the homotopy type, the
heat flow used in the preceding proof comes in useful again. The upshot
is the following result, which explicitly describes the homotopy structure
of $\Vau\left(\StandardFibration\right)$ in terms of the principal
circle action and the first integral cohomology of the base:
\begin{cor}[Vertical automorphisms of oriented circle fiberings, II]
\label{cor:The-Vau-group-of-oriented-circle-bundle-II}Let $\StandardFibration$
be a smooth oriented circle bundle over $B$, assumed to be principal
$\CircleGroup$-bundle as above. Then the vertical automorphism group
of $\StandardFibration$ deformation retracts onto the direct product
of $\CircleGroup$ and the first integral cohomology
of $B$:
\[
\Vau\left(\StandardFibration\right)\simeq\CircleGroup\times H^{1}\left(B;\mathbb{Z}\right).
\]
Here, the identity component $\CircleGroup$ is embedded as the subgroup
consisting of uniform translations on all fibers, while the component
group $H^{1}\left(B;\mathbb{Z}\right)\cong\left[B,\CircleGroup\right]$
has each of its members being represented by a parametrized translation
on the fibers.
\end{cor}

\begin{proof}
We first invoke \prettyref{prop:The-Vau-group-of-oriented-circle-bundle-I}
to deformation retract $\Vau\left(\StandardFibration\right)$ onto
$\MappingSpace^{\infty}\left(B,\CircleGroup\right)$. The latter group
has the desired (weak) homotopy type:
\[
\MappingSpace^{\infty}\left(B,\CircleGroup\right)\simeq\MappingSpace\left(B,\CircleGroup\right)\simeq\CircleGroup\times H^{1}\left(B;\mathbb{Z}\right).
\]
Here, the first equivalence between $\MappingSpace^{\infty}\left(B,\CircleGroup\right)$
and $\MappingSpace\left(B,\CircleGroup\right)$ follows from the general
smooth approximations for current groups (see e.g., \cite[Theorem A.3.7]{MR1935553}),
which in particular establishes isomorphisms between $\pi_{i}\MappingSpace^{\infty}\left(B,\CircleGroup\right)$
and $\pi_{i}\MappingSpace\left(B,\CircleGroup\right)$ for all $i$;
while the second equivalence between $\MappingSpace\left(B,\CircleGroup\right)$
and $\CircleGroup\times H^{1}\left(B;\mathbb{Z}\right)$ follows from
the general homotopy types for current groups with coefficients in
Eilenberg–MacLane spaces $K\left(\mathbb{Z},n\right)$ (see e.g.,
\cite{MR105106}), which in particular establishes isomorphisms between
$\pi_{i}\MappingSpace\left(B,\CircleGroup\right)$ and $H^{1-i}\left(B;\mathbb{Z}\right)$
for all $i$. Specifically, we can further keep track of the subspace
realizing this homotopy type as follows:
\begin{itemize}
\item For the component group $\pi_{0}\MappingSpace^{\infty}\left(B,\CircleGroup\right)$,
we confirm its asserted description by the following isomorphisms:
\[
\pi_{0}\MappingSpace^{\infty}\left(B,\CircleGroup\right)\cong\left[B,\CircleGroup\right]\cong\hom\left(\pi_{1}\left(B\right),\pi_{1}\left(\CircleGroup\right)\right)\cong\hom\left(H_{1}\left(B;\mathbb{Z}\right),\mathbb{Z}\right)\cong H^{1}\left(B;\mathbb{Z}\right).
\]
Here, the crux is at the second isomorphism from $\left[B,\CircleGroup\right]$
to $\hom\left(\pi_{1}\left(B\right),\pi_{1}\left(\CircleGroup\right)\right)$,
as given by the construction of induced homomorphisms, which is indeed
invertible thanks to the aspherical-ness of the target space $\CircleGroup$,
as desired.
\item For the identity component $\MappingSpace_{0}^{\infty}\left(B,\CircleGroup\right)$,
we need to confirm its asserted deformation retraction onto $\CircleGroup$.
Since null-homotopic maps $B\to\CircleGroup$ always lift, $\MappingSpace_{0}^{\infty}\left(B,\CircleGroup\right)$
admits the central extension $\mathbb{Z}\hookrightarrow\MappingSpace^{\infty}\left(B\right)\overset{q}{\twoheadrightarrow}\MappingSpace_{0}^{\infty}\left(B,\CircleGroup\right)$.
Thus arguing as in the proof of \prettyref{prop:The-Vau-group-of-oriented-circle-bundle-I},
we can use the heat flow to construct the desired deformation retraction:
\[
H\colon\MappingSpace_{0}^{\infty}\left(B,\CircleGroup\right)\times\left[0,1\right]\to\MappingSpace_{0}^{\infty}\left(B,\CircleGroup\right),\qquad H_{t}\circ q\coloneqq q\circ\Heat B_{t}.
\]
More specifically, running the heat flow \prettyref{eq:S1-Vau-proof-step-heat-flow}
yields a deformation retraction of $\MappingSpace^{\infty}\left(B\right)$
onto the subgroup $\mathbb{R}$ of constant functions, which thanks
to its $\mathbb{Z}$-equivariance can be descended to a deformation
retraction of $\MappingSpace_{0}^{\infty}\left(B,\CircleGroup\right)$
onto the subgroup $\CircleGroup$ of constant maps (hence corresponding
to the subgroup of $\Vau\left(\StandardFibration\right)$ consisting
of uniform translations), as desired.
\end{itemize}
Lastly, since both groups $\CircleGroup$ and $H^{1}\left(B;\mathbb{Z}\right)$
are embedded as subgroups of the abelian group $\MappingSpace^{\infty}\left(B,\CircleGroup\right)$,
their semidirect product must actually be direct product. This completes
the proof of \prettyref{cor:The-Vau-group-of-oriented-circle-bundle-II},
as desired.
\end{proof}
We are now in a position to apply the results as developed in this
subsection to yield precise information about low-dimensional manifolds
and their circle fiberings.

\section{Case Study: Dimension Two\label{sec:Case-study-dimension2}}

We specialize the preceding study in \prettyref{sec:oriented-circle-fiberings}
to the case of closed, 2-dimensional total space $E$. Our goal is
the following theorem, which completely determines the homotopy type
(hence topological type too) of the moduli space $\Fib_{+}\left(E,\Sphere 1\right)$
of smooth oriented circle fiberings in this case:
\begin{thm}[Space of oriented circle fiberings on surfaces]
\label{thm:T2-case-main-theorem}The only closed surface that admits
smooth oriented circle fiberings is the 2-torus $\Torus 2$, on which
all such fiberings form a single equivalence class modeled on the
standard meridian fibration $\StandardFibration\colon\Torus 2\to\Sphere 1$,
whose corresponding moduli space of fiberings $\Fib\left(\StandardFibration\right)$
admits a deformation retraction onto a discrete space $\ZPrim 2$
of coprime pairs of integers:
\[
\Fib_{+}\left(\Torus 2,\Sphere 1\right)=\Fib\left(\StandardFibration\right)\simeq\ZPrim 2.
\]
More precisely, this deformation retract consists of those fiberings
that are affinely equivalent to $\StandardFibration$ (aka. the ``rational
linear fiberings''), which is identified with $\ZPrim 2$ by extracting%
the ``slope''.
\end{thm}

In what follows, we shall justify that $\Fib\left(\StandardFibration\right)$
is indeed the only non-void case, and that its homotopy type is indeed
as claimed. The essence of our proof for the latter claim is to homotopy-equivalently
rigidify the various (infinite-dimensional) transformation groups
of $\StandardFibration$, by means of exploiting a suitable structure
on the ambient space — in the current case, the standard \emph{affine structure}
on $\Torus 2$. This scheme is succinctly summarized in the following
two commutative diagrams:
\[
\xymatrix{\Vau\left(\StandardFibration\right)\ar@{^{(}->}[r] & \Aut\left(\StandardFibration\right)\ar@{->>}[r] & \Diff\left(B\right)_{\StandardFibration}\\
\Vau_{\Affine}\left(\StandardFibration\right)\ar@{^{(}->}[r]\ar@{^{(}->}[u]^{\simeq} & \Aut_{\Affine}\left(\StandardFibration\right)\ar@{->>}[r]\ar@{^{(}..>}[u]^{\simeq} & \Diff_{\Affine}\left(B\right)_{\StandardFibration}\ar@{^{(}->}[u]^{\simeq}\\
\CircleGroup\times\mathbb{Z}\ar@{^{(}->}[r]\ar@{=}[u]^{\cong} & \CircleGroup^{2}\rtimes\mathrm{Aff}\left(1,\mathbb{Z}\right)\ar@{->>}[r]\ar@{=}[u]^{\cong} & \CircleGroup\rtimes C_{2}\ar@{=}[u]^{\cong}
}
\]
and
\[
\xymatrix{\Aut\left(\StandardFibration\right)\ar@{^{(}->}[r] & \Diff\left(\Torus 2\right)\ar@{->>}[r] & \Fib\left(\StandardFibration\right)\\
\Aut_{\Affine}\left(\StandardFibration\right)\ar@{^{(}->}[r]\ar@{^{(}->}[u]^{\simeq} & \Affine\left(\Torus 2\right)\ar@{->>}[r]\ar@{^{(}->}[u]^{\simeq} & \Fib_{\Affine}\left(\StandardFibration\right)\ar@{^{(}..>}[u]^{\simeq}\\
\CircleGroup^{2}\rtimes\mathrm{Aff}\left(1,\mathbb{Z}\right)\ar@{^{(}->}[r]\ar@{=}[u]^{\cong} & \CircleGroup^{2}\rtimes\GL 2Z\ar@{->>}[r]\ar@{=}[u]^{\cong} & \ZPrim 2\ar@{=}[u]^{\cong}
}
.
\]
The detailed explanation for these two diagrams will occupy the main
part of the remaining of this subsection.

\subsection{The classification}

To begin with, let us use \prettyref{prop:Classifying-oriented-circle-fiberings}
to classify smooth oriented circle fiberings on closed surface:
\begin{lem}
\label{lem:T2-case-classification}The only closed surface that admits
smooth oriented circle fiberings is the 2-torus $\Torus 2$, on which
all such fiberings form a single equivalence class:
\[
\Fib_{+}\left(\Torus 2,\Sphere 1\right)=\Fib\left(\StandardFibration\right).
\]
More precisely, this unique class can be represented by the (oriented)
standard meridian fibration $\StandardFibration\colon\Torus 2\to\Sphere 1$
(as given by the first coordinate projection of $\Sphere 1\times\Sphere 1$).
\end{lem}

\begin{proof}
Suppose that $E$ is a closed surface. Then the claim is equivalent
to saying that the desired classification $\ClassFib_{+}\left(E,\Sphere 1\right)$
(consisting of all equivalence classes of smooth oriented circle fiberings
on $E$) is completely determined as
\[
\ClassFib_{+}\left(E,\Sphere 1\right)=\begin{cases}
\left\{ \StandardFibration\right\}  & \text{if \ensuremath{E\approx\Torus 2}},\\
\emptyset & \text{otherwise}
\end{cases}\qquad\text{(\ensuremath{E} closed surface)}.
\]
To this end, first note that for any oriented circle bundle $E\to B$,
the base space $B$ must be a closed 1-manifold, hence $B\approx\Sphere 1$.
Thus we can apply \prettyref{prop:Classifying-oriented-circle-fiberings}
with the constraint $B\approx\Sphere 1$, which yields the following
injection via the Euler class:
\[
\bigsqcup_{\text{\ensuremath{E} closed surface}}\ClassFib_{+}\left(E,\Sphere 1\right)\quad\hookrightarrow\quad H^{2}\left(\Sphere 1;\mathbb{Z}\right)/\pi_{0}\Diff\left(\Sphere 1\right)=0.
\]
Thus there is only one candidate for the desired bundle classes $E\to\Sphere 1$;
i.e., the trivial one, which can be represented by the standard meridian
fibration $\StandardFibration\colon\Torus 2\to\Sphere 1$ on the 2-torus
$\Torus 2$. Therefore, we conclude that under our assumption on
$E$, the non-void cases of $\ClassFib_{+}\left(E,\Sphere 1\right)$
are exactly exhausted by $\ClassFib_{+}\bigl(\Torus 2,\Sphere 1\bigr)=\left\{ \StandardFibration\right\} $,
as desired.
\end{proof}

\subsection{The base transformation group}

Having obtained $\left\{ \StandardFibration\right\} $ as the complete
set of representatives from the preceding result, we next study its
symmetries. We first consider its basic transformation group $\Diff\left(B\right)_{\StandardFibration}$
and use \prettyref{cor:Base-transformations-of-an-oriented-circle-bundle}
to determine its homotopy type:
\begin{lem}
\label{lem:T2-case-base-transformations}The basic transformation
group of $\StandardFibration\colon\Torus 2\to\Sphere 1$ deformation
retracts onto its subgroup consisting of those that are affine:
\begin{equation}
\Diff\left(\Sphere 1\right)_{\StandardFibration}\simeq\CircleGroup\rtimes C_{2}.\label{eq:T2-case-base-transformations}
\end{equation}
More precisely, the identity component $\CircleGroup$ is embedded
as the subgroup consisting of translations on the base circle, while
the component group $C_{2}$ is generated by a distinguished reflection,
acting on $\CircleGroup$ by inversion.
\end{lem}

\begin{proof}
This is clear on the level of identity components, where the desired
deformation retraction from $\Diff_{0}\left(\Sphere 1\right)$ onto
$\CircleGroup$ was already well known: e.g., by running the heat
flow as in the proof of \prettyref{prop:The-Vau-group-of-oriented-circle-bundle-I},
or by running the affine homotopy on its universal cover. As for the
$\pi_{0}$-level, we need to show that those components making up
$\Diff\left(\Sphere 1\right)_{\StandardFibration}$ attain the full
mapping class group $\pi_{0}\Diff\left(\Sphere 1\right)\cong C_{2}$.
This is already clear since $\StandardFibration$ is trivial (so that
all the basic transformations can lift to automorphisms). Alternatively,
we can apply \prettyref{cor:Base-transformations-of-an-oriented-circle-bundle}
to the (trivial) action of $\pi_{0}\Diff\left(\Sphere 1\right)\cong C_{2}$
on $H^{2}\left(\Sphere 1;\mathbb{Z}\right)=0$, for which the stabilizer
$\pi_{0}\Diff\left(\Sphere 1\right)_{\StandardFibration}$ clearly
attains the full group $C_{2}$, as desired.
\end{proof}

\subsection{The vertical automorphism group}

Having understood the basic transformations from the preceding result,
we next consider the vertical automorphism group $\Vau\left(\StandardFibration\right)$
and use \prettyref{cor:The-Vau-group-of-oriented-circle-bundle-II}
to determine its homotopy type:
\begin{lem}
\label{lem:T2-case-vertical-automorphisms}The vertical automorphism
group of $\StandardFibration\colon\Torus 2\to\Sphere 1$ deformation
retracts onto its subgroup consisting of those that are affine:
\begin{equation}
\Vau\left(\StandardFibration\right)\simeq\CircleGroup\times\mathbb{Z}.\label{eq:T2-case-vertical-automorphisms}
\end{equation}
More precisely, the identity component $\CircleGroup$ is embedded
as the subgroup consisting of uniform translations on all fibers,
while the component group $\mathbb{Z}$ is generated by the Dehn twist
along a distinguished fiber, acting on $\CircleGroup$ trivially.
\end{lem}

\begin{proof}
Applying \prettyref{cor:The-Vau-group-of-oriented-circle-bundle-II}
with $B=\Sphere 1$, we obtain a deformation retraction of $\Vau\left(\StandardFibration\right)$
onto $\CircleGroup\times H^{1}\left(\Sphere 1;\mathbb{Z}\right)\cong\CircleGroup\times\mathbb{Z}$.
We need to verify that this deformation retract as explicitly described
in \prettyref{cor:The-Vau-group-of-oriented-circle-bundle-II} agrees
with what is claimed here. This is already clear on the level of identity
components $\Vau_{0}\left(\StandardFibration\right)\simeq\CircleGroup$.
As for the $\pi_{0}$-level, the proposition tells us that $H^{1}\left(\Sphere 1;\mathbb{Z}\right)\cong\left[\Sphere 1,\CircleGroup\right]$
has each of its members being represented by a parametrized translation
$\Sphere 1\to\CircleGroup$ on the fibers. In other words, $H^{1}\left(\Sphere 1;\mathbb{Z}\right)\cong\mathbb{Z}$
encodes the degree of this circle map $\Sphere 1\to\CircleGroup$,
so that it is generated by the degree-$1$ map $e^{i\theta}\mapsto\theta$,
which in turn corresponds to the vertical automorphism that translates
the circle fiber over each $e^{i\theta}\in\Sphere 1$ by $\theta$
— this is exactly the Dehn twist as desired.
\end{proof}

\subsection{The ambient space and its diffeomorphism group}

So far, only the fiber structure and the base structure were concerned;
to proceed, we bring in the ambient structure. For the ambient 2-torus,
the homotopy type of its diffeomorphism group was known to Earle and
Eells in \cite{MR276999}:
\begin{thm}[\noun{Earle–Eells}]
\label{thm:T2-case-ambient-diffeomorphisms}The diffeomorphism group
of $\Torus 2$ admits the following deformation retract:
\begin{equation}
\Diff\left(\Torus 2\right)\simeq\Affine\left(\Torus 2\right)\cong\CircleGroup^{2}\rtimes\GL 2Z,\label{eq:T2-case-ambient-diffeomorphisms}
\end{equation}
where $\Affine\left(\Torus 2\right)$ is the affine diffeomorphism
group of the standard Euclidean $\Torus 2$.
\end{thm}

Here, the affine structure of the Euclidean torus $\Torus 2=\mathbb{R}^{2}/\mathbb{Z}^{2}$
is induced from the standard Euclidean plane $\mathbb{R}^{2}$, with
respect to which the affine diffeomorphisms form the group $\Affine\left(\Torus 2\right)$.
In turn, this group can be described in terms of the linear structure
on the cover $\mathbb{R}^{2}$; more specifically, each affine diffeomorphism
on $\Torus 2$ has a unique representation by its lift to $\mathbb{R}^{2}$
of the form
\[
\bigl(x,y\bigr)\mapsto A\bigl(x,y\bigr)+\bigl(a,b\bigr)\qquad\text{with}\qquad\bigl(a,b\bigr)\in\mathbb{R}^{2},\ A\in\GL 2Z.
\]
This yields the desired isomorphism $\Affine\left(\Torus 2\right)\cong\CircleGroup^{2}\rtimes\GL 2Z$,
with the semidirect product corresponding to the canonical action
of $\GL 2Z$ on $\CircleGroup^{2}$. Our remaining goal is to use
this homotopy model of $\Diff\left(\Torus 2\right)$ to induce the
desired homotopy models of $\Aut\left(\StandardFibration\right)$
and of $\Fib\left(\StandardFibration\right)$; for this purpose, recall
that our model fibration $\StandardFibration\colon\Torus 2\to\Sphere 1$
is descended from the coordinate projection $\mathbb{R}^{2}\to\mathbb{R}$
given by $\bigl(x,y\bigr)\mapsto x$.

\subsection{The homotopy core of the automorphism group}

We first consider the group $\Aut\left(\StandardFibration\right)$
of automorphisms of $\StandardFibration$. By restricting (to $\Aut\left(\StandardFibration\right)$)
the preceding homotopy model $\Affine\left(\Torus 2\right)$ of $\Diff\left(\Torus 2\right)$,
we obtain the following candidate for a homotopy model of $\Aut\left(\StandardFibration\right)$:
\begin{equation}
\Aut\left(\StandardFibration\right)\hookleftarrow\Aut\left(\StandardFibration\right)\cap\Affine\left(\Torus 2\right)\cong\CircleGroup^{2}\rtimes\AffineGroupZ.\label{eq:T2-case-automorphisms-pre}
\end{equation}
Here, $\mathrm{Aff}\left(1,\mathbb{Z}\right)\coloneqq\mathbb{Z}^{1}\rtimes\GL 1Z=\mathbb{Z}\rtimes C_{2}$\nomenclature[AffnZ]{$\mathrm{Aff}\left(n,\mathbb{Z}\right)$}{the general affine group of degree $n$ over $\mathbb{Z}$ (i.e., $\mathbb{Z}^{n}\rtimes\GL nZ$)}
denotes the general affine group of degree 1 over $\mathbb{Z}$, which
can be viewed as the subgroup of $\GL 2Z$ consisting of matrices
of the form $\left(\begin{smallmatrix}\epsilon & 0\\
n & 1
\end{smallmatrix}\right)$ for $\epsilon=\pm1$ and $n\in\mathbb{Z}$. This leads us to the
following result, which asserts that the above embedding \prettyref{eq:T2-case-automorphisms-pre}
is indeed a homotopy equivalence:
\begin{thm}
\label{thm:T2-case-automorphisms}The automorphism group of $\StandardFibration\colon\Torus 2\to\Sphere 1$
deformation retracts onto its subgroup consisting of those that are
affine:
\begin{equation}
\Aut\left(\StandardFibration\right)\simeq\CircleGroup^{2}\rtimes\AffineGroupZ,\label{eq:T2-case-automorphisms}
\end{equation}
where $\CircleGroup^{2}\rtimes\AffineGroupZ$ is embedded as a subgroup
via \prettyref{eq:T2-case-automorphisms-pre}.
\end{thm}

\begin{proof}
Consider the surjective map
\[
\CircleGroup^{2}\rtimes\mathrm{Aff}\left(1,\mathbb{Z}\right)\twoheadrightarrow\CircleGroup\rtimes C_{2},\qquad\left(\left(\begin{smallmatrix}a\\
b
\end{smallmatrix}\right),\left(\begin{smallmatrix}\epsilon & 0\\
n & 1
\end{smallmatrix}\right)\right)\mapsto\left(a,\epsilon\right).
\]
This projection yields a finite-dimensional fibration (resp., extension)
that is embedded into our infinite-dimensional fibration (resp., extension)
$\Vau\left(\StandardFibration\right)\hookrightarrow\Aut\left(\StandardFibration\right)\twoheadrightarrow\Diff\left(B\right)_{\StandardFibration}$,
as in the following commutative diagram:
\[
\xymatrix{\Vau\left(\StandardFibration\right)\ar@{^{(}->}[r] & \Aut\left(\StandardFibration\right)\ar@{->>}[r] & \Diff\left(B\right)_{\StandardFibration}\\
\CircleGroup\times\mathbb{Z}\ar@{^{(}->}[r]\ar@{^{(}->}[u] & \CircleGroup^{2}\rtimes\mathrm{Aff}\left(1,\mathbb{Z}\right)\ar@{->>}[r]\ar@{^{(}->}[u] & \CircleGroup\rtimes C_{2}\ar@{^{(}->}[u]
}
.
\]
Here, the right, left, and middle vertical arrows are given by the
embeddings \prettyref{eq:T2-case-base-transformations}, \prettyref{eq:T2-case-vertical-automorphisms},
and \prettyref{eq:T2-case-automorphisms-pre}, respectively; in other
words, they are embedded as the ``affine models'' of $\Diff\left(B\right)_{\StandardFibration}$,
$\Vau\left(\StandardFibration\right)$, and $\Aut\left(\StandardFibration\right)$,
from which we see that the diagram indeed commutes.

Turning to homotopy types, we use the (surjective) fibration structures
on both rows to induce two long exact sequences of homotopy groups.
Since by \prettyref{lem:T2-case-base-transformations} and \prettyref{lem:T2-case-vertical-automorphisms}
the right embedding of $\CircleGroup\rtimes C_{2}$ into $\Diff\left(B\right)_{\StandardFibration}$
and the left embedding of $\CircleGroup\times\mathbb{Z}$ into $\Vau\left(\StandardFibration\right)$
are both homotopy equivalences, we can deduce by the five lemma that
the middle embedding of ${\CircleGroup^{2}\rtimes\AffineGroupZ}$
into $\Aut\left(\StandardFibration\right)$ is a homotopy equivalence
too, as desired. Lastly, the promotion to deformation retraction
follows from \prettyref{prop:Promoting-weak-homotopy-type}, which
completes the proof of \prettyref{thm:T2-case-automorphisms}.
\end{proof}

\subsection{The homotopy core of the moduli space}

Finally, we turn to our main object of study: the moduli space $\Fib\left(\StandardFibration\right)$
of fiberings modeled on $\StandardFibration$. By projecting (to $\Fib\left(\StandardFibration\right)$)
the preceding homotopy model $\Affine\left(\Torus 2\right)$ of $\Diff\left(\Torus 2\right)$,
we obtain the following candidate for a homotopy model of $\Fib\left(\StandardFibration\right)$:
\begin{equation}
\Fib\left(\StandardFibration\right)\hookleftarrow\Affine\left(\Torus 2\right)\cdot\left\{ \StandardFibration\right\} \approx\ZPrim 2.\label{eq:T2-case-space-of-fiberings-pre}
\end{equation}
Here, $\ZPrim 2$\nomenclature[ZPn]{$\ZPrim n$}{the discrete space of all setwise-coprime $n$-tuples of integers}
denotes the discrete space of coprime pairs of integers, which can
be viewed as the homogeneous space of $\GL 2Z$ consisting of vectors
$\left(p,q\right)\in\ZPrim 2$ in the orbit of $\left(0,1\right)$.
By construction, $\ZPrim 2$ is embedded in $\Fib\left(\StandardFibration\right)$
as the subspace of (oriented) \emph{rational linear fiberings} on
$\Torus 2$, which are those that are \emph{affinely equivalent}
to the standard one $\StandardFibration$ (as characterized by the
``slope'' ranging in $\ZPrim 2$). Therefore, proving that the above
embedding \prettyref{eq:T2-case-space-of-fiberings-pre} is indeed
a homotopy equivalence will complete the final step of our proof of
\prettyref{thm:T2-case-main-theorem}:
\begin{proof}[Proof of \prettyref{thm:T2-case-main-theorem}]
The claim regarding classification was proved in \prettyref{lem:T2-case-classification},
which in particular shows that $\Fib_{+}\left(\Torus 2,\Sphere 1\right)=\Fib\left(\StandardFibration\right)$.
In turn, to determine the homotopy type of $\Fib\left(\StandardFibration\right)$,
consider the surjective map
\[
\CircleGroup^{2}\rtimes\GL 2Z\twoheadrightarrow\ZPrim 2,\qquad\left(\left(\begin{smallmatrix}a\\
b
\end{smallmatrix}\right),\left(\begin{smallmatrix}m & p\\
n & q
\end{smallmatrix}\right)\right)\mapsto\left(\begin{smallmatrix}m & p\\
n & q
\end{smallmatrix}\right)\left(\begin{smallmatrix}0\\
1
\end{smallmatrix}\right)=\left(\begin{smallmatrix}p\\
q
\end{smallmatrix}\right).
\]
This projection yields a finite-dimensional fibration that is embedded
into our infinite-dimensional fibration $\Aut\left(\StandardFibration\right)\to\Diff\left(\Torus 2\right)\to\Fib\left(\StandardFibration\right)$,
as in the following commutative diagram:
\[
\xymatrix{\Aut\left(\StandardFibration\right)\ar@{^{(}->}[r] & \Diff\left(\Torus 2\right)\ar@{->>}[r] & \Fib\left(\StandardFibration\right)\\
\CircleGroup^{2}\rtimes\mathrm{Aff}\left(1,\mathbb{Z}\right)\ar@{^{(}->}[r]\ar@{^{(}->}[u] & \CircleGroup^{2}\rtimes\GL 2Z\ar@{->>}[r]\ar@{^{(}->}[u] & \ZPrim 2\ar@{^{(}->}[u]
}
,
\]
Here, the middle, left, and right vertical arrows are given by the
embeddings \prettyref{eq:T2-case-ambient-diffeomorphisms}, \prettyref{eq:T2-case-automorphisms},
and \prettyref{eq:T2-case-space-of-fiberings-pre}, respectively;
in other words, they are embedded as the ``affine models'' of $\Diff\left(\Torus 2\right)$,
$\Aut\left(\StandardFibration\right)$, and $\Fib\left(\StandardFibration\right)$,
from which we see that the diagram indeed commutes.

Turning to homotopy types, we use the (surjective) fibration structures
on both rows to induce two long exact sequences of homotopy groups.
Since by \prettyref{thm:T2-case-ambient-diffeomorphisms} and \prettyref{thm:T2-case-automorphisms}
the middle embedding of $\CircleGroup^{2}\rtimes\GL 2Z$ into $\Diff\left(\Torus 2\right)$
and the left embedding of $\CircleGroup^{2}\rtimes\mathrm{Aff}\left(1,\mathbb{Z}\right)$
into $\Aut\left(\StandardFibration\right)$ are both homotopy equivalences,
we can deduce by the five lemma that the right embedding of $\ZPrim 2$
into $\Fib\left(\StandardFibration\right)$ is a homotopy equivalence
too, as desired. Lastly, the promotion to deformation retraction follows
from \prettyref{prop:Promoting-weak-homotopy-type}, which completes
the proof of \prettyref{thm:T2-case-main-theorem}.
\end{proof}

\section{Case Study: Dimension Three\label{sec:Case-study-dimension3}}

We specialize the preceding study in \prettyref{sec:oriented-circle-fiberings}
to the case of closed, orientable, 3-dimensional total space $E$.

\subsection{The case of flat geometry}

First, let us consider those total spaces $E$ with \emph{flat} geometry,
which serves as a natural generalization of the previous 2-torus case.
Our goal is the following theorem, which completely determines the
homotopy type (hence topological type too) of the moduli space $\Fib_{+}\left(E,\Sphere 1\right)$
of smooth oriented circle fiberings on $E$ in this case:
\begin{thm}[Space of oriented circle fiberings on flat 3-manifolds]
\label{thm:T3-case-main-theorem}The only closed, orientable, flat
3-manifold that admits smooth oriented circle fiberings is the 3-torus
$\Torus 3$, on which all such fiberings form a single equivalence
class modeled on the standard meridian fibration $\StandardFibration_{0}\colon\Torus 3\to\Torus 2$,
whose corresponding moduli space of fiberings $\Fib\left(\StandardFibration_{0}\right)$
admits a deformation retraction onto a discrete space $\ZPrim 3$
of (setwise) coprime triples of integers:
\[
\Fib_{+}\left(\Torus 3,\Sphere 1\right)=\Fib\left(\StandardFibration_{0}\right)\simeq\ZPrim 3.
\]
More precisely, this deformation retract consists of those fiberings
that are affinely equivalent to $\StandardFibration_{0}$, which is
identified with $\ZPrim 3$ by extracting the ``slope''.
\end{thm}

As before, we need to justify that $\Fib\left(\StandardFibration_{0}\right)$
is indeed the only non-void case, and that its homotopy type is indeed
as claimed. The essence of our proof for the latter claim is to homotopy-equivalently
rigidify the various (infinite-dimensional) transformation groups
of $\StandardFibration_{0}$, by means of exploiting a suitable structure
on the ambient space — in the current case, the standard \emph{affine structure}
on $\Torus 3$. This scheme is succinctly summarized in the following
two commutative diagrams:
\[
\xymatrix{\Vau\left(\StandardFibration_{0}\right)\ar@{^{(}->}[r] & \Aut\left(\StandardFibration_{0}\right)\ar@{->>}[r] & \Diff\left(\Torus 2\right)_{\StandardFibration_{0}}\\
\Vau_{\Affine}\left(\StandardFibration_{0}\right)\ar@{^{(}->}[r]\ar@{^{(}->}[u]^{\simeq} & \Aut_{\Affine}\left(\StandardFibration_{0}\right)\ar@{->>}[r]\ar@{^{(}..>}[u]^{\simeq} & \Diff_{\Affine}\left(\Torus 2\right)_{\StandardFibration_{0}}\ar@{^{(}->}[u]^{\simeq}\\
\CircleGroup\times\mathbb{Z}^{2}\ar@{^{(}->}[r]\ar@{=}[u]^{\cong} & \CircleGroup^{3}\rtimes\mathrm{Aff}\left(2,\mathbb{Z}\right)\ar@{->>}[r]\ar@{=}[u]^{\cong} & \CircleGroup^{2}\rtimes\GL 2Z\ar@{=}[u]^{\cong}
}
\]
and
\[
\xymatrix{\Aut\left(\StandardFibration_{0}\right)\ar@{^{(}->}[r] & \Diff\left(\Torus 3\right)\ar@{->>}[r] & \Fib\left(\StandardFibration_{0}\right)\\
\Aut_{\Affine}\left(\StandardFibration_{0}\right)\ar@{^{(}->}[r]\ar@{^{(}->}[u]^{\simeq} & \Affine\left(\Torus 3\right)\ar@{->>}[r]\ar@{^{(}->}[u]^{\simeq} & \Fib_{\Affine}\left(\StandardFibration_{0}\right)\ar@{^{(}..>}[u]^{\simeq}\\
\CircleGroup^{3}\rtimes\mathrm{Aff}\left(2,\mathbb{Z}\right)\ar@{^{(}->}[r]\ar@{=}[u]^{\cong} & \CircleGroup^{3}\rtimes\GL 3Z\ar@{->>}[r]\ar@{=}[u]^{\cong} & \ZPrim 3\ar@{=}[u]^{\cong}
}
\]
Since the proof for the case of flat geometry is essentially the same
as the 2-torus case, we shall omit the detailed justifications and
turn to other geometries.

\subsection{The case of elliptic geometry}

Next, let us consider those total spaces $E$ with \emph{elliptic}
geometry, which serves as our first example of the ``twisted'' case.
Our goal is the following theorem, which completely determines the
homotopy type (hence topological type too) of the moduli space $\Fib_{+}\left(E,\Sphere 1\right)$
of smooth oriented circle fiberings on $E$ in this case:
\begin{thm}[Space of oriented circle fiberings on elliptic 3-manifolds]
\label{thm:elliptic-case-main-theorem}The only closed, elliptic
3-manifolds that admit smooth oriented circle fiberings are the lens
spaces $\Lens e1$ for $e>0$, on each of which all such fiberings
form a single equivalence class modeled on the standard Hopf fibration
$\StandardFibration_{e}\colon\Lens e1\to\Sphere 2$, whose corresponding
moduli space of fiberings $\Fib\left(\StandardFibration_{e}\right)$
admits a deformation retraction onto a pair of disjoint 2-spheres
if $e=1,2$, and a discrete space of two elements otherwise:
\[
\Fib_{+}\left(\Lens e1,\Sphere 1\right)=\Fib\left(\StandardFibration_{e}\right)\simeq\begin{cases}
\Sphere 2\sqcup\Sphere 2 & \text{if \ensuremath{e=1,2}},\\
\Sphere 0 & \text{if \ensuremath{e\geq3}}.
\end{cases}
\]
More precisely, this deformation retract consists of those fiberings
that are congruent to $\StandardFibration_{e}$ (aka. the ``Hopf
fiberings''), which is identified with $\Sphere 2\sqcup\Sphere 2$
by extracting the ``direction'' and ``chirality'' if $e=1,2$,
and is the oppositely-oriented pair $\left\{ \StandardFibration_{e},\StandardFibration_{-e}\right\} $
otherwise.
\end{thm}

In what follows, we shall justify that $\Fib\left(\StandardFibration_{e}\right)$
for $e>0$ are indeed the only non-void cases, and that their homotopy
types are indeed as claimed. The essence of our proof for the latter
claim is to homotopy-equivalently rigidify the various (infinite-dimensional)
transformation groups of $\StandardFibration$, by means of exploiting
a suitable structure on the ambient space — in the current case, the
standard \emph{metric structure} on $\Lens e1$. We shall focus on
the case $e=2$ (i.e., $E\approx\RP 3$, the real projective 3-space),\footnote{On the other hand, the case $e=1$ (i.e., $E\approx\Sphere 3$) is
treated in a joint work \cite{deturck2025homotopytypespacefibrations}
with a more concrete, classical approach; while the remaining cases
$e\geq3$ are similar and turn out to be simpler (with the spaces
of fiberings being contractible up to orientation).} for which the scheme is succinctly summarized in the following two
commutative diagrams:
\[
\xymatrix{\Vau\left(\StandardFibration_{2}\right)\ar@{^{(}->}[r] & \Aut\left(\StandardFibration_{2}\right)\ar@{->>}[r] & \Diff\left(B\right)_{\StandardFibration_{2}}\\
\Vau_{\Isometry}\left(\StandardFibration_{2}\right)\ar@{^{(}->}[r]\ar@{^{(}->}[u]^{\simeq} & \Aut_{\Isometry}\left(\StandardFibration_{2}\right)\ar@{->>}[r]\ar@{^{(}..>}[u]^{\simeq} & \Diff_{\Isometry}\left(B\right)_{\StandardFibration_{2}}\ar@{^{(}->}[u]^{\simeq}\\
\CircleGroup\ar@{^{(}->}[r]\ar@{=}[u]^{\cong} & \SOrthogonal 3\times\CircleGroup\ar@{->>}[r]\ar@{=}[u]^{\cong} & \SOrthogonal 3\ar@{=}[u]^{\cong}
}
\]
and
\[
\xymatrix{\Aut\left(\StandardFibration_{2}\right)\ar@{^{(}->}[r] & \Diff\left(\RP 3\right)\ar@{->>}[r] & \Fib\left(\StandardFibration_{2}\right)\\
\Aut_{\Isometry}\left(\StandardFibration_{2}\right)\ar@{^{(}->}[r]\ar@{^{(}->}[u]^{\simeq} & \Isometry\left(\RP 3\right)\ar@{->>}[r]\ar@{^{(}->}[u]^{\simeq} & \Fib_{\Isometry}\left(\StandardFibration_{2}\right)\ar@{^{(}..>}[u]^{\simeq}\\
\SOrthogonal 3\times\CircleGroup\ar@{^{(}->}[r]\ar@{=}[u]^{\cong} & \left(\SOrthogonal 3\times\SOrthogonal 3\right)\rtimes C_{2}\ar@{->>}[r]\ar@{=}[u]^{\cong} & \Sphere 2\sqcup\Sphere 2\ar@{=}[u]^{\cong}
}
\]
The detailed explanation for these two diagrams will occupy the main
part of the remaining of this subsection.

\subsection{The classification}

To begin with, let us use \prettyref{prop:Classifying-oriented-circle-fiberings}
to classify smooth oriented circle fiberings on closed, elliptic 3-manifolds:
\begin{lem}
\label{lem:elliptic-case-classification}The only closed, elliptic
3-manifolds that admit smooth oriented circle fiberings are the lens
spaces $\Lens e1$ for $e>0$, on each of which all such fiberings
form a single equivalence class:
\[
\Fib_{+}\left(\Lens e1,\Sphere 1\right)=\Fib\left(\StandardFibration_{e}\right),\qquad\forall e>0.
\]
More precisely, for each $e>0$, this unique class can be represented
by the (oriented) standard Hopf fibration $\StandardFibration_{e}\colon\Lens e1\to\Sphere 2$
(as descended from the standard Hopf map on $\Sphere 3$).
\end{lem}

\begin{proof}
Suppose that $E$ is a closed, elliptic 3-manifold (hence necessarily
orientable). Then the claim is equivalent to saying that the desired
classification $\ClassFib_{+}\left(E,\Sphere 1\right)$ (consisting
of all equivalence classes of smooth oriented circle fiberings on
$E$) is completely determined as
\[
\ClassFib_{+}\left(E,\Sphere 1\right)=\begin{cases}
\left\{ \StandardFibration_{e}\right\}  & \text{if \ensuremath{E\approx\Lens e1} for \ensuremath{e>0}},\\
\emptyset & \text{otherwise}
\end{cases}\qquad\text{(\ensuremath{E} closed, elliptic 3-manifold)}.
\]
To this end, first note that for any oriented circle bundle $E\to B$,
the base space $B$ must be a closed, orientable, elliptic 2-manifold,
hence $B\approx\Sphere 2$. Thus we can apply \prettyref{prop:Classifying-oriented-circle-fiberings}
with the constraint $B\approx\Sphere 2$, which yields the following
injection via the Euler class:
\[
\bigsqcup_{\text{\ensuremath{E} closed, elliptic 3-manifold}}\ClassFib_{+}\left(E,\Sphere 1\right)\quad\hookrightarrow\quad H^{2}\left(\Sphere 2;\mathbb{Z}\right)/\pi_{0}\Diff\left(\Sphere 2\right)\cong\mathbb{Z}/C_{2}.
\]
Thus there is a $\mathbb{Z}$'s worth of candidates for the desired
bundle classes $E\to\Sphere 2$; i.e., the ones classified by the
Euler class $e\in\mathbb{Z}$, each of which can be represented by
the canonical circle fibration on a certain Seifert manifold:
\[
\StandardFibration_{e}\colon\left(\Sphere 2,1/e\right)\to\Sphere 2\qquad\left(e\in\mathbb{Z}\right).
\]
Here, the Seifert manifold $\left(\Sphere 2,1/e\right)$ is obtained
by Dehn filling the solid torus that kills the slope $1/e$ on the
boundary torus. For such Seifert manifolds, their first integral homology
groups are isomorphic to $\mathbb{Z}/e\mathbb{Z}$, so that their
diffeomorphism types can also be distinguished by the Euler number
at least up to a sign:
\[
\left(\Sphere 2,1/e\right)\approx\left(\Sphere 2,1/e'\right)\iff\left|e\right|=\left|e'\right|.
\]
In fact, these diffeomorphism types can be represented by the familiar
lens spaces:
\[
\left(\Sphere 2,1/e\right)\approx\begin{cases}
\Sphere 1\times\Sphere 2 & e=0\\
\Lens{\left|e\right|}1 & e\neq0
\end{cases}.
\]
Among these 3-manifolds, $\Sphere 1\times\Sphere 2$ admits (and only
admits) $\mathbb{R}\times\Sphere 2$ geometry, while $\Lens{\left|e\right|}1$
for each $e\neq0$ admits (and only admits) elliptic geometry. Thus
under our constraint on the total space, a complete list of representatives
for bundle classification can be given by $\left\{ \StandardFibration_{e}\mid e\neq0\right\} $;
furthermore since $\pi_{0}\Diff\left(\Sphere 2\right)\cong C_{2}$
acts on $H^{2}\left(\Sphere 2;\mathbb{Z}\right)\cong\mathbb{Z}$ by
negation, this list is coarsened to $\left\{ \StandardFibration_{e}\mid e>0\right\} $
for a desired fibering classification. Therefore, we conclude that
under our assumption on $E$, the non-void cases of $\ClassFib_{+}\left(E,\Sphere 1\right)$
are exactly exhausted by $\ClassFib_{+}\bigl(\Lens e1,\Sphere 1\bigr)=\left\{ \StandardFibration_{e}\mid e>0\right\} $,
as desired.
\end{proof}

\subsection{The base transformation group}

Having obtained $\left\{ \StandardFibration_{e}\mid e>0\right\} $
as the complete set of representatives from the preceding result,
we next study their symmetries. Let us from now on focus on the case
$e=2$ (i.e., $E\approx\RP 3$, the real projective 3-space). We first
consider its basic transformation group $\Diff\left(B\right)_{\StandardFibration_{2}}$
and use \prettyref{cor:Base-transformations-of-an-oriented-circle-bundle}
to determine its homotopy type:
\begin{lem}
\label{lem:RP3-case-base-transformations}The basic transformation
group of $\StandardFibration_{2}\colon\RP 3\to\Sphere 2$ deformation
retracts onto its subgroup consisting of those that are isometric:
\begin{equation}
\Diff\left(\Sphere 2\right)_{\StandardFibration_{2}}\simeq\SOrthogonal 3.\label{eq:RP3-case-base-transformations}
\end{equation}
More precisely, $\SOrthogonal 3$ is embedded as the subgroup consisting
of rotations on the base 2-sphere.
\end{lem}

\begin{proof}
This is clear on the level of identity components, where the desired
deformation retraction from $\Diff_{0}\left(\Sphere 2\right)$ onto
$\SOrthogonal 3$ was already known by \cite{MR112149}. As for the
$\pi_{0}$-level, we need to show that those components making up
$\Diff\left(\Sphere 1\right)_{\StandardFibration_{2}}$ form the trivial
subgroup of the mapping class group $\pi_{0}\Diff\left(\Sphere 2\right)\cong C_{2}$.
To see this, we can apply \prettyref{cor:Base-transformations-of-an-oriented-circle-bundle}
to the (negation) action of $\pi_{0}\Diff\left(\Sphere 2\right)\cong C_{2}$
on $H^{2}\left(\Sphere 2;\mathbb{Z}\right)=\mathbb{Z}$, for which
the stabilizer $\pi_{0}\Diff\left(\Sphere 1\right)_{\StandardFibration_{2}}$
of $2\in\mathbb{Z}$ clearly is the trivial group, as desired.
\end{proof}

\subsection{The vertical automorphism group}

Having understood the basic transformations from the preceding result,
we next consider the vertical automorphism group $\Vau\left(\StandardFibration_{2}\right)$
and use \prettyref{cor:The-Vau-group-of-oriented-circle-bundle-II}
to determine its homotopy type:
\begin{lem}
\label{lem:RP3-case-vertical-automorphisms}The vertical automorphism
group of $\StandardFibration_{2}\colon\RP 3\to\Sphere 2$ deformation
retracts onto its subgroup consisting of those that are isometric:
\begin{equation}
\Vau\left(\StandardFibration_{2}\right)\simeq\CircleGroup.\label{eq:RP3-case-vertical-automorphisms}
\end{equation}
More precisely, $\CircleGroup$ is embedded as the subgroup consisting
of uniform translations on all fibers.
\end{lem}

\begin{proof}
Applying \prettyref{cor:The-Vau-group-of-oriented-circle-bundle-II}
with $B=\Sphere 2$, we obtain the desired deformation retraction
of $\Vau\left(\StandardFibration_{2}\right)$ onto $\CircleGroup\rtimes H^{1}\left(\Sphere 2;\mathbb{Z}\right)\cong\CircleGroup$.
\end{proof}

\subsection{The ambient space and its diffeomorphism group}

So far, only the fiber structure and the base structure were concerned;
to proceed, we bring in the ambient structure. For the ambient real
projective 3-space, the homotopy type of its diffeomorphism group
was known to Bamler and Kleiner in \cite{bamler2019ricci}:
\begin{thm}[\noun{Bamler–Kleiner}]
\label{thm:RP3-case-ambient-diffeomorphisms}The diffeomorphism group
of $\RP 3$ admits the following deformation retract:
\begin{equation}
\Diff\left(\RP 3\right)\simeq\Isometry\left(\RP 3\right)\cong\left(\SOrthogonal 3\times\SOrthogonal 3\right)\rtimes C_{2},\label{eq:RP3-case-ambient-diffeomorphisms}
\end{equation}
where $\Isometry\left(\RP 3\right)$ is the isometry group of the
standard round $\RP 3$.
\end{thm}

Here, the metric structure of the round $\RP 3=\Sphere 3/C_{2}$ is
induced from the standard round 3-sphere $\Sphere 3$, with respect
to which the isometries form the group $\Isometry\left(\RP 3\right)$.
In turn, this group can be described in terms of the group structure
on the cover $\Sphere 3$ (i.e., as the group of unit quaternions);
more specifically, each isometry on $\RP 3$ has a unique representation
by its lift to $\Sphere 3$ of the form
\[
p\mapsto q_{1}p^{\varepsilon}q_{2}^{-1}\qquad\text{with}\qquad q_{1},q_{2}\in\Sphere 3,\ \varepsilon=\pm1.
\]
This yields the desired isomorphism $\Isometry\left(\RP 3\right)\cong\left(\SOrthogonal 3\times\SOrthogonal 3\right)\rtimes C_{2}$,
with the semidirect product corresponding to the action of $C_{2}$
on $\SOrthogonal 3\times\SOrthogonal 3$ by interchanging factors.
Our remaining goal is to use this homotopy model of $\Diff\left(\RP 3\right)$
to induce the desired homotopy models of $\Aut\left(\StandardFibration_{2}\right)$
and of $\Fib\left(\StandardFibration_{2}\right)$; for this purpose,
recall that our model fibration $\StandardFibration_{2}\colon\RP 3\to\Sphere 2$
is descended from the Hopf map $\Sphere 3\to\Sphere 2$ given by $q\mapsto qiq^{-1}$.%

\subsection{The homotopy core of the automorphism group}

We first consider the group $\Aut\left(\StandardFibration_{2}\right)$
of automorphisms of $\StandardFibration_{2}$. By restricting (to
$\Aut\left(\StandardFibration_{2}\right)$) the preceding homotopy
model $\Isometry\left(\RP 3\right)$ of $\Diff\left(\RP 3\right)$,
we obtain the following candidate for a homotopy model of $\Aut\left(\StandardFibration_{2}\right)$:%
\begin{equation}
\Aut\left(\StandardFibration_{2}\right)\hookleftarrow\Aut\left(\StandardFibration_{2}\right)\cap\Isometry\left(\RP 3\right)\cong\SOrthogonal 3\times\CircleGroup.\label{eq:RP3-case-automorphisms-pre}
\end{equation}
Here, $\SOrthogonal 3\times\CircleGroup$ is the subgroup of $\Isometry_{0}\left(\RP 3\right)\cong\SOrthogonal 3\times\SOrthogonal 3$
consisting of those represented by $\left(q_{1},q_{2}\right)\in\Sphere 3\times\Sphere 3$
with $q_{2}\in\CircleGroup$, where $\CircleGroup\subseteq\Sphere 3$
is the subgroup of unit complex numbers. This leads us to the following
result, which asserts that the above embedding \prettyref{eq:RP3-case-automorphisms-pre}
is indeed a homotopy equivalence:
\begin{thm}
\label{thm:RP3-case-automorphisms}The automorphism group of $\StandardFibration_{2}\colon\RP 3\to\Sphere 2$
deformation retracts onto its subgroup consisting of those that are
isometric:
\begin{equation}
\Aut\left(\StandardFibration_{2}\right)\simeq\SOrthogonal 3\times\CircleGroup,\label{eq:RP3-case-automorphisms}
\end{equation}
where $\SOrthogonal 3\times\CircleGroup$ is embedded as a subgroup
via \prettyref{eq:RP3-case-automorphisms-pre}.
\end{thm}

\begin{proof}
Consider the surjective map
\[
\SOrthogonal 3\times\CircleGroup\twoheadrightarrow\SOrthogonal 3,\qquad\left(q_{1},q_{2}\right)\mapsto q_{1}.
\]
This projection yields a finite-dimensional fibration (resp., extension)
that is embedded into our infinite-dimensional fibration (resp., extension)
$\Vau\left(\StandardFibration_{2}\right)\hookrightarrow\Aut\left(\StandardFibration_{2}\right)\twoheadrightarrow\Diff\left(B\right)_{\StandardFibration_{2}}$,
as in the following commutative diagram:
\[
\xymatrix{\Vau\left(\StandardFibration_{2}\right)\ar@{^{(}->}[r] & \Aut\left(\StandardFibration_{2}\right)\ar@{->>}[r] & \Diff\left(B\right)_{\StandardFibration_{2}}\\
\CircleGroup\ar@{^{(}->}[r]\ar@{^{(}->}[u] & \SOrthogonal 3\times\CircleGroup\ar@{->>}[r]\ar@{^{(}->}[u] & \SOrthogonal 3\ar@{^{(}->}[u]
}
.
\]
Here, the right, left, and middle vertical arrows are given by the
embeddings \prettyref{eq:RP3-case-base-transformations}, \prettyref{eq:RP3-case-vertical-automorphisms},
and \prettyref{eq:RP3-case-automorphisms-pre}, respectively; in other
words, they are embedded as the ``isometric models'' of $\Diff\left(B\right)_{\StandardFibration_{2}}$,
$\Vau\left(\StandardFibration_{2}\right)$, and $\Aut\left(\StandardFibration_{2}\right)$,
from which we see that the diagram indeed commutes.

Turning to homotopy types, we use the (surjective) fibration structures
on both rows to induce two long exact sequences of homotopy groups.
Since by \prettyref{lem:RP3-case-base-transformations} and \prettyref{lem:RP3-case-vertical-automorphisms}
the right embedding of $\SOrthogonal 3$ into $\Diff\left(B\right)_{\StandardFibration_{2}}$
and the left embedding of $\CircleGroup$ into $\Vau\left(\StandardFibration_{2}\right)$
are both homotopy equivalences, we can deduce by the five lemma that
the middle embedding of ${\SOrthogonal 3\times\CircleGroup}$ into
$\Aut\left(\StandardFibration_{2}\right)$ is a homotopy equivalence
too, as desired. Lastly, the promotion to deformation retraction follows
from \prettyref{prop:Promoting-weak-homotopy-type}, which completes
the proof of \prettyref{thm:RP3-case-automorphisms}.
\end{proof}

\subsection{The homotopy core of the moduli space}

Finally, we turn to our main object of study: the moduli space $\Fib\left(\StandardFibration_{2}\right)$
of fiberings modeled on $\StandardFibration_{2}$. By projecting (to
$\Fib\left(\StandardFibration_{2}\right)$) the preceding homotopy
model $\Isometry\left(\RP 3\right)$ of $\Diff\left(\RP 3\right)$,
we obtain the following candidate for a homotopy model of $\Fib\left(\StandardFibration_{2}\right)$:
\begin{equation}
\Fib\left(\StandardFibration_{2}\right)\hookleftarrow\Isometry\left(\RP 3\right)\cdot\left\{ \StandardFibration_{2}\right\} \approx\Sphere 2\sqcup\Sphere 2.\label{eq:RP3-case-space-of-fiberings-pre}
\end{equation}
By construction, $\Sphere 2\sqcup\Sphere 2$ is embedded in $\Fib\left(\StandardFibration_{2}\right)$
as the subspace consisting of those fiberings that are metrically
equivalent (i.e., congruent) to $\StandardFibration_{0}$, which are
characterized by the ``direction'' and ``chirality'' ranging in
$\Sphere 2\sqcup\Sphere 2$. Therefore, proving that the above embedding
\prettyref{eq:RP3-case-space-of-fiberings-pre} is indeed a homotopy
equivalence will complete the final step of our proof of \prettyref{thm:elliptic-case-main-theorem}:
\begin{proof}[Proof of \prettyref{thm:elliptic-case-main-theorem}]
The claim regarding classification was proved in \prettyref{lem:elliptic-case-classification},
which in particular shows that $\Fib_{+}\left(\Lens e1,\Sphere 1\right)=\Fib\left(\StandardFibration_{e}\right)$
for each $e>0$. In turn, we need to determine the homotopy type of
$\Fib\left(\StandardFibration_{e}\right)$, for which we focus on
the case $e=2$. (The case $e=1$ is treated in a joint work \cite{deturck2025homotopytypespacefibrations}
with a more concrete, classical approach; while the remaining cases
$e\geq3$ are similar and simpler, with the moduli spaces being contractible
up to orientation.) Consider the surjective map
\[
\left(\SOrthogonal 3\times\SOrthogonal 3\right)\rtimes C_{2}\twoheadrightarrow\Sphere 2_{+}\sqcup\Sphere 2_{-},\qquad\left(\left(q_{1},q_{2}\right)\in\Sphere 3\times\Sphere 3,\pm1\right)\mapsto q_{2}iq_{2}^{-1}\in\Sphere 2_{\pm}.
\]
This projection yields a finite-dimensional fibration that is embedded
into our infinite-dimensional fibration $\Aut\left(\StandardFibration_{2}\right)\to\Diff\left(\RP 3\right)\to\Fib\left(\StandardFibration_{2}\right)$,
as in the following commutative diagram:
\[
\xymatrix{\Aut\left(\StandardFibration_{2}\right)\ar@{^{(}->}[r] & \Diff\left(\RP 3\right)\ar@{->>}[r] & \Fib\left(\StandardFibration_{2}\right)\\
\SOrthogonal 3\times\CircleGroup\ar@{^{(}->}[r]\ar@{^{(}->}[u] & \left(\SOrthogonal 3\times\SOrthogonal 3\right)\rtimes C_{2}\ar@{->>}[r]\ar@{^{(}->}[u] & \Sphere 2\sqcup\Sphere 2\ar@{^{(}->}[u]
}
.
\]
Here, the middle, left, and right vertical arrows are given by the
embeddings \prettyref{eq:RP3-case-ambient-diffeomorphisms}, \prettyref{eq:RP3-case-automorphisms},
and \prettyref{eq:RP3-case-space-of-fiberings-pre}, respectively;
in other words, they are embedded as the ``isometric models'' of
$\Diff\left(\RP 3\right)$, $\Aut\left(\StandardFibration_{2}\right)$,
and $\Fib\left(\StandardFibration_{2}\right)$, from which we see
that the diagram indeed commutes.

Turning to homotopy types, we use the (surjective) fibration structures
on both rows to induce two long exact sequences of homotopy groups.
Since by \prettyref{thm:RP3-case-ambient-diffeomorphisms} and \prettyref{thm:RP3-case-automorphisms}
the middle embedding of $\left(\SOrthogonal 3\times\SOrthogonal 3\right)\rtimes C_{2}$
into $\Diff\left(\RP 3\right)$ and the left embedding of $\SOrthogonal 3\times\CircleGroup$
into $\Aut\left(\StandardFibration_{2}\right)$ are both homotopy
equivalences, we can deduce by the five lemma that the right embedding
of $\Sphere 2\sqcup\Sphere 2$ into $\Fib\left(\StandardFibration_{2}\right)$
is a homotopy equivalence too, as desired. Lastly, the promotion to
deformation retraction follows from \prettyref{prop:Promoting-weak-homotopy-type},
which completes the proof of \prettyref{thm:elliptic-case-main-theorem}.
\end{proof}

\subsection{Completing the landscape in dimension three }

Let us summarize our proof of \prettyref{thm:elliptic-case-main-theorem}
and put it into perspective. For $\Lens 11=\Sphere 3$, the Smale
conjecture was proved—first partially at the $\pi_{0}$-level by \noun{Cerf}
\cite{MR229250}—by work of \noun{Hatcher} \cite{MR701256}:
\[
\Diff\left(\Sphere 3\right)\simeq\Isometry\left(\Sphere 3\right)\cong\SOrthogonal 4\rtimes C_{2}.
\]
Here, the isometry group has two components distinguished by the property
of orientation-preserving, with the identity component $\SOrthogonal 4$
being realized by quaternion multiplication on the 3-sphere as ${\Sphere 3\times\Sphere 3}/\left\langle \left(-1,-1\right)\right\rangle =\Sphere 3\tildetimes\Sphere 3$,
so that our main commutative diagram reads
\[
\xymatrix{\Aut\left(\StandardFibration_{1}\right)\ar@{^{(}->}[r] & \Diff\left(\Sphere 3\right)\ar@{->>}[r] & \Fib\left(\StandardFibration_{1}\right)\\
\Sphere 3\tildetimes\Sphere 1\ar@{^{(}->}[r]\ar@{^{(}->}[u] & \left(\Sphere 3\tildetimes\Sphere 3\right)\rtimes C_{2}\ar@{->>}[r]\ar@{^{(}->}[u] & \Sphere 2\sqcup\Sphere 2\ar@{^{(}->}[u]
}
.
\]
For $\Lens 21=\RP 3$, the (generalized) Smale conjecture was first
proved by work of \noun{Bamler} and \noun{Kleiner} (\cite{MR4536904,bamler2019ricci})
using Ricci flow, and later in a different approach by work of \noun{Ketover}
and \noun{Liokumovich} (\cite{MR4921571}) using minimal surfaces
and min-max theory:
\[
\Diff\left(\RP 3\right)\simeq\Isometry\left(\RP 3\right)\cong\mathrm{P}\SOrthogonal 4\rtimes C_{2}
\]
Here, the isometry group has two components distinguished by the property
of orientation-preserving, with the identity component $\mathrm{PSO}\left(4\right)$
being realized by quaternion multiplication on the 3-sphere as $\Sphere 3/_{\left\langle -1\right\rangle }\times\Sphere 3/_{\left\langle -1\right\rangle }=\SOrthogonal 3\times\SOrthogonal 3$,
so that our main commutative diagram reads
\[
\xymatrix{\Aut\left(\StandardFibration_{2}\right)\ar@{^{(}->}[r] & \Diff\left(\RP 3\right)\ar@{->>}[r] & \Fib\left(\StandardFibration_{2}\right)\\
\SOrthogonal 3\times\Sphere 1\ar@{^{(}->}[r]\ar@{^{(}->}[u] & \left(\SOrthogonal 3\times\SOrthogonal 3\right)\rtimes C_{2}\ar@{->>}[r]\ar@{^{(}->}[u] & \Sphere 2\sqcup\Sphere 2\ar@{^{(}->}[u]
}
.
\]
By contrast, for $\Lens e1$ with $e>2$, the moduli spaces of fiberings
are contractible up to the choice of orientation: For $\Lens e1$
with $e=2n+1>2$ odd, the (generalized) Smale conjecture is proved
by work of \noun{Hong}, \noun{Kalliongis}, \noun{McCullough}, and
\noun{Rubinstein} (\cite{MR2976322}): 
\[
\Diff\left(L\left(2n+1,1\right)\right)\simeq\Isometry\left(L\left(2n+1,1\right)\right)\cong\Sphere 3\tildetimes\left(\Sphere 1\cup j\Sphere 1\right)\qquad\left(n>0\right).
\]
Here, all isometries are orientation preserving, and the isometry
group $\Sphere 3\tildetimes\left(\Sphere 1\cup j\Sphere 1\right)$
is realized by quaternion multiplication on the 3-sphere as ${{\Sphere 3\times\left(\Sphere 1\cup j\Sphere 1\right)}/\left\langle \left(-1,-1\right)\right\rangle }\cong{\Sphere 3\tildetimes\left(\Sphere 1\cup j\Sphere 1\right)}$,
so that our main commutative diagram reads
\[
\xymatrix{\Aut\left(\StandardFibration_{2n+1}\right)\ar@{^{(}->}[r] & \Diff\left(L\left(2n+1,1\right)\right)\ar@{->>}[r] & \Fib\left(\StandardFibration_{2n+1}\right)\\
\Sphere 3\tildetimes\Sphere 1\ar@{^{(}->}[r]\ar@{^{(}->}[u] & \Sphere 3\tildetimes\left(\Sphere 1\cup j\Sphere 1\right)\ar@{->>}[r]\ar@{^{(}->}[u] & \left\{ \StandardFibration_{2n+1},\StandardFibration_{-2n-1}\right\} \ar@{^{(}->}[u]
}
.
\]
For $\Lens e1$ with $e=2n+2>2$ odd, the (generalized) Smale conjecture
is also proved by work of \noun{Hong}, \noun{Kalliongis}, \noun{McCullough},
and \noun{Rubinstein} (\cite{MR2976322}): 
\[
\Diff\left(L\left(\StandardFibration_{2n+2},1\right)\right)\simeq\Isometry\left(L\left(\StandardFibration_{2n+2},1\right)\right)\cong\SOrthogonal 3\times\Orthogonal 2\qquad\left(n>0\right).
\]
Here, all isometries are orientation preserving, and the isometry
group $\SOrthogonal 3\times\Orthogonal 2$ is realized by quaternion
multiplication on the 3-sphere as ${\Sphere 3/_{\left\langle -1\right\rangle }\times\left(\Sphere 1\cup j\Sphere 1\right)/_{\left\langle -1\right\rangle }}\cong{\SOrthogonal 3\times\Orthogonal 2}$,
so that our main commutative diagram reads
\[
\xymatrix{\Aut\left(\StandardFibration_{2n+2}\right)\ar@{^{(}->}[r] & \Diff\left(L\left(2n+2,1\right)\right)\ar@{->>}[r] & \Fib\left(\StandardFibration_{2n+2}\right)\\
\SOrthogonal 3\times\Sphere 1\ar@{^{(}->}[r]\ar@{^{(}->}[u] & \SOrthogonal 3\times\Orthogonal 2\ar@{->>}[r]\ar@{^{(}->}[u] & \left\{ \StandardFibration_{2n+2},\StandardFibration_{-2n-2}\right\} \ar@{^{(}->}[u]
}
.
\]
As shown in \prettyref{lem:elliptic-case-classification} (and the
proof thereof), these are the only elliptic cases and, more generally,
the only 3-dimensional total spaces that admit smooth oriented circle
fibrations over a spherical base are of the form $\left(\Sphere 2,\left(1,e\right)\right)$
for $e\geq0$ over the 2-sphere, as obtained by Dehn filling the solid
torus that kills the slope $1/e$ on the boundary torus. The cases
$e>0$ correspond exactly to the lens spaces $\Lens e1$ as already
treated above, and the remaining case $e=0$ corresponds to the trivial
fibration $\StandardFibration_{0}$ of the direct product $\Sphere 2\times\Sphere 1$.
In this case, the homotopy type of its total diffeomorphism group
was known by work of \noun{Hatcher} \cite{MR624946}:
\[
\Diff\left(\Sphere 2\times\Sphere 1\right)\simeq\Omega\SOrthogonal 3\times\Orthogonal 3\times\Orthogonal 2,
\]
where $\Omega\SOrthogonal 3$ corresponds to the ``rota-translational''
diffeomorphisms (which are not accounted for by isometries); our main
commutative diagram reads
\[
\xymatrix{\Aut\left(\StandardFibration_{0}\right)\ar@{^{(}->}[r] & \Diff\left(\Sphere 2\times\Sphere 1\right)\ar@{->>}[r] & \Fib\left(\StandardFibration_{0}\right)\\
\Orthogonal 3\times\Sphere 1\ar@{^{(}->}[r]\ar@{^{(}->}[u] & \Omega\SOrthogonal 3\times\Orthogonal 3\times\Orthogonal 2\ar@{->>}[r]\ar@{^{(}->}[u] & \widetilde{\Omega\SOrthogonal 3}\ar@{^{(}->}[u]
}
.
\]
In words, this shows that the moduli space of oriented circle fiberings
on $\Sphere 2\times\Sphere 1$ deformation retracts to a double cover
$\widetilde{\Omega\SOrthogonal 3}$, the oriented based loop space
of $\SOrthogonal 3$. This completes the (3-dimensional, oriented)
case over a spherical base ($B=\Sphere 2$), with the $e\leq1$ cases
first treated in \cite{deturck2025homotopytypespacefibrations,wang2024homotopytypespacefiberings}.
Similarly, one can proceed to the other cases such as the one over
a flat base ($B=\Torus 2$), with the eligible total spaces classified
by the analogous construction $\left(\Torus 2,\left(1,e\right)\right)$
for $e\geq0$ over the 2-torus, as obtained by Dehn filling a trivial
circle bundle over the punctured torus that kills the slope $1/e$
on the boundary torus. The trivial case $e=0$ has been treated above
in \prettyref{thm:T3-case-main-theorem} already:
\[
\Fib_{+}\left(\Torus 3,\Sphere 1\right)=\Fib\left(\StandardFibration_{0}\colon\Torus 3\to\Torus 2\right)\simeq\ZPrim 3.
\]
On the other hand, the remaining cases $e>0$ correspond to the mapping
tori of $\Torus 2$ with reducible monodromy $\left(\begin{smallmatrix}1 & e\\
0 & 1
\end{smallmatrix}\right)\in\pi_{0}\Diff\left(\Torus 2\right)$, which is denoted by $\MappingTorus{\Torus 2}e$ and admits the natural
structure of an oriented circle fibration $\StandardFibration_{e}$
over the base torus, unique up to fibering equivalence. In this case,
the total space is of $\Nil$ geometry, for which the homotopy type
of diffeomorphism group was known by work of \noun{Bamler} and \noun{Kleiner}
\cite{MR4685084}, so that we can apply our framework as before to
deduce that the corresponding moduli space of fiberings deformation
retracts onto a discrete space of the pair $\StandardFibration_{e},\StandardFibration_{-e}$:
\[
\Fib_{+}\left(\MappingTorus{\Torus 2}e,\Sphere 1\right)=\Fib\left(\StandardFibration_{e}\colon\MappingTorus{\Torus 2}e\to\Torus 2\right)\simeq\Sphere 0\qquad\left(e>0\right).
\]
Lastly, one can approach similarly the remaining cases over a hyperbolic
base ($B=S_{g}$ for $g\geq2$), and show that the corresponding moduli
space of fiberings $\Fib_{+}\left(\left(S_{g},\left(1,e\right)\right),\Sphere 1\right)$
deformation retracts to $\Sphere 0$ for all $g\geq2$ and $e\geq0$.
But it is worth noticing that the preceding cases were already known
in others authors' work—most notably the result deduced from the work
of \noun{Hong}, \noun{Kalliongis}, \noun{McCullough}, and \noun{Rubinstein}
\cite{MR2976322} that $\Fib_{+}\left(E^{3},\Sphere 1\right)$ has
contractible components if the 3-dimensional total space $E^{3}$
is Haken. In conclusion, this completes a solution to the problem
of finding ``homotopy cores'' (a minimal deformation retract that
encodes the topological structure of such a moduli space of fibrations)
for all (closed, orientable) $3$-manifold $E$.

\newpage\pagestyle{plain}\bibliographystyle{bibtotoc}
\phantomsection\addcontentsline{toc}{chapter}{\bibname}
\makeatletter
\@ifpackageloaded{latexml}{\newcommand{\etalchar}[1]{$^{#1}$}
\providecommand{\bysame}{\leavevmode\hbox to3em{\hrulefill}\thinspace}
\providecommand{\MR}{\relax\ifhmode\unskip\space\fi MR }
\providecommand{\MRhref}[2]{%
  \href{http://www.ams.org/mathscinet-getitem?mr=#1}{#2}
}
\providecommand{\href}[2]{#2}

}{

}
\makeatother

~

{\footnotesize\noun{Department of Mathematics \& MIT Open Learning, Massachusetts Institute
of Technology}}\\
{\footnotesize\emph{Email address}:\enskip{}\code{ziqifang@mit.edu}}{\footnotesize\par}
\end{document}